\documentclass[10pt, reqno]{amsart}
\usepackage[myheadings]{fullpage}

\usepackage{mathtools}
\usepackage{amsmath, amsthm, amssymb, mathabx, verbatim, setspace, enumerate, tensor}
\usepackage{relsize}
\usepackage[mathscr]{euscript}
\usepackage[all,cmtip]{xy}
\usepackage{relsize}
\usepackage[toc,page]{appendix}
\usepackage{stmaryrd}
\usepackage{tikz-cd}

\usepackage{graphicx}

\usepackage{stackengine}
\usepackage{scalerel}

\xyoption{rotate}
\xyoption{pdf}

\usepackage[ampersand]{easylist}

\usepackage{tikz-cd}

\usepackage{url}
\usepackage{hyperref}

\newtheorem{introthm}{Theorem}

\newtheorem{thm}[subsubsection]{Theorem} 
\newtheorem{prop}[subsubsection]{Proposition}
\newtheorem{lemma}[subsubsection]{Lemma}
\newtheorem{cor}[subsubsection]{Corollary}

\theoremstyle{definition}

\theoremstyle{remark}
\newtheorem{rmk}[subsubsection]{Remark}
\newtheorem{exmp}[subsubsection]{Example}
\newtheorem{defn}[subsubsection]{Definition}

\newcommand{\fp}{\mathfrak{p}}

\DeclareMathOperator{\Hom}{Hom}
\newcommand{\colim}{\operatornamewithlimits{colim}}

\DeclareMathOperator{\Pro}{Pro}
\DeclareMathOperator{\Ind}{Ind}

\DeclareMathOperator{\Ext}{Ext}

\DeclareMathOperator{\Coh}{Coh}

\DeclareMathOperator{\dmod}{-mod}

\DeclareMathOperator{\dcomod}{-comod}

\DeclareMathOperator{\Ch}{Ch}

\DeclareMathOperator{\Perv}{Perv}
\DeclareMathOperator{\Spec}{Spec}

\DeclareMathOperator{\Sym}{Sym}
\DeclareMathOperator{\shHom}{\mathscr{H}\mathnormal{om}}

\DeclareMathOperator{\Mod}{Mod}
\renewcommand{\mod}{\operatorname{-mod}}
\newcommand{\APerf}{{\mathrm{APerf}}}
\DeclareMathOperator{\pt}{pt}

\DeclareMathOperator{\coker}{coker}

\DeclareMathOperator{\oppp}{op}

\DeclareMathOperator{\End}{End}

\DeclareMathOperator{\QCoh}{QC}

\DeclareMathOperator{\Alg}{Alg}
\DeclareMathOperator{\Perf}{Perf}
\DeclareMathOperator{\Frac}{Frac}
\DeclareMathOperator{\IndCoh}{QC^{!}}

\DeclareMathOperator{\Tot}{Tot}

\DeclareMathOperator{\Map}{Map}
\DeclareMathOperator{\Fun}{Fun}

\newcommand{\bZ}{\mathbb{Z}}

\newcommand{\bA}{\mathbb{A}}
\newcommand{\bG}{\mathbb{G}}
\newcommand{\bD}{\mathbb{D}}

\newcommand{\bL}{\mathbb{L}}
\newcommand{\bF}{\mathbb{F}}
\newcommand{\bT}{\mathbb{T}}
\newcommand{\cO}{\mathscr{O}}

\newcommand{\cK}{\mathscr{K}}
\newcommand{\cE}{\mathscr{E}}
\newcommand{\cN}{\mathscr{N}}
\newcommand{\cG}{\mathscr{G}}

\newcommand{\cF}{\mathscr{F}}

\newcommand{\cV}{\mathscr{V}}
\newcommand{\cZ}{\mathscr{Z}}
\newcommand{\BL}{B_{\chL}}
\newcommand{\cS}{\mathscr{S}}

\newcommand{\cW}{\mathscr{W}}

\newcommand{\Lag}{\on{Lag}}

\newcommand{\on}{\operatorname}

\newcommand{\hCoh}{\widehat{\Coh}}

\newcommand{\Av}{\on{Av}}
\newcommand{\Sing}{\bT^{*\,[\mh 1]}}

\newcommand{\opp}{{\oppp}}

\newcommand{\ra}{\rightarrow}

\newcommand{\mf}{\mathfrak}
\newcommand{\bs}{\backslash}
\newcommand{\cat}{\mathbf}
\newcommand{\wh}{\widehat}
\newcommand{\wt}{\widetilde}

\newcommand{\fg}{\mathfrak{g}}

\newcommand{\Vect}{\mathrm{Vect}}
\newcommand{\mon}{\on{-mon}}
\newcommand{\id}{\mathrm{id}}

\newcommand{\tens}[1]{%
  \mathbin{\mathop{\otimes}\limits_{#1}}%
}

\newcommand{\utimes}[1]{%
  \mathbin{\mathop{\times}\limits_{#1}}%
}

\newcommand{\rihgtarrow}{\rightarrow}

\newcommand{\St}{\on{St}}

\mathchardef\md="2D
\mathchardef\mh="2D

\newcommand{\sss}{\subsubsection{}}

\newcommand{\IC}{\IndCoh}

\newcommand{\mr}{\mathring}

\newcommand{\ft}{\check{\mathfrak{t}}}
\newcommand{\chT}{\check{T}}
\newcommand{\sC}{\mathbf{C}}
\newcommand{\sD}{\mathbf{D}}
\newcommand{\chP}{\check{P}}
\newcommand{\chG}{\check{G}}
\newcommand{\chB}{\check{B}}
\newcommand{\sM}{\mathscr{M}}
\newcommand{\chL}{\check{L}}
\newcommand{\chN}{\check{N}}
\newcommand{\chQ}{\check{Q}}
\newcommand{\chR}{\check{R}}
\newcommand{\DGCat}{\on{DGCat}_{cont}}
\newcommand{\fchh}{\check{\mathfrak{h}}}
\newcommand{\hj}{\hspace{.07cm} \widehat{\hspace{-.07cm} j}}
\newcommand{\chI}{I}

\newcommand{\chH}{\check{H}}
\newcommand{\halpha}{\check{\alpha}}
\newcommand{\fsl}{\mathfrak{sl}}
\newcommand{\Oblv}{\operatorname{Oblv}}
\newcommand{\PN}{{{}_{\chP}\check{N}}}
\newcommand{\RN}{{{}_{\chR}\check{N}}}
\newcommand{\QN}{{{}_{\chQ}\check{N}}}
\newcommand{\Dmod}{\operatorname{D-mod}}

\renewcommand{\varinjlim}{\colim}
\renewcommand{\varprojlim}{\lim}

\newcommand{\mmod}{\on{-mod}}
\newcommand{\na}[1]{\wt{\cN}_{#1}}
\newcommand{\nb}[1]{\wt{\cN}_{#1}'}

\newcommand{\nc}[1]{\wt{\mf{g}}_{#1}}

\newcommand{\ga}[1]{\wt{\mf{g}}_{#1}}

\newcommand{\QC}{\QCoh}

\newcommand{\Lamda}{\Lambda}

\newcommand{\fH}{\mathbf{H}}

\newcommand{\hCohst}{\hCoh{\vphantom{\Coh}}^*}
\newcommand{\hCohsh}{\hCoh{\vphantom{\Coh}}^!}
\begin{document}

\title{A Langlands dual realization of coherent sheaves on the nilpotent cone}
\date{\today}
\author{Harrison Chen and Gurbir Dhillon}
\begin{abstract}
Let $G$ and $\chG$ be Langlands dual connected reductive groups. We establish a monoidal equivalence of $\infty$-categories between equivariant quasicoherent sheaves on the formal neighborhood of the nilpotent cone in $G$ and Steinberg--Whittaker D-modules on the loop group of $\chG$, as conjectured by Bezrukavnikov.  More generally, we establish equivalences between various spectral and automorphic realizations of affine Hecke categories and their modules, confirming conjectures of Bezrukavnikov.
\end{abstract}

\maketitle

\setcounter{tocdepth}{2}
\tableofcontents
\begingroup

\section{Introduction}

\subsubsection{} 
In the influential work \cite{roma hecke}, Bezrukavnikov established an equivalence between spectral and automorphic versions of the affine Hecke category. This result fits into a wider class of conjectures, the local geometric Langlands correspondence, and establishes them for tamely ramified local systems with unipotent monodromy.

The present work has two main goals. First, we record proofs of several conjectures introduced by Bezrukavnikov and his collaborators in their work, including an automorphic realization of coherent sheaves on the nilpotent cone.

Second, we provide some analysis of issues of homological algebra around Bezrukavnikov's equivalence, including various $\infty$-categorically enhanced versions of his result which are required elsewhere in the local Langlands program and the representation theory of affine Lie algebras. 

We now describe both goals in slightly more detail; the reader may safely skip this motivational material and  find a precise statement of our main results in Theorem \ref{t:thma} below.

\subsubsection{} 
We begin with a discussion of our first goal.  
In {\em loc. cit.} Bezrukavnikov conjectured several further compatibilities of the equivalence. These match automorphic and spectral realizations of certain standard modules for the affine Hecke category, as well as categories of intertwining operators between such modules. The first goal of the present work is to confirm those conjectures, one of which was previously obtained by Bezrukavnikov--Losev \cite{BL}. 

Relatedly, we prove a further compatibility, namely we provide a monoidal equivalence between the category of coherent sheaves on the formal neighborhood of the nilpotent cone of a connected reductive group $G$ with the category of bi-Steinberg--Whittaker sheaves on the loop group of its Langlands dual $\chG$, confirming a conjecture of Bezrukavnikov (unpublished).  For further discussion of this result, including precise definitions and an expected application to giving an alternative to the argument of \cite{roma hecke}, we refer the reader to Section \ref{biwhittaker sec} below.

\begin{rmk} 
We note in passing some arithmetic consequences of the equivalences that we prove. To state a representative example, fix a parahoric subgroup of the loop group $\chG_F$ lying in the arc subgroup $\chG_O$. Then, by applying categorical trace of Frobenius, the corresponding Hecke algebra may be identified with the Hochschild homology of a parabolic version of the Steinberg stack, where one replaces the full Springer resolution by a corresponding partial Springer resolution. The resulting Langlands parametrization of irreducible representations of a $p$-adic group with nontrivial fixed vectors for the parahoric is a celebrated result of Kazhdan--Lusztig in the case of an Iwahori subgroup, and for $\chG_O$ is the parametrization of unramified representations due to Satake. We are unsure if  the remaining cases are in the literature. 
\end{rmk}



 \subsubsection{} 
Let us now discuss our second goal, namely to establish various technical and structural results about Bezrukavnikov's equivalence itself.  In particular, in line with the expected formulation of local geometric Langlands as an equivalence of certain $(\infty, 2)$-categories, we produce cocompleted $\infty$-categorical enhancements of the appearing categories and equivalences. Moreover, we in fact provide two such enhancements for the equivalences, corresponding to two natural full subcategories of small objects on either side. 

In performing this analysis, we encounter on the automorphic side a natural derived local system on any algebraic variety, its {\em cofree monodromic sheaf}, and an associated category of \emph{nearly compact} objects in the monodromic Hecke category. On the spectral side, we dually study the category of {\em continuous ind-coherent sheaves} on a formal scheme. These notions, while closely related to previous work, may not have appeared elsewhere in the literature.

\subsection{Statement of results}\label{s:intromainresults}

\subsubsection{} To state our main result, we need to first introduce notation, which is admittedly somewhat involved, albeit mostly standard. The reader may therefore wish to briefly skim the contents of Sections \ref{sss:monday} and \ref{sss:friday} before reading the statement of Theorem \ref{t:thma}, and referring back only as needed.

\subsubsection{}\label{sss:monday} Let $\chG$ denote a pinned connected reductive algebraic group over an algebraically closed field $k$ of characteristic zero, with corresponding Borel subgroup $\chB$. 
Let $G$ denote the corresponding Langlands dual group, with corresponding Borel subgroup $B$. We denote the field of Laurent series and its ring of integers by $$F = k(\!(t)\!) \quad \quad O = k[\![t]\!],$$respectively.  We consider the loop group $\chG_F$ as an ind-scheme over $k$ on the automorphic side.  In this setting, by sheaves we mean the category of D-modules, or equivalently since all our stacks are ind-finite-orbit the category of constructible sheaves.

\subsubsection{}\label{main thm sec} On the automorphic side, we will encounter various categories of D-modules on $\chG_F$, where we vary the equivariance conditions on the left and right. As a primary example, we write $I \subset \chG_F$ for the Iwahori subgroup corresponding to the Borel $\chB$, and we might impose $I$-equivariance for sheaves on  $\chG_F$ with respect to left translations or right translations.

Associated to each equivariance condition on the automorphic side we consider will be a formal stack equipped with a map to the formal neighborhood of the nilpotent cone $$\fg^\wedge_\cN / G \hookrightarrow \fg/G$$ on the spectral side. Continuing the above example, the stack corresponding to $I$-equivariance is the (stacky) Springer resolution 
$$\wt{\cN}/G := (G \times^B \mf{n})/G \simeq \mf{n}/B \rightarrow \fg/G.$$

The basic assertion, as conjectured by Bezrukavnikov, is that the automorphic category one obtains by choosing two equivariance conditions for left and right translations is equivalent to the spectral category one obtains by taking coherent sheaves on the corresponding fiber product over $\mf{g}^\wedge_{\cN}/G$. 
For example, if we impose $I$-equivariance on both sides, we arrive at the statement
$$\Dmod(I \bs \chG_F / I) \simeq \QC^!_{\cN[1]}(\wt{\cN}/G  \underset{\fg/G} \times \wt{\cN}/G),$$
i.e., Bezrukavnikov's original equivalence of affine Hecke categories (see Section \ref{def sheaves} and Section \ref{app ss} for a discussion of the spectral categories of sheaves and singular support conditions).

\subsubsection{}\label{sss:morning} As we just stated, $I$-equivariance is exchanged with the Springer resolution. Let us denote this relationship by
\begin{equation} \label{e:EYE} \chG_F/I \overset{\mathbb{L}}{\longleftrightarrow} \wt{\cN}/G.\end{equation}

A technical, but important, variant is the following. One can relax the condition of $I$-equivariance to $I$-monodromicity, i.e., one takes the full subcategory of sheaves generated by strictly equivariant objects (see Section \ref{s:defmon} for discussion). This allows sheaves to have unipotent, rather than only trivial, monodromy along Iwahori orbits. 

On the spectral side, the corresponding thickening is obtained by passing to the full subcategory of coherent sheaves on the Grothendieck--Springer resolution $\wt{\fg}/G$ generated by pushforwards from $\wt{\cN}/G$.  Equivalently, these are coherent sheaves on $\wt{\mf{g}}/G$ set-theoretically supported on $\wt{\cN}/G$, or coherent sheaves on the formal completion $\wt{\mf{g}}^\wedge_{\wt{\cN}}$.  That is, we have 
\begin{equation} \label{e:EYE2}\chG_F/I\mon \overset{\mathbb{L}}{\longleftrightarrow} \wt{\fg}^\wedge_{\wt{\cN}}/G.\end{equation}

The equivalences arising from \eqref{e:EYE} and \eqref{e:EYE2} are the main results of \cite{roma hecke}. The new equivalences we are after arise from two further classes of equivariance conditions. Namely, as we explain next, we will either pass from \eqref{e:EYE} to equivariance for larger parahorics than the Iwahori, or turn on a nontrivial character sheaf on the pro-unipotent radical of $I$, i.e., impose a possibly degenerate Whittaker condition.

\subsubsection{} We begin by discussing the parahoric case.  Fix a standard parabolic subgroup $\chP \supset \chB$ of $\chG$. On the one hand, associated to $\chP$ is a standard parahoric subgroup  $I_{\chP} \subset \chG_F$, namely the preimage of $\chP$ under the evaluation map $\chG_O \rightarrow \chG.$   On the automorphic side, we may therefore consider strict equivariance for the parahoric $I_{\chP}$.

On the spectral side, associated to $\chP$ is a standard parabolic subgroup $P \supset B$ of $G$. Associated to $P$ is a parabolic variant of the Springer resolution. Namely, if we write $\mf{n}_P \subset \mf{p}$ for the nilradical of the Lie algebra of $P$, we have the partial Springer resolution $$\wt{\cN}_P/G := (G \times^P \mf{n}_P)/G \simeq \mf{n}_P/P \ra \mf{g}^\wedge_{\cN}/G.$$
In \cite{roma hecke}, Bezrukavnikov proposed the relationship
$$\chG/I_{\chP} \overset{\bL}{\longleftrightarrow} \wt{\cN}_P/G.$$
If one pairs with the conditions \eqref{e:EYE2}, this in particular implies an equivalence
$$ \Dmod(I\mon \bs \chG_F / I_{\chP}) \simeq \QC^!_{\cN[1]}(\wt{\fg}^\wedge_{\wt{\cN}}/G \underset{\fg/G} \times \wt{\cN}_P/G).$$
This important case was established previously by Bezrukavnikov--Losev \cite{BL}. 
\subsubsection{} As in the case of Iwahori, Bezruvanikov proposed a monodromic variant of the above parahoric correspondence.  Namely, letting $I'_{\chP} \subset I_{\chP}$ denote the derived subgroup of $I_{\chP}$, we may consider on the automorphic side $I'_{\chP}$-equivariant sheaves which are monodromic for the quotient torus $$I_{\chP}/I'_{\chP} \simeq \chP/[\chP, \chP].$$We denote this equivariance condition by $(I_{\chP}, I'_{\chP})\mon$ below.

On the spectral side, one forms the corresponding thickening of the partial Springer resolution as follows. If we denote the center of the Lie algebra of $P$ by $\mf{z}(\mf{p})$, consider the stack 
$$\wt{\cN}_P'/G := (G \times^P (\mf{n}_P + \mf{z}(\mf{p})))/G \simeq (\mf{n}_P + \mf{z}(\mf{p}))/P \ra \mf{g}/G.$$
The formal stack of interest is again its formal completion along nilpotent elements
$$(\wt{\cN}_P')^{\wedge}_{\cN}/G \ra \mf{g}^\wedge_{\cN}/G.$$
The desired monodromic variant of the parahoric correspondence, as proposed by Bezrukavnikov, reads as follows
$$\chG_F/(I_{\chP}, I'_{\chP})\mon \overset{\bL}{\longleftrightarrow} (\wt{\cN}_P')^\wedge_{\cN}/G.$$

\subsubsection{} The remaining equivariance condition to discuss is the case of Whittaker invariants. 
Let us begin with the automorphic side. Fix a parabolic $\chP \subset \chG$. If we write $\chL$ for the Levi quotient of $\chP$, the pinning for $\chG$ affords us a a distinguished maximal unipotent subgroup $N^-_{\chL}$ in general position to the image of $\chB$, along with a nondegenerate character sheaf $\psi$ of $N^-_{\chL}$. Let us write $_{\chP}\mathring{I}$ for the preimage of $N^-_{\chL}$ under the quotient map 
$$I_{\chP} \rightarrow \chL,$$
and let us denote its induced character sheaf again by $\psi$. With this, the equivariance condition of interest is almost
$$\chG_F/ _{\chP}\mathring{I}, \psi.$$
However, this is too large, as, somewhat informally speaking, it receives contributions from all the unramified principal series, and not only the trivial representation of the torus. Therefore, we pass to its full subcategory of {\em Steinberg--Whittaker} sheaves, 
$$\chG_F/ _{\chP}\mathring{I}, \psi, \St \hookrightarrow \chG_F/ _{\chP}\mathring{I}, \psi,$$
as studied by Campbell and the first named author in \cite{campbelldhillon}, as well as by Arinkin--Bezrukavnikov (unpublished); see Definition \ref{define steinberg whit} for a precise definition.

On the spectral side, we may consider a parabolic version of the Grothendieck--Springer resolution
$$\wt{\mf{g}}_P/G  := (G \times^P \mf{p})/G \simeq \mf{p}/P \ra \mf{g}/G,$$
and again the formal stack of interest is obtained by  completing along the nilpotent elements
$$(\wt{\mf{g}}_P)^\wedge_{\cN}/G \ra \mf{g}^\wedge_{\cN}/G.$$
This gives rise to a correpsondence
$$\chG/{}_{\chP}\mr{I}, \psi, \St \overset{\bL}{\longleftrightarrow} (\wt{\mf{g}}_P)^\wedge_{\cN}/G.$$

\label{sss:evening}

\subsubsection{} Having listed the various equivariance conditions and the corresponding formal stacks, let us briefly comment on the relationship to local geometric Langlands. This is supposed to be, following Arinkin--Gaitsgory, an equivalence of 2-categories between the 2-category of categories acted on by $\Dmod(\chG_F)$ and ind-coherent sheaves of categories with nilpotent singular support on the moduli space of $G$-connections on the formal punctured disk $\mathscr{D}^\times$
\begin{equation} \label{e:loclang}\mathbb{L}: \Dmod(\chG_F)\mod \simeq \on{2-QC^!}_{\cN}(\on{LocSys}_{G}(\mathscr{D}^\times)).\end{equation}
For a torus $\chT$, the equivalence of 2-categories \eqref{e:loclang} follows from the monoidal equivalence 
$$\Dmod(\chT_F) \simeq \QCoh(\on{LocSys}_{T}(\mathscr{D}^\times)),$$
which exchanges, writing $\mathscr{D}$ for the formal non-punctured disk, the modules 
$$\mathbb{L} (\Dmod(\chT_F/\chT_O)) \simeq \QCoh(\on{LocSys}_T(\mathscr{D})).$$
By parabolically inducing to $\chG$ and $G$, one arrives at the prediction 
$$\mathbb{L}(\Dmod(\chG_F / \chN_F \chT_O)) \simeq \QCoh( \on{LocSys}_T(\mathscr{D}) \underset{\on{LocSys}_T(\mathscr{D}^\times)} \times \on{LocSys}_B(\mathscr{D}^\times) ) \simeq \QCoh(\wt{\cN}/ G).$$
As one has a canonical identification, due to Raskin,
$$\Dmod(\chG_F / \chN_F \chT_O) \simeq \Dmod(\chG_F / I),$$
by taking endomorphisms one arrives at the prediction of a monoidal equivalence
\begin{equation} \label{e:matchends}\mathbf{End}_{\Dmod(\chG_F)\mod}(\Dmod(\chG_F / I)) \simeq \mathbf{End}_{\on{2-QC^!}_{\cN}(\on{LocSys}_G(\mathscr{D}^\times))}( \on{QCoh}(\wt{\cN}/G)).\end{equation}
The left-hand side is given by the affine Hecke category, i.e., $\Dmod(I \bs \chG_F / I)$. To compute the right-hand side, note that one has a natural map $$\fg/G \rightarrow \on{LocSys}_G(\mathscr{D}^\times), \quad \quad X \mapsto d + X \frac{dt}{t}.$$
After further restricting to the formal neighborhood of the nilpotent cone, the resulting map 
\begin{equation} \label{e:tameunip} \fg^\wedge_{\cN}/G \rightarrow \on{LocSys}_{G}(\mathscr{D}^\times)\end{equation}
is radicial and formally \'etale, so that one has an inclusion 
$$\on{2-QC^!}_{\cN}(\fg^\wedge_{\cN}/G) \hookrightarrow \on{2-QC^!}_{\cN}(\on{LocSys}_{G}(\mathscr{D}^\times)).$$
Using this, one may therefore compute $$\mathbf{End}_{\on{2-QC^!}_{\cN}(\on{LocSys}_G(\mathscr{D}^\times))}( \on{QCoh}(\wt{\cN}/G)) \simeq  \QC^!_{\cN[1]}( \wt{\cN}/G \underset{\fg^\wedge_{\cN}/G} \times \wt{\cN}/G).$$
Therefore, \eqref{e:matchends} amounts to a monoidal equivalence
$$\Dmod(I \bs \chG_F / I) \simeq \QC^!_{\cN[1]}( \wt{\cN}/G \underset{\fg^\wedge_{\cN}/G} \times \wt{\cN}/G).$$
This is the formulation of Bezrukavnikov's equivalence required for local geometric Langlands, cf. \cite{Arinkinnotes}. To obtain this from \cite{roma hecke} requires some work, particularly the analysis of singular support on the spectral side, and this is among the things provided by Theorem \ref{main thm} below. 

On the automorphic side of local Langlands, one can consider the full sub-2-category of categorical representations generated by their Iwahori equivariant objects$$\Dmod(\chG_F)\mod^{I\on{-gen}} \hookrightarrow \Dmod(\chG_F)\mod,$$and one has an equivalence
$$\mathbf{Fun}_{\Dmod(\chG_F)\mod}( \Dmod(\chG_F/I), -): \Dmod(\chG_F)\mod^{I\on{-gen}} \simeq \Dmod(I \bs \chG_F / I)\mod.$$

On the spectral side of \eqref{e:loclang}, one can consider sheaves of categories supported on the tame unipotent locus
\eqref{e:tameunip}, i.e., 
$$\on{2-QC^!}_{\cN}(\fg^\wedge_{\cN}/G) \hookrightarrow \on{2-QC^!}_{\cN}(\on{LocSys}_{G}(\mathscr{D}^\times)),$$and again one has an equivalence
$$\mathbf{Fun}_{\on{2-QC}^!_{\cN}(\fg^\wedge_{\cN}/G)}( \QCoh(\wt{\cN}/G), - ): \on{2-QC^!}_{\cN}(\fg^\wedge_{\cN}/G)  \simeq \QC^!_{\cN \cap \cN[1]}( \wt{\cN}/G \underset{\fg^\wedge_{\cN}/G} \times \wt{\cN}/G)\mod.$$
Therefore, Bezrukavnikov's equivalence implies the equivalence of 2-categories
\begin{equation}\mathbb{L}: \Dmod(\chG_F)\mod^{I\on{-gen}} \simeq \on{2-QC^!}_{\cN}(\fg^\wedge_{\cN}/G),\label{e:loclangtame}\end{equation}
i.e., local geometric Langlands over the tame, unipotent locus.

Having explained this, let us revisit the matching of equivariance conditions and formal stacks over $\fg^{\wedge}_{\cN}/G$ described in Sections \ref{sss:morning}-\ref{sss:evening}. To fix ideas, let us focus on the cases
$$\chG_F/I_{\chP} \overset{\mathbb{L}}{\longleftrightarrow}  \wt{\cN}_P / G \quad \quad \text{and} \quad \quad \chG_F/ _{\chP}\mathring{I}, \psi, \St \overset{\mathbb{L}}{\longleftrightarrow} (\wt{\fg}_P)^\wedge_\cN / G,$$
the others being technical variations thereof. These matchings, which were stated somewhat informally before, in the setup  of \eqref{e:loclangtame} are meant to arise from 
$$\mathbb{L}(\Dmod(\chG_F/I_{\chP})) \simeq \QC(\wt{\cN}_P / G) \quad \quad \text{and} \quad \quad \mathbb{L}(\Dmod(\chG_F/ _{\chP}\mathring{I}, \psi, \St)) \simeq \QC((\wt{\fg}_P)^\wedge_\cN / G).$$
These predictions, in turn, arise from applying parabolic induction to two general expectations in local geometric Langlands. Namely, recalling that we write $\mathscr{D}$ for the non-punctured formal disk, one is meant to exchange the standard generators of the unramified parts
$$\mathbb{L}(\Dmod(\chG_F/\chG_O)) \simeq \QC( \on{LocSys}_G(\mathscr{D})).$$
In addition, one is meant to have tempered geometric Langlands, i.e., that Whittaker invariants are exchanged with global sections of sheaves of categories
$$\mathbb{L}(\Dmod(\chG_F/\chN_F, \psi)) \simeq \QC(\on{LocSys}_G(\mathscr{D}^\times)),$$which after restricting to tame parts simplifies to
$$\mathbb{L}(\Dmod(\chG_F/ \mathring{I}, \psi, \St)) \simeq \QC(\fg^\wedge_\cN / G).$$

\subsubsection{} 
We now state our main theorem, a precise statement of the above discussion, which appears in the body as Theorem \ref{thm conjectures} (with the $\infty$-categorical lift of Bezrukavnikov's theorem appearing as Theorem \ref{roman theorem}).  We note that in the statements below, we rewrite the formal completions along $\cN$ in the above discussions as support conditions on categories of sheaves.

\begin{introthm} \label{t:thma}\label{main thm}
There are $\infty$-categorical equivalences between the automorphic and spectral affine Hecke categories
$$\begin{tikzcd}[row sep=0ex, column sep=tiny]
\Dmod(I_{\chP} \bs \chG_F / I_{\chQ}) \arrow[r, "\simeq"] & \IndCoh_{\cN[1]}(\wt{\cN}_P/G \utimes{\mf{g}/G} \wt{\cN}_Q/G) & \text{(strict-strict)}  \\
\Dmod((I_{\chP}, I'_{\chP})\mon\!\bs \chG_F / I_{\chQ}) \arrow[r, "\simeq"]& \IndCoh_{\cN[1]}(\wt{\cN}'_P/G \utimes{\mf{g}/G} \wt{\cN}_Q/G) & \text{(monodromic-strict)} \\
\Dmod(\mathring{I}, \psi_P \bs \chG_F / I_{\chQ}) \arrow[r, "\simeq"] & \IndCoh(\wt{\mf{g}}_P/G \utimes{\mf{g}/G} \wt{\cN}_Q/G)  &  \text{(Whittaker-strict)} \\
\Dmod((I_{\chP},I'_{\chP})\mon\! \bs \chG_F / (I_{\chQ}, I'_{\chQ})\mon) \arrow[r, "\simeq"]& \IndCoh_{\cN \cap \cN[1]}(\wt{\cN}'_P/G \utimes{\mf{g}/G} \wt{\cN}'_Q/G) & \text{(monodromic-monodromic)} \\
\Dmod(\mathring{I},\psi_P \bs \chG_F / (I_{\chQ}, I'_{\chQ})\mon) \arrow[r, "\simeq"]& \IndCoh_{\cN}(\wt{\mf{g}}_P/G \utimes{\mf{g}/G} \wt{\cN}'_Q/G) & \text{(Whittaker-monodromic)} \\
\Dmod(\mathring{I}, \psi_P, \St \bs \chG_F / \mathring{I}, \psi_Q, \St) \arrow[r, "\simeq"] & \IndCoh_{\cN}(\wt{\mf{g}}_P/G \utimes{\mf{g}/G} \wt{\mf{g}}_Q/G) & \text{(bi-Whittaker)}
 \end{tikzcd}$$
 compatible with monoidal and module structures, and the functorialities discussed in Section \ref{s:mainresults}.
\end{introthm}

We note that in the equivalences above, the spectral side categories may all be given the support and singular support condition $\cN \cap \cN[1]$.  In cases where one of these conditions is redundant, we have simplified the expression by omitting it (see Section \ref{app ss} for a discussion of singular support, and Section \ref{section ssupp cond food} for a discussion of how singular and classical supports arise in the above categories).

Our results also hold in the setting as formulated in \cite{roma hecke}, i.e. taking $F = \overline{\mathbb{F}}_q(\!(t)\!)$ and the category of $\ell$-adic sheaves on the automorphic side, and defining varieties over $\overline{\mathbb{Q}}_\ell$ on the spectral side.  Here, pullback along geometric Frobenius is exchanged with pullback along scaling by $q$, which may be deduced $\infty$-categorically using our methods from Appendix A.

\subsubsection{}  \label{sss:friday} Finally, we discuss a finer aspect of the equivalences of Theorem \ref{t:thma}, namely how they match certain small full subcategories of interest.  The nature of subcategories depends on whether one is imposing a strict or monodromic equivariance condition on the automorphic side, or dually working with stack or a formal stack on the spectral side. Let us begin with a discussion of the strict case.

\subsubsection{} On the automorphic side, the categories of D-modules on the quotient stacks which appear are all compactly generated. Therefore, one natural class of objects to consider are the compact objects themselves, i.e., {\em safe} D-modules. In what follows, for a stack $X$, we denote by $\Dmod_s(X)$ its category of safe D-modules.

While the safe objects are a very natural subcategory to consider, they fail to contain many basic objects of interest. For example, the monoidal unit in the strict affine Hecke category, i.e. the dualizing sheaf on $I \bs I / I$
$$\omega_{I \bs I / I} \in \Dmod(I \bs \chG_F / I),$$
fails to be compact, essentially due to the infinite cohomological amplitude of its endomorphisms. However, it does generate the larger category of {\em coherent} D-modules on this stratum. Explicitly, letting $C_\bullet(\chT)$ denote the de Rham homology of $\chT$, one has an identification
$$\Dmod(I \bs I / I) \overset{\simeq}{\longrightarrow} C_\bullet(\chT)\mod,$$
where the equivalence takes a D-module $\sM$ to a $C_\bullet(\chT)$-module whose underlying vector space is the $!$-restriction of $\sM$ to a point.  In particular, the dualizing sheaf corresponds to the augmentation module for $C_\bullet(\chT)$.  Under this identification, compact objects correspond to perfect $C_\bullet(\chT)$-modules, whereas coherent objects correspond to $C_\bullet(\chT)$-modules with perfect underlying vector spaces. In general, for a stack $X$ we may consider its category of coherent D-modules, which we denote by $\Dmod_c(X)$. 

Spectrally, for a Steinberg stack $\cZ/G$, i.e., one of the fiber products appearing on the spectral side of Theorem \ref{t:thma}, recall that the category associated to it under consideration is$$\IndCoh_{\cN \cap \cN[1]}(\cZ/G),$$i.e., ind-coherent with both a nilpotent singular support and a classical support condition.  As before, to pass to small categories, we can take the compact objects in this category, i.e. the category $\Coh_{\cN \cap \cN[1]}(\cZ/G)$ of coherent sheaves with the same singular and classical support conditions.  

To see how to find an enlargement of it, corresponding to passing from safe to coherent D-modules, let us again consider the strict spectral affine Hecke category.  Under the convolution monoidal structure, the monoidal unit is the pushforward of the dualizing sheaf of the stacky Springer resolution along the diagonal $\Delta$, i.e.
$$\Delta_* \omega_{\wt{\cN}/G} \in \IndCoh_{\cN[1]}(\wt{\cN}/G \utimes{\mf{g}/G} \wt{\cN}/G).$$
This object does not have nilpotent singular support, so is not a compact object, matching expectations from the automorphic side.  It is, however, a coherent sheaf.

As this example suggests, we show that the automorphic enlargement from safe D-modules to coherent D-modules corresponds spectrally to the enlargement from coherent sheaves with nilpotent singular support to all coherent sheaves in all the equivalences of Theorem \ref{t:thma}.  Furthermore, these enlargements have abstract characterizations in terms of natural $t$-structures.  Namely, given a stable $\infty$-category with a reasonable $t$-structure, one can consider what we call the {\em pseudo-compact objects}, i.e., objects $c$ which are $t$-bounded and for which $\Hom(c, -)$ commutes with uniformly bounded below filtered colimits.\footnote{Pseudo-compact objects are also sometimes called coherent in the algebro-geometric literature, but we avoid this term to avoid overloading it. In addition, the condition for $\Hom(c, -)$ without the assumption of $t$-boundedness is termed {\em almost compactness}.}   The enlarged categories of coherent sheaves in both the automorphic and spectral affine Hecke categories are precisely the pseudo-compact objects for their natural $t$-structures.\footnote{Note the equivalences of Theorem \ref{t:thma} are not $t$-exact. However, they may be shown to be $t$-bounded, which does imply the agreement of pseudo-compact objects, although we employ a different argument in Theorem \ref{thm conjectures}.}

\subsubsection{} If instead of strict equivariance we impose monodromic equivariance, the  difference between compact and pseudo-compact objects disappears. Nonetheless, one again encounters two natural small subcategories, in a manner essentially Koszul dual to the previous discussion. 

As before, we may start with the compact objects. To  identify an interesting enlargement of them, we again contemplate the fact that many basic objects of interest fail to be compact. Once again, the prototypical example is the monoidal unit of the monodromic Hecke category 
$$\delta^\wedge_e \in \Dmod(I\mon \!\bs \chG_F / I\mon).$$
This fails to be compact as it is a unipotent local system on $I$ of infinite rank. However, its maximal constant, as opposed to merely unipotent, subsystem {\em is} finite rank, and this holds after pulling back along any finite cover of $I$. It is this form of finiteness which leads to the definition of the larger category of {\em nearly compact} objects, i.e., monodromic objects whose equivariant averages become compact, cf. Definitions \ref{d:defnc1}, \ref{d:ncnear2}.

Explicitly, letting $C^\bullet(\chT)$ denote the de Rham cohomology of $\chT$, one has an identification 
$$\Dmod(I\mon \bs I / I\mon) \simeq C^\bullet(\chT)\mod,$$
where the equivalence takes a D-module $\sM$ to its global sections. This identification matches the subcategory of compact D-modules with the subcategory of perfect $C^\bullet(\chT)$-modules. The monoidal unit $\delta^\wedge_e$ is exchanged with the augmentation module for $C^\bullet(\chT)\mod$, and again the {nearly compact} objects in this category, i.e., $$\Dmod_{n.c.}(I\mon \bs I / I\mon),$$are exchanged with the $C^\bullet(\chT)$-modules whose underlying vector spaces are perfect.

Spectrally, recall that the monodromic Hecke categories can be realized as either coherent sheaves on a Steinberg stack with a nilpotent classical support condition, or as coherent sheaves on the completion along nilpotent elements.  The monoidal unit is the pushforward of the dualizing sheaf on the formal completion along the diagonal, i.e.
$$\Delta_* \omega_{\wt{\mf{g}}^\wedge_{\cN}/G} \in \IndCoh_{\cN}(\wt{\mf{g}}/G \utimes{\mf{g}/G} \wt{\mf{g}}/G) \simeq \IndCoh(\wt{\mf{g}}^\wedge_{\cN} \utimes{\mf{g}^\wedge_{\cN}/G} \wt{\mf{g}}^\wedge_{\cN}).$$
Explicitly, we can compute $\omega_{\wt{\mf{g}}^\wedge_{\cN}/G}$ by taking the sheafy local cohomology of the dualizing sheaf of the stack $\wt{\mf{g}}/G$ along nilpotent elements.  This is not a coherent sheaf but an ind-coherent sheaf, given by the colimit of dualizing sheaves over infinitesimal thickenings of the nilpotent substack.

The monoidal unit is not compact, but it belongs to a natural (still small) enlargement of the subcategory of compact objects which we call the category of {\em ind-continuous sheaves} on the formal completion (see Section \ref{sec completions}), which we denote by $\wh{\on{Coh}}(-)$:
$$\hCoh_{\cN}(\wt{\mf{g}}/G \utimes{\mf{g}/G} \wt{\mf{g}}/G) \simeq \hCoh(\wt{\mf{g}}^\wedge_{\cN}/G \utimes{\mf{g}^\wedge_{\cN}/G} \wt{\mf{g}}^\wedge_{\cN}/G).$$
These are defined in terms of the $t$-structure as well; they are $t$-bounded ind-coherent sheaves whose $!$-restrictions to any finite infinitesimal thickening of the nilpotent elements are almost compact, and by this definition one immediately sees that the monoidal unit belongs to this subcategory.

Having introduced the categories of nearly compact objects on the automorphic side and continuous ind-coherent sheaves on the spectral side, one of our basic results in this paper is that the monodromic equivalences in Theorem \ref{t:thma} exchanges the two, cf. \ref{thm conjectures}.

\subsubsection{} We formally summarize the above discussion in the following.
\begin{introthm}\label{thm small obj}
The equivalences in Theorem \ref{t:thma} match the following small subcategories.
$$\Dmod_s \leftrightarrow \Coh_{\cN \cap \cN[1]}, \;\;\;\;\;\;\;\;\;\; \Dmod_c \leftrightarrow \Coh_{\cN},$$
$${\Dmod}_{s,n.c.} \leftrightarrow \hCoh_{\cN \cap \cN[1]}, \;\;\;\;\;\;\;\;\;\; \Dmod_{c,n.c.} \leftrightarrow \hCoh_{\cN}.$$
\end{introthm}

\begin{rmk}
It is expected that the Koszul duality described in \cite{BezYun} and the linear Koszul duality in \cite{MR} should be exchanged under the equivalences in Theorems \ref{t:thma} and \ref{thm small obj}, which requires a discussion of mixed structures on the automorphic side and a $\bG_m$-equivariance on the spectral side (announced in \cite{ho li}).  Since these structures are absent in our setting, this will result in a mismatch: support in the singular direction is naturally $\bG_m$-graded by the cohomological grading, but the classical support is not.
\end{rmk}

\subsubsection{} To orient the reader, let us highlight a few cases of Theorem \ref{t:thma} which were already known and play a basic role in this subject. 

First, there is the identification of the spherical and anti-spherical modules for the affine Hecke category (and similarly for the monodromic affine Hecke category)  \cite{AB}
$$\Dmod(I \bs \chG_F / \chG_O) \simeq \IndCoh_{\cN[1]}(\wt{\cN}/G \times_{\mf{g}/G} \{0\}/G), \;\;\;\;\;\;\;\;\;\; \Dmod(I \bs \chG_F / \mathring{I}, \psi) \simeq \IndCoh(\wt{\cN}/G).$$
In addition, there is the following derived geometric Satake equivalence and its ind-coherent version \cite{BF}
$$\Dmod(\chG_O \bs \chG_F / \chG_O) \simeq \IndCoh_{\cN[1]}(\{0\}/G \times_{\mf{g}/G} \{0\}/G), \;\;\; \Ind\Dmod_c(\chG_O \bs \chG_F / \chG_O) \simeq \IndCoh(\{0\}/G \times_{\mf{g}/G} \{0\}/G).$$Finally, we highlight the equivalence of Arkhipov--Bezrukavnikov--Ginzburg \cite{ABG} and its equivariant version
$$\Dmod(\chG_O \bs \chG_F/\chI\mon) \simeq \QC^!( \{0\}/G \times_{\mf{g}/G} \wt{\mf{g}}/G) \quad \quad \Dmod(\chG_O \bs \chG_F/\chI) \simeq \QC^!( \{0\}/G \times_{\mf{g}/G} \wt{\cN}/G)  ;$$
the latter is stated for example in \cite{GaIC}.

\subsubsection{}\label{biwhittaker sec} Among the new equivalences obtained in Theorem \ref{t:thma}, we would like to draw attention to one in particular. Namely, we obtain an automorphic realization of coherent sheaves on the (formal completion of the) nilpotent cone in terms of a certain full subcategory of bi-Whittaker sheaves on the loop group. A closely related result appears in \cite{BezNilCone}, and the result itself was conjectured by Bezrukavnikov (unpublished).
\begin{introthm}
There is an equivalence of monoidal categories \label{t:thmb}
\begin{equation} \label{e:bistwhit}
\begin{tikzcd}\Dmod(\mathring{I}, \psi, \St \bs \chG_F / \mathring{I}, \psi, \St) \arrow[r, "\simeq"] & \on{QCoh}(\wh{\cN}/G).\end{tikzcd}
\end{equation}
\end{introthm}

We emphasize that Theorem \ref{t:thmb} is Koszul dual to the derived Satake equivalence of Bezrukavnikov--Finkelberg \cite{BF}. In particular, the appearance of the nilpotent cone as a singular support condition in the latter equivalence is transformed into its appearance as a naive support condition in Theorem \ref{t:thmb}.

For us, part of the appeal of Theorem B is the following. While we obtain Theorem B essentially by a monadicity argument from the main result of \cite{roma hecke}, we expect the logic may be reversed. Namely, as in \cite{AB}, one can use geometric Satake and fusion to directly obtain the desired monoidal functor 

$$\QCoh(\wh{\cN}/G) \rightarrow  \cat{End}_{\Dmod(\chG_F)}(\Dmod(\chG_F/\mathring{I}, \psi, \St)) \simeq \Dmod(\mathring{I}, \psi, \St \bs \chG_F / \mathring{I}, \psi, \St).$$
We expect this can be directly checked to be an equivalence which moreover matches the modules 
$$\QCoh(\wt{\cN} / G) \simeq \Dmod(\mathring{I}, \psi \backslash \chG_F / I),$$
i.e., the equivalence of \cite{AB}; we note that a very similar idea already appears in the work of Bezrukavnikov--Riche--Rider \cite{BezRicheRider}.  From here, one may directly take endomorphisms to obtain 
$$\cat{End}_{\QCoh(\wh{\cN}/G)}(\QCoh(\wt{\cN}/G)) \simeq \QCoh(\wt{\cN}/G  \underset{\wh{\cN}/G}{\times} \wt{\cN}/G),$$
and similarly on the automorphic side. That is, up to renormalizing, one can recover the identification of affine Hecke categories simply by taking endomorphisms.

A detailed proof along these lines, in the case of the universal Betti affine Hecke category, i.e., with monodromy in all of $G/G$ rather than the unipotent cone, is the subject of ongoing work joint with Jeremy Taylor.

\subsection{Organization of the paper} 

\subsubsection{} The overall organization of the paper is as follows. In Sections \ref{app dag} and \ref{spectral hecke sec}, we provide some analysis purely on the spectral side, and in Sections \ref{sec small ob} and \ref{sec autom hecke} some complementary analysis purely on the automorphic side. Finally, in Section \ref{s:mainresults} we obtain our main results. Let us now spell this out in more detail.

\subsubsection{} In Section \ref{app dag} we establish foundational definitions and results which are relevant for discussing spectral side affine Hecke categories.  Section \ref{dag sec} contains a review of some  standard definitions in derived algebraic geometry, namely of variants of derived categories of quasi-coherent sheaves, as well as a review of Grothendieck duality and formal completions.  In Section \ref{sec completions} we introduce in Definition \ref{wcoh defn} an enlargement of the category of coherent sheaves on a formal completion, the \emph{continuous ind-coherent sheaves}. We establish the main technical criterion for checking membership in this category in Theorem \ref{wcoh prop qs} as well as a Grothendieck existence theorem.  In Section \ref{app ss} we review the notion of singular support.

\subsubsection{}In Section \ref{spectral hecke sec} we  establish basic results regarding convolution formalism in the algebro-geometric context, and then apply these results to spectral affine Hecke categories.  In Section \ref{convolution sec}, we discuss the convolution monoidal structure on categories of coherent sheaves and establish a  basic property, namely a relative tensor product formula in Theorem \ref{convolution composition} involving the convolution of singular support conditions.  We then apply this formalism to spectral affine Hecke categories in \ref{convolution hecke}, in particular explaining the appearance of nilpotent classical and singular support conditions in Proposition \ref{reduction using monoidal}.  In Section \ref{sec algebra obj} we discuss how certain monads can be expressed via algebra objects in monoidal categories of coherent sheaves under convolution, and in Section \ref{spectral whittaker sec} this formalism is applied to affine Hecke categories to describe the category of partial Whittaker affine Hecke categories in terms of monodromic affine Hecke categories.  Finally, in Section \ref{sec spectral monodrom}, we discuss a method for relating strict affine Hecke categories and centrally monodromic affine Hecke categories via Koszul duality.

\subsubsection{}In Section \ref{sec small ob}, we introduce various small subcategories of categories of categories of D-modules on ind-schemes and stacks of possibly infinite type which appear on the automorphic side of Langlands duality.  In Section \ref{dmod placid} we review D-modules on \emph{placid ind-schemes}, i.e. ind-schemes which may be written as a limit under smooth projections, followed by a colimit under closed embeddings, and prove the equivalence of two notions of smallness, {\em safety} and {\em coherence}, that may be imposed on objects in this category.  In Section \ref{dmod placid equiv}, we consider the same in the stacky setting, where the two notions of smallness now diverge.  In Section \ref{sec cofree} we explore the Koszul dual phenomenon in the setting of monodromic D-modules which we call \emph{near compactness}. In particular, we relate near compact objects to a certain derived local system which exists on any algebraic variety, namely its {\em cofree monodromic sheaf}. We establish a criterion for membership in the category of nearly compact objects in Proposition \ref{p:oneorbit} in the case of a stack with a single orbit, and in the case of multiple strata and introduce certain natural generators of such categories called \emph{cofree-monodromic} standard and tilting sheaves in Corollary \ref{c:whoisncshv} and Proposition \ref{p:classtilting}.

\subsubsection{}In Section \ref{sec autom hecke}, we establish techniques for expressing automorphic partial affine Hecke categories in terms of affine Iwahori-Hecke categories.  In Section \ref{group def}, we discuss a Koszul duality which allows us to compute strict and monodromic invariants in terms of each other by killing operators in Hochschild cohomology via categorical tensor products.  In Section \ref{whit monad}, we discuss how certain monads in automorphic monodromic affine Hecke categories can be expressed via algebra structures on various cofree-monodromic tilting sheaves, thus giving a description of partial Whittaker affine Hecke categories in terms of the monodromic affine Hecke category.  In Section \ref{autom convolve sec}, we record a useful formula for the relative tensor product of certain modules for affine Hecke categories.

\subsubsection{}Finally, in Section \ref{s:mainresults}, we put it all together and prove the main Theorem \ref{thm conjectures}.  Appendix \ref{app infcat} contains results we use to lift 1-categorical statements to $\infty$-categorical statements.

\vspace{.5cm}
	
	\noindent {\bf Conventions and notation.}  We let $k$ denote an algebraically closed field of characteristic zero.  See Appendix \ref{app infcat} for $\infty$-categorical notions.  Let $\Vect_k$ denote the $\infty$-category of $k$-modules, i.e. the dg nerve of the category of chain complexes over $k$, which is monoidal with a $t$-structure.  A dg ring is an algebra object in $\Vect_k^{\leq 0}$, and for a dg ring $R$ the category $D(R)$ is the $\infty$-category of $R$-modules in $\Vect_k$.  Let $\DGCat$ denote the $\infty$-category of presentable $k$-linear (in particular, cocomplete and stable) $\infty$-categories and colimit-preserving functors.

\label{sec conventions}

\vspace{.5cm}
	
	\noindent {\bf Acknowledgments.}  It is a pleasure to thank Pramod Achar, Roman Bezrukavnikov, Justin Campbell, Tsao-Hsien Chen, Stefan Dawydiak, Joakim F\ae rgeman,  Dennis Gaitsgory, Daniel Halpern-Leistner, Yau-Wing Li, Ivan Losev, Sergey Lysenko, David Mehrle, Sam Raskin, Simon Riche, Cheng-Chiang Tsai, Zhiwei Yun, and Xinwen Zhu for helpful related discussions.  G.D. was supported by an NSF Postdoctoral Fellowship under grant No. 2103387.

\section{Renormalized categories of sheaves on formal completions}\label{app dag}

\subsection{Overview}

\subsubsection{} Our main goal in this section is to define and study the category of {\em continuous ind-coherent sheaves} on a formal stack. The precise contents are as follows.

\sss
After collecting some preliminaries in derived algebraic geometry in Section \ref{dag sec}, we introduce the continuous ind-coherent sheaves in Section \ref{sec completions}. The main result here is a useful criterion for testing whether an ind-coherent sheaf on a formal stack is continuous, which is provided in Theorem \ref{wcoh prop qs}. In the final Section \ref{app ss}, we review some basic facts about singular support for ind-coherent sheaves in preparation for Section \ref{spectral hecke sec}.

In particular, the reader may wish to skim Section \ref{sec completions} before continuing to Section \ref{spectral hecke sec}, and refer back only as needed.

\subsection{Derived algebraic geometry}\label{dag sec}

\subsubsection{}
Let us recall some basic definitions from derived algebraic geometry for the reader's convenience, and to establish terminology.  We refer the reader to \cite{DAG, HLP} for details.  Every stack of interest to us will be a stack  quotient of a finite type derived scheme by a reductive algebraic group.  However, some results hold in greater generality, and we state them as such.

\subsubsection{}

We introduce the following terminology conventions (recall Section \ref{sec conventions}).  Following \cite{DAG}, recall that a \emph{derived ring} is an algebra object in $\Vect_k^{\leq 0}$, and for a derived ring $R \in \cat{DRng}_k$ we define $D(R)$ to be the category of module objects for $R$ in the $\Vect_k^{\leq 0}$-module category $\Vect_k$.  A connective dg algebra defines a derived ring $R$, and the dg nerve of the dg derived category of complexes of $R$-modules is equivalent to $D(R)$.  We define the category of \emph{affine derived schemes} by $\cat{DAff}_k := \cat{DRng}_k^{\opp}$.  

We say an affine derived scheme $S = \Spec(A)$ is \emph{almost finite type} over a classical Noetherian $k$-algebra $R$ if $\pi_0(A)$ is finitely generated over $R$ and each $\pi_i(A)$ is a finitely generated $\pi_0(A)$-module, and we say $S$ is \emph{Noetherian} if $\pi_0(A)$ is Noetherian and $\pi_i(A)$ is finitely generated as a $\pi_0(A)$-module.  Every almost finite type ring is automatically Noetherian, but not conversely.  We say $S$ is \emph{reduced} if $A \simeq \pi_0(A)$ and $\pi_0(A)$ is reduced.

A \emph{prestack} is a functor $X: \cat{DAff}^{\opp}_k \rightarrow \cat{S}$ where $\cat{S}$ is the $\infty$-category of spaces; by the $\infty$-Yoneda lemma, there is a fully faithful embedding $\Spec: \cat{DAff}_k \hookrightarrow \cat{PreSt}_k$.  A \emph{stack} $X \in \cat{St}_k$ is a prestack which is a sheaf in the \'{e}tale topology (see Chapter 2, Remark 4.1.4 in \cite{DAG}).   There is a sheafification functor $L: \cat{PreSt}_k \rightarrow \cat{St}_k$, which is left adjoint to the forgetful functor, and therefore commutes with colimits.

Let $R$ be a classical Noetherian $k$-algebra. A \emph{locally almost of finite type} or \emph{laft} prestack $X$ over a classical ring $R$ is a convergent prestack $X$ such that the $n$-coconnective truncations $\tau^{\leq n} X$ are left Kan extended from their restrictions to finite type over $R$ affine $n$-coconnective affine schemes.  See Chapter 2, Section 1.7 of \cite{DAG} for details.    A stack is \emph{reduced} if it is equivalent via the natural map to its reduced stack, i.e. the sheafification of the left Kan extension of its restriction to reduced affine schemes.  A stack is \emph{locally Noetherian} if it is equivalent via the natural map to its Noetherianization, i.e. the sheafification of the left Kan extension of its restriction to Noetherian affine schemes.\footnote{This is a somewhat non-standard notion, and we will not use it outside of this definition.  We note that the Noetherianization of an affine scheme need not be an affine scheme: for example, the Noetherianization of $\Spec k[x_1, x_2, \ldots]$ is the ind-scheme $\mathbb{A}^\infty$.  On the other hand, the reduced scheme of an affine scheme is always affine.}

An \emph{Artin stack} will simply mean a derived 1-Artin stack (see Chapter 2, Section 4 of \cite{DAG}).  We refer the reader to Definition 1.1.8 of \cite{QCA} for the notion of a \emph{QCA stack}, and to Definition 3.1 of \cite{BFN} for the notion of a \emph{perfect} stack.  We say a stack is \emph{geometric} if it is an 1-Artin stack with affine diagonal.


\subsubsection{}\label{sec formal stack} 
We now discuss formal completions and formal stacks, following \cite{dgind}.  Let $i: Z \rightarrow X$ be a map of prestacks, and denote by $\wh{X}_Z$ the formal completion of $X$ along $Z$, i.e. the prestack with $S$-points:
$$\wh{X}_Z(S) = \{\eta \in X(S) \mid \eta \text{ extends a point } \eta' \in Z(\pi_0(S)^{\mathrm{red}})\}.$$
We let $\wh{i}: \wh{X}_Z \rightarrow X$ denote the natural map.  The notion of a formal completion makes sense for any map of prestacks, but we specialize to the case where $X$ is a stack, and $Z \subset X$ a closed substack.

We say a prestack $X$ is a \emph{formal scheme} (resp. \emph{stack}) if its underlying reduced prestack is a scheme (resp. 1-Artin stack) and if the map
$$\colim_{\substack{Z_\alpha \subset X \text{ closed} \\ Z_\alpha^{\mathrm{red}} = X^{\mathrm{red}}}} Z_\alpha \overset{\simeq}{\longrightarrow} X$$
is an equivalence.

\subsubsection{}\label{def sheaves}
We consider various categories of sheaves on prestacks, which will be defined via right Kan extension.  That is, we first define our category of sheaves, along with a pullback functor, on some fixed full subcategory $\cat{C} \subset \cat{DAff}_k$, i.e., a functor
$$\cS{h}: \cat{C}^{\opp} \rightarrow \cat{Cat}_\infty.$$ 
We then extend it to the category of prestacks by the formula\footnote{This makes sense whether or not $X$ itself is a left Kan extension off its restriction to $\cat{C}$.}
$$\cS{h}(X) := \lim_{\substack{S \in \cat{C} \\ \eta \in X(S)}} \cS{h}(S),$$
see Appendix A of \cite{HLP}, Chapter 3 of \cite{DAG}, and Section 10 of \cite{indcoh} for details.  
The specific sheaf theories of interest for us are as follows.

\subsubsection{} We begin with quasi-coherent sheaves. 

\begin{defn} The category of \emph{quasi-coherent sheaves} on an affine derived scheme $S$ is $$\QCoh(S) := D(\cO_S),$$with pullback along $f: S' \rightarrow S$ given by the derived tensor product $$f^*(-) := - \underset{\cO_S}{\otimes} \cO_{S'}.$$
The functor $f_*$ is the right adjoint to $f^*$.  The category $\QCoh(X)$ of quasicoherent sheaves on any prestack $X$ is defined via the above  via right Kan extension.\end{defn}

We recall there is a canonical $t$-structure on $\QCoh(X)$ which is the usual $t$-structure when $X$ is affine, and for general $X$ characterized by the property that $f^*$ is right $t$-exact. Using this, within all quasi-coherent sheaves on a prestack we may isolate the following non-cocomplete full subcategory.

\begin{defn} The category of \emph{almost perfect complexes} on an affine derived scheme $S$ is the full subcategory $$\APerf(S) \subset \QCoh(S)$$consisting of complexes $\cF$ such that the truncations $\tau^{\geq n} \cF$ are compact in $\QCoh(S)^{\geq n}$ for all $n$.  Almost perfect complexes are preserved by $*$-pullback, thus defining $\APerf(X)$ for any prestack $X$.  Explicitly, $\APerf(X) \subset \QCoh(X)$ is the full subcategory of objects whose $*$-pullback to any affine $S$ is almost perfect.\end{defn}  

If $S$ is a Noetherian affine derived scheme, we have explicitly that these are the eventually connective complexes with coherent cohomology, i.e., 
$$\APerf(S) = D^-_{\mathrm{coh}}(\cO_S) \subset \QCoh(S)^-.$$
More generally, by considering the $t$-exact smooth $*$-pullback to an atlas, if $X$ is an Artin stack then $\APerf(X) \subset \QCoh(X)^-$ is the full subcategory of right $t$-bounded objects with coherent cohomology.

\subsubsection{} We next review the definition of ind-coherent sheaves.  Here, we restrict ourselves to the Noetherian case for simplicity and refer the reader to \cite{indcoh} for generalizations.

\begin{defn}
Fix a classical Noetherian $k$-algebra $R$.  For an almost finite type affine $S$ over $R$, we denote the bounded derived category of $\cO_S$-modules with coherent cohomology by 
$$\Coh(S) := D^b_{\mathrm{coh}}(\cO_S).$$
The category of \emph{ind-coherent sheaves} on $S$ is the ind-completion of $\Coh(S)$, i.e., 
$$\IndCoh(S) := \Ind( \Coh(S)).$$

For any map $f: S' \rightarrow S$ of affine schemes, the functor $f_*$ on quasi-coherent sheaves restricts to a functor $f_*: \Coh(S') \ra \QCoh(S)^+ = \IndCoh(S)^+$, thus defines a functor $f_*: \IndCoh(S') \rightarrow \IndCoh(S)$ by ind-completion.  

If $f$ is a map of laft of (possibly non-affine) schemes over $R$, there is a $!$-pullback 
$$f^!: \IndCoh(S) \rightarrow \IndCoh(S'),$$
which is defined by the property that when $f$ is an open immersion $f^! = f^*$, and when $f$ is proper $f^!$ is the right adjoint to $f_*$, see Theorem 5.2.2 in \cite{indcoh} and Chapter 5, Section 3 in \cite{DAG}.  When $f: S' \ra S$ is smooth, the functor $f_*$ has a left adjoint $f^*$.  

For any laft prestack $X$ over $R$, the category $\IndCoh(X)$ is then defined via right Kan extension; see Sections 10.1 and 10.2 of \cite{indcoh} for details. 
\end{defn}

    

\sss When $X$ is an Artin stack, following \cite{QCA} and Appendix A of \cite{BNP convolution}, the small subcategory of \emph{coherent complexes} $$\Coh(X) \subset \IndCoh(X)$$ is defined to be the full subcategory consisting of objects  $\cF$ such that $p^*\cF \in \Coh(U)$ for every smooth map $p: U \rightarrow X$, or equivalently for any smooth atlas $p$.   This definition has an analogue within quasi-coherent sheaves $\Coh(X) \subset \QCoh(X)$ which is a canonically equivalent category by Lemma 3.2.2 in \cite{QCA}.  We note that it is generally not true that the canonical map 
$$\Ind(\Coh(X)) \longrightarrow \IndCoh(X)$$
is an equivalence, but this holds when $X$ is a QCA stack by \emph{op. cit.}, Theorem 3.3.5.

\sss Let us now discuss $t$-structures on categories of ind-coherent sheaves on formal stacks. As bounded coherent complexes are closed under truncation,  for $S$ affine the category $\IndCoh(S)$ inherits a $t$-structure from $\QCoh(S)$ by definition. For an Artin stack $X$, we consider the unique $t$-structure on $\IndCoh(X)$  such that for any smooth  map $p: U \rightarrow X$ from $U$ affine the pullback $$p^*: \IndCoh(X) \rightarrow \IndCoh(U)$$is $t$-exact.   For a locally Noetherian\footnote{This guarantees that the closed embeddings are finite type.} formal stack
$$X = \colim X_\alpha,$$
there is a unique $t$-structure such that the pushforward functors $$i_{\alpha*}: \IndCoh(X_\alpha) \ra \IndCoh(X)$$ are $t$-exact, see Section 2.5 of \cite{dgind}.


\subsubsection{} In the the case of laft prestacks over a classical Noetherian $k$-algebra $R$, using the above $t$-structure, we may consider the following subcategories of ind-coherent sheaves.  In the affine case, they may also be realized as subcategories of quasi-coherent sheaves, since ind-coherent and quasi-coherent sheaves agree on eventually coconnective objects.


\begin{defn} 
Fix a classical Noetherian $k$-algebra $R$.  The \emph{$!$-almost perfect complexes} on a Noetherian affine derived scheme $S$ over $R$ are the full subcategory of ind-coherent sheaves consisting of eventually coconnective complexes with coherent cohomology 
$$\APerf^!(S) := D^+_{\mathrm{coh}}(\cO_S) \subset  \IndCoh(S)^+ = \QCoh(S)^+.$$
By Lemma 7.1.7 of \cite{indcoh}, $!$-pullback $f^!$ along an almost finite-type map $f$ is left $t$-exact up to a shift and preserves coherence, and thus defines a functor $$f^!: \APerf^!(S) \rightarrow \APerf^!(S').$$
For any $X$ laft over $R$, the category $\APerf^!(X) \subset \IndCoh(X)$ is then defined via right Kan extension. 
It is not a full subcategory of $\QCoh(X)$ since the $!$-pullback is not well-defined on quasi-coherent sheaves.
\end{defn}

The above definition is somewhat abstract.  However, we have the following  explicit identification, which applies in our cases of interest.
\begin{prop}
Let $R$ be a classical Noetherian $k$-algebra.  If $X$ is a laft Artin stack over $R$, then $\APerf^!(X) \subset \QCoh^!(X)$ is the full subcategory of eventually coconnective complexes with coherent cohomology.  If $X$ is a laft formal stack over $R$, then $\APerf^!(X) \subset \IndCoh(X)$ is the full subcategory of eventually coconnective complexes whose $!$-restriction to any closed Artin substack has coherent cohomology.
\end{prop}



\subsubsection{}\label{shriek pullback} We again fix a classical Noetherian ring $R$.  Given an almost finite-type map $f: X \rightarrow Y$ of left prestacks over $R$, we have a $*$-pullback on quasi-coherent sheaves 
$$f^*: \APerf^*(Y) \rightarrow \APerf^*(X)$$ and a $!$-pullback on ind-coherent sheaves
$$f^!: \APerf^!(Y) \rightarrow \APerf^!(X).$$
We would like to next explain that Grothendieck-Serre duality allows for passage between these two operations. To do so, we first review its formulation  over an arbitrary classical Noetherian affine base scheme $K = \Spec R$.

The definition of Grothendieck-Serre duality involves the dualizing complex.  Typically, one defines the dualizing complex on a laft stack $p: X \ra \Spec k$ over $k$ by taking $\omega_X = p^!k$.  However, when $X$ is not laft over $k$, this formula does not apply.  We will require this in the following example.
\begin{exmp}\label{complete stack}
Let $S = \Spec(A)$ be a finite type affine scheme over $k$, and $S_0$ a closed subscheme with defining ideal $I$.  Let $$\wh{A} := \lim_n A/I^n, \;\;\;\;\;\;\; \wh{S} := \Spec \wh{A}$$
be the completion of the ring $A$ with respect to $I$ and its corresponding affine scheme.

Further suppose we have a reductive algebraic group $G$ acting on $S$, such that $S_0 \subset S$ is $G$-closed and generated by $G$-invariants,\footnote{This assumption means the scheme $\wh{X}$ has an induced $G$-action.  For example, for $A = k[x]$, $I = (x)$, and $G = \bG_m = \Spec k[z,z^{-1}]$ acting on $x$ by weight -1, the completed ring $k[\![x]\!]$ does not inherit a $k[z,z^{-1}]$-comodule structure.} i.e. so that $$S_0 = \pi_0(S \times_{S/\!/G} S_0/\!/G).$$
Moreover, letting $\wh{S/\!/G} = \Spec \wh{A^G}$ where $\wh{A^G}$ is the completion of $A^G$ along $I^G$, we have $$\wh{S} = \pi_0(S \times_{S/\!/G} \wh{S/\!/G}).$$  This fiber product diagram is equivariant for the trivial $G$-action on $S/\!/G$, thus we have an induced $G$-action on $\wh{S}$, and by descent, the following squares are classically Cartesian in $\cat{St}_k$:
$$\begin{tikzcd}
S_0/G \arrow[d] \arrow[r] & \wh{S}/G \arrow[d] \arrow[r] & S/G \arrow[d] \\
S_0/\!/G \arrow[r] & \wh{S/\!/G} \arrow[r] & S/\!/G.
\end{tikzcd}$$
The resulting stack $\wh{S}/G$ is QCA and geometric, and also perfect by Proposition 3.21 of \cite{BFN}.  It is not finite type over $k$, but it is finite type over $\wh{A^G}$.
\end{exmp}

To overcome this in more general formulations of Grothendieck-Serre duality, cf. Section 6.5 of \cite{BNP convolution}, we fix a \emph{choice} of dualizing complex defined as follows.  
\begin{defn} An object $\omega_K \in \QCoh(K)$ is called a {\em dualizing complex} if it satisfies the following conditions.

\begin{enumerate}
    \item The object $\omega_K$ is $!$-almost perfect, i.e., lies in $\APerf^!(K)$,
    \item the object $\omega_K$ has finite injective dimension, and 
    \item the canonical map $\cO_K \rightarrow \shHom_K(\omega_K, \omega_K)$ is an equivalence.
\end{enumerate}
\end{defn}
In general $\omega_K$ is not uniquely determined by these conditions, and in particular may be twisted by any cohomologically shifted line bundle.  If a choice of $\omega_K$ exists, we then have a compatible assignment of dualizing complexes for any $X$ almost finite type over $K$, i.e., 
$$p: X \rightarrow K,$$as we may take $\omega_X = p_X^! \omega_K$. 

The basic statement of Grothendieck--Serre duality then reads as follows.
\begin{prop}\label{serre}
Let $K = \Spec(R)$ denote a classical Noetherian affine scheme which admits a dualizing complex $\omega_K$, and let $p: X \rightarrow K$ be a geometric stack almost of finite type over $K$. Then $p^!\omega_K$ is a dualizing complex on $X$ and we have a Grothendieck-Serre duality equivalence
$$\bD_X = \shHom_X(-, p^!\omega_K): \APerf^*(X)^{\opp} \overset{\simeq}{\rightarrow} \APerf^!(X)$$
satisfying the following compatibilities:
\begin{enumerate}[(1)]
\item Duality $\bD$ commutes with proper pushforward, i.e. for $f: X \ra Y$ proper, $\bD_Y \circ f_*^{\opp} \simeq f_* \circ \bD_X$.
\item If $f: X \rightarrow Y$ is a schematic finite-type map of finite-type geometric stacks over $K$, then $\bD$ intertwines $f^*$ and $f^!$, i.e. $f^! \circ \mathbb{D}_Y  \simeq \bD_X \circ f^{*,\opp}$.
\item The functor $\bD_X$ is $t$-bounded, i.e. it is both left $t$-exact up to a finite shift and right $t$-exact up to a finite shift, or equivalently takes $t$-bounded objects to $t$-bounded objects.\footnote{Note that taking the opposite category exchanges the notions of left and right boundedness.}
\end{enumerate}
\end{prop}
\begin{proof}
The main statement and third property are Proposition 6.5.3 of \cite{BNP convolution}. For the remaining statements, since all functors including $\bD$ are smooth local and all stacks are geometric, we may assume that $X$ and $Y$ are derived schemes.  

For the first compatibility, note that $f_*$ preserves complexes with coherent cohomology if $f$ is proper, and that $f_*$ is $t$-bounded up to a shift when $f$ is finite type; thus, $f_*$ restricts to functors on $*$ and $!$ almost-perfect complexes, and the commuting relation follows via the formula $$\shHom_Y(f_*\cF, \omega_Y) \simeq f_*\shHom_X(\cF, f^!\omega_Y).$$

For the second compatibility, note that $f^*$ and $f^!$ preserve $*$- and $!$-almost perfect complexes respectively by construction.  By the construction of $f$ for schematic maps (see Section 6.3 of \cite{indcoh}), it suffices to check the claim when $f$ an open immersion and when $f$ is proper.  When $f$ is an open embedding the claim is clear since the functor $\bD$ is smooth local; when $f$ is proper, the statement follows from the first compatibility by passing to right adjoints.  
\end{proof}

\begin{exmp}\label{gs duality ex}
Consider the embedding
$$j: K = \Spec k[\![x_1, \ldots, x_n]\!] \hookrightarrow \bA^n_k = \Spec k[x_1, \ldots, x_n].$$
We can take $\omega_K := j^*\omega_{\bA^n_k}$.  The map $i: \Spec k \hookrightarrow K$ is finite-type, and we have a canonical identification
$$i^!\omega_K \simeq \Hom_{k[\![x_1, \ldots, x_n]\!]}(k, \omega_K) \simeq \Hom_{k[x_1, \ldots, x_n]}(k, \omega_{\bA^1_k}) \simeq (i \circ j)^! \omega_{\bA^n_k} \simeq k,$$
i.e. this Grothendieck-Serre duality for stacks over $K$ defined by $\omega_K$ is compatible with the usual duality over the closed point of $K$ (see also \cite[\href{https://stacks.math.columbia.edu/tag/0AU3}{Section 0AU3}]{stacks-project}).
\end{exmp}

\subsubsection{} Recall the notion of a formal completion and formal stack from Section \ref{sec formal stack}.  In \cite{dgind} it is established that, in the case where $X$ is a scheme, any such formal completion of a scheme along a closed subscheme is in fact a formal scheme.  We now establish the analogous fact for stacks.
\begin{prop}\label{dgind stack}
Let $X$ be an Artin stack, and $Z \subset X$ a closed substack.  We have that 
$$\wh{X}_Z = \colim_{|Z_\alpha| \subset |Z|} Z_\alpha$$
where $|Z_\alpha|$ ranges over all closed substacks $Z_\alpha \subset X$ such that $\pi_0(Z_\alpha)^{\mathrm{red}} \subset \pi_0(Z)^{\mathrm{red}}$.  In particular, $\wh{X}_Z$ is a formal stack.
\end{prop}
\begin{proof}
We first write $\wh{X}_Z$ as a colimit.  Fix an atlas $U$ for $X$, let $U^{(n)}$ denote the $n$th term in the corresponding Cech complex $\mathrm{Cech}(U/X)$, and set $$\wh{U}^{(n)}_Z := U^{(n)} \times_X \wh{X}_Z.$$Similarly,  for a closed embedding $Z_\alpha \hookrightarrow \wh{X}_Z$, we set $$U_\alpha^{(n)} :=Z_\alpha \times_X U^{(n)}.$$We claim that $\colim U_\alpha^{(n)} \simeq \wh{U}^{(n)}_Z$; given this, by descent (Chapter 2, Corollary 4.3.3 of \cite{DAG}) we have
$$\colim_\alpha Z_\alpha \simeq \colim_\alpha |\mathrm{Cech}(U_\alpha/Z_\alpha)| \simeq |\colim_\alpha  \mathrm{Cech}(\wh{U}_\alpha/Z_\alpha)| \simeq |\mathrm{Cech}(U/X)| \simeq \wh{X}_Z.$$
For the claim, we show as in Proposition 6.5.5 of \cite{dgind} that any closed embedding $Y \hookrightarrow \wh{U}^{(n)}_Z$ factors through some $U_\alpha^{(n)}$.  Consider the Cartesian square
$$\begin{tikzcd}
 U^{(n+i)} \arrow[r, "q"] \arrow[d, "p"'] & U^{(n)} \arrow[d] \\
U^{(i)} \arrow[r] & X
\end{tikzcd}$$
and let $Y^{(i)}$ be the scheme-theoretic image of $q^{-1}(Y)$ along $p$.  The $Y^{(i)}$ assemble into a simplicial diagram of closed subschemes of the $U^{(i)}$ descending to a closed substack $Z_\alpha \subset X$, and we can take these to be our $U_\alpha^{(n)}$.
\end{proof}


\sss

As we now explain, the above proposition implies that the categories of quasi-coherent and ind-coherent sheaves on a formal completion have familiar presentations via support conditions on the ambient stack.

Namely, let $X$ be a laft Artin stack over an classical Noetherian ring $R$. Let $Z$ be a closed substack with open complement $U$, and denote their inclusions by $$i: Z \hookrightarrow X \quad \quad \text{and} \quad \quad j: U \hookrightarrow X, $$respectively. Consider the restriction functors $$j^*: \QCoh(X) \rightarrow \QCoh(U) \quad \quad \text{and} \quad \quad  j^!: \IndCoh(X) \rightarrow \IndCoh(U).$$The categories of \emph{quasi-coherent sheaves supported on $Z$} and  \emph{ind-coherent sheaves supported on $Z$} are defined to be the kernels of these functors, and are denoted by $\QCoh_Z(X)$ and $\IndCoh_Z(X)$, respectively. 

We then have the following.
\begin{cor}
We have equivalences of categories
$$\QCoh(\wh{X}_Z) = \lim_{|Z_\alpha| \subset |Z|} \QCoh(Z_\alpha), \;\;\;\;\;\;\;\;\;\;\;\; \IndCoh(\wh{X}_Z) = \lim_{|Z_\alpha| \subset |Z|} \IndCoh(Z_\alpha)$$
and similarly for the full subcategories $\APerf$ and $\APerf^!$, where the transition functors are given by $*$- and $!$-pullback respectively.  We also have
$$\IndCoh(\wh{X}_Z) = \colim_{|Z_\alpha| \subset |Z|} \IndCoh(Z_\alpha)$$
where the transition functors are given by $*$-pushforward.  
Furthermore, we have the following functors and identifications.
\begin{enumerate}[(1)]
    \item The continuous functor $\wh{i}^{\,!}: \IndCoh(X) \rightarrow \IndCoh(\wh{X}_Z)$ admits a fully faithful left adjoint which we denote $$\wh{i}_*: \IndCoh(\wh{X}_Z) \rightarrow \IndCoh(X),$$with essential image $\IndCoh_Z(X)$.
    \item The continuous functor $\wh{i}^*: \QCoh(X) \rightarrow \QCoh(\wh{X}_Z)$ admits a (non-continuous) right adjoint which we denote $$\wh{i}_*: \QCoh(\wh{X}_Z) \rightarrow \QCoh(X)$$ and a fully faithful left adjoint which we denote $$\wh{i}_+: \QCoh(\wh{X}_Z) \rightarrow \QCoh(X),$$with essential image $\QCoh_Z(X)$.
\end{enumerate}
\end{cor}
\begin{proof}
The first statements follow by Lemma 5.1.5.5 of \cite{HTT}, i.e. since the functors $\QCoh$ and $\IndCoh$ take colimits in $\cat{PreSt}_k$ to limits, and by passing to right adjoints as in Chapter 1, Section 2.4 of \cite{DAG}.  
The functors $\wh{i}^{\,!}$ and $\wh{i}^*$ are automatically well-defined and continuous, thus guaranteeing existence of right adjoints.  The left adjoint of $\wh{i}^{\,!}$ is by the same argument as in Section 7.4.3 in \cite{dgind}, and the left adjoint of $\wh{i}^*$ is by Section 7.1.6 in \emph{op. cit.}  Finally, the fully faithfulness of the functors $\wh{i}_*$ are by the same arguments as in Propositions 7.1.3 and 7.4.5 in \cite{dgind}, noting that the assertions are smooth local and using a smooth atlas instead of a Zariski atlas.
\end{proof}

In view of this result, we will often implicitly identify the categories of ind-coherent sheaves $\IndCoh(\wh{X}_Z)$ and $\IndCoh_Z(X)$.  For quasi-coherent sheaves, we implicitly identify $\QCoh(\wh{X}_Z)$ and $\QCoh_Z(X)$ via the left adjoint $\wh{i}_+$ and not the non-continuous right adjoint $\wh{i}_*$. 

\subsubsection{} We end this subsection with two remarks.  
\begin{rmk} First, by the above presentation, we see that the compact objects of $\IndCoh(\wh{X}_Z)$ are just $\Coh(\wh{X}_Z)$, the category of coherent complexes on $X$ set-theoretically supported on $Z$.  This category need not include the dualizing complex $\omega_{\wh{X}_Z}$ which is in general not finitely generated. 
\end{rmk}

\begin{rmk}
Second, for any prestack, in particular on a formal stack $\mf{X}$, the categories $\QCoh(\mf{X})$ and $\IndCoh(\mf{X})$ are equipped with canonical $t$-structures with the property that for any nilthickening $i: Z \hookrightarrow X$, the functor $i_*$ is $t$-exact (for both $\QCoh$ and $\IndCoh$).  When $\wh{i}: \mf{X} = \wh{X}_Z \hookrightarrow X$ is a formal completion, the functor $\wh{i}^*: \QCoh(X) \ra \QCoh(\wh{X}_Z)$ is right $t$-exact and its right adjoint $\wh{i}_*$ is $t$-exact (again for both).  For ind-coherent sheaves, the functor $\wh{i}^!: \IndCoh(X) \ra \IndCoh(\wh{X}_Z)$ is left $t$-exact and its left adjoint $\wh{i}_*$ is $t$-exact. \end{rmk}

\subsection{Continuous ind-coherent sheaves}\label{sec completions}

\subsubsection{}
We now define the categories of continuous ind-coherent and pro-coherent sheaves on formal stacks.
\begin{defn}\label{wcoh defn}
Let $\mf{X}$ be a formal scheme which is laft over a classical Noetherian ring $R$.  We define the categories of \emph{continuous ind-coherent sheaves} and \emph{continuous pro-coherent sheaves} to be the respective full subcategories
$$\hCohsh(\mf{X}) \subset \APerf^!(\mf{X}), \;\;\;\;\;\;\;\;\; \hCohst(\mf{X}) \subset \APerf^*(\mf{X})$$
consisting of $t$-bounded objects for the canonical $t$-structures.  When omitting the superscript, we default to the $!$-version.  When $\mf{X} = \wh{X}_Z$ is a formal completion, we sometimes abusively write $\hCoh_Z(X) := \hCoh(\wh{X}_Z)$.
\end{defn}

These categories may be viewed as Grothendieck dual versions of the derived category of coherent sheaves on a formal completion in the sense of \cite[\href{https://stacks.math.columbia.edu/tag/0EHN}{Tag 0EHN}]{stacks-project}.   Note that our notation is chosen to be compatible with that in \cite{DAG}; the category $\Coh(\wh{X}_Z)$ in \cite{HLP} corresponds to our category $\hCohst$.

\begin{rmk}\label{rmk t structures}
We can realize the above categories as full subcategories:
$$\hCohst(\wh{X}_Z) \subset \APerf^*(\wh{X}_Z) \subset \QCoh(\wh{X}_Z), \;\;\;\;\;\;\;\;\;\; \hCohsh(\wh{X}_Z) \subset \APerf^!(\wh{X}_Z) \subset \IndCoh(\wh{X}_Z).$$ 
These subcategories are defined by a boundedness condition on the $t$-structure and coherence of cohomology.  Since $\wh{i}_*:$ is $t$-exact for the canonical $t$-structure in either setting, we can easily interpret the $t$-boundedness condition within $\IndCoh_Z(X)$ and $\QCoh_Z(X)$.  For the almost perfect condition, we require that $\cF \in \IndCoh_Z(X)$ (resp. $\QCoh_Z(X)$) has the property that for any closed immersion $i': Z' \hookrightarrow X$ with set-theoretic image contained in $|Z|$, we have that $i'^!\cF \in \IndCoh(Z')$ (resp. $i'^*\cF \in \QCoh(Z')$) has coherent (but not necessarily bounded) cohomology.
\end{rmk}

\medskip

We next explain an instructive example of the category $\hCoh(\wh{X}_Z)$ and its relatives.
\begin{exmp}
Let $A = k[x]$ and $I = (x)$, so that $\wh{A} = k[\![x]\!]$.  Take $X = \Spec(A)$, $Z = \Spec(A/I)$.
\begin{enumerate}[(1)]
    \item The category of quasi-coherent sheaves $\QCoh(\wh{X}_Z)$ has two embeddings into $\QCoh(X)$: the essential image of $\wh{i}_*$ is the category of derived $x$-complete $k[x]$-complexes, i.e., objects $M$ of $\QC(X)$ such that $$M \simeq \wh{M} := \wh{i}_*\wh{i}^*M,$$while the essential image of $\wh{i}_+$ is $\QCoh_Z(X)$, the category of locally $x$-nilpotent $k[x]$-complexes.

    The category of ind-coherent sheaves $\IndCoh(\wh{X}_Z)$ identifies with $\IndCoh_Z(X)$, the derived category of $k[x]$-modules which are locally $x$-nilpotent.
    \item The category of coherent sheaves $\Coh_Z(X)$ identifies with the bounded derived category of finitely generated locally $x$-nilpotent $k[x]$-modules, and compactly generates $\IndCoh(\wh{X}_Z)$.  It contains the augmentation module $k$, but does not contain the dualizing sheaf $$\omega_{\wh{X}_Z} \simeq k(\!(x)\!)dx/k[\![x]\!]dx$$nor the completed structure sheaf $\cO_{\wh{X}_Z} \simeq k[\![x]\!]$.
    \item The category of continuous pro-coherent sheaves $\hCohst(\wh{X}_Z)$ identifies via $\wh{i}_*$ with the pretriangulated hull of the completed structure sheaf, i.e.,  derived $x$-complete $k[x]$-modules which are perfect as $k[\![x]\!]$-modules.  This category includes the completed structure sheaf, but not the dualizing sheaf, nor infinite sums of the completed structure sheaf.
    \item The category of continuous ind-coherent sheaves $\hCohsh(\wh{X}_Z)$ identifies via $\wh{i}_*$ with the $k[x]$-modules which are locally $x$-nilpotent such that the kernel of the $x$-action is perfect over $A/I = k$.  This includes the dualizing sheaf, but not the structure sheaf, nor infinite sums of the dualizing sheaf.
\end{enumerate}
Informally, one can think of objects of $\hCohsh(\wh{X}_Z)$ as ind-finite with respect to the $I$-adic topology, while objects of $\hCohst(\wh{X}_Z)$ may be thought of as being pro-finite.
\end{exmp}

Finally, the extension of Grothendieck-Serre duality (Proposition \ref{serre}) to almost-perfect complexes on formal completions is immediate given the above description and definitions (see also Section 2.6 of \cite{dgind}).
\begin{cor}\label{inf serre}
Assume that $X$ is a geometric Artin stack, almost of finite type over a Noetherian affine scheme $K = \Spec R$ admitting a dualizing complex, and $Z \subset X$ a closed substack over $K$.  There is a $t$-bounded (for the canonical $t$-structures) Grothendieck duality equivalence
$$\bD_{\wh{X}_Z}: \begin{tikzcd}\APerf^*(\wh{X}_Z)^{\opp} \arrow[r, "\simeq"] & \APerf^!(\wh{X}_Z) \end{tikzcd} $$
defined compatibly with the functors (where $i_\alpha: Z_\alpha \hookrightarrow \wh{X}_Z$ is as in Proposition \ref{dgind stack}):
$$\begin{tikzcd}
\APerf^*(Z_\alpha)^{\opp} \arrow[r, shift left, "i_{\alpha*}^{\opp}"] \arrow[d, "\bD_{Z_\alpha}"'] & \APerf^*(\wh{X}_Z)^{\opp} \arrow[l, shift left, "i_\alpha^{*,\opp}"] \arrow[r, hook, "\wh{i}_*^{\opp}", shift left] \arrow[d, "\bD_{\wh{X}_Z}"'] & \APerf^*(X)^{\opp} \arrow[d, "\bD_X"] \arrow[l, shift left, "\wh{i}^{*, \opp}"] \\
\APerf^!(Z_\alpha)  \arrow[r, shift left, "i_{\alpha*}"] & \APerf^!(\wh{X}_Z) \arrow[r, hook, "\wh{i}_*", shift left] \arrow[l, shift left, "i_\alpha^{!}"]  & \APerf^!(X). \arrow[l, shift left, "\wh{i}^{\,!}"]
\end{tikzcd}$$
In particular, $\bD_{\wh{X}_Z}$ restricts to an equivalence $\hCohst(\wh{X}_Z) \simeq \hCohsh(\wh{X}_Z)$, and $\bD_{\wh{X}_Z}(\cO_{\wh{X}_Z}) \simeq \omega_{\wh{X}_Z}$.
\end{cor}
\begin{proof}
We use the compatibility of Grothendieck-Serre duality with pullbacks in Proposition \ref{serre} to conclude the existence of the functor.  The $t$-boundedness claims follow via Remark \ref{rmk t structures}.
\end{proof}

\subsubsection{} 
We now establish a few basic properties of the category of continuous ind-coherent sheaves.
\begin{lemma}\label{wcoh prop}
In the set-up of Corollary \ref{inf serre}, we have the following:
\begin{enumerate}[(a)]
\item The functor $\bD_{\wh{X}_Z}$ defines an equivalence $\hCohst(\wh{X}_Z)^{\opp} \simeq \hCohsh(\wh{X}_Z)$.
\item Assume that $X - Z$ is quasi-compact.  If $\cF \in \Coh(X)$, then $\wh{i}^{\,!}\cF \in \hCoh(\wh{X}_Z)$.
\item We have $\Coh_Z(X) \simeq \Coh(\wh{X}_Z) \subset \hCoh(\wh{X}_Z)$.
\item The category $\hCoh(\wh{X}_Z)$ is small.
\end{enumerate}
\end{lemma}
\begin{proof}
Statement (a) follows since $\bD_{\wh{X}_Z}$ is $t$-bounded by Corollary \ref{inf serre}.   
For (b), note that $\wh{i}^{\,!}\cF \in \APerf^!(\wh{X}_Z)$  by construction. To see its $t$-boundedness, consider the open embedding $$j: X - Z \hookrightarrow X,$$and the associated distinguished triangle$$\wh{i}_*\wh{i}^{\,!}\cF  \rightarrow \cF \rightarrow j_*j^* \cF \xrightarrow{+1}.$$ To see the $t$-boundedness of $i^! \cF$, it is therefore enough to see the same holds for $j_* j^* \cF$, but the latter may be computed via a bounded Cech complex since $X -Z$ is finite type over affine by assumption.  Statement (c) is clear. To see statement (d), note that the category of almost perfect complexes on a derived scheme is small, and a countable limit, or more generally any small limit, of small categories is small.
\end{proof}

\subsubsection{}  Given an ind-coherent sheaf on a formal completion $\wh{X}_Z$, it is not immediately obvious from the definition how to effectively determine whether it is continuous or not, as {\em a priori} one has to check a coherence condition on every nil-thickening of $Z$, as well as the $t$-boundedness of the original object.   

As we will justify below,  it in fact suffices to check the coherence condition for the $!$-restriction to  $Z$ itself, and no other nil-thickening, provided $Z \hookrightarrow X$ is a quasi-smooth closed embedding. Moreover, under the same assumption on $Z$, one can verify the  $t$-boundedness condition after $!$-restriction to $Z$ as well.  We emphasize that the appearance of quasi-smoothness does not result in any loss of generality, as any closed substack has a quasi-smooth approximation with the same underlying subset.

That is, we will show the following. 

\begin{thm}\label{wcoh prop qs}
Let $X$ be a laft Artin stack over a classical Noetherian ring $R$, and $i: Z \hookrightarrow X$ a closed substack.  If $i$ is quasi-smooth, then for any $\cF \in \IndCoh(\wh{X}_Z)$, we have that $\cF \in \hCoh(\wh{X}_Z)$ if and only if $i^! \cF \in \Coh(Z)$.
\end{thm}

\sss
We now record the main ingredients for the proof 
 in the following two lemmas.

\begin{lemma}\label{reflect t boundedness}
Assume the set-up of Theorem \ref{wcoh prop qs}.  Then the functor $i^!: \IndCoh(\wh{X}_Z) \rightarrow \IndCoh(Z)$ preserves and reflects left $t$-boundedness.  Furthermore, assuming that the closed embedding $i: Z \hookrightarrow X$ is quasi-smooth, the functor $i^!$ preserves and reflects $t$-boundedness.
\end{lemma}
\begin{proof}
It is standard that $i^!$ preserves left $t$-boundedness for general $i$.  When $i$ is quasi-smooth, $i^!$ preserves $t$-boundedness by a standard local calculation using Koszul complexes, which are bounded.

We now show that $i^!$ reflects left $t$-boundedness, i.e. that if $i^! \cF \in \IndCoh(Z)^{\geq 0}$, then $\cF \in \IndCoh_Z(X)^{\geq 0}$. Equivalently, we must show for any $\cG \in \IndCoh_Z(X)^{\leq -1}$ that \begin{equation} \label{e:bronto}\pi_0(\Hom(\cG, \cF)) \simeq 0.\end{equation} We first reduce to the case where $\cG$ is compact.  Since the $t$-structure is compatible with filtered colimits by construction, cf. Section 1.2.1 of \cite{indcoh}, we may write $\cG = \colim_\alpha \cG_\alpha$ for $\cG_\alpha \in \Coh_Z(X)^{\leq -1}$, and note that $$\Hom(\cG, \cF) \simeq \lim_\alpha \Hom(\cG_\alpha, \cF).$$There is a spectral sequence with $E_2$ page $R^p \lim_\alpha \Ext_X^{q}(\cG_\alpha, \cF)$ converging to the cohomology of this complex. In particular, on the zeroth diagonal, 
we encounter the terms $R^p\lim_\alpha \Ext^{-p}(\cG_\alpha, \cF)$. As we have assumed the vanishing of $\Ext_X^q(\cG_\alpha, \cF)$ for all $q \leq 0$, these graded pieces all vanish and we conclude that \eqref{e:bronto} holds, as desired.

Thus, we may assume that $\cG$ is compact.  Noting that any  object in $\Coh_Z(X)^{\leq -1}$ may be written as a finite extension of objects in the essential image of $i_*: \Coh(Z)^{\leq -1} \rightarrow \Coh_Z(X)^{\leq -1}$,   by devissage we may reduce to checking vanishing for objects of the form $i_* \cG$ for $\cG \in \Coh(Z)^{\leq -1}$.  But our assumption that $i^! \cF \in \IndCoh(Z)^{\geq 0}$ implies, via adjunction, the desired vanishing $$\pi_0 \Hom(i_* \cG, \cF) = \pi_0 \Hom(\cG, i^! \cF) \simeq 0.$$Hence, we conclude that $i^!$ reflects left $t$-boundedness, as desired.

Next, we show that $i^!$ reflects $t$-boundedness under the assumption of quasi-smoothness.  
  As the assertion is local, we may assume that $Z$ is the derived vanishing locus of a sequence of functions $f_1, \cdots, f_r$ on $X$.  For each $0 \leq i \leq r$, let $Z_i$ denote the derived vanishing locus of $f_1, \cdots, f_i$, so that $Z_0 = X$ and $Z_r = Z$.  We therefore obtain a sequence of quasi-smooth closed immersions
$$Z_r = Z \hookrightarrow Z_{r-1} \hookrightarrow \cdots \hookrightarrow Z_1 \hookrightarrow Z_0 = X.$$
By induction on the codimension $k$, it is enough to show that if the $!$-restriction of $\cF$ to $Z_k$ is $t$-bounded for some $k, 1 \leqslant k \leqslant r$, then so is the $!$-restriction of $\cF$ to $Z_{k-1}$.

We induct on this sequence; suppose we know that the $!$-restriction to $Z_k$ is $t$-bounded, and consider the (abusively denoted) $!$-restriction functor $i^!: \IndCoh_Z(X_{k-1}) \rightarrow \IndCoh_Z(X_k)$. Assume that $i^! \cF \in \IndCoh_Z(X_k)$ is $t$-bounded.  Since we have already shown that $!$-pullbacks reflect left $t$-boundedness, we may assume that $\cF \in \IndCoh_Z(X_{k-1})^+$.  Ind-coherent sheaves and quasi-coherent sheaves are equivalent\footnote{Under the inductive $t$-structure, which is compatible with the canonical $t$-structure on $\IndCoh_Z(X_{k-1})$, and where we identify $\QCoh(\wh{X_{k-1}}_Z) \simeq \QCoh_Z(X_{k-1})$ via the $t$-exact functor $\wh{i}_+$.} on eventually coconnective objects, so we may take $\cF \in \QCoh_Z(X_{k-1})^+$.  In the setting of quasi-coherent sheaves, we have $i^!\cF = \mathrm{fib}(f: M \rightarrow M)$, thus there is a spectral sequence converging to $i^! \cF$ degenerating at the page $$E_2^{p,0} = \ker(f: H^p(\cF) \rightarrow H^p(\cF)) \quad \text{and} \quad E_2^{p,1} = \coker(f: H^p(\cF) \rightarrow H^p(\cF)).$$Note that since $i^!$ is conservative (i.e. since $Z \subset X_{k-1}$), we have that $E_2^{p,0} = E_2^{p,1} = 0$ if and only if $H^p(\cF) = 0$.  Thus $i^!\cF$ can only be right $t$-bounded if $H^p(\cF) = 0$ for $p \gg 0$, i.e. if $\cF$ is right $t$-bounded, establishing the inductive step, and the claim.
\end{proof}

\begin{rmk}
One can understand Lemma \ref{reflect t boundedness} instead via Koszul duality.  Let us assume $Z$ is smooth, so that $\IndCoh(Z) = \QCoh(Z)$.  In this case, $i^!: \IndCoh(\wh{X}_Z) \rightarrow \IndCoh(Z)$ naturally upgrades via Barr-Beck-Lurie to a functor 
$$\wt{i}^!: \IndCoh(\wh{X}_Z) \rightarrow i^!i_*\dmod_{\IndCoh(Z)} \simeq R\End_X(i_*\cO_Z)\dmod_{\IndCoh(Z)}.$$ 
Here, we use that the monad $i^!i_*$ is monoidal, therefore is determined by its value on the monoidal unit $\cO_Z$.  By the quasi-smoothness assumption, this sheaf of dg algebras can be computed locally by a Koszul resolution, thus is a semi-free dg algebra (over $\cO_Z$) with generators in cohomological degree $1$.  This algebra is bounded with coherent cohomology, so $i^! \cF$ is bounded with coherent cohomology if and only if $H^\bullet(\wt{i}^!\cF)$ is finitely generated as a $H^\bullet(R\End_X(i_*\cO_Z))$-module.

A result similar appears as Theorem 3.2 in \cite{raskin dg lie} for the case of a ``quasi-quasi-smooth closed embedding'' of a point $\{x\} \hookrightarrow X$, i.e. such that the relative cotangent complex is perfect in degrees $[-2, \infty)$ rather than $[-1, \infty)$.  It is interesting to ponder generalizations of Lemma \ref{reflect t boundedness} where one removes the quasi-smoothness assumption.
\end{rmk}

\begin{lemma}\label{technical lemma coh com}
Let $X$ be a locally Noetherian Artin stack and let $i: Z \hookrightarrow X$ be a derived nilthickening, i.e. $\pi_0(i): \pi_0(Z) \hookrightarrow \pi_0(X)$ is defined by a nilpotent ideal sheaf.  Assume that $\cF \in \IndCoh(X)$ is left $t$-bounded, i.e. $\cF \in \IndCoh(X)^+ = \QCoh(X)^+$.  Then, $\cF$ has coherent cohomology if and only if $i^!\cF$ has coherent cohomology.
\end{lemma}
\begin{proof}
We may assume that $X = \Spec A$ is a Noetherian affine derived scheme for a dg ring $A$, and let $Z = \Spec A/I$.  That $i^!$ preserves coherence of cohomology follows since for an $A$-module $M$, we have $i^! M = \Hom_A(A/I, M)$, and one can take a semi-free $A$-resolution of $A/I$ which is finite rank in each degree by the Noetherian assumption.  

We now show that $i^!$ reflects coherence of cohomology.  Let us first consider the case where $Z, X$ are classical, and $M$ is an $A$-module in degree 0 such that $R^ki^!M$ is coherent for every $k$.  In fact, we will only need the weaker assumption that $R^0 i^! M$ is coherent.  In this case, there is a finite filtration $$0 = M_0 \subset M_1 \subset \cdots \subset M_r = M$$such that $I$ acts trivially on the successive quotients $M_i/M_{i-1}, 1 \leqslant i \leqslant r$.  Thus, $R^0i^!(M_i/M_{i-1}) = M_i/M_{i-1}$, and by the long exact sequence obtained by applying $Ri^!$ to $$0 \rightarrow M_1 \rightarrow M \rightarrow M/M_1 \rightarrow 0,$$we find that $M_1$ is coherent since it is a submodule of $R^0i^!M$, which is coherent by assumption.  Since $R^\bullet i^!$ preserves coherence, this in turn implies that $R^1i^! M_1$ is coherent, thus $R^0i^!(M/M_i)$ is coherent, and we may induct on $i$, i.e. consider $M_2/M_1 \subset M/M_1$, et cetera.  We conclude that each subquotient is coherent, thus $M$ is.

Now, we assume that $M$ is a left $t$-bounded complex, say in degrees $\geq 0$.  In this case we have an exact triangle $$H^0(M) \rightarrow M \rightarrow M^{\geq 1} \xrightarrow{+1},$$ thus an exact triangle $i^! H^0(M) \rightarrow i^! M \rightarrow i^! M^{\geq 1} \xrightarrow{+1}$.  For degree reasons, we have $R^0i^! H^0(M) = H^0(i^!M)$, which is coherent by assumption.  Thus, $H^0(M)$ is coherent, by the above argument.  Since $i^!$ preserves coherence, $i^! H^0(M)$ has coherent cohomology.  By the long exact sequence, $i^!M^{\geq 1}$ has coherent cohomology.  We may be continue inducting on the degree, i.e. consider the exact triangle $$H^1(M) \rightarrow i^! M^{\geq 1} \rightarrow i^! M^{\geq 2} \xrightarrow{+1},$$ to conclude that $H^i(M)$ is coherent for all $i$, thus $M$ has coherent cohomology.

Now, we remove the assumption that $Z, X$ are classical.  In this case, we have a closed immersion $i_0: \pi_0(Z) \hookrightarrow Z$.  Let $M \in \IndCoh(X)^+$ such that $i^!M \in \IndCoh(Z)$ has coherent cohomology.  Then, $$i_0^! i^! M \in \IndCoh(\pi_0(Z))$$has coherent cohomology, so we have reduced to the case where $Z$ is classical.  Furthermore, via the factorization $$\begin{tikzcd}
\pi_0(Z) \arrow[r, "z_0", hook] &  \pi_0(X) \arrow[r, "z"] & X
\end{tikzcd}$$and using the fact that $!$-restriction reflects coherence for classical nilthickenings, we may conclude that $z^! \cF$ has coherent cohomology, thus reducing showing that $!$-restriction along $i: Z = \pi_0(X) \hookrightarrow X$ reflects coherence.  Without loss of generality, we may take $M \in \IndCoh(X)^{\geq 0}$; there is an exact triangle $$H^0(M) = i_* M'\rightarrow M \rightarrow M^{\geq 1} \xrightarrow{+1}$$for some $M' \in \IndCoh(X)^\heartsuit = \IndCoh(\pi_0(X))^\heartsuit$.  This gives rise to an exact triangle $$i^! H^0(M) = i^! i_* M' \rightarrow i^! M \rightarrow i^! M^{\geq 1} \xrightarrow{+1}.$$ Note that we have $R^0i^!(i_* M') = i_*M'$, and since $H^0(i^! M)$ is coherent by assumption, this implies that $i_* M' = H^0(M)$ is coherent by the associated long exact sequence.  Since $i^!$ preserves coherence, we have that $i^!H^0(M)$ has has coherent cohomology, and we may induct as before on the cohomological degree to conclude that $H^i(M)$ is coherent for all $i$.  This completes the proof.
\end{proof}

\begin{proof}[Proof of Theorem \ref{wcoh prop qs}]
Suppose $\cF \in \hCoh(\wh{X}_Z)$; then $i^!\cF \in \Coh(Z)$ is $t$-bounded by Lemma \ref{wcoh prop} and has coherent cohomology by Lemma \ref{technical lemma coh com}.  Now, suppose $i^! \cF \in \Coh(Z)$.  Then, $\cF$ is $t$-bounded by Proposition \ref{wcoh prop}.  Next, for any closed embedding $i': Z' \hookrightarrow \wh{X}_Z$, we can locally choose a nilthickening $i'': Z''\hookrightarrow \wh{X}_Z$ containing both $Z$ and $Z'$.  We apply Lemma \ref{technical lemma coh com} to conclude that $i''^!\cF$ has coherent cohomology, and then again to conclude that $i'^!\cF$ has coherent cohomology.
\end{proof}

\begin{rmk}\label{star version}
The proof and statement of Lemma \ref{technical lemma coh com} works in the setting of $*$-pullbacks, and replacing left $t$-boundedness with right $t$-boundedness, i.e. taking $\cF \in \QCoh(X)^-$.  Lemma \ref{reflect t boundedness} holds with the exception that $*$-restriction does not reflect right $t$-boundedness.  For example, take $$i: Z = \Spec k \hookrightarrow X = \Spec k[t]/t^2$$and let $M = \colim_{n \in \bZ^{\geq 0}} k[n] \in \IndCoh(X)$ denote the usual acyclic injective complex.  This complex is not right $t$-bounded, but one can verify that $i^* M = 0$.
\end{rmk}

\subsubsection{} Finally, we relate the category $\hCoh(\wh{X}_Z)$ to more familiar categories, e.g. coherent sheaves on the non-finite type stacks in Example \ref{complete stack} which appear in \cite{roma hecke}, via a derived variant of the Grothendieck existence theorem due to Halpern-Leistner--Preygel.
\begin{thm}[Grothendieck existence, \cite{HLP}]\label{groth exist}
Let $G$ be a reductive group acting on a Noetherian derived scheme $X$, let $i: S_0 \hookrightarrow S$ be a closed immersion of Noetherian schemes, and let $f: X/G \rightarrow S$ be a map of stacks.  Consider the Cartesian diagram
$$\begin{tikzcd}
X_0/G \arrow[r, "i'"] \arrow[d] & X/G \arrow[d, "f"] \\
S_0 \arrow[r, "i", hook] & S.
\end{tikzcd}$$
Suppose that
\begin{enumerate}[(a)]
\setlength{\itemindent}{2em}
\item the induced map $X \rightarrow S$ is (relatively) finite type projective-over-affine,
\item the Noetherian scheme $S$ admits a dualizing complex (and fix one such dualizing complex $\omega_S$),
\item $X$ admits a $G$-equivariant line bundle which is ample relative to $S$,
\item $H^0(f_*\cO_{X/G}) \in \QCoh(S)^\heartsuit$ is finitely generated as a module,
\item the functor $\wh{i}^*: \APerf^*(S) \rightarrow \APerf^*(\wh{S}_{S_0})$ is an equivalence.
\end{enumerate}
Then, we have compatible equivalences
$$\wh{i'}^*: \APerf(X/G) \overset{\simeq}{\longrightarrow} \APerf(\wh{X/G}_{X_0/G}), \;\;\;\;\;\;\;\;\;\;  \wh{i'}^*: \Coh(X/G) \overset{\simeq}{\longrightarrow} \hCohst(\wh{X/G}_{X_0/G}).$$
If we are given a choice of dualizing complex $\omega_S$ and the map $f$ is laft, then under the resulting Grothendieck duality we have equivalences 
$$\wh{i'}^!: \APerf^!(X/G) \overset{\simeq}{\longrightarrow} \APerf^!(\wh{X/G}_{X_0/G}), \;\;\;\;\;\;\;\;\;\; \wh{i'}^!: \Coh(X/G) \overset{\simeq}{\longrightarrow} \hCohsh(\wh{X/G}_{X_0/G}).$$
\end{thm}
\begin{proof}
That $\wh{i'}^*$ is an equivalence is a special case of Theorem 4.2.1 of \cite{HLP}, applying the criterion for a stack to be cohomologically projection in Proposition 4.2.3 of \emph{op. cit}.  The fact that the restriction to $t$-bounded objects is an equivalence follows from Lemma \ref{wcoh prop} and $t$-exactness of $\wh{i}_*$ (i.e. the adjoint equivalences restrict to $t$-bounded objects).  We apply Corollary \ref{inf serre} to deduce the final claim.
\end{proof}

In our cases of interest, we often take $S$ to be the spectrum of a completed ring, i.e. 
$$S_0 := \Spec A/I \subset  S := \Spec \lim A/I^n$$
where $A$ is a finite-type $k$-algebra and $I \subset A$ an ideal.  In this setting, it is known that the condition (e) above is satisfied (see Example 1.1.2 and Proposition 2.1.4 of \cite{HLP}), i.e. we have that $\hCoh(\wh{S}_{S_0}) \simeq \Coh(S)$.

\subsection{Singular support and ind-coherent sheaves}\label{app ss}

\subsubsection{}
We briefly recall the notion of singular support of coherent sheaves on quasismooth Artin stacks from \cite{AG}, and refer the reader to \emph{op. cit.} for details.
\begin{defn}\label{def ss}
We say a derived Artin stack $X$ is \emph{quasi-smooth} (or derived lci) if its cotangent complex $\bL_X$ is perfect in degrees $[-1, \infty)$ (see Lemma 8.1.2 in \emph{op. cit.}).  Any Artin stack with a quasi-smooth atlas is quasi-smooth.  More generally, we say a map $f: X \rightarrow Y$ of (derived) Artin stacks is quasi-smooth if the above is true for its relative cotangent complex $\bL_f$.  Let $X$ be a quasi-smooth (derived) Artin stack.  
\begin{enumerate}[(1)]
    \item We define the (classical) \emph{odd cotangent space} by\footnote{Note that $\bL_X^\vee$ is, by assumption, is perfect in degrees $(-\infty, 1]$, thus $H^1(\bL_X^\vee)$ is a quotient of $\bL_X^\vee$.}
$$\Sing_X := \pi_0(\Spec_X(\Sym_X \bL_X^\vee[1])) = \Spec_X \Sym_X H^1(\bL_X^\vee).$$
Given a fixed presentation of $\bL_X$ as a complex of locally free sheaves whose $(-1)$-term is $\cE$, we obtain a closed embedding $\Sing_X \subset \mathbb{E}_X := \Spec_X \Sym_X \cE^\vee$.  
\item To any ind-coherent sheaf $\cF \in \IndCoh(X)$, we have a closed conical subvariety called the \emph{singular support} $\on{SS}(\cF) \subset \Sing_X$.  For any closed conical subset $\Lambda \subset \Sing_X$, we have a full subcategory $\IndCoh_\Lambda(X) \subset \IndCoh(X)$ of ind-coherent sheaves with singular support inside $\Lambda$, which is compactly generated by $\Coh_\Lambda(X) \subset \Coh(X)$, the full subcategory of coherent sheaves with same restriction on singular support.
\item If $f: X \rightarrow Y$ is a schematic map between quasi-smooth derived Artin stacks, then we have the following functorial properties of singular support (see Section 7.1 of \cite{AG}).  Consider the correspondence
$$\begin{tikzcd}
\Sing_X & \Sing_Y \times_Y X \arrow[l, "df"'] \arrow[r, "\pi_f"] & \Sing_Y.
\end{tikzcd}$$
Then, for $\Lambda_X \subset \Sing_X$ and $\Lambda_Y \subset \Sing_Y$, we define
$$f(\Lambda_X) := \overline{\pi_f(df^{-1}(\Lambda_X))}, \;\;\;\;\;\;\; f^{-1}(\Lambda_Y) := \overline{df(\pi_f^{-1}(\Lambda_Y))}.$$
By Proposition 7.1.3 of \emph{op. cit.}, these singular support conditions are preserved by pushforward and pullback functors.
\end{enumerate}
\end{defn}

The following is Corollary 4.5.10 of \cite{AG}, which will allow us to pass between categories with different support conditions.
\begin{prop}\label{nil to coh}
Let $X$ be a quasi-smooth stack, and let $\Lambda \subset \Sing_X$ be a singular support condition containing the zero section.  Then, there are natural adjoint functors 
$$\iota_\Lambda: \begin{tikzcd} \IndCoh_\Lambda(X) \arrow[r, shift left] & \IndCoh(X) \arrow[l ,shift left] \end{tikzcd} :\Gamma_\Lambda$$
such that $\iota_\Lambda$ is fully faithful and $t$-exact, and $\Gamma_\Lambda \circ \iota_\Lambda = \mathrm{id}_{\IndCoh_\Lambda(X)}$. 
\end{prop}

If $\Lambda_1 \subset \Lambda_2$ are nested singular support conditions, we can also use the above functors to pass between $\IndCoh_{\Lambda_1}(X)$ and $\IndCoh_{\Lambda_2}(X)$.  Furthermore, we often abusively denote the endofunctor $\iota_{\Lambda_1}^{\Lambda_2} \Gamma_{\Lambda_1}^{\Lambda_2}: \IndCoh_{\Lambda_2}(X) \rightarrow \IndCoh_{\Lambda_2}(X)$ by simply $\Gamma_{\Lambda_1}$. 

\begin{rmk}\label{renormalize indcoh}
We have the following analogue of the criterion established in Theorem \ref{wcoh prop qs}, now in the setting of singular supports rather than classical supports.  Let $i: Z \hookrightarrow X$ be a quasi-smooth closed immersion of quasi-smooth QCA stacks, and let $\Lambda_{X} := i^!\Sing_{X}$.  In particular, $\{0\}_Z \subset \Lambda_{X}$, so the fully faithful functor $\Perf(Z) \hookrightarrow \Coh_{\Lambda_{X}}(Z)$ induces a fully faithful functor $\QCoh(Z) \hookrightarrow \IndCoh_{\Lambda_{X}}(Z)$, thus realizing $\Coh(Z)$ as a full subcategory of $\IndCoh_{\Lambda_{X}}(Z)$ (of possibly non-compact objects).  Consider the adjoint functors
$$\begin{tikzcd}
\IndCoh_{\Lambda_{X}}(Z) \arrow[r, "i_*", shift left] & \IndCoh(X) \arrow[l, "i^!", shift left].
\end{tikzcd}$$
We have the following characterizations of certain small subcategories in terms of these functors (see Proposition 7.6.4 of \cite{AG}).
\begin{enumerate}[(1)]
    \item The subcategory $\Coh(Z) \subset \IndCoh_{\Lambda_{X}}(Z)$ of possibly non-compact objects is characterized as the full subcategory consisting of objects whose image under $i_*$ lands in $\Coh(X) \subset \IndCoh(X)$.
    \item The subcategory of compact objects $\Coh_{\Lambda_{X}}(Z) \subset \IndCoh_{\Lambda_{X}}(Z)$ is characterized as the full subcategory generated by the essential image of $i^!$ applied to compact objects, i.e. $i^! \Coh(X) \subset \IndCoh(Z)$.
\end{enumerate}

Koszul dually, still assuming that $i: Z \hookrightarrow X$ is quasi-smooth, we now consider the classical support condition defined by $Z$.  Then, we have adjoint functors:
$$\begin{tikzcd}
\IndCoh(Z) \arrow[r, "i_*", shift left] & \arrow[l, "i^!", shift left] \IndCoh_Z(X)
\end{tikzcd}$$
which pick out the following small subcategories.
\begin{enumerate}[(1)]
    \item By Theorem \ref{wcoh prop qs} the subcategory $\hCoh_Z(X) \subset \IndCoh_Z(X)$ of possibly non-compact objects is characterized as the full subcategory consisting of objects whose image under $i^!$ lands in $\Coh(Z)$.
    \item The subcategory of compact objects $\Coh_Z(X) \subset \IndCoh_Z(X)$ is characterized as the full subcategory generated by the essential image of $i_*$ applied to compact objects, i.e. $i_* \Coh(Z) \subset \IndCoh(\wh{X}_Z)$.
\end{enumerate}
\end{rmk}

\section{Convolution categories and spectral affine Hecke categories}\label{spectral hecke sec}

\subsection{Overview}

\sss
In this section we continue our discussion of the spectral side of the equivalences.

In the previous section, we reviewed and established relevant results about individual categories of ind-coherent sheaves. Here, we study the intertwining operators between such categories in good situations, where they identify with categories of ind-coherent sheaves on fiber products, perhaps with some prescribed singular support condition. Explicitly, these act via integral transforms, and the composition is given by  convolution.

In Section \ref{convolution sec} we prove basic structural results on these categories. The main new theorem here is Theorem \ref{convolution composition}, which identifies the relative tensor product of two composable convolution categories with another convolution category with a singular support condition.   In Section \ref{convolution hecke} we apply this formalism to spectral affine Hecke categories, i.e. categories of ind-coherent sheaves on variants of Steinberg stacks.  In particular, in Proposition \ref{reduction using monoidal} we explain how nilpotent singular and classical support conditions naturally arise when taking relative tensor products over monodromic and strict affine Hecke categories respectively.

In Section \ref{sec algebra obj} we identify certain natural monads on convolution categories and the algebra objects they correspond to.  In Section \ref{spectral hecke algebra} we apply this to characterize partial spectral affine Hecke categories as modules for certain algebra objects in the full affine Hecke category, cf. Corollary \ref{completed spectral coalgebra}.   Finally,  in Section \ref{sec spectral monodrom} we describe a cohomological Koszul duality that allows us to pass between ind-coherent sheaves on certain Steinberg stacks and their formal thickenings.

\subsection{Algebro-geometric convolution structures}\label{convolution sec}

\sss

Our starting point is  the following special case of Theorem 1.1.3 or Theorem 3.0.2 of \cite{BNP convolution}, which we will refer to repeatedly and state for convenience.
\begin{thm}[Ben-Zvi, Nadler, Preygel]\label{thm bnp}
Let $S$ be a smooth Artin stack, and assume that $X, Y$ are smooth and perfect, with $X, Y$ proper over $S$.  The $!$-integral transform and $*$-integral transform induce equivalences of monoidal categories
$$\begin{tikzcd} \Coh(X \times_S Y) \arrow[r, "\simeq"', "\bF"] & \cat{Fun}_{\Perf(S)}(\Coh(X), \Coh(Y)),\end{tikzcd}.$$
Furthermore, these equivalences are natural in the sense that for any proper map $S \rightarrow S'$, we have an induced\footnote{I.e. since $S \rightarrow S'$ is proper it is separated, thus the diagonal is a closed immersion.} closed immersion $\iota: X \times_{S'} Y \rightarrow X \times_S Y$ and a commuting square
$$\begin{tikzcd}
\Coh(X \times_S Y) \arrow[r, "\bF"] \arrow[d, "\iota_*"] & \cat{Fun}_{\Perf(S)}(\Coh(X), \Coh(Y)) \arrow[d] \\
\Coh(X \times_{S'} Y) \arrow[r, "\bF"] & \cat{Fun}_{\Perf(S')}(\Coh(X), \Coh(Y)).
\end{tikzcd}$$
In particular, $\iota_*$ is monoidal.
\end{thm}

\subsubsection{}
The categories of integral transforms in the above theorem are examples of convolution categories; we now establish notations and conventions.
\begin{defn}\label{convolution setup}
Let $S$ be a smooth Artin stack.  Fix an indexing set $I$, and let $X_i$ ($i \in I$) be smooth and perfect stacks, such that the structure maps $p_i: X_i \rightarrow S$ are proper.  To any pair $(i,j)$ (possibly non-distinct) we define the corresponding \emph{convolution space} by $Z_{ij} := X_i \times_S X_j$, and for any triple $(i,j,k)$ (possibly non-distinct) we have a correspondence:
$$\begin{tikzcd}
Z_{ij} \times Z_{jk} = X_i \times_S X_j \times X_j \times_S X_k & \arrow[l, "\delta"'] \arrow[r, "\pi"] Z_{ijk} := X_i \times_S X_j \times_S X_k & Z_{ik} = X_i \times_S X_k.
\end{tikzcd}$$
For singular support conditions $\Lambda_{ij} \subset \Sing_{Z_{ij}}$, $\Lambda_{jk} \subset \Sing_{Z_{jk}}$.  We define the \emph{convolution of singular supports} by (see Section \ref{app ss}):
$$\Lambda_{ij} \star \Lambda_{jk} := \pi_*\delta^!(\Lambda_{ij} \boxtimes \Lambda_{jk}).$$
We note that the singular support condition includes a classical support condition.
\end{defn}

\begin{exmp}[Convolution of supports]
Consider the convolution spaces $Z_{12} = X_1 \times_S X_2$ and $Z_{23} = X_2 \times_S X_3$.  We will compute the convolution of their odd cotangent bundles $\Sing_{Z_{12}} \star \Sing_{Z_{23}}$.  Fix a $k$-point $\eta = (x_1,x_2,x_3) \in (Z_{123})(k)$ mapping to $s \in S(k)$, and let $dp_{x_i}^*: \bT^*_{S,s} \rightarrow \bT^*_{X_i,x_i}$ denote the map on cotangent spaces.  The presentation of the $Z_{ij}$ as fiber products give identifications
$$\Sing_{Z_{12} \times Z_{23}}|_{\delta(\eta)} = \{(\omega_{12}, \omega_{23}) \in (\bT^*_{S,s})^2 \mid dp_{x_1}^* \omega_{12} = dp_{x_2}^* \omega_{12} = 0, dp_{x_2}^*\omega_{23} = dp_{x_3}^*\omega_{23} = 0\},$$ 
$$\Sing_{Z_{123}}|_\eta = \{(\omega_{12}, \omega_{23}) \in (\bT^*_{S,s})^2 \mid dp_{x_1}^* \omega_{12} = 0, dp_{x_2}^*\omega_{23} = 0, dp_{x_2}^* \omega_{12} = dp_{x_2}^*\omega_{23}\},$$
$$\Sing_{Z_{13}}|_{\pi(\eta)} = \{\omega \in \bT^*_{S,s} \mid dp_{x_1}^* \omega = dp_{x_3}^*\omega = 0\}.$$
The map $\delta^!$ sends $(\omega_{12}, \omega_{23}) \mapsto (\omega_{12}, \omega_{23})$, and the map $\pi_*$ sends $\omega \mapsto (\omega, \omega)$.  Thus, over a $k$-point $\pi(\eta) = (x_1, x_3) \in (Z_{13})(k)$, we have
$$(\Sing_{Z_{12}} \star \Sing_{Z_{23}})|_{\pi(\eta)} = \bigcup_{x_2 \in p_2^{-1}(s)} \{\omega \in \bT^*_{S, s} \mid dp_{x_1}^* \omega = dp_{x_2}^*\omega = dp_{x_3}^* \omega = 0\} \subseteq  \Sing_{Z_{13}, \pi(\eta)}.$$
Note that for points not in the image of $\pi$, the singular support locus is empty.
\end{exmp}

\medskip
While we state a few basic results in greater generality, in practice all of our singular support conditions on the various convolution spaces arise via a single singular support condition on the base $S$ in the following way.
\begin{defn}\label{def support on base}
Let $\Lambda \subset \bT^*_S$ be a conical closed subset of the classical cotangent bundle of $S$.  Then for any $X_i, X_j$ over $S$ in the above set-up, we have a singular support condition $\Lambda_S$ on $Z_{ij} = X_i \times_S X_j$ defined by embedding $\Sing_{Z_{ij}} \subset \bT^*_S \times_S Z_{ij}$ and taking $\Lambda_S := (\Lambda \times_S Z_{ij}) \cap \Sing_{Z_{ij}}$. 
\end{defn}

\subsubsection{}
We establish the following basic properties of convolution categories.  Recall that a compactly generated category is \emph{semi-rigid} if the left and right duals of compact objects are compact \cite{bn ctcg}, and \emph{rigid} if in addition, the monoidal unit is compact \cite{1affine}.
\begin{prop}\label{convolution first prop}
Fix (closed and conical) singular support conditions $\Lambda_{ij}$ as above, and denote the corresponding closed classical support condition by $\Lambda_{ij}^{\mathrm{cl}} \subset Z_{ij}$.  We have the following.
\begin{enumerate}[(a)]
\item Suppose that $\Lambda_{ij} \star \Lambda_{jk} \subset \Lambda_{ik}$.  We have a \emph{convolution functor} $\star: \IndCoh_{\Lambda_{ij}}(Z_{ij}) \otimes \IndCoh_{\Lambda_{jk}}(Z_{jk}) \longrightarrow \IndCoh_{\Lambda_{jk}}(Z_{jk})$ preserving compact objects, defined by the correspondence
$$\begin{tikzcd}
X_i \times_S X_j \times X_j \times_S X_k & \arrow[l, "\delta"'] \arrow[r, "\pi"] X_i \times_S X_j \times_S X_k & X_i \times_S X_j,
\end{tikzcd}\;\;\;\;\;\;\;\;\;\;\;\;\;\;\; \cK_1 \star^! \cK_2 := \pi_*\delta^!(\cK_1 \boxtimes \cK_2).$$
\item Assume that $\Lambda_{ii}  \star \Lambda_{ii} \subset \Lambda_{ii}$ and that $\Lambda_{ii}^{\mathrm{cl}} \subset Z_{ii}$ is square, i.e. of the form $Y \times_S Y$ for some $Y \subset X$.  Then, convolution makes $\IndCoh_{\Lambda_{ii}}(Z_{ii})$ into a semi-rigid associative (i.e. $E_1$) monoidal category.  If $\on{SS}(\Delta_* \omega_X) \subset \Lambda_{ii}$, then the category is rigid.  In particular, $\IndCoh(Z_{ii})$ is rigid.
\item In the set-up of (b), consider the adjoint functors $(\iota_{\Lambda_{ii}}, \Gamma_{\Lambda_{ii}})$ between $\IndCoh(Z_{ii})$ and $\IndCoh_{\Lambda_{ii}}(Z_{ii})$ of Proposition \ref{nil to coh}.  The functor $\iota_{\Lambda_{ii}}$ is oplax monoidal, and the functor $\Gamma_{\Lambda_{ii}}$ is lax monoidal.
\item Assume that $\Lambda_{ij} \star \Lambda_{jj} \subset \Lambda_{ij}$ and $\Lambda_{ii} \star \Lambda_{ij} \subset \Lambda_{ii}$.  Then the category $\IndCoh_{\Lambda_{ij}}(Z_{ij})$ is a $\IndCoh_{\Lambda_{ii}}(Z_{ii})-\IndCoh_{\Lambda_{jj}}(Z_{jj})$-bimodule category under convolution.
\item Assume that $\Lambda_{ij} \star \Lambda_{jj} \subset \Lambda_{ij}$, $\Lambda_{jj} \star \Lambda_{jk} \subset \Lambda_{jk}$, and $\Lambda_{ij} \star \Lambda_{jk} \subset \Lambda_{ik}$.  Then the convolution product descends to the relative tensor product, i.e. to a functor
$$\IndCoh_{\Lambda_{ij}}(Z_{ij}) \tens{\IndCoh_{\Lambda_{jj}}(Z_{jj})} \IndCoh_{\Lambda_{jk}}(Z_{jk}) \longrightarrow \IndCoh_{\Lambda_{ik}}(Z_{ik}).$$
\end{enumerate}
\end{prop}

Note that all conditions in the following statements are satisfied if we fix $\Lambda_S$ as in Definition \ref{def support on base} and take $\Lambda_{ij} := \Lambda_S$ for varying $i, j$, with the exception that in (b), the condition $SS(\Delta_*\omega_{X_i}) \subset \Lambda_S$ may fail (i.e. $\IndCoh_{\Lambda_S}(Z_{ii})$ is always semi-rigid but may not be rigid).

\begin{exmp}
By Proposition \ref{convolution first prop}, we have that $\IndCoh_{\cN}(\wt{\mf{g}} \times_{\mf{g}} \wt{\mf{g}})$ and $\IndCoh_{\cN[1]}(\wt{\cN} \times_{\mf{g}} \wt{\cN})$ are semi-rigid (but not rigid) monoidal categories.  For $\IndCoh_{\cN}(\wt{\mf{g}} \times_{\mf{g}} \wt{\mf{g}})$, the object $\Delta_* \omega_{\wt{\mf{g}}}$ is not supported over nilpotent elements, thus the monoidal unit $\Delta_* \Gamma_{\cN} \omega_{\wt{\mf{g}}}$ is obtained by applying a local cohomology functor which does not preserve compactness.  For $\IndCoh_{\cN[1]}(\wt{\cN} \times_{\mf{g}} \wt{\cN})$ the singular support of $\Delta_* \omega_{\wt{\cN}}$ is over a given Borel $B$ consists of the singular codirections $\mf{n}^\perp \subset \mf{g}^*$, which may be non-nilpotent.
\end{exmp}

\begin{proof}
Statement (a) follows by Proposition 7.1.3 in \cite{AG}, and preservation of compact objects follows since $\pi$ is proper (since each $p_i: X_i \rightarrow S$ is proper) and $\delta$ is quasi-smooth (since each $X_i$ is smooth) and affine (since each $X_i$ has affine diagonal).

For statement (b), we omit the index $i$ for notational convenience.  Let us first take the maximal singular support locus, i.e. empty condition, in the setting of small categories.  Here, the convolution category $\Coh(Z)$ attains the structure of an associative (i.e. $E_1$) monoidal $\infty$-category by Section 5.2 of \cite{DAG} or by transporting the monoidal structure on endofunctor categories constructed in Section 4.7.1 of \cite{HA} across the equivalence in Theorem \ref{thm bnp}.  Its ind-completion $\IndCoh(Z)$ then automatically attains the structure of a monoidal category by Proposition 4.4 of \cite{BFN}.

Any full subcategory of a monoidal category which is closed under the monoidal product and contains the unit is a monoidal category.  The condition $\Lambda \star \Lambda \subset \Lambda$ guarantees it is closed under the product.  However, it is possible that the monoidal unit $\Delta_* \omega_X$ does not belong to the category $\IndCoh_\Lambda(Z)$ if the classical support condition imposed by $\Lambda$ does not contain the diagonal $\Delta(X) \subset Z$, and we need to address this issue.

For now, let $Y \subset X$ be any closed reduced substack defining a classical support condition, and $Y^2 := Y \times_S Y \subset Z$ the corresponding square classical support condition on $Z$.  We let $\iota_Y: \IndCoh_{Y}(X) \ra \IndCoh(X)$ denote the $t$-exact inclusion of sheaves supported on $Y$, and $\Gamma_Y$ its left $t$-exact right adjoint local cohomology functor.  By Theorem 1.3.1 in \cite{BNP convolution}, the ind-completion of the equivalence in Theorem \ref{thm bnp} restricts to an equivalence
$$\IndCoh(X \times_S X)^+ \overset{\simeq}{\longrightarrow} \cat{Fun}^{L, +}_{\QC(S)}(\QCoh(X), \QCoh(X))$$
where $\cat{Fun}^{L, +}$ indicates continuous functors which are left $t$-exact up to a shift.  We first argue that the equivalence identifies the full subcategories
$$\begin{tikzcd}
\IndCoh(X \times_S X)^+ \arrow[r, "\simeq"] & \cat{Fun}^{L,+}_{\QCoh(S)}(\QCoh(X), \QCoh(X)) \\ 
\IndCoh_{Y^2}(X \times_S X)^+ \arrow[u, hook, "\iota_{Y^2}"] \arrow[r, dotted, "\simeq"] & \cat{Fun}^{L,+}_{\QCoh(S)}(\QCoh_{Y}(X), \QCoh_{Y}(X)). \arrow[u, hook]
\end{tikzcd}$$
where the arrow on the right sends a functor $F: \QCoh_{Y}(X) \ra \QCoh_{Y}(X)$ to the composition $\iota_{Y} \circ F \circ \Gamma_{Y}$, and that the vertical functors are oplax monoidal.  In particular, the category $\IndCoh_{Y^2}(X \times_S X)^+$ on the left inherits a convolution monoidal structure.

We first need to show that the right vertical arrow is fully faithful.  It has a right adjoint given by the assignment $G \mapsto \Gamma_Y \circ G \circ \iota_Y$, and the unit of the adjunction
$$F \rightarrow \Gamma_Y \circ \iota_Y \circ F \circ \Gamma_Y \circ \iota_Y$$
is an equivalence since $\Gamma_Y \circ \iota_Y \simeq \mathrm{id}_{\QCoh_Y(X)}$ via the unit for the adjunction $(\iota_Y, \Gamma_Y)$.

Next, we show that the two subcategories coincide.  Let $\cK \in \IndCoh_{Y^2}(X \times_S X)^+$ be an integral kernel.  We need to show that the functor $F_{\cK}$ it defines can be written in the form $F_{\cK} \simeq \iota_Y \circ F' \circ \Gamma_Y$ for some $F'$.  However, as the essential image of $F_{\cK}$ is a subcategory of $\QCoh_Y(X)$, we may simply take $F'$ to be the restriction of $F_{\cK}$ to $\QCoh_Y(X)$.

Conversely, consider a functor of the form $F = \iota_Y \circ F' \circ \Gamma_Y$; it corresponds to some integral kernel $\cK \in \IndCoh(X \times_S X)^+$.  Let $j: U = X - Y \hookrightarrow X$ be the complement.  The restriction of the functor $F: \QCoh(X) \rightarrow \QCoh(X)$ to $\QCoh(U)$ is given by the integral kernel obtained by restricting $\cF$ to $U \times_S X$.  Since $\Gamma_Y$ is zero on $\QCoh(U)$, this means $\cK|_{U \times_S X} = 0$.  Now, consider $j^* \circ F: \QCoh(X) \rightarrow \QCoh(U)$; since $j^* \circ \iota_Y = 0$, this functor is zero, and it is given by the integral kernel obtained by restricting $\cF$ to $X \times_S U$, so $\cK|_{X \times_S U} = 0$.  In particular, $\cK$ is supported on $Y^2$ as desired.

Note that by construction, the functor
$$ \cat{Fun}^{L,+}_{\QCoh(S)}(\QCoh_{Y}(X), \QCoh_{Y}(X)) \hookrightarrow \cat{Fun}^{L,+}_{\QCoh(S)}(\QCoh(X), \QCoh(X))
$$
is oplax monoidal, i.e. using the unit for the adjunction $(\iota_Y, \Gamma_{Y})$.  It then follows that $\iota_{Y^2}: \IndCoh_{Y^2}(X \times_S X)^+ \hookrightarrow \IndCoh(X \times_S X)^+$ is an oplax monoidal functor of monoidal categories.  We need to show that $\IndCoh_{Y^2}(X \times_S X)$ is monoidal, and that the functor
$$\iota_{Y^2}: \IndCoh_{Y^2}(X \times_S X) \hookrightarrow \IndCoh(X \times_S X)$$
is oplax monoidal.  The structure maps for the monoidal structure of $\IndCoh_{Y^2}(X \times_S X)^+$ are $t$-bounded; thus, they restrict to functors $\Coh_{Y^2}(X \times_S X)^{\times n} \ra \IndCoh_{Y^2}(X \times_S X)^{\times m}$ for various $n, m$.  For example, the monoidal unit is given by a functor $\Vect_k = \Coh_{Y^2}(X \times_S X)^{\times 0} \ra \IndCoh_{Y^2}(X \times_S X)^+$ and the product is given by a functor $\Coh_{Y^2}(X \times_S X)^{\times 2} \ra \IndCoh(X \times_S X)^+$.  Then, ind-completion defines a monoidal structure on $\IndCoh_{Y^2}(X \times_S X)$, and the induced functor $\iota_{Y^2}: \IndCoh_{Y^2}(X \times_S X) \hookrightarrow \IndCoh(X \times_S X)$ is lax monoidal.

Now take $Y = \Lambda^{\mathrm{cl}}$, and consider the full subcategory of the monoidal category
$$\IndCoh_{\Lambda}(X \times_S X) \hookrightarrow \IndCoh_{(\Lambda^{\mathrm{cl}})^2}(X \times_S X).$$ 
Since $\Lambda \star \Lambda \subset \Lambda$, the monoidal product preserves the full subcategory.  Since $\Lambda^{\mathrm{cl}}$ is square, the subcategory contains the monoidal unit, since $\Coh_{\Lambda^{\mathrm{cl}}}(X \times_S X) \subset \IndCoh_\Lambda(Z)$.  Thus, $\IndCoh_{\Lambda}(X \times_S X)$ inherits the structure of a monoidal category.  Moreover, the inclusion $\iota_\Lambda: \IndCoh_\Lambda(X \times_S X) \hookrightarrow \IndCoh(X \times_S X)$ is the composition of an oplax monoidal functor with a monoidal functor, thus is oplax monoidal.  Its right adjoint is then lax monoidal, establishing (c).

The semi-rigidity claim in (b) follows from Theorem 3.25 in \cite{BCHN}.  Namely, via the proof of the theorem, we find that for the $!$-transform the right dual of $\cK$ is $\bD_{Z}(\cK) \otimes \omega_{Z/X}$, and the left dual is $\bD_Z(\cK)$.  Since tensoring by shifted line bundles and applying the Grothendieck-Serre duality functor do not change singular support (Proposition 4.7.3 in \cite{AG}), we see that $\Coh_\Lambda(Z) \subset \IndCoh(Z)$ is closed under taking left duals and right duals, thus is semi-rigid following (b).  If in addition the monoidal unit is compact, then the category is rigid.

Statement (d) follows by a similar argument as statement (b), i.e. via the identification on compact objects $\Coh(X_i \times_S X_j) \simeq \cat{Fun}_{\Perf(S)}(\Coh(X_i), \Coh(X_j))$ realizing $\IndCoh(X_i \times_S X_j)$ as a left module category for $\IndCoh(X_i \times_S X_i)$ under convolution in a similar way as above, and noting that the conditions on singular support guarantee that the functors restrict to the given subcategories.  

For statement (e), note that  the relative tensor product $\cat{M} \otimes_{\cat{A}} \cat{N}$ for (large) presentable $\infty$-categories is defined to be the colimit of the cosimplicial relative bar complex:
$$\begin{tikzcd}
\cdots \arrow[r, shift left=0.25ex] \arrow[r, shift right=0.25ex] \arrow[r, shift left=0.75ex] \arrow[r, shift right=0.75ex] & \cat{M} \otimes \cat{A} \otimes \cat{A} \otimes \cat{N} \arrow[r] \arrow[r, shift left] \arrow[r, shift right] & \cat{M} \otimes \cat{A} \otimes \cat{N} \arrow[r, shift left] \arrow[r, shift right] & \cat{M} \otimes \cat{N} \arrow[r] & \cat{M} \otimes_{\cat{A}} \cat{N}.
\end{tikzcd}$$
The identification of convolution categories with endofunctor categories in Theorem \ref{thm bnp} is for small categories, where such colimits are not guaranteed to exist.  However, the resulting augmented cosimplicial diagram of small categories
$$\begin{tikzcd}[column sep=small]
\cdots \arrow[r] \arrow[r, shift left] \arrow[r, shift right] &  \Coh_{\Lamda_{ij}}(Z_{ij}) \otimes \Coh_{\Lamda_{jj}}(Z_{jj}) \otimes \Coh_{\Lamda_{jk}}(Z_{jk}) \arrow[r, shift left] \arrow[r, shift right] & \Coh_{\Lamda_{ij}}(Z_{ij}) \otimes \Coh_{\Lamda_{jk}}(Z_{jk}) \arrow[r] & \Coh_{\Lamda_{ik}}(Z_{ik})
\end{tikzcd}$$
induces an augmented cosimplicial diagram of large categories, thus establishing the required factorization through the colimit:
$$\begin{tikzcd}
\arrow[r, shift left] \arrow[r, shift right]  & \IndCoh_{\Lambda_{ij}}(Z_{ij}) \otimes \IndCoh_{\Lambda_{jk}}(Z_{jk}) \arrow[r] & \IndCoh_{\Lambda_{ij}}(Z_{ij}) \tens{\IndCoh_{\Lambda_{jj}}(Z_{jj})} \IndCoh_{\Lambda_{jk}}(Z_{jk}) \arrow[r, dotted] & \IndCoh_{\Lambda_{ik}}(Z_{ik}).
\end{tikzcd}$$
\end{proof}

\begin{rmk}[$!$-convolution vs. $*$-convolution]\label{convolution duality}
The convolution above is defined using $!$-pullbacks; sometimes we denote it by $\star^!$.  On the other hand, one can define convolution using $*$-pullbacks:
$$\cK_1 \star^* \cK_2 := p_*\Delta^*(\cK_1 \boxtimes \cK_2)$$
with monoidal unit $\Delta_* \cO_X$.  We have an equivalence of monoidal categories
$$(\IndCoh(Z), \star^*) \longrightarrow (\IndCoh(Z), \star^!), \;\;\;\;\;\;\;\;\;\; \cK \mapsto \cK \otimes_{\cO_Z} \omega_Z.$$
Note that since $Z$ is quasi-smooth, $\omega_Z$ is a shifted line bundle.  Thus Proposition \ref{convolution first prop} applies just as well to the $*$-convolution and the difference is largely a matter of convenience.  Note that this difference is \emph{not} given by a Grothendieck-Serre duality equivalence $\Coh(Z)^{\opp} \simeq \Coh(Z)$, which would induce an equivalence of large categories $\Pro(\Coh(Z))^{\opp} \simeq \IndCoh(Z)$ due to contravariance.
\end{rmk}

\subsubsection{}\label{completed renorm}
We will also be interested in the following two set-ups, which can be gainfully interpreted as sub-cases of the above.
\begin{enumerate}[(1)]
    \item Let $f: X \rightarrow S$ be as above, $Z = X \times_S X$, and let $S_0 \subset S$ be a closed substack.  We denote by $X_0, Z_0$ the corresponding fibers in $X, Z$, and the formal completion of $X, Z$ along these closed substacks by $\wh{X}_0, \wh{Z}_0$.  These are not smooth geometric stacks.  However, there is a canonical equivalence $\IndCoh(\wh{Z}_0) \simeq \IndCoh_{Z_0}(Z)$ (see Proposition \ref{dgind stack}), and thus the former is monoidal under convolution.   The monoidal product restricts to compact objects  $\Coh_{Z_0}(Z) = \IndCoh_{Z_0}(Z) \cap \Coh(Z)$, but the monoidal unit is not compact.
    \item Take $S$ and $\wh{S}$ to be from the set-up in Example \ref{complete stack}, let $X$ and $Z$ be as above, and denote $\wh{X} = X \times_S \wh{S}$ and $\wh{Z} = Z \times_S \wh{S}$.  These stacks are not finite type, thus they do not fall within our set-up for convolution.  However, passing through the Grothendieck existence theorem (Theorem \ref{groth exist}) we may realize $\Coh(\wh{Z})$ as a small monoidal subcategory of possibly non-compact objects $\hCoh(\wh{Z}_0) \subset \IndCoh_{Z_0}(Z)$ (see Definition \ref{wcoh defn}) containing the monoidal unit.
\end{enumerate}

\medskip

\subsubsection{}
We now study the convolution functor
$$\star: \IC(X_1 \times_S X_2) \tens{\IC(X_2 \times_S X_2)} \IC(X_2 \times_S X_3) \rightarrow \IC(X_1 \times_S X_3).$$
It will turn out to always be fully faithful, and the essential image can be characterized in terms of the following singular support condition.  We note that a related albeit slightly less general result appears as Proposition 3.30 in \cite{BCHN} using methods from \cite{BNP affine}.

\begin{thm}\label{convolution composition}
Assume that $\Lambda_{12} \star \Lambda_{22} \subset \Lambda_{12}$ and $\Lambda_{22} \star \Lambda_{23} \subset \Lambda_{23}$.  Convolution induces an equivalence
$$\Phi: \IndCoh_{\Lambda_{12}}(Z_{12}) \tens{\IndCoh_{\Lambda_{22}}(Z_{22})} \IndCoh_{\Lambda_{23}}(Z_{23}) \overset{\simeq}{\longrightarrow} \IndCoh_{\Lambda_{12} \star \Lambda_{23}}(Z_{13}).$$
\end{thm}
\begin{proof}
Recall that $Z_{123} = X_1 \times_S X_2 \times_S X_3$.  To determine the essential image of the functor, we apply Theorem 1.3.13 and Proposition 7.6.4 of \cite{AG}, which characterize the essential image of proper pushforwards and affine quasi-smooth pullbacks via the functoriality of singular support defined in Definition \ref{def ss}.  In our setting, $\pi: Z_{123} \rightarrow Z_{13}$ is proper since $X_2 \rightarrow S$ is proper, $\delta: Z_{123} \rightarrow Z_{12} \times Z_{23}$ is affine since $X_2$ is a geometric stack and $\delta$ is quasi-smooth since $X_2$ is smooth.  Applying \emph{loc. cit.} we deduce that $\Phi$ generates $\IndCoh_{\Lambda_{12} \star \Lambda_{23}}(Z_{13})$ under colimits.

We now argue that the convolution functor $\Phi$ is fully faithful.  For convenience, let us  take the $*$-convolution.  We establish some shorthand; let $\cat{A} = \IndCoh_{\Lambda_{22}}(Z_{22})$ be the monoidal category and consider the module categories $\cat{M} = \IndCoh_{\Lambda_{12}}(Z_{12})$ and $\cat{N}  = \IndCoh_{\Lambda_{23}}(Z_{23})$.  Consider the insertion functor from Proposition \ref{convolution first prop}
$$i: \cat{M} \otimes \cat{N} \longrightarrow \cat{M} \otimes_{\cat{A}} \cat{N}.$$
By Chapter 1, Proposition 8.7.2 in \cite{DAG}, this functor admits a (continuous) right adjoint $i^R$, which is conservative since $i$ generates under colimits (here we use the singular support conditions and Proposition 7.6.4. and Theorem 7.8.2 of \cite{AG}).  Thus, by Barr-Beck-Lurie, insertion defines a natural equivalence 
$$\begin{tikzcd} (i^R \circ i)\dmod_{\cat{M} \otimes \cat{N}} \arrow[r, "\simeq"', "i"] & \cat{M} \otimes_{\cat{A}} \cat{N}.\end{tikzcd}$$ 
Let $\Phi^R$ be the right adjoint to the convolution functor; under this equivalence the endofunctor $\Phi^R \Phi \in \End(\cat{M} \otimes_{\cat{A}} \cat{N})$ is taken to the endofunctor $i^R \Phi^R \Phi i \simeq \Delta_* p^! p_* \Delta^* \in \End(\cat{M} \otimes \cat{N})$.  Since the forgetful functor $(i^R \circ i)\dmod_{\cat{M} \otimes \cat{N}} \rightarrow \cat{M} \otimes \cat{N}$ is conservative, to show that the unit for the adjunction $(\Phi, \Phi^R)$ is an equivalence (and thus $\Phi$ is fully faithful), we need to show that the unit map $i^R \circ i \rightarrow \Delta_* p^! p_* \Delta^*$ is an equivalence.

By Proposition C.2.3 in \cite{1affine} (the conditions of which are satisfied by any semi-rigid category by the arguments in Proposition D.2.2 and Corollary D.2.4 in \emph{op. cit.}), given a semi-rigid monoidal category $\cat{A}$, $\cat{A}$-module categories $\cat{M}, \cat{N}$, and letting $a$ denote the action functor (with right adjoint $a^R$), we have a ``base change'' natural commuting diagram of functors 
$$\begin{tikzcd}
\cat{M} \otimes \cat{A} \otimes \cat{N} \arrow[d, "{\id \otimes a}"'] & \cat{M} \otimes \cat{N} \arrow[d, "i"] \arrow[l, "{a^R \otimes \mathrm{id}}"'] \\
\cat{M} \otimes \cat{N} & \cat{M} \otimes_{\cat{A}} \cat{N}. \arrow[l, "i^R"']
\end{tikzcd}$$
Let $\iota: Z_{1223} = X_1 \times_S X_2 \times_S X_2 \times_S X_3 \rightarrow Z_{12} \times Z_{23} = X_1 \times_S X_2 \times X_2 \times_S Y_3$ denote the inclusion.  Applying this to our setting, we realize $i^R \circ i$ via the base change diagram
$$\begin{tikzcd}[column sep=-12ex]
& & X_1 \times_S X_2 \times_S X_2 \times_S X_3 \arrow[dr, "\Delta_1"] \arrow[dl, "\Delta_2"'] & & \\
& X_1 \times_S X_2 \times_S X_2 \times X_2 \times_S X_3 \arrow[dr, "\Delta_1"] \arrow[dl, "p_1"'] & & X_1 \times_S X_2 \times X_2 \times_S X_2 \times_S X_3 \arrow[dr, "p_3"] \arrow[dl, "\Delta_3"'] \\
X_1 \times_S X_2 \times X_2 \times_S X_3 & & X_1 \times_S X_2 \times X_2 \times_S X_2 \times X_2 \times_S X_3 & & X_1 \times_S X_2 \times X_2 \times_S X_3
\end{tikzcd}$$
where $\Delta_k$ means the diagonal map from the $k$th factor of $X$ to the $k$th and $(k+1)$th factor, and $p_i$ projects out the $i$th factor of $X$.  That is, 
$$i^R \circ i \simeq (\mathrm{id} \otimes a) \circ (a^R \otimes \mathrm{id}) = p_{3*} \Delta_3^* \Delta_{1*} p_1^! \simeq p_{3,*}\Delta_{1*}\Delta_2^* p_1^! \simeq \iota_* \iota^*(- \otimes \omega_{p_1}) \simeq \iota_*(\Delta_2^*\omega_{p_1} \otimes \iota^*(-)).$$
On the other hand, we perform a base change for the functor $\Delta_* p^! p_* \Delta^*$:
$$\begin{tikzcd}[column sep=-6ex]
& & X_1 \times_S X_2 \times_S X_2 \times_S X_3 \arrow[dr, "q_2"] \arrow[dl, "q_1"'] & & \\
& X_1 \times_S X_2 \times_S X_3 \arrow[dr] \arrow[dl, "\Delta"'] & & X_1 \times_S X_2 \times_S X_3 \arrow[dr, "\Delta"] \arrow[dl] \\
X_1 \times_S X_2 \times X_2 \times_S X_3 & & Z_1 \times_S X_3 & & X_1 \times_S X_2 \times X_2 \times_S X_3
\end{tikzcd}$$
where $q_i$ projects out the $i$th factor of $X$.  That is,
$$\Delta_* p^! p_* \Delta^* \simeq \Delta_* q_{2,*} q_1^! \Delta^* \simeq \iota_*(\omega_{q_1} \otimes \iota^*(-)).$$
Now, note that $q_1$ is the base change of $p_1$ across $\Delta_2$, thus $\Delta_2^*\omega_{p_1} \simeq \omega_{q_1}$, establishing the claim.
\end{proof}

In particular, the tensor product of the full categories of ind-coherent sheaves is as follows.
\begin{cor}
Let $\Lambda = \Sing_{Z_{12}} \star \Sing_{Z_{23}}$.  Convolution induces an equivalence
$$\IndCoh(Z_{12}) \tens{\IndCoh(Z_{22})} \IndCoh(Z_{23}) \overset{\simeq}{\longrightarrow} \IndCoh_{\Lambda}(Z_{13}).$$
\end{cor}

\subsection{Convolution of spectral affine Hecke categories}\label{convolution hecke}

\sss
We now apply the formalism from the previous section to our setting of spectral affine Hecke categories.  We first define a few variants of the Steinberg stack.
\begin{defn}\label{steinberg notation}
We consider the following variants of the Springer and Grothendieck-Springer resolutions, which are all smooth varieties (or stacks) proper over $\mf{g}$ (or $\mf{g}/G$) and vector bundles over a partial flag variety $G/P$.  Let $P \supset B$ be a parabolic subgroup with Lie algebra $\mf{p}$, Levi quotient $\nu_P: \mf{p} \rightarrow \mf{l}_P$ and unipotent radical $\mf{n}_P := \nu^{-1}(\{0\}) \subset \mf{p}$.  We denote by $\mf{z}_L := \mf{z}(\mf{l})$ the center of Levi quotient, and define $\mf{z}_P := \nu^{-1}(\mf{z}_L) \subset \mf{p}$; we emphasize that $\mf{z}_P$ is not the center of $\mf{p}$.  We define the following vector bundles over $G/P$:
$$\wt{\cN}_P := G \times^P \mf{n}_P, \;\;\;\;\;\;\;\; \wt{\cN}'_P := G \times^P \mf{z}_P, \;\;\;\;\;\;\; \wt{\mf{g}}_P := G \times^P \mf{p},$$
$$\wt{\cN}_P/G \simeq \mf{n}_P/P, \;\;\;\;\;\;\; \wt{\cN}'_P/G \simeq \mf{z}_P/P, \;\;\;\;\;\;\; \wt{\mf{g}}_P/G \simeq \mf{p}/P.$$
We may view the above stacks as $G$-equivariant vector subbundles $\cV_P \subset G/P \times \mf{g}$ over $G/P$, or equivalently as $P$-subrepresentations $V_P \subset \mf{g}$.  For another parabolic $Q$ with Lie algebra $\mf{q}$ and a $\cV_Q$ as above, we will consider the corresponding variants of the (derived) \emph{Steinberg schemes} (resp. \emph{Steinberg stack}):
$$\cZ_{V_P, W_Q} := \cV_P \utimes{\mf{g}} \cV_Q, \;\;\;\;\;\;\;\;\;\; \cZ_{V_P, W_Q}/G \simeq V_P/P \utimes{\mf{g}/G} V_Q/Q$$
and the formal completion over the nilpotent cone by
$$\wh{\cZ}_{V_P, W_Q} = \cZ_{V_P, W_Q} \times_{\mf{h}/\!/W} \wh{\mf{h}/\!/W}_{\{0\}}.$$
\end{defn}

We define the singular support conditions (using Definition \ref{def support on base}) which arise naturally from convolution in this set-up.
\begin{defn}\label{def nilss}
Note that $\bT^*(\mf{g}/G) \subset (\mf{g} \times \mf{g}^*)/G$.    We let $\cN \subset \mf{g}$ or $\cN \subset \mf{g}^*$ denote the (reduced) nilpotent cone.  We define the \emph{classical nilpotent cone} $\cN$ and the \emph{odd nilpotent cone} $\cN[1]$ to be the singular support conditions defined via Defintion \ref{def support on base} by the $G$-equivariant subsets of $\mf{g} \times \mf{g}^*$:
$$\cN := \cN \times \mf{g}^*, \;\;\;\;\;\;\;\;\;\; \cN[1] := \mf{g} \times \cN.$$
In particular, $\cN$ imposes only a support condition in the classical support directions, while $\cN[1]$ imposes only a support condition in the singular support directions.
\end{defn}

\subsubsection{}\label{section ssupp cond food}

We now show how these (singular) support conditions arise in nature.  We begin by computing the odd cotangent bundle on these Steinberg stacks.
\begin{lemma}\label{convolve ss description}\label{auto nilpotent}
Suppose $V_P, V_Q \subset \mf{g}$ are $P$-closed and $Q$-closed subspaces respectively.  Let $\cZ = \cZ_{V_P, W_Q}$, and consider the projection $p: \Sing_{\cZ} \rightarrow \cZ \rightarrow G/P \times G/Q$.  Its fiber over $(gP, hQ) \in G/P \times G/Q$ is:
$$p^{-1}(gP, hQ) = \left\{(x, y) \in  \mf{g} \times \mf{g}^* \left| \begin{array}{c} y([x, -]) = 0 \\ x \in gV_Pg^{-1} \cap hW_Qh^{-1}, \\ y \in gV_P^\perp g^{-1} \cap hW_Q^\perp h^{-1} \end{array}\right. \right\}.$$
\end{lemma}
\begin{proof}
Let us fix a point $$\eta = (gP, hQ, x) \in \cZ(k) = (\cV_P \times_{\mf{g}} \cV_Q)(k) \subset (G/P \times G/Q \times \mf{g})(k),$$where $x \in \mf{g}(k)$ and $g, h \in G(k)$ subject to the condition $x \in gV_Pg^{-1} \cap hV_Qh^{-1}.$  
The map $G \times^P V_P \ra \mf{g}$ descends from the adjoint action map $G \times V_P \ra \mf{g}$, whose differential on tangent spaces takes the form $$\mf{g} \oplus V_P \ra \mf{g}, \quad \quad (v, p) \mapsto g[v, x]g^{-1} + gpg^{-1}.$$  
The vanishing locus for the pullback is $(\mf{g}^*)^x \cap gV_P^\perp g$.   A similar formula holds for the pullback to $\cW_Q$, and the formula follows.
\end{proof}

We observe that if either $V_P = \mf{n}_P$ or $W_Q = \mf{n}_Q$ (i.e. if either factor is ``strictly nilpotent''), then the \emph{classical} support of $\cZ_{V_P, W_Q}$ must be nilpotent, while if either $V_P = \mf{p}$ or $W_Q = \mf{q}$ (i.e. if either factor is ``Whittaker''), then it is the \emph{singular} support of $\cZ_{V_P, W_Q}$ that must be nilpotent.

\subsubsection{}

We now characterize the above support conditions as those which arise naturally via Theorem \ref{convolution composition}.  Namely, the classical nilpotence condition $\cN$ arises when convolving over the strictly nilpotent spectral affine Hecke category $\IndCoh(\wt{\cN}/G \times_{\mf{g}/G} \wt{\cN}/G)$, while the singular nilpotence condition $\cN[1]$ arises when convolving over the ``completed'' or Whittaker spectral affine Hecke category $\IndCoh(\wt{\mf{g}}/G \times_{\mf{g}/G} \wt{\mf{g}}/G)$.
\begin{prop}\label{reduction using monoidal}
Consider $\cZ = \cZ_{V_P, W_Q}$ as the convolution of $\cZ_1 = \cZ_{V_P, \mf{n}}$ and $\cZ_2 = \cZ_{\mf{n}, W_Q}$.  We have $\Sing_{\cZ_1/G} \star \Sing_{\cZ_2/G} = \cN \cap \Sing_{\cZ/G}.$  Thus, convolution induces an equivalence of categories:
$$\begin{tikzcd}
\IndCoh(\cV_P/G \times_{\mf{g}/G} \wt{\cN}/G) \tens{\IndCoh(\na{}/G \times_{\mf{g}/G} \na{}/G)} \IndCoh(\wt{\cN}/G \times_{\mf{g}/G} \cW_Q/G) \arrow[rr, "\simeq"] & & \IndCoh_{\cN}(\cV_P/G \times_{\mf{g}/G} \cW_Q/G).
\end{tikzcd}$$

Consider $\cZ = \cZ_{V_P, W_Q}$ as the convolution of $\cZ_1 = \cZ_{V_P, \mf{b}}$ and $\cZ_2 = \cZ_{\mf{b}, W_Q}$.  We have $\Sing_{\cZ_1/G} \star \Sing_{\cZ_2/G} = \cN[1] \cap \Sing_{\cZ/G}$. Thus, convolution induces an equivalence of categories:
$$\begin{tikzcd}
\IndCoh(\cV_P/G \times_{\mf{g}/G} \nc{}/G) \tens{\IndCoh(\nc{}/G \times_{\mf{g}/G} \nc{}/G)} \IndCoh(\nc{}/G \times_{\mf{g}/G} \cW_Q/G) \arrow[rr, "\simeq"] & & \IndCoh_{\cN[1]}(\cV_P/G \times_{\mf{g}/G} \cW_Q/G).
\end{tikzcd}$$
\end{prop}
\begin{proof}
Let us first consider the case $\cZ_1 = \cZ_{V_P, \mf{n}}$ and $\cZ_2 = \cZ_{\mf{n}, V_Q}$.  Fix a $k$-point $\eta = (gP, hQ, x) \in (G/P \times G/Q \times \mf{g})(k)$.  The fiber over this point in the convolution of supports is
$$(\Sing_{\cZ_1} \star \Sing_{\cZ_2})_\eta = \left(\bigcup_{\substack{\mf{b}' \in G/B \\ x \in \mf{n}'}}  (\mf{g}^*)^x \cap gV_P^\perp g^{-1} \cap hV_Q^\perp h^{-1} \cap (\mf{n}')^\perp \right)$$
$$\simeq \left(\bigcup_{\substack{\mf{b}' \in G/B \\ x \in \mf{n}'}}  \mf{g}^x \cap  gV_P^\perp g^{-1} \cap hV_Q^\perp h^{-1} \cap \mf{b}' \right)$$
where the equivalence is obtained by passing through the Killing form.  First, note that unless $x$ is nilpotent, this set is empty, thus we have a nilpotent classical support condition.  For the singular support, we wish to show that for fixed $\eta$ with $x$ nilpotent, the above set is equal to the full odd cotangent fibers from Lemma \ref{convolve ss description}.  This follows from the fact that for any commuting elements $[x, y] = 0$, there is a Borel containing them.

This completes the proof of the first statement.  The second statement is analogous, except that we replace the condition that $x \in \mf{n}'$ with the condition that $x \in \mf{b}'$, removing the nilpotence condition for $x$, and we require that $y \in (\mf{b}')^\perp \simeq \mf{n}'$, imposing a nilpotence condition on the singular support instead.
\end{proof}

\subsection{Algebra objects in convolution categories}\label{sec algebra obj}

\sss Recall the set-up of Definition \ref{convolution setup}.  Let $X, Y, S$ be smooth geometric stacks, and suppose we have a sequence
$$\begin{tikzcd}
X \arrow[r, "\phi"'] \arrow[rr, bend left, "f_X"] & Y \arrow[r, "f_Y"'] & S
\end{tikzcd}$$
where $\phi$, $f_Y$ and $f_X$ are proper.  Letting ${}^r\phi := \mathrm{id} \times \phi: X \times_S X \ra X \times_S Y$, this defines convolution spaces related by the Cartesian square
$$\begin{tikzcd}
Z_Y := X \times_Y X \arrow[r, "\iota", hook] \arrow[d] & Z_S := X \times_S X \arrow[d, "{}^r\phi"] \\
X \arrow[r, hook] & X \times_S Y
\end{tikzcd}$$  
where we note that $\iota$ is a closed immersion since $Y$ is separated over $S$.  We fix a closed substack $S_0 \subset S$, and denote by $X_0, Y_0, Z_{S,0}, Z_{Y,0}$ the corresponding fibers.  Equivalently, we can take the singular support condition $\Lambda_0 = \bT^*_S \times_S S_0$ (as in Definition \ref{def support on base}).  For brevity, we denote the local cohomology functor with respect to any of these supports by $\Gamma_{\Lambda_0}$.  Note that by base change the functor $\Gamma_{\Lambda_0}$ commutes with all pullback and pushforward functors for maps over $S$.

\medskip

\subsubsection{}
We now define certain canonical (co)algebra objects in the convolution categories $\IndCoh(Z_S)$ and $\IndCoh(Z_Y)$ one can assign to the above set-up.
\begin{prop}\label{def coalgebra O}
Under the integral kernel $\mathbb{F}$ correspondence of Theorem \ref{thm bnp}, we have natural equivalences
$$\bF({\omega_{Z_Y}}) \simeq \phi^!\phi_* \in \cat{End}_{\Perf(Y)}(\Coh(X)), \;\;\;\;\;\;\;\;\;\; \bF({\iota_* \omega_{Z_Y}}) \simeq \phi^! \phi_* \in \cat{End}_{\Perf(S)}(\Coh(X))$$
as monads or algebra objects in a monoidal category. (resp. $\bF({\cO_{Z_Y}}) \simeq \phi^* \phi_*$ and $\bF({\iota_* \cO_{\omega_{Z_Y}}}) \simeq \phi^* \phi_*$ as comonads or coalgebra objects).  Furthermore, the possibly non-compact object $\Gamma_{\Lambda_0} \omega_{Z_Y}$ is an algebra object in $\IndCoh_{Z_{Y,0}}(Z_Y)$.
\end{prop}
\begin{proof}
Convolution with $\omega_{Z_Y}$ is equivalent to the pull-push (from the left to right) along the diagram 
$$X \longleftarrow Z_Y \times X = X \times_Y X \times X \longleftarrow X \times_Y X \longrightarrow X$$
i.e. is given by the functor $\pi_{1*}\pi_2^! \simeq \phi^!\phi_*$ by base change along the diagram
$$\begin{tikzcd}
Z_Y = X \times_Y X \arrow[r, "\pi_2"] \arrow[d, "\pi_1"] & X \arrow[d, "\phi"] \\
X \arrow[r, "\phi"] & Y.
\end{tikzcd}$$
This establishes the first statement.  The second follows since $\iota_*$ is monoidal, and the third statement follows by Proposition \ref{convolution first prop}(c).
\end{proof}

\begin{rmk}\label{fingernails}
A priori, we could have considered a singular support condition in place of the support condition $S_0 \subset S$.  However, since $\omega_{Z_Y}$ is perfect, the only relevant direction in which one can restrict the support is classical.  For any $\Lambda_S \subset \bT^*_S$ (as in Definition \ref{def support on base}), in the above and following we can replace $\Lambda_S$ with the classical support condition $S_0 := \Lambda_S \cap (S \times \{0\}) \subset \bT^*_S$ (or singular support $\bT^*_S \times_S S_0$).
\end{rmk}

\subsubsection{}\label{algebra explicit}
Let us describe some of the structure of this (co)algebra object-wise.  We assume $S_0 = S$; otherwise one can apply $\Gamma_{\Lambda_0}$ to obtain analogous statements.
\begin{enumerate}[(1)]
\item Let $\Delta_{X/Y}: X \rightarrow Z_Y = X \times_Y X$ be the relative diagonal.  The unit for the algebra (resp. counit for the coalgebra) arises by applying $\iota_*$ to the unit for the adjunction $(\Delta_*, \Delta^!)$ evaluated at $\omega_{Z_Y}$ (resp. the counit for the adjunction $(\Delta^*, \Delta_*)$ evaluated at $\cO_{Z_Y}$).
\item Let $p: X \times_Y X \times_Y X \rightarrow Z_Y = X \times_Y X$ be the projection off the middle factor.  The multiplication for the algebra (resp. comultiplication for the coalgebra) arises by applying $\iota_*$ to the counit for the adjunction $(p_*, p^!)$ evaluated at $\omega_{Z_Y}$ (resp. the unit for the adjunction $(p^*, p_*)$ evaluated at $\cO_{Z_Y}$).
\item We define the $n$-fold convolution space by
$$Z_Y^{(n)} := \overbrace{X \times_Y \cdots \times_Y X}^{n+1}, \;\;\;\;\;\;\; Z^{(0)}_Y = X, \;\;\;\;\;\;\; Z^{(-1)}_Y = Y.$$
This indexing is chosen such that the simplicial category $\IndCoh(Z_Y^{(\bullet)})$ presents $\IndCoh(Z_Y)$ as a monoidal category, i.e. the convolution of a sheaf on $Z_Y^{(n)}$ and $Z_Y^{(m)}$ gives a sheaf on $Z_Y^{(n+m)}$. 
Let $p: Z_Y \rightarrow Y$ denote the projection and $q: Z_Y^{(n)} \rightarrow Z_Y = X \times_Y X$ be the projection away from the middle factors.  We have identifications of the $n$-fold convolutions
$$\overbrace{\omega_{Z_Y} \star^! \cdots \star^! \omega_{Z_Y}}^{n} \simeq q_* \omega_{Z_Y^{(n)}} \simeq p^!(\phi_*\phi^!)^{n-1} \omega_Y,$$
$$\overbrace{\cO_{Z_Y} \star^* \cdots \star^* \cO_{Z_Y}}^{n}  \simeq q_* \cO_{Z_Y^{(n)}} \simeq p^*(\phi_*\phi^*)^{n-1} \cO_Y,$$
where the rightmost equivalences arise via iterated base change.  The structure maps are induced by the counit (resp. unit) for the adjunction $(\phi_*, \phi^!)$ (resp. $(\phi^*, \phi_*)$).
\end{enumerate}

\medskip

\subsubsection{}
Let $\Lambda_S$ be as in Remark \ref{fingernails}.  Next, we aim to express $\IndCoh(X \times_S Y)$ in terms of $\IndCoh(X \times_S X)$ via the above (co)algebra objects.  We consider the functors 
$$
{}^r\phi_*:  \begin{tikzcd}
\IndCoh_{\Lambda_S}(X \times_S X) \arrow[r, shift left] & \arrow[l, shift left] \IndCoh_{\Lambda_S}(X \times_S Y) \end{tikzcd} :{}^r\phi^!$$
and likewise the pair $({}^r\phi^*, {}^r\phi_*)$.  Consider the monad ${}^r\phi^!\,{}^r\phi_*$ and the comonad ${}^r\phi^*\,{}^r\phi_*$ on $\IndCoh(X \times_S X)$.  We leave it to the reader to verify that the functors above preserve the singular support condition $\Lambda_S$.  We first establish monadicity.
\begin{prop}\label{pre spectral monadic}
If $\phi: X \rightarrow Y$ is surjective on geometric points, then the functor ${}^r\phi^!$ is monadic and ${}^r\phi^*$ is comonadic.
\end{prop}
\begin{proof}
Since $\IndCoh_{\Lambda_S}(X \times_S X)$ is presentable it has all small limits and colimits (Definition 5.5.0.18 and Proposition 5.5.2.4 in \cite{HTT}), and since ${}^r\phi^*$ are ${}^r\phi^!$ are twists of each other, they are both left and right adjoints, thus they commute with limits and colimits.  They are conservative due to surjectivity (Proposition 6.2.2 in \cite{DAG}), thus the claim follows by Lurie-Barr-Beck (Theorem 4.7.3.5 in \cite{HA}).
\end{proof}

Given monadicity, we may now identify the $\IndCoh_{\Lambda_S}(X \times_S X)$-module category $\IndCoh_{\Lambda_S}(X \times_S Y)$ as a category of modules for a natural algebra object in $\IndCoh_{\Lambda_S}(X \times_S X)$.
\begin{prop}\label{spectral monadic}
Assume $\phi$ is surjective on geometric points.  There is a natural identification as algebras of the monad ${}^r\phi^!\,{}^r\phi_*$ on the category $\IndCoh_{\Lambda_S}(Z_S)$ with right $!$-convolution by $\Gamma_{\Lambda_0} \iota_*\omega_{Z_Y}$.  Thus, by Barr-Beck-Lurie, we have an equivalence of left $\IndCoh_{\Lambda_S}(Z_S)$-module categories:
$$\begin{tikzcd}
\IndCoh(X \times_S Y) \arrow[d, "{}^r\phi^!", shift left] & \arrow[l, "\simeq"'] \iota_*\omega_{Z_Y}\dmod_{\IC(Z_S)} \arrow[dl, shift right] \\
\IndCoh(X \times_S X). \arrow[u, "{}^r\phi_*", shift left] \arrow[ur, shift right] &
\end{tikzcd}$$
Similarly, there is an identification as coalgebras of the comonad ${}^r\phi^*\,{}^r\phi_*$ with right convolution by $\Gamma_{\Lambda_0} \iota_* \cO_{Z_Y}$, and an equivalence of left $\IndCoh(Z_S)$-module categories:
$$\begin{tikzcd}
\IndCoh(X \times_S Y) \arrow[d, "{}^r\phi^*"', shift right] & \arrow[l, "\simeq"'] \iota_*\cO_{Z_Y}\dcomod_{\IC(Z_S)} \arrow[dl, shift right] \\
\IndCoh(X \times_S X). \arrow[u, "{}^r\phi_*"', shift right] \arrow[ur, shift right] &
\end{tikzcd}$$
\end{prop}
\begin{proof}
We prove the algebra version; the coalgebra version is similar. The functors ${}^r\phi_*, {}^r\phi^!$ are evidently equivariant for the left $\IndCoh(Z)$-action by convolution; for any monoidal category $\cat{A}$, viewing $\cat{A}_L := \cat{A}$ as the regular left module, we have a natural inverse equivalences
$$\begin{tikzcd}
\cat{A} \arrow[rr, "{a \mapsto  a \star -}", shift left] & & \cat{End}_{\cat{A}}(\cat{A}_L) \arrow[ll, "{F(1_{\cat{A}})} \mapsfrom F", shift left]
\end{tikzcd}$$
i.e. every $\cat{A}$-linear endofunctor is given by right convolution with some object, and the object can be recovered by evaluation at the monoidal unit.  Applying this general fact to our setting, the endofunctor ${}^r\phi^!\,{}^r\phi_*$ is given by convolution with the object  ${}^r\phi^!\,{}^r\phi_* \Delta_*\Gamma_{\Lambda_0} \omega_X \in \IndCoh(Z)$.  By base change, this object is $\Gamma_{\Lambda_0} \iota_* \omega_{Z_Y}$.  Thus, we have an identification of endofunctors ${}^r\phi^!\,{}^r\phi_*(-) \simeq \Gamma_{\Lambda_0} \iota_* \omega_{Z_Y} \star -$.   Finally, this is an identification as algebra objects in $\cat{End}(\IndCoh(Z))$ by the commuting equivalences of monoidal categories
$$\begin{tikzcd}
\iota_* \omega_{\cZ_Y} \in \Coh(Z) \arrow[r, "\simeq"] \arrow[d, "\simeq", shift left=5] & \cat{End}_{\Coh(Z)}(\Coh(Z)) \ni {}^r\phi^!{}^r\phi_* \\
\;\;\;\;\;\;\;\;\;\;\;\;\;\;\;\;\;\;\;\;\;\;\;\phi^!\phi_* \in \cat{End}_{\Perf(Y)}(\Coh(X)) \arrow[ru, "\simeq"] & 
\end{tikzcd}$$
where the diagonal arrow is given by base change.
\end{proof}

\subsection{Algebra objects in spectral affine Hecke categories}\label{spectral hecke algebra}\label{spectral whittaker sec}

\sss
We now apply the above generalities to affine Hecke categories, in particular the Whittaker affine Hecke categories which will consist of the bulk of this section.  Let $\mf{h} = \mf{b}/[\mf{b}, \mf{b}]$ denote the universal Cartan, and $W = W_G$ the universal Weyl group.   For a parabolic $P$, we denote its quotient Levi by $L_P$ (and the corresponding Lie algebras by $\mf{p}, \mf{l}_P$); we have a subgroup $W_P \subset W$, and an isomorphism $W_P \simeq W_{L_P}$.

Recall the Grothendieck-Springer resolution $\wt{\mf{g}} = G \times^B \mf{b}$ and the partial Grothendieck-Springer resolution $\ga{P} := G \times^P \mf{p}$, and denote the eigenvalue map $$\nu_P: \wt{\mf{g}}_P \rightarrow \mf{l}_P/\!/L_P \simeq \mf{h}/\!/W_P$$with extremes $\nu_B: \wt{\mf{g}} \rightarrow \mf{h}$ and $\nu_G: \mf{g} \rightarrow \mf{h}/\!/W$.  

\medskip

\subsubsection{}
Following the set-up and notation of Section \ref{sec algebra obj}, we take
$$\begin{tikzcd}
X = \mf{b}/B = \wt{\mf{g}}/G \arrow[r, "\phi"] & Y = \mf{p}/P = \wt{\mf{g}}_P/G \arrow[r, "f_Y"] &  S = \mf{g}/G.
\end{tikzcd}$$
For now, we do not place any restrictions on singular support.  Later, we will restrict the classical support to the classical nilpotent cone $\cN \subset \mf{g}$.  We establish the following notation for this section only: recall the notation from Definition \ref{steinberg notation}; we denote the usual Steinberg by $\cZ := \cZ_{\mf{b}, \mf{b}} = \wt{\mf{g}} \times_{\mf{g}} \wt{\mf{g}}$, the partial Steinberg by $\cZ_P := \wt{\mf{g}} \times_{\wt{\mf{g}}_P} \wt{\mf{g}}$, and recall that $\cZ_{\mf{b}, \mf{p}} = \wt{\mf{g}} \times_{\mf{g}} \wt{\mf{g}}_P$.  Then we have maps
$${}^r\phi: \cZ/G \longrightarrow \cZ_{\mf{b}, \mf{p}}/G , \;\;\;\;\;\;\;\;\;\;\;\;\;\;\;\;\; \iota: \cZ_P/G  \hookrightarrow \cZ/G.$$
Note that $\iota$ is a closed immersion since it is base-changed from the relative diagonal $\wt{\mf{g}}_P \hookrightarrow \wt{\mf{g}}_P \times_{\mf{g}} \wt{\mf{g}}_P$, which is a closed immersion since the map $\wt{\mf{g}}_P \rightarrow \mf{g}$ is separated, and that ${}^r\phi$ is base-changed from $\phi$.  We also denote, as in \ref{algebra explicit},
$$\cZ_P^{(n)} = \overbrace{\wt{\mf{g}} \times_{\wt{\mf{g}}_P} \cdots \times_{\wt{\mf{g}}_P} \wt{\mf{g}}}^{n+1}, \;\;\;\;\;\;\; \;\;\; \cZ_P^{(0)} = \wt{\mf{g}}, \;\;\;\;\;\;\; \;\;\; \cZ_P^{(-1)} = \wt{\mf{g}}_P.$$

\medskip

\subsubsection{}
Noting that the map $\phi$ is surjective and applying Proposition \ref{spectral monadic} to the above setting we have the following immediate observation.
\begin{cor}\label{spectral whittaker convolution}
There is a natural identification ${}^r\phi^!\, {}^r\phi_*(-) \simeq - \star \iota_* \omega_{\cZ_P/G}$ as monads in $\IndCoh(\cZ/G)$ under $!$-convolution, and ${}^r\phi^*\, {}^r\phi_*(-) \simeq - \star \iota_* \cO_{\cZ_P/G}$ as comonads under $*$-convolution.  Thus, we have equivalences of left $\IndCoh(\cZ/G)$-module categories
$$\iota_*\omega_{\cZ_P/G}\dmod_{\IndCoh(\cZ/G)} \simeq \IndCoh(\cZ_{\mf{b}, \mf{p}}/G), \;\;\;\;\;\;\;\;\;\; \iota_*\cO_{\cZ_P/G}\dcomod_{\IndCoh(\cZ/G)} \simeq \IndCoh(\cZ_{\mf{b}, \mf{p}}/G).$$
\end{cor}

Our goal is now to give an explicit characterization of this algebra (or coalgebra) object.  We begin by observing that the algebra and coalgebra objects have the same underlying object.  Recall that a derived stack $X$ is \emph{Calabi-Yau of dimension $n$} if there is an equivalence $\omega_X \simeq \cO_X[n]$. 
\begin{prop}\label{calabi yau}
The derived schemes $\cZ_P^{(n)}$ (for $n \geq -1$) are Calabi-Yau of dimension $\dim(G)$.  Thus, the derived stacks $\cZ_P^{(n)}/G$ are Calabi-Yau of dimension 0.
\end{prop}
\begin{proof}
We induct on $n$.  For $n=-1$, it suffices to show that for any affine algebraic group $H$, the adjoint quotient $\mf{h}/H$ is Calabi-Yau of dimension 0.  Let $p: \mf{h} \ra \mf{h}/H$ be the atlas; we have $H$-equivariant equivalences
$$\cO_{\mf{h}} \simeq \det(\mf{h})^{-1} \otimes \omega_{\mf{h}}[-\dim H] \simeq \det(\mf{h})^{-1} \otimes  p^! \omega_{\mf{h}/H}[-\dim H] \simeq \det(\mf{h})^{-1} \otimes \det(\mf{h}) \otimes p^* \omega_{\mf{h}/H} \simeq p^* \omega_{\mf{h}/H}.$$
Since $p^*$ is conservative and $p^* \cO_{\mf{h}/H} \simeq \cO_H$, this means that $\cO_{\mf{h}/H} \simeq \omega_{\mf{h}/H}$.  To induct, we note that given a diagram
$$\begin{tikzcd}
\arrow[d] W \times_Y X \arrow[r] & X \arrow[d, "f"]\\
W \arrow[r, "g"] & Y
\end{tikzcd}$$
where $f$ is quasi-smooth and relatively Calabi-Yau (i.e. $\omega_{X/Y} \simeq \cO_X$, or $f^! \simeq f^*$), and $W$ is Calabi-Yau, then $W \times_Y X$ is Calabi-Yau of the expected dimension.  Furthermore, we note that if $f: X \ra Y$ is quasi-smooth and $X, Y$ are Calabi-Yau, then $f$ is relatively Calabi-Yau.  The inductive step follows.
\end{proof}

\begin{rmk}
The claim holds if we replace $\wt{\mf{g}}$ above by $\wt{\mf{g}}_Q$ for $Q \subset P$.
\end{rmk}

\subsubsection{}
By a standard dimension count (i.e. semi-smallness of the partial Grothendieck-Springer resolutions), the a priori derived schemes $\cZ$ and $\cZ_P$ are in fact underived (i.e. a fiber product of smooth schemes is underived if the dimension is equal to its virtual or expected dimension).  The argument by dimension count fails for the iterated fiber products $\cZ_P^{(n)}$ for $n \geq 2$.  However, we have the following result, well-known to experts though we could not find a reference (and which we apply to the setting $Q = B$).  It implies that, although $\cO_{\cZ_P^{(n)}}$ may not live in the heart of the standard $t$-structure, its pushforward to $\cZ_P = \cZ_P^{(1)}$ does in fact live in the heart (as well as its pushfoward to $\wt{\mf{g}} = \cZ_P^{(0)}$ and $\wt{\mf{g}}_P = \cZ_P^{(-1)}$).  Since these are the objects that describe the algebra structure on $\iota_*\omega_{\cZ_P/G}$, this will imply that $\iota_*\omega_{\cZ_P/G}$, which a priori lives in a derived category, may be studied as an object in the abelian category $\Coh(\cZ/G)^\heartsuit$.

\begin{prop}\label{affinization calculation}
Let $Q \subset P$ be parabolics.  Let $\mu_Q^P: \wt{\mf{g}}_Q \rightarrow \wt{\mf{g}}_P$.  We have a canonical identification of the derived pushforward:
$$(\mu_Q^P)_* \cO_{\wt{\mf{g}}_Q} \simeq \cO_{\wt{\mf{g}}_P} \tens{\cO(\mf{h}/\!/W_P)} \cO(\mf{h}/\!/W_Q), \;\;\;\;\;\;\;\;\;\; (\mu_Q^P)_* \omega_{\wt{\mf{g}}_Q} \simeq \omega_{\wt{\mf{g}}_P} \tens{\cO(\mf{h}/\!/W_P)} \cO(\mf{h}/\!/W_Q).$$
In particular, $(\mu_Q^P)_* \cO_{\wt{\mf{g}}_Q} \simeq (\mu_Q^P)_* \omega_{\wt{\mf{g}}_Q}[-\dim(G)]$ is in the heart of the standard $t$-structure.
\end{prop}
\begin{proof}
The statement for $\omega$ follows from that for $\cO$ given Proposition \ref{calabi yau}.  Let $f: X \rightarrow Y$ be a proper birational map of schemes, with $Y$ normal.  We claim that $R^0f_* \cO_X = \cO_Y$.  Working affine locally on $Y$, we may assume $Y$ is affine.  Since $f$ is proper, $f_* \cO_X$ is a coherent sheaf on $Y$.  Since $f$ is birational, $\cO(Y) \subset \cO(X) \subset \Frac(\cO(Y))$.  Since $Y$ is normal, it is integrally closed, and since $\cO(X)$ is finitely generated as an $\cO(Y)$-module this means that $\cO(X) = \cO(Y)$.

We apply the above discussion to the map
$$f: \wt{\mf{g}}_Q \rightarrow \wt{\mf{g}}_P \times_{\mf{h}/\!/W_P} \mf{h}/\!/W_Q.$$
The map is evidently proper, since $\wt{\mf{g}} \rightarrow \mf{g}$ is proper, and $\mf{g} \rightarrow \mf{g} \times_{\mf{h}/\!/W} \mf{h}/\!/W_Q$ is affine, thus separated.  It is well-known to be birational on the regular semisimple locus (see, for example, Lemma 3.1.42 in \cite{CG}).  To see that the higher cohomology vanishes, we apply the Grauert-Riemenschneider vanishing theorem, along with the identification of $\cO$ with top differential forms in Proposition \ref{calabi yau}.  

Thus, it remains to verify that $\wt{\mf{g}}_P \times_{\mf{h}/\!/W_P} \mf{h}/\!/W_Q$ is normal.  By Proposition 6.14.1 in \cite{EGAIV.2}, the base change of a normal variety along a flat map with normal geometric fibers is normal.   Thus it suffices to check that $\wt{\mf{g}}_P \rightarrow \mf{h}/\!/W_P$ is flat with normal geometric fibers, both following from theorems of Kostant.  For flatness, note that the map $\nu_P$ factors as$$\wt{\mf{g}}_P \rightarrow \mf{l}_P \rightarrow \mf{l}_P/\!/L_P = \mf{h}/\!/W_P;$$ the first map has fiber $G \times^Q (x + \mf{n}_Q)$ for $x \in \mf{l}_Q$, which are in particular equidimensional; the fibers of the second map are equidimensional by Theorem 6.7.2 of \cite{CG}.  Thus flatness follows by miracle flatness.  The first map are smooth, and in particular its fibers are normal, and the fibers of the second map are normal by Theorem 6.7.3 of \emph{op. cit.}, thus the fibers of $\nu_P$ are normal.  This completes the proof.
\end{proof}

We note the following, though we do not use it. 
\begin{cor}
For parabolics $B \subset Q \subset P \subset G$, we have a natural sequence of maps 
$$\wt{\mf{g}} \longrightarrow \wt{\mf{g}}_Q \longrightarrow \wt{\mf{g}}_P \longrightarrow \mf{g}$$
whose affinization is
$$\mf{g} \times_{\mf{h}/\!/W} \mf{h} \longrightarrow \mf{g}\times_{\mf{h}/\!/W} \mf{h}/\!/W_Q \longrightarrow \mf{g}\times_{\mf{h}/\!/W} \mf{h}/\!/W_P \longrightarrow \mf{g}.$$
Furthermore, $H^i(\wt{\mf{g}}_P, \cO_{\wt{\mf{g}}_P}) = 0$ for $i > 0$.
\end{cor}

In particular, introducing the $G$-stackiness eliminates the difference in shifts, and the algebra object $\iota_* \omega_{\cZ_P/G}$ (resp. the coalgebra object $\iota_*\cO_{\cZ_P/G}$) and all of its iterated self-$!$-convolutions (resp. $*$-convolutions) are in the heart of the canonical $t$-structure of $\IndCoh(\cZ/G)$.  We thus have the following explicit identification of the algebra structure of $\iota_* \omega_{\cZ_P/G}$ and coalgebra structure of $\iota_*\cO_{\cZ_P}$.
\begin{cor}\label{spectral whittaker coalgebra}
There is a natural identification of the $n$-fold convolution 
$$\overbrace{\iota_* \omega_{\cZ_P/G} \star^! \cdots \star^! \iota_*\omega_{\cZ_P/G}}^n \simeq \iota_* \left( \omega_{\cZ_P/G} \tens{\cO(\mf{h}) \mkern-18mu \tens{\cO(\mf{h}/\!/W_P)}\mkern-18mu \cO(\mf{h})} \overbrace{\cO(\mf{h}) \tens{\cO(\mf{h}/\!/W_P)} \cdots \tens{\cO(\mf{h}/\!/W_P)} \cO(\mf{h})}^{n+1} \right) \in \Coh(\cZ/G),$$
where the algebra structure maps are induced by the averaging map $\cO(\mf{h}) \rihgtarrow \cO(\mf{h})^{W_P}$ taking 
$$f \mapsto |W_P|^{-1}\sum_{w \in W_P} w \cdot f.$$ Likewise, we have
$$\overbrace{\iota_* \cO_{\cZ_P/G} \star^* \cdots \star^* \iota_*\cO_{\cZ_P/G}}^n \simeq \iota_* \left( \cO_{\cZ_P/G} \tens{\cO(\mf{h}) \mkern-18mu \tens{\cO(\mf{h}/\!/W_P)}\mkern-18mu \cO(\mf{h})} \overbrace{\cO(\mf{h}) \tens{\cO(\mf{h}/\!/W_P)} \cdots \tens{\cO(\mf{h}/\!/W_P)} \cO(\mf{h})}^{n+1} \right) \in \Coh(\cZ/G),$$
where the coalgebra structure maps are given by insertion of the constant function.  In particular, the iterated convolutions above are in the heart of the standard $t$-structure.
\end{cor}
\begin{proof}
We prove the $!$-statement; the $*$-statement is analogous.  Since $\iota_*$ is monoidal, we can compute the $n$-fold convolution of $\omega_{\cZ_P}$ (or $\cO_{\cZ_P}$) on $\cZ_P$ rather than $\cZ$, i.e. $\iota_*\omega_{\cZ_P} \star \cdots \star \iota_*\omega_{\cZ_P} \simeq \iota_*(\omega_{\cZ_P} \star \cdots \star \omega_{\cZ_P})$.  The calculation of $\omega_{\cZ_P} \star \cdots \star \omega_{\cZ_P}$ follows from the general calculation in Section \ref{algebra explicit}, and by Proposition \ref{affinization calculation} which implies that
$$\phi_*\phi^!\omega_{\wt{\mf{g}}_P} \simeq \phi_* \omega_{\wt{\mf{g}}} \simeq \omega_{\wt{\mf{g}}_P} \otimes_{\cO(\mf{h}/\!/W_P)} \cO(\mf{h}).$$
The claims regarding the structure maps follow by computing the counit and unit:
$$\phi_* \phi^! \omega_{\wt{\mf{g}}_P} \simeq \omega_{\wt{\mf{g}}_P} \otimes_{\cO(\mf{h}/\!/W_P)} \cO(\mf{h}) \rightarrow \omega_{\wt{\mf{g}}_P}, \;\;\;\;\;\;\; \cO_{\wt{\mf{g}}_P} \rightarrow \phi_*\phi^* \cO_{\wt{\mf{g}}_P} \simeq \cO_{\wt{\mf{g}}} \otimes_{\cO(\mf{h}/\!/W_P)} \cO(\mf{h}).$$
The latter claim is clear; for the former, if $W$ is a finite group and $\phi: \wt{X} \rightarrow X$ is a $W$-torsor, then the Serre trace map $f_*\omega_{\wt{X}} \rightarrow \omega_X$ is locally given by the averaging map.  Thus, the same must be true for the counit for the adjunction $(\phi_*, \phi^!)$ for $\phi: \wt{\mf{g}} \rightarrow \wt{\mf{g}}_P$ after restricting to the regular semisimple locus, which determines the map.
\end{proof}

\subsubsection{}

Having understood $\iota_*\omega_{\cZ_P/G} \in \Coh(\cZ/G)^\heartsuit$ as an algebra object in an abelian category (i.e. in the notation of Definition \ref{alg in heart}, $\iota_*\omega_{\cZ_P/G} \in \Alg^\heartsuit(\Coh(\cZ/G))$), in view of Proposition \ref{algebra in heart prop} it is thus entirely determined by 1-categorical data.  Our next step is to characterize this 1-categorical object in terms of its ``monodromy'' bimodule structure, i.e. via the map of stacks $\cZ/G \rightarrow \mf{h} \times_{\mf{h}/\!/W} \mf{h}$.

Recall that the we have a canonical map $\iota_*\omega_{\cZ_P/G} \ra \omega_{\cZ/G}$ of algebra objects coming from the counit of an adjunction, and as we saw above these are objects in the heart of the natural $t$-structure.  We now show that it is a subobject and characterize it.  We note an analogous characterization was obtained in Proposition 5.8 of \cite{BL} in the strict case.
\begin{prop}\label{spectral whittaker characterization}
The algebra object (resp. coalgebra object)
$$\iota_*\omega_{\cZ_P/G} \in \IndCoh(\cZ/G)^{\heartsuit}, \;\;\;\;\;\;\;\iota_*\cO_{\cZ_P/G} \in \IndCoh(\cZ/G)^\heartsuit$$ 
is the minimal subobject of $\omega_{\cZ/G}$ (resp. quotient of $\cO_{\cZ/G}$) with the property that its restriction to the regular semisimple locus is precisely 
$$\left.\iota_*\omega_{\cZ_P/G}\right\vert_{\cZ^{rs}/G} = \bigoplus_{w \in W_P} \omega_{\cZ_w^{rs}/G}, \;\;\;\;\;\;\;\;\;\;  \left.\iota_*\cO_{\cZ_P/G}\right\vert_{\cZ^{rs}/G} = \bigoplus_{w \in W_P} \cO_{\cZ_w^{rs}/G}$$
where $\cZ_w^{rs} \subset \cZ^{rs}$ is the graph of the action of $w \in W$.
\end{prop}
\begin{proof}
That the map $\cO_{\cZ/G} \rightarrow \iota_*\cO_{\cZ_P/G}$ is surjective follows since both $\cZ$ and $\cZ_P$ are classical schemes and $\iota$ is a closed immersion.  Thus we have an exact triangle
$$\cK \longrightarrow \cO_{\cZ/G} \longrightarrow \iota_*\cO_{\cZ_P/G}$$
with all three objects living in the heart.  Applying the Grothendieck duality functor, we have
$$\iota_*\omega_{\cZ_P/G} \longrightarrow \omega_{\cZ/G} \longrightarrow \bD_{\cZ/G}(\cK).$$
Since $\cK$ is in the heart and $\cZ/G$ is Gorenstein of dimension 0 (i.e. its dualizing complex lives in the heart), we have $\bD_{\cZ/G}(\cK) \in \Coh(\cZ/G)^{\geq 0}$.  In particular, since $\iota_*\omega_{\cZ_P/G}$ and $\omega_{\cZ/G}$ live in the heart, the left map must be injective (and in particular we also conclude that $\bD_{\cZ/G}(\cK)$ lives in the heart). 

We now check the minimality property.  First, by miracle flatness and a dimension count using semi-smallness of the map $\wt{\mf{g}} \rightarrow \wt{\mf{g}}_P$, we deduce that $\cZ_P$ and $\cZ$ are flat over $\mf{h}$, using either projection.  Next, we note that $\cZ_P^{rs}$ is the union of the graphs of action by $w$ for $w \in W_P$.

Now, let us work in the following general setting.  Let $S$ be an irreducible locally Noetherian scheme with a dense open subscheme $U \subset S$.  Let $X$ be a locally Noetherian scheme flat over $S$, and $Z \subset S$ a closed subscheme flat over $S$.  Then, we claim that $Z$ is the minimal closed subscheme of $X$ whose restriction to $U$ Is $Z \times_S U \subset X \times_S U$.

To see this, let $Z'$ be any other closed subscheme such that $Z' \times_S U = Z \times_S U$ as closed subschemes of $X \times_S U$.  By replacing $Z'$ with $Z' \cap Z$, we can assume that $Z' \subset Z$.  Now, we have a surjection $\cO_Z \twoheadrightarrow \cO_{Z'}$ which is an isomorphism over $U$, thus its kernel is supported over the complement of $U$.  In particular, $Z$ has an associated point over the complement of $U$.  Since flat morphisms of locally Noetherian schemes take associated points to associated points, and the only associated point of $S$ is generic, this cannot be.

Now, applying this to our situation, where $S = \mf{h}$ and $U = \mf{h}^{rs}$, we note that by flatness of $\cZ_P$ and $\cZ$, the claimed characterization follows for $\cO_{\cZ_P}$.  The claim for $\omega_{\cZ_P}$ follows by Grothendieck duality.
\end{proof}

\begin{rmk}
In particular, $\iota_* \omega_{\cZ_P/G} \simeq \iota_*\cO_{\cZ_P/G}$ is both a quotient and subobject of $\omega_{\cZ/G} \simeq \cO_{\cZ/G}$.  The map $$\iota_*\omega_{\cZ_P/G} \rightarrow  \iota_*\omega_{\cZ/G} \simeq \iota_*\cO_{\cZ/G} \rightarrow \iota_*\cO_{\cZ_P/G}$$arising by composing the inclusion and quotient is not a splitting; for example, when $B = P$ it takes the constant function $1$ to the pullback of the function $\prod_{\alpha > 0} (1 - \alpha) \in \cO(\mf{h})$ cutting out the root hyperplanes.
\end{rmk}

\medskip

\subsubsection{}\label{sec completed spectral coalgebra}

The Steinberg scheme that appears naturally as a monodromic spectral affine Hecke category is not $\cZ = \wt{\mf{g}} \times_{\mf{g}} \wt{\mf{g}}$ but its completion along nilpotent elements; thus we must base-change the above set-up to this formal completion.  Let $\wh{\mf{h}/\!/W}_{\{0\}}$ denote the formal completion of $\mf{h}/\!/W$ along $\{0\}$. We have the following Cartesian squares of formal stacks (see Example \ref{complete stack}):
$$\begin{tikzcd}
\wh{\cZ_P}_{\cN}/G \arrow[r, "\wh{\iota}", hook] \arrow[d, hook] &\wh{\cZ}_{\cN}/G \arrow[r] \arrow[d, hook, "\wh{i}"] &\wh{\mf{h}}_{\{0\}} \arrow[r] \arrow[d, hook] &\wh{\mf{h}/\!/W}_{\{0\}} \arrow[d, hook] \\
\cZ_P/G \arrow[r, "\iota", hook] &\cZ/G \arrow[r] &\mf{h} \arrow[r] &\mf{h}/\!/W.
\end{tikzcd}$$
We let $\Gamma_{\cN} = \wh{i}_*\wh{i}^{!}: \IndCoh(\cZ/G) \rightarrow \IndCoh(\cZ/G)$ denote the local cohomology functor.  To apply the Grothendieck existence theorem (Theorem \ref{groth exist}), we use the notation
$$\mf{h}^\wedge/\!/W := \Spec \cO(\wh{\mf{h}/\!/W}_{\{0\}}), \;\;\;\;\;\;\;\;\;\; \mf{h}^\wedge =  \mf{h} \times_{\mf{h}/\!/W} \mf{h}^\wedge/\!/W = \Spec \cO(\wh{\mf{h}}_{\{0\}})$$
and, letting $j: \mf{h}^\wedge/\!/W \hookrightarrow \mf{h}/\!/W$ denote the natural map of affine schemes, we choose the dualizing complex to be the $*$-restriction as in Example \ref{gs duality ex}:
$$\omega_{\mf{h}^\wedge/\!/W} := j^* \omega_{\mf{h}/\!/W}.$$
We likewise define $\cZ_P^\wedge$ and $\cZ^\wedge$  by base change; since $\mf{h}, \cZ_P, \cZ$ are flat over $\mf{h}/\!/W$, the resulting schemes are underived.

\medskip

We now  prove the following analogue of Proposition \ref{spectral whittaker characterization} in this formally completed setting.  Given a stable $\infty$-category $\cat{C}$ with a $t$-structure, we use the notation $\cat{C}^\heartsuit[-r]$ to denote the full abelian subcategory of objects cohomologically concentrated in degree $r$.
\begin{cor}\label{completed spectral coalgebra}
We have a natural equivalence of left $\IndCoh(\wh{\cZ}_{\cN})$-module categories
$$\iota_* \omega_{(\wh{\cZ_P})_{\cN}/G}\dmod_{\IndCoh(\wh{\cZ}_{\cN}/G)} \simeq \IndCoh((\wh{\cZ_P})_{\cN}/G).$$
The algebra object $\iota_* \omega_{(\wh{\cZ_P})_{\cN}/G}$ is cohomologically concentrated in degree $\dim(H)$, i.e.
$$\Gamma_{\cN}(\iota_*\omega_{\cZ_P/G}) \simeq \iota_* \omega_{(\wh{\cZ_P})_{\cN}/G} \in \IndCoh(\cZ/G)^\heartsuit[-\dim(H)].$$
The same is true for the algebra object $\omega_{\wh{\cZ}_{\cN}/G}$, and the canonical map
$$\left(\iota_* \omega_{(\wh{\cZ_P})_{\cN}/G} \hookrightarrow \omega_{\wh{\cZ}_{\cN}/G} \right) \in \IndCoh(\cZ/G)^\heartsuit[-\dim(H)]$$
is injective.  Moreover, this map corresponds under the Grothendieck existence equivalence 
to an injective map of objects
\begin{equation}\label{food}
    \left(\omega_{\cZ_P^\wedge/G} \hookrightarrow \omega_{\cZ^\wedge/G}\right) \in \Coh(\cZ^\wedge/G)^\heartsuit
    \end{equation}
or, Grothendieck dually, to a surjective map of objects
\begin{equation}\label{hungry}\left(\cO_{\cZ^\wedge/G} \twoheadrightarrow \cO_{\cZ_P^\wedge/G} \right) \in \Coh(\cZ^\wedge/G)^\heartsuit
\end{equation}
satisfying the characterizing property in Proposition \ref{spectral whittaker characterization}.
\end{cor}
\begin{proof}
The first claim follows directly since the local cohomology functor is monoidal by Proposition \ref{def coalgebra O} and Corollary \ref{spectral whittaker convolution}.  Next, we establish the claim that $\Gamma_{\cN}(\iota_*\omega_{\cZ_P/G})$ is cohomologically concentrated in degree $r := \dim(H)$, which likewise gives the same for $\Gamma_{\cN}(\omega_{\cZ/G})$.  By base change and $t$-exactness of $\iota_*$, it suffices to check the claim for $\Gamma_{\cN}(\omega_{\cZ_P})$.  We may ignore $G$-equivariance, since pullback along $\cZ_P \rightarrow \cZ_P/G$ (or $(\wh{\cZ_P})_{\cN} \rightarrow (\wh{\cZ_P})_{\cN}/G$) is $t$-exact, i.e. we wish to check that $\Gamma_{\cN}(\omega_{\cZ_P})$ is cohomologically concentrated in degree $r$.

We make the following general claim: let $X \rightarrow S$ be a map of schemes, and let $S_0 \subset S$ be a lci closed subscheme of codimension $r$.  Let $\Gamma_0$ denote the local cohomology functor with respect to $S_0$ (or $X_0 = X \times_S S_0$).  Then, $\Gamma_0$ sends $S$-flat sheaves in $\QCoh(X)^\heartsuit$ to objects in $\QCoh(X)^\heartsuit[-r]$, i.e. complexes cohomologically concentrated in degree $r$, and the functor $H^r \circ \Gamma_0$ is exact on short exact sequences of $S$-flat sheaves.

In particular, since the claim is local on $S$, we may assume $S = \Spec(A)$ is affine, and define the closed subscheme $S_0 = \Spec(A/f_1, \ldots, f_r)$ by a regular sequence $f_1, \ldots, f_r$.  Then, we have a system of formal neighborhoods
$$S_0 = \Spec(A/f_1, \ldots, f_r) \hookrightarrow  \cdots \hookrightarrow S_n = \Spec(A/f_1^{n+1}, \ldots, f_r^{n+1}) \hookrightarrow \cdots \hookrightarrow S = \Spec(A)$$
where each $i_n: S_n \hookrightarrow S$.  Note that we have
$$\Gamma_0(\cF) = \colim_n i_{n*} i_n^! \cF.$$
Thus, it suffices to show that $i_{n*}i_n^! \cF$ is cohomologically concentrated in degree $r$, and that the transition maps $H^r(i_{n*}i_n^! \cF) \rightarrow H^r(i_{n+1*} i_{n+1}^! \cF)$ are injective (i.e. so that the derived colimit is the colimit).  Both claims follow by direct calculation, i.e. by flatness of $\cF$ we can check the claims on the Koszul resolutions themselves, whereby the vanishing claim follows since Koszul resolutions are self-dual up to  a shift, and injectivity follows since the transition map is given on $n$th cohomology of Koszul complexes by multiplication by $f_1 \cdots f_r$ (which must be a non-zerodivisor).

Finally, in our setting, we apply the above to the setting where $X = \cZ_P$, which is flat over $S = \mf{h}$ on either side.

Next, we check the conditions of Theorem \ref{groth exist}, taking $S = \mf{h}^\wedge/\!/W$.  Conditions (a) and (c) may be easily checked; for (b) we let $\omega_S = \cO_{\mf{h}}[\dim(H)]$.  For (d), we note that $\wh{\cZ}$ is proper over $\Spec \cO(\wh{\mf{g}}_{\cN})$, whose $G$-equivariant global sections is equal to $\cO(\wh{\mf{h}})$.  For (e), we use the classical Grothendieck existence theorem. 

It is automatic that Theorem \ref{groth exist} sends the dualizing sheaf to the dualizing sheaf; moreover, by the Calabi-Yau property in Proposition \ref{calabi yau}, we have that the dualizing sheaves are concentrated in degree 0.  To see that the map is injective, it suffices to show that the cone of \eqref{food} is also in degree 0.  To see this, we can pass through Grothendieck duality, i.e. it suffices to check that the map \eqref{hungry} is surjective, but this is clear.

Finally, having passed through the Grothendieck existence theorem, the arguments in Proposition \ref{spectral whittaker characterization} carry over in the same way to the analogous stacks over $\mf{h}^\wedge/\!/W$. 
\end{proof}

\begin{rmk}
Note that by Lemma \ref{wcoh prop}, our algebra object of interest is $\omega_{\wh{\cZ}_{\cN}}$ is no longer a compact object of $\IndCoh_{\cN}(\cZ)$; however it does live in the small subcategory $\hCoh(\wh{\cZ}_{\cN})$, thus is compact in the renormalization $\wh{\IndCoh}_{\cN}(\cZ) := \Ind(\hCoh_{\cN}(\cZ))$.   Furthermore, the algebra object of interest is $\Gamma_{\cN}(\iota_* \omega_{\cZ_P/G}) \simeq \iota_* \omega_{\wh{\cZ_P}_{\cN}/G}$, while the coalgebra object is $\Gamma_{\cN}(\iota_*\cO_{\cZ_P/G})$.  This latter object is \emph{not} equivalent to the Grothendieck-Serre dual object $\cO_{\wh{\cZ_P}_{\cN}}$; recall (see Remark \ref{convolution duality}) that the relationship between $*$ and $!$-convolutions is given by twisting by the dualizing bundle, not by duality.
\end{rmk}

\subsection{Centrally monodromic affine Hecke categories and Koszul duality}\label{sec spectral monodrom}

\sss
We begin by discussing some generalities on Koszul duality from Section 3 in \cite{preygel MF} (see in particular Construction 3.1.5), in the setting of the convolution structures in Section \ref{convolution sec}.

As before, we fix a smooth, perfect, locally Noetherian stack $S$, and let $W, Y$ be smooth, perfect stacks which are proper over $S$.  In addition, we fix a vector space $V$, and assume that $Y$ has a map to $V$ such that the derived fiber
$$Y_0 := Y \times_V \{0\}$$ 
is still smooth (it is automatically perfect and proper over $S$).  We define
$$X := W \times_S Y, \;\;\;\;\;\;\;\;\;\; Z := Y \times_S Y.$$
The category $\IndCoh(X)$ will be considered as a right module category for $\QCoh(V)$, and $\IndCoh(Z)$ as a $\QCoh(V)$-bimodule category.  Via convolution, the category $\IndCoh(X)$ has a right action of the monoidal category $\IndCoh(Z)$ compatible with the actions of $\QCoh(V)$, i.e. the action map descends
$$\IndCoh(X) \otimes_{\QCoh(V)} \IndCoh(Z) \longrightarrow \IndCoh(X).$$
Furthermore, we consider the derived fibers
$$X_0 := W \times_S Y_0, \;\;\;\;\;\;\;\;\;\; {}_0Z_0 := Y_0 \times_S Y_0.$$
By Theorem 1.2.4 in \cite{BNP convolution}, $\IndCoh(X_0)$ is a module category for the monoidal $\IndCoh({}_0Z_0)$.

Moreover, $\IndCoh(X_0)$ attains a right $\Mod(\Sym V[-2])$-module structure, and $\IndCoh({}_0Z_0)$ a compatible bimodule structure, via Koszul duality.
That is, letting $i: X_0 \hookrightarrow X$ be the inclusion map, the $!$-restriction may be identified
\begin{equation}\label{chickadee}
\begin{tikzcd}[column sep=20ex]
    \Coh(X) \arrow[r, "i^!"] \arrow[d, equals]&  \Coh(X_0) \arrow[d, "{\text{\cite[Thm. 1.2.4]{BNP convolution}}}"]\\
    \Coh(X) \arrow[r, "{\cF \mapsto (M \mapsto M \otimes \cF)}"'] &  \cat{Fun}_{\Coh(V)}(\Coh(\{0\}), \Coh(X)).
\end{tikzcd}
\end{equation}
The category $\Coh(\{0\} \times_V \{0\}) \simeq \cat{End}_{\Coh(V)}(\Coh(\{0\})$ acts naturally on this category by convolution.  Thus, the ind-completion $\IndCoh(\{0\} \times_V \{0\})$ acts on $\IndCoh(X_0)$.  Moreover, we have a monoidal Koszul duality equivalence
\begin{equation}\label{kd eq}
\begin{tikzcd}
((\IndCoh(\{0\} \times_V \{0\}), \star) \arrow[r, "\simeq"] & \Sym V[-2]\mod, \otimes).
\end{tikzcd}\end{equation}
That is, $\IndCoh(X_0)$ has a natural right $\Sym V[-2]$-linear structure arising from the $\Sym V^*$-linear structure on $\IndCoh(X)$, and likewise for $\IndCoh({}_0Z_0)$.

By Corollary 3.2.4 of \cite{preygel MF}\footnote{The setting in \emph{op. cit.} is for $X$ smooth, but this is not used in the proof.} (see also Proposition A.6 of \cite{ihes}) the functor
$$\IndCoh(X_0) \tens{\Sym V[-2]} \Vect_k \hookrightarrow \IndCoh(X)$$
is fully faithful, with essential image $\IndCoh_{X_0}(X)$.  Furthermore, this category has a module structure for
$$\End_{\Perf(\Sym V[-2])}(\Perf(k)) \simeq \Mod_{f.d.}(\Sym V[-1]).$$
The Koszul dual algebra to $\Sym V[-1]$ is $\End_{\Sym V[-1]}(k, k) \simeq \cO(\wh{V})$ where $\wh{V}$ is the completion of $V$ at $0$; then, similarly to (\ref{kd eq}) this gives rise to a $\cO(\wh{V})$-linear structure on $\IndCoh_{X_0}(X) \simeq \IndCoh(\wh{X}_{X_0})$ which agrees with the original $\cO(\wh{V})$-module structure arising via pullback.

By compatibility of the actions, the $\Mod(\Sym V[-2])$-action on $\IndCoh(X_0)$ may also be realized via the right action on the monoidal unit of $\IndCoh({}_0Z_0)$ followed by action of $\IndCoh({}_0Z_0)$ on $\IndCoh(X_0)$.  In particular, there is a map
\begin{equation}\label{central mon hh eq}
    \Sym V[-2] \longrightarrow \End^\bullet_{{}_0Z_0}(1_{\IndCoh({}_0Z_0)}) \hookrightarrow HH^\bullet(\IndCoh(X_0)).
\end{equation}

\subsubsection{} In this section we set
$$\cZ_P := \wt{\cN}_P \times_{\mf{g}} \wt{\cN}_P, \;\;\;\;\;\;\;\;\;\; \cZ'_P := \wt{\cN}_P' \times_{\mf{g}} \wt{\cN}_P'$$
and denote the corresponding partial affine Hecke categories and monodromic affine Hecke category by
$$\fH_P := \IndCoh(\cZ_P/G), \;\;\;\;\;\;\;\;\;\; \fH_P' := \IndCoh(\cZ'_P/G), \;\;\;\;\;\;\;\;\;\fH := \IndCoh(\wh{\cZ}/G).$$
Consider the $(\fH, \fH_P)$ and $(\fH, \fH_P')$-bimodule categories
$$\IndCoh_{\cN}(\wt{\mf{g}}/G \times_{\mf{g}/G} \wt{\cN}_P/G), \;\;\;\;\;\;\;\;\;\; \IndCoh_{\cN}(\wt{\mf{g}}/G \times_{\mf{g}/G} \wt{\cN}_P'/G).$$ 
We now establish a Koszul dual relationship between the two as follows.
\begin{prop}\label{spectral central}
Considering the right $\fH_P$-action on $\IndCoh_{\cN}(\wt{\mf{g}}/G \times_{\mf{g}/G} \wt{\cN}_P/G)$, (\ref{central mon hh eq}) gives an injection
$$\Sym \mf{z}_{L_P}[-2] \hookrightarrow \End(1_{\fH_P}) \hookrightarrow HH^\bullet_{\fH}(\IndCoh_{\cN}(\wt{\mf{g}}/G \times_{\mf{g}/G} \wt{\cN}_P/G))$$
identifying $\Sym \mf{z}_{L_P}[-2]$ with the subalgebra of $\End(1_{\fH_P})$ generated in degree 2.  In turn, this gives rise to an equivalence of left $\fH$-module categories
$$\IndCoh_{\cN}(\wt{\mf{g}}/G \times_{\mf{g}/G} \wt{\cN}_P'/G) \simeq \IndCoh_{\cN}(\wt{\mf{g}}/G \times_{\mf{g}/G} \wt{\cN}_P/G) \tens{\Sym \mf{z}_{L_P}[-2]} \Vect_k.$$
\end{prop}
\begin{proof}  
We apply the above discussion to
$$X = \wt{\mf{g}}/G \utimes{\mf{g}/G} \wt{\cN}_P'/G, \;\;\;\;\;\;\; Z = \cZ_P'/G, \;\;\;\;\;\;\; V = \mf{z}_{L_P}.$$
using the calculation in Proposition \ref{affinization calculation}. 
Taking the $!$-fiber at $0 \in \mf{z}_{L_P}$, we have 
a right action of $\fH_P$ on $\IndCoh_{\cN}(\wt{\mf{g}}/G \times_{\mf{g}/G} \wt{\cN}_P/G)$, and a compatible action of $\Sym \mf{z}_{L_P}[-2]$ on both.  
Furthermore, we have a fully faithful functor
$$\IndCoh_{\cN}(\wt{\mf{g}}/G \times_{\mf{g}/G} \wt{\cN}_P/G) \tens{\Mod(\Sym \mf{z}_{L_P}[-2])} \Vect_k \hookrightarrow \IndCoh_{\cN}(\wt{\mf{g}}/G \times_{\mf{g}/G} \wt{\cN}_P'/G)$$
identifying the left category with the subcategory generated under colimits by the essential image of $\kappa_*$ where $\kappa: \nc{} \times_{\mf{g}} \na{P} \rightarrow \nc{} \times_{\mf{g}} \nb{P}$, i.e. sheaves set-theoretically supported over the nilpotent cone $\IndCoh_{\cN}(\wt{\mf{g}}/G \times_{\mf{g}/G} \wt{\cN}_P'/G)$. 

It remains to establish that the map $\Sym \mf{z}_{L_P}[-2] \rightarrow \End(1_{\fH_P})$ is injective and that $\mf{z}_{L_P} \simeq \Ext^2(1_{\fH_P})$.  The map is injective since the Koszul dual map $\cO(\mf{z}_{L_P}) \rightarrow \cO(\cZ_P'/G)$ is injective, and for surjectivity in degree 2 we postpone the argument to the proof of Theorem \ref{thm conjectures}.  Given the statement in \emph{loc. cit.} for strict affine Hecke categories, i.e. the monoidal identification
$$\IndCoh_{\cN[1]}(\cZ_P) \simeq \Dmod(I_{\chP} \bs \chG_F / I_{\chP})$$
it then  becomes evident that the endomorphisms of the monoidal unit are $C^\bullet(B\chP)$, which is a formal dg algebra generated in even degrees with $H^2(B\chP) \simeq \mf{z}_{L_P}$ as desired.   This concludes the proof.
\end{proof}

\begin{rmk}
    We note that although we use a statement from Theorem \ref{thm conjectures} involving strict affine Hecke categories, Proposition \ref{spectral central} above is not used in the proof of that statement, so our logic is not circular.
\end{rmk}

\section{Small objects in automorphic categories}\label{sec small ob}

\subsection{Overview}

\sss In this section, we would like to record some basic facts about various notions of `small' D-modules, and some relations between them. Let us now provide a preliminary discussion of what this entails.

\subsubsection{} A notion of being small is simply being a compact object. Such objects will be fit into larger small categories in the two following ways.

\subsubsection{}\label{dmod defn} First, consider a QCA stack $X$, which is roughly an Artin stack whose points have affine automorphism groups, cf. Definition 1.1.8 of \cite{QCA}. On such a stack, a basic result of {\em op. cit.} is that one typically has more coherent D-modules than compact ones. That is, if we denote the \emph{coherent} D-modules by $\Dmod_c(X)$, and the compact, i.e., {\em safe}, D-modules, by $\Dmod_s(X)$, the inclusion 
$$\Dmod_s(X) \hookrightarrow \Dmod_c(X)$$
is typically not an equivalence.  We recall the prototypical example of such a coherent but non-compact object is the constant sheaf on the classifying stack of a finite dimensional group with nontrivial reductive quotient.


Coherent D-modules may be characterized by the explicit property that they actually `look' coherent, i.e., after pulling back along any smooth map from a scheme, one obtains a finite complex of D-modules with finitely generated cohomology groups. Alternatively, such D-modules may be characterized intrinsically within the category, equipped with its $t$-structure, as the {\em pseudo-compact objects}.\footnote{These objects are typically called coherent or $t$-coherent, i.e. in the setting of a category with an appropriate $t$-structure, the objects which are almost compact and $t$-bounded.  We introduce the term pseudo-compact to avoid overloading the term coherent.}

One can similarly make sense of coherent D-modules on certain infinite dimensional schemes and stacks, such as a loop group or  Hecke stack, and they admit similar characterizations. We spell this out in Sections \ref{dmod placid} and \ref{dmod placid equiv} below.

\subsubsection{}  The second notion of `small' objects arises when studying on usual varieties their unipotent derived local systems, i.e., the full subcategory of D-modules generated under colimits by the constant sheaf.\footnote{The unipotence here is somewhat an artifact of our case of interest. Replacing the constant sheaf with any other semisimple local system is straightforward, but a treatment of all local systems at once for a general variety seems more delicate.} 

On the one hand, one has the objects which are compact in the category of all D-modules. These are simply objects in the pre-triangulated hull of the constant D-module itself, i.e., bounded complexes whose cohomology sheaves are finite rank unipotent local systems. 

However, one frequently encounters, in addition to these objects, a few more derived local systems. In some instances, e.g., the  free-monodromic sheaves of Bezrukavnikov--Yun \cite{BezYun}, the complexes are bounded but with cohomology sheaves of infinite rank. In other instances, e.g., the monoidal unit of the monodromic Hecke category $$\Dmod(G\mon\!\bs G / G\mon)$$of a semisimple group $G$, the complexes are unbounded but with cohomology sheaves of finite rank. In general, e.g., for the monodromic Hecke category of a general reductive group, or the character sheaves of categorical representations generated by their $G$-invariant vectors, both phenomena appear. One would like to articulate in what sense the encountered objects are nonetheless far from general unipotent derived local systems. 

We make the following basic observation about such objects. Namely, one can characterize such local systems $\mathscr{E}$ by the property that their de Rham cohomology is finite dimensional. The prototypical example here is the \emph{cofree-monodromic} sheaf on a variety $X$, cf. Definition \ref{cofree mon def}, which is a certain iterated self-extension of the constant sheaf, cf. Section \ref{s:cofree} below, characterized by the property that its global sections are $k$ in degree zero. We then describe a similar enlargement for the category of perverse sheaves on a stratified space with unipotent monodromy along the strata in Section \ref{s:lasts4}, which for a single stratum recovers the previous discussion.  We call the resulting objects {\em nearly compact}, cf. Definition \ref{d:defnc1}.

\subsubsection{} It worth emphasizing that these two more expansive notions of small objects are in fact closely related by Koszul duality. 

Let us first articulate this in representation theoretic language. Namely, one can understand the issue of safety vs. coherence, when working with a quotient stack $X/G$, as an instance of the phenomenon that for a $\Dmod(G)$-module $\sC$, the forgetful functor
\begin{equation}\label{e:big}\on{Oblv}: \sC^G \rightarrow \sC^{G\mon}\end{equation}may send noncompact objects to compact objects, and so the preimage of the compact objects of $\sC$ under $\on{Oblv}$ gives an enlargement of the compact objects of $\sC^G$.  

One may similarly understand the appearance of cofree-monodromic sheaves as an instance of the fact that the averaging functor 
\begin{equation}\label{e:bowwow}\on{Av}^G_*: \sC^{G\mon} \rightarrow \sC^G\end{equation}may send noncompact objects to compact objects. Again, the preimage of the compact objects of $\sC^G$ under $\on{Av}^G_*$ gives an enlargement of the compact objects of $\sC^{G\mon}$.  

In this way, the situations of \eqref{e:big} and \eqref{e:bowwow} are swapped by Koszul duality for Hecke categories \cite{BezYun}. 

Somewhat alternatively, this may be situated in the appearance of Koszul duality in homotopy theory, as follows.\footnote{We emphasize that the general connection of Koszul duality in topology to renormalization phenomena and Koszul duality in representation theory is not new, e.g., may be perceived in various forms in the works of Ben-Zvi--Nadler as well as Arinkin--Gaitsgory--Kazhdan--Raskin--Rozenblyum--Varshavsky, cf. in particular Appendix E.2.6 of \cite{AGKRRV} where it is (almost) explicit.} A beautiful result from the algebro-topological literature, with its origins in the work of Adams--Hilton, Eilenberg--Moore, Rothenberg--Steenrod, and Dwyer, is that for suitable spaces $X$, e.g., a simply connected finite CW complex, one has a Koszul duality between cochains on $X$ and chains on the based loop space $\Omega X$, 

$$\Hom_{C^\bullet(X)}(k, k) \simeq C_\bullet(\Omega X) \quad \quad \text{and} \quad \quad \Hom_{C_\bullet(\Omega X)}(k, k) \simeq C^\bullet(X),$$
where $k$ denotes natural augmentation modules for $C^\bullet(X)$ and $C_\bullet(\Omega X)$. 

For a stack of the form $X = BG$, the usual category of D-modules on it is $C_\bullet(\Omega X)\mod$, and instead working with ind-coherent D-modules yields $C^\bullet(X)\mod$. For varieties $X$, the subcategory of D-modules generated by the constant sheaf is $C^\bullet(X)\mod$, and its renormalization wherein the cofree-monodromic sheaf is compact is, for suitable $X$, $C_\bullet(\Omega X)\mod$.

\subsection{Coherent D-modules on placid ind-schemes}\label{dmod placid}

\subsubsection{} Recall that on a scheme $X$ of finite type, its category of D-modules is of finite cohomological dimension. As such, compact objects are simply bounded complexes of D-modules with coherent cohomology, which we will simply call {\em coherent} below.

We now show the same phenomenon persists for placid ind-schemes. 

\subsubsection{} Let $Z$ be a placid scheme. I.e., it admits a presentation as a cofiltered inverse limit
$$Z \simeq \varprojlim Z_\alpha,$$
where each $Z_\alpha$ is finite type and all the transition maps $\pi_{\alpha, \beta}: Z_\beta \rightarrow Z_\alpha$ are smooth and affine. 

Recall that its category of D-modules may be presented as 
$$\Dmod(Z) \simeq \varinjlim \Dmod(Z_\alpha),$$
where the transition maps are $\pi^*$. It carries a unique $t$-structure for which each 
$\pi^*_\alpha: \Dmod(Z_\alpha) \rightarrow \Dmod(Z)$
is $t$-exact, which is automatically compatible with filtered colimits.

\subsubsection{} Let $\sC$ be a category equipped with a right complete $t$-structure, which we always will assume is compatible with filtered colimits. We recall that an object $c$ of $\sC^+$ is called {\em pseudo-compact} if for any filtered diagram of objects $c_\gamma$ of $\sC^{\geqslant 0}$, one has 

$$\varinjlim \Hom(c, c_\gamma) \simeq \Hom(c, \varinjlim c_\gamma).$$
I.e., homomorphisms out of $c$ commute with uniformly bounded below colimits.

\begin{lemma} \label{l:nosurprises}Suppose that $\sC$ is compactly generated, and the compact objects of $\sC$ are closed under truncation functors. Then any pseudo-compact object of $\sC$ is compact. 
\end{lemma}

\begin{proof} Suppose $c$ is an pseudo-compact object, and write it as a filtered colimit of compact objects $c \simeq \varinjlim c_\gamma$. As $c$ lies in $\sC^{\geqslant n}$ for some $n \in \mathbb{Z}$, and our $t$-structure is compatible with filtered colimits, we have 
$$c \simeq \tau^{\geqslant n} \varinjlim c_\gamma \simeq \varinjlim \tau^{\geqslant n} c_\gamma.$$By the pseudo-compactness of $c$, it follows that $c$ is a retract of $\tau^{\geqslant n} c_{\gamma'}$ for some $c_\gamma'$. By our assumption that compact objects are closed under truncation functors, it follows that $c$ is a summand in a compact object, hence compact.  \end{proof}

\subsubsection{} Let us return to the setting of D-modules on a placid scheme. We now have the following assertion. 

\begin{prop} Let $c$ be an object of $\Dmod(Z)$. Then the following are equivalent:\label{p:cohplacidscheme}

\begin{enumerate}
    \item $c$ is pseudo-compact,
    \item $c$ may be written as $\pi^*_\alpha(c_\alpha)$ for a coherent D-module on $Z_\alpha$ for some $\alpha$,
    \item $c$ is compact.
\end{enumerate}

\end{prop}

\begin{proof} It is standard and 2. and 3. are equivalent, and it follows that 3. implies 1. To see that 1. implies 3., we may apply Lemma \ref{l:nosurprises}. \end{proof}

We call an object satisfying the equivalent conditions of the above proposition a coherent D-module on $Z$. 

\subsubsection{} Let $Y$ be a placid ind-scheme. I.e., it admits a presentation as a filtered colimit $$Y \simeq \varinjlim Y_\beta,$$
where each $Y_\beta$ is a placid scheme and the transition maps $\iota_{\gamma, \beta}: Y_\beta \rightarrow Y_\gamma$ are finitely presented closed embeddings. 

Recall that its category of D-modules may be presented as 
$$\Dmod(Y) \simeq \varinjlim \Dmod(Y_\beta),$$
where the transition maps are $\iota_{\gamma, \beta, *}$. With this, any compact object of $\Dmod(Y)$ is the $*$-extension of a coherent D-modules from $Y_\beta$. We call such objects coherent D-modules on $Y$.

\begin{prop}\label{p:cohplacidindscheme} Suppose $Y$ admits a {\em dimension theory}, which is roughly a consistent assignment of `semi-infinite' dimensions to compact open subschemes, cf. Section 6.10 of \cite{raskindmod}.
Consider the associated $t$-structure on $\Dmod(Y)$. Then an object $c$ of $\Dmod(Y)$ is compact if and only if it is pseudo-compact.
\end{prop}

\begin{proof}By the construction of the $t$-structure, the compact objects of $\Dmod(Y)$ are bounded and closed under truncation functors. In particular, they are pseudo-compact, and the converse follows again by Lemma \ref{l:nosurprises}.  
\end{proof}

\subsection{Equivariant coherent D-modules on certain placid ind-schemes}
\label{dmod placid equiv}

\sss
We now extend the contents of the previous section to an equivariant setting which is sufficient for our purposes, namely after taking the quotient by the action of a group scheme $H$ whose prounipotent radical is of finite codimension. 

The main result of this subsection is that there is now a difference between coherent and compact objects, and the former may be characterized as the pseudo-compact objects.

\subsubsection{} Let $X$ be a placid scheme equipped with an action of $H$. In this case, we have an action of $\Dmod(H)$ on $\Dmod(X)$, and in particular may consider the equivariant objects

$$\on{Oblv}: \Dmod(X)^H \rightleftarrows \Dmod(X): \on{Av}_*^H.$$

We recall that the left-hand category carries a unique $t$-structure with the property that $\on{Oblv}$ is $t$-exact. In particular, as $\on{Oblv}$ is conservative, the $t$-structure is compatible with filtered colimits, and so we may speak of pseudo-compact objects.

\begin{prop}  Let $X$ be a placid scheme; for an object $c$ of $\Dmod(X)^H$, the following are equivalent. 

\begin{enumerate}
    \item $\on{Oblv}(c)$ is compact, i.e, a coherent D-module, 
    \item $c$ is pseudo-compact.
\end{enumerate}
    \label{p:coheqplacidscheme}
\end{prop}

\begin{proof}To see that 2. implies 1., by Proposition \ref{p:cohplacidscheme}
it is enough to see that $\on{Oblv}(c)$ is pseudo-compact. However, this follows by adjunction from the fact that $\on{Av}_*^H$ is left exact.

To see that 1. implies 2., recall that $\Dmod(X)^H$ may be presented as the limit of the simplicial diagram 
$$\Dmod(X)^H \simeq \varprojlim_{n \in \Delta^{op}} \Dmod(X) \otimes \Dmod(H)^{\otimes n}.$$
In particular, if $c$ and $d$ are objects of $\Dmod(X)^H$ such that $\on{Oblv}(c)$ is compact, we may compute the homomorphisms between them as the direct product totalization of a bicomplex of vector spaces 
$$\Hom_{\Dmod(X)}(c, d) = \Tot_n\left( \Hom_{\Dmod(X)}(c,d) \otimes C^\bullet(H)^{\otimes n}\right).$$
In particular, given a filtered colimit  $\varinjlim d_\gamma$ where all the $d_\gamma \in \Dmod(X)^{H, \geqslant 0}$, as $c$ is bounded below we have that $\Hom_{\Dmod(X)}(c,d_\gamma)$ are all uniformly bounded below. 

In particular, to see the canonical map
$$\varinjlim \Hom_{\Dmod(X)^H}(c, d_\gamma) \rightarrow \Hom_{\Dmod(X)^H}(c, \varinjlim d_\gamma)$$is an equivalence, by the discussed boundedness all the relevant direct product totalizations coincide with their direct sum totalizations, so we may pass the filtered colimit through the totalization, as desired. \end{proof}

\subsubsection{} Now suppose that $Y$ is a placid ind-scheme acted on by $H$. For simplicity, we will assume there exists a presentation $\varinjlim Y_\beta$ as above such that each $Y_\beta$ moreover carries an action of $H$ and the transition maps are $H$-equivariant. 

Again, given an object $c$ of $\Dmod(Y)^H$, we may call it coherent if $\on{Oblv}(c)$ is. Explicitly, any such object is $*$-extended from some $Y_\beta$, where it is an pseudo-compact object of $\Dmod(Y_\beta)^H$ by Proposition \ref{p:coheqplacidscheme}. As in Proposition \ref{p:cohplacidindscheme}, we may obtain the following. 

\begin{prop} Suppose $Y$ admits a dimension theory, and consider the associated $t$-structure on $\Dmod(Y)^H$. Then an object $c$ of $\Dmod(Y)^H$ is coherent if and only if it is pseudo-compact.  
\end{prop}

\begin{proof} By our assumptions, we have
$$\Dmod(Y)^H \simeq \varinjlim \Dmod(Y_\beta)^H,$$where the $t$-structures on each $\Dmod(Y_\beta)^H$ are a shift of those discussed in Proposition \ref{p:coheqplacidscheme}, and the transition functors $\iota_{\gamma, \beta, *}$ are $t$-exact and admit continuous right adjoints $\iota_{\gamma, \beta}^!$.

If $c$ is coherent, by definition it is the $*$-extension of a coherent object of $c_\beta$ of $\Dmod(Y_\beta)^H$ for some $\beta$. Moreover, $c_\beta$ is pseudo-compact by Proposition \ref{p:coheqplacidscheme}. The pseudo-compactness of $c$ now follows from the exactness of $\iota_{\beta, *}$ and the left exactness of $\iota_{\beta}^!$. 

Conversely, consider an object $c$ which is pseudo-compact. As with any object of $\Dmod(Y)^H$, we have that 
$$c \simeq \varinjlim i_{\beta, *} i_\beta^! c.$$
By the pseudo-compactness of $c$, it follows that $c$ is a summand of $i_{\beta, *} i_\beta^! c$ for some $\beta$, and in particular is of the form $i_{\beta, *} c_\beta$. By the fully faithfulness and $t$-exactness of $i_{\beta, *}$, it follows that $c_\beta$ is an pseudo-compact object of $\Dmod(Y_\beta)^H$, and hence coherent by Proposition \ref{p:coheqplacidscheme}. 
\end{proof}

Fix a prounipotent subgroup $K \subset H$ of finite codimension. Recall that the forgetful functor $$\on{Oblv}: \Dmod(Y)^K \rightarrow \Dmod(Y)$$is now fully faithful. In particular, an object $c$ of $\Dmod(Y)^K$ is compact if and only if $\on{Oblv}(c)$ is, hence in this case pseudo-compactness is equivalent to compactness.

We therefore obtain the following. 

\begin{cor}\label{coherent is compact} For $Y, H,$ and $K$ as above, and an object $c$ of $\Dmod(Y)^H$, the following are equivalent. 
\begin{enumerate}
    \item $c$ is coherent,
    \item $\on{Oblv}(c) \in \Dmod(Y)^K$ is compact. 
\end{enumerate}
\end{cor}

\subsection{Cofree-monodromic sheaves}\label{sec cofree}

\sss
In the previous subsection, we studied a particular case of the following general phenomenon. Given a categorical representation $\sC$ of a group scheme $H$ whose prounipotent radical is of finite codimension.  the forgetful functor 
$$\on{Oblv}: \sC^H \rightarrow \sC$$
may send non-compact objects to compact ones. In the presence of a suitable $t$-structure, one can nonetheless argue such non-compact objects are pseudo-compact. 

\subsubsection{}\label{sss:nearcompactdef}
In the remaining parts of this section, we study the Koszul dual phenomenon for the averaging functor (see Section \ref{averaging defin section}).
$$\on{Av}_*^H: \sC^{H\mon} \rightarrow \sC^H.$$
\begin{defn} \label{d:defnc1} Let us say an object $c$ of $\sC^{H\mon}$ is {\em nearly compact} if $\on{Av}_*^H(c)$ is compact.     
\end{defn}

As we will see below,  the nearly compact objects, and their ind-completion, are familiar in geometric representation theory, and appear naturally in the study of Koszul duality. Moreover, when the category $\sC$ in question is sheaves on a variety,  the study of nearly compact objects is nearly the same as the study of certain derived local systems, namely the {\em cofree-monodromic sheaves}. We will begin with a discussion of the latter objects, and then return later to the connection with near compactness.


\subsubsection{} Let $X$ be a placid scheme. We recall that in general the components $\pi_0(X)$ of $X$ form a profinite set, but for simplicity we will assume this is in fact finite. For $\gamma \in \pi_0(X)$, we will denote by $X_\gamma$ the corresponding component.

Let us denote the full subcategory of D-modules on $X$ generated by the constant D-module by 
$$\iota: \begin{tikzcd} \on{IndLisse}(X)^{uni} \arrow[r, hook, shift left] & \arrow[l, shift left] \Dmod(X)\end{tikzcd} : \on{Av}^{uni}_*,$$ 
which we think of as (certain) derived local systems with unipotent monodromy. This category is also known as the category of ind-lisse unipotent sheaves \cite{AGKRRV}.   We have a tautological equivalence 
$$\on{IndLisse}(X)^{uni} \simeq C^\bullet(X)\mod,$$
where $C^\bullet(X)$ denotes the de Rham cohomology of $X$, i.e., the endomorphisms of the constant D-module. 

In the right hand category, there are (at least) two reasonable notions of a small object. On the one hand, one has the perfect complexes, corresponding to  the pretriangulated envelope of the constant sheaf itself. On the other hand, one may also consider the category of $C^\bullet(X)$-modules whose underlying vector space is perfect. Under the assumption that $C^\bullet(X)$ is perfect, this contains the compact objects, but typically more.

\subsubsection{} To identify this more explicitly, let us recall the following basic maneuvers with modules over algebras like $C^\bullet(X)$. 

\begin{prop} \label{p:coconnectivealgs}
Let $A$ be a coconnective\footnote{Note that this means $H^n(A) = 0$ for $n < 0$, which is {\em a priori} less strict than requiring that $A$ is quasi-isomorphic to a dg algebra $B$ such that $B^n = 0$ for $n < 0$.} $dg$-algebra and $S$ a semisimple algebra, i.e. an algebra concentrated in degree zero which is completely reducible as a left module over itself.  Suppose we are given a map 
$$S \rightarrow A,$$
inducing an isomorphism $S \simeq H^0(A)$. Then the following hold. 
\begin{enumerate}
    \item For any $A$-module $M$ and integer $n$, there exists a (non-unique) exact triangle in $A\mod$
$$M^{> n} \rightarrow M \rightarrow M^{\leqslant n} \xrightarrow{+1}$$
inducing isomorphisms $H^i(M^{> n}) \simeq H^i(M)$, for $i > n$, and $H^i(M) \simeq H^i (M^{\leqslant n}),$ for $i \leqslant n$.

\item For any $S$-module $L$ concentrated in cohomological degree zero, and integer $n$, there exists an $A$-module $$K(L, n)$$
characterized up to non-unique isomorphism by the properties that the cohomology of $K(L, n)$ is concentrated in degree $n$, and is isomorphic to $L$ as an $H^0(A)$-module.

\end{enumerate}

\end{prop}

\begin{proof} Let us sketch a proof, essentially following Propositions 3.3 and 3.9 of \cite{dgi}. Recall the induction and restriction adjunction 
$$\on{Ind}: S\mod \rightleftarrows A\mod: \on{Oblv}.$$
By the coconnectivity of $S$, $S\mod$ carries a $t$-structure, characterized by the property that the restriction to $\Vect$ is $t$-exact.  By the semisimplicity of $S$ and coconnectivity of $A$, the functor $\on{Oblv} \circ \on{Ind}$ is moreover left $t$-exact.

For 1., we first construct the object $M^{\leq n}$ and the map $M \rightarrow M^{\leq n}$.  Consider the restriction $\on{Oblv}(M)$, and choose a section $s: \tau^{> n} \Oblv(M) \rightarrow \Oblv(M)$ of the canonical truncation map 
$\on{Oblv}(M) \rightarrow \tau^{> n} \Oblv(M).$  By adjunction, we obtain a map 
$$\on{Ind}(\tau^{> n} \Oblv(M)) \rightarrow M,$$
which by construction induces an isomorphism on $H^{n+1}$ and surjections on $H^i$, for $i > n+1$.

Replace $M$ by the cone $M_1$, which has $H^i(M) \simeq H^i(M_1)$ for $i \leq n$.  Apply the same process to $M_1$ to produce $M_2$, and iterate countably many times. Then the natural map $M \rightarrow \varinjlim M_j$ satisfies the required property of $M \rightarrow M^{\leqslant n}$, and take $M^{> n} = \mathrm{fib}(M \rightarrow M^{\leq n})$.

Note that, by construction, for any module $N$ with nonvanishing cohomology only in degrees $\leqslant n$, the restriction map 
\begin{equation}\Hom(M^{\leqslant n}, N) \rightarrow \Hom(M, N) \label{e:res}\end{equation}
is a surjection on $H^0$ and an isomorphism on $H^i$ for  $i > 0$. 

For 2., by shifting we may assume $n = 0$. For existence, we may construct $(\on{Ind}(L))^{\leqslant 0}$ as above. For uniqueness up to isomorphism, given another object $K(L, 0)$ satisfying the listed properties, by adjunction we obtain a map $\on{Ind}(L) \rightarrow K(L, 0)$ inducing the identity on $H^0$. By Equation \ref{e:res}, we may extend this to a map from $(\on{Ind}(L))^{\leqslant 0}$, which is necessarily an equivalence. 
\end{proof}

\begin{exmp}
Let us apply the construction of $M^{\leq n} = \colim M_j$ in the above proof to the case of $$A = M = C^\bullet(S^1), \quad \quad \text{with} \quad n = 0.$$
To do so, write $C^\bullet(S^1) = k[\eta]/(\eta^2)$, where $\eta$ has cohomological degree one. Then, if we introduce an auxiliary variable $x$ of degree zero with $d(x) = \eta$, we have
$$M^{\leq 0} = \colim M_j := \colim k[\eta][x]/x^{i+1} \simeq k.$$
I.e., in this case, the above construction recovers the usual Koszul complex. 
\end{exmp}

\begin{exmp} (Minimal resolutions) Let $R$ be a local Noetherian $k$-algebra with maximal ideal $\mathfrak{m}$ and residue field $k$. Consider the full subcategory $\QC^!(R)^{\on{nilp}}$ of $\QC^!(R)$ compactly generated by $k$ itself, i.e., modules on which $\mathfrak{m}$ acts locally nilpotently.\footnote{Equivalently, this is ind-coherent sheaves on the formal completion of $\on{Spec }R$ at its closed point.} One has a tautological equivalence
$$\zeta: \QC^!(R)^{\on{nilp}} \simeq \End(k)^{op}\mod,$$
which exchanges the injective envelope $I$ of $k$ on the left-hand side with the augmentation module $k$ on the right-hand side. For a coconnective object $M$ of the left-hand side, a choice of truncation filtration 
$$\zeta(M)^{\geqslant 0} \leftarrow \zeta(M)^{\geqslant 1} \leftarrow \cdots$$
as in the proposition yields, after applying $\zeta^{-1}$, a minimal injective resolution of $M$ equipped with its stupid truncation filtration. \label{ex:minres}\end{exmp}

By combining both assertions of Proposition \ref{p:coconnectivealgs}, we obtain the following. 
\begin{cor} \label{c:whoisfd} For $A$ a coconnective dg algebra satisfying the assumption of Proposition \ref{p:coconnectivealgs}, consider the forgetful functor
$$\Oblv: A\mod \rightarrow \Vect.$$
Then the preimage of the compact objects in $\Vect$ is the pretriangulated envelope of the modules $K(L, 0)$, as $L$ runs over the simple $S$-modules. 
\end{cor}

\subsubsection{} \label{s:cofree}By our assumption that $X$ has finitely many components, we may apply the preceding to 
$$A = C^\bullet(X) \quad \text{and} \quad S = H^0(X).$$  
The simple $S$-modules $L_\gamma$ are in canonical bijection with the components  $\gamma \in \pi_0(X)$. We have the corresponding object $K(L_\gamma, 0)$ of $C^\bullet(X)\mod$, and we call the associated D-module $$\delta^\wedge_\gamma := \underline{k}_X \underset{C^\bullet(X)} \otimes K(L_\gamma, 0)$$
a {\em cofree-monodromic sheaf} supported on $X_\gamma$. As a special case of Corollary \ref{c:whoisfd}, we deduce the following.

\begin{cor}\label{c:cofree=fdcoh} For an object $M$ of $\on{IndLisse}(X)^{uni}$, its de Rham cohomology 
$$\Hom( \underline{k}_X, M) \in \Vect$$is perfect if and only if $M$ lies in the pretriangulated envelope of the cofree-monodromic sheaves  $\delta^\wedge_\gamma, \gamma \in \pi_0(X).$

In particular, for a component $X_\gamma$, its constant sheaf lies in the pretriangulated envelope of $\delta^\wedge_\gamma$ if and only if $C^\bullet(X_\gamma)$ is perfect, in which case one has a canonical isomorphism
$$\underline{k}_{X_\gamma} \simeq \Hom(\delta^\wedge_\gamma, \underline{k}_{X_\gamma}) \underset{\End(\delta^\wedge_\gamma)} \otimes \delta^\wedge_\gamma.$$
\end{cor}

\subsubsection{} Let us provide a more hands on description of the cofree-monodromic sheaves $\delta^\wedge_\gamma$.  To do so, pick a point $x \in X(k)$ in the component $\gamma$. Note that $*$-pullback to $x$ defines an augmentation module
$$C^\bullet(X) \rightarrow k_x,$$
and we have $k_x \simeq K(L_\gamma, 0)$.
\begin{defn}\label{cofree mon def}
We call the corresponding object
$$\delta^\wedge_x := \underline{k}_X \underset{C^\bullet(X)}{\otimes} k_x \in \on{IndLisse}(X)^{uni}$$
{\em the cofree-monodromic sheaf} corresponding to $x$. 
\end{defn}

\begin{lemma} The $*$-fiber of the cofree-monodromic sheaf $\delta^\wedge_x$ at any point $y \in X(k)$ is given by 
$$k_x \underset{C^\bullet(X)}{\otimes} k_y.$$
In particular, the $*$-fiber at $x$ is given by the self intersection $k_x \underset{C^\bullet(X)}{\otimes} k_x$.
\end{lemma}

\begin{proof} This is immediate from the fact that, as $\on{IndLisse}(X)^{uni}$ consists of ind-holonomic objects, $*$-restriction is defined on them, and the corresponding functor 
$$C^\bullet(X)\mod \rightarrow \Vect$$is given by tensoring with $k_y$. 
\end{proof}
 
\begin{rmk} Suppose $k = \mathbb{C}$. For suitable $X$, e.g., finite type and simply connected, the above self intersection canonically identifies with cochains on the based loop space of the analytification $X_\mathbb{C}$, i.e.,  $$ k_x \underset{C^\bullet(X)} \otimes k_x \simeq C^\bullet(\Omega X_\mathbb{C}).$$ 
\end{rmk}

\begin{exmp}\label{Cofree group} 
Suppose $H$ is an affine group scheme, possibly of infinite type, whose prounipotent radical is finite codimension. If we write $H^\circ$ for the neutral component of $H$, then we may write $C^\bullet(H^\circ) \simeq \Sym(\fp)$. Moreover, we have a canonical basepoint, namely the identity $e$ of $H$, and the resulting sheaf $\delta_e^\wedge$ is canonically isomorphic to the monoidal unit  $\delta^\wedge$ of $\Dmod(H\mon\!\bs H / H\mon).$ 

By the preceding Lemma, we have its $*$-fiber is given by 
$$i_e^*(\delta^\wedge) \simeq \Sym(\fp[1]).$$
For any $x \in H$, if we write $\delta_x$ for the delta D-module supported at $x$, one obtains the corresponding cofree-monodromic sheaf by translation
$$\delta^\wedge_x \simeq \delta_x \star \delta^\wedge.$$
\end{exmp}

\subsubsection{} Here are some other basic properties of the cofree-monodromic sheaves. First, for a map $X \rightarrow Y$, one has adjoint functors 
\begin{equation} \label{e:adjconst}\pi^*: \on{IndLisse}(Y)^{uni} \rightleftarrows \on{IndLisse}(X)^{uni}: \pi_*^{uni},\end{equation}
where $\pi^*$ is the usual $*$-pullback, and $\pi_*^{uni}$ is the composition 

$$\on{IndLisse}(X)^{uni} \hookrightarrow \Dmod(X) \xrightarrow{\pi_*} \Dmod(Y) \xrightarrow{\on{Av}_*^{uni}} \on{IndLisse}(Y)^{uni}.$$

\begin{lemma} For any $x \in X(k)$, there is a canonical pullback map and pushforward isomorphism 
$$\pi^* (\delta_{\pi(x)}^\wedge) \rightarrow \delta^\wedge_{x} \quad \text{and} \quad \pi_*^{uni}(\delta^\wedge_x) \simeq \delta_{\pi(x)}^\wedge.$$
\end{lemma}

\begin{proof} This follows from the fact that the adjunction \eqref{e:adjconst} identifies with restriction and induction along the map $C^\bullet(Y) \rightarrow C^\bullet(X)$. \end{proof}

We also note that cofree-monodromic sheaves behave simply on products. Namely, given a product $$X \xleftarrow{\pi_X} X \times Y \xrightarrow{\pi_Y} Y,$$
and a point $(x,y) \in X \times Y (k)$, one has a canonical isomorphism 
$$\delta^\wedge_{(x,y)} \simeq \pi_X^*(\delta^\wedge_x) \overset * \otimes \pi_Y^*(\delta^\wedge_y).$$

\subsubsection{}  Let us finally mention a few other basic properties of the cofree monodromic sheaf. These will not be needed in what follows, and the reader may proceed directly to the next subsection. 

\sss
For a category $\sC$ equipped with a $t$-structure $\tau$, let us denote its left completion with respect to $\tau$ by $\sC^{l.c.}$. Recall this comes equipped with a canonical map which sends an object of $\sC$ to the formal inverse limit of its bounded below truncations
 $$\sC \rightarrow \sC^{l.c.}, \quad \quad c \mapsto \lim_{n \in \mathbb{Z}}  \tau^{\geqslant n} c.$$
Let us denote the left completion of $\on{IndLisse}(X)^{uni}$ by
$$\on{QLisse}(X)^{uni} := \on{IndLisse}(X)^{uni, l.c.};$$
these are the unipotent quasi-lisse sheaves in the sense of  \cite{AGKRRV}.

\sss
Note that the embedding 
$$\on{IndLisse}(X)^{uni} \hookrightarrow \Dmod(X),$$
being by definition $t$-exact, induces an embedding of their left completions with respect to their $t$-structures
$$\begin{tikzcd}\on{IndLisse}(X)^{uni} \ar[rr, hook] \ar[d] &&\Dmod(X) \ar[d] \\ \on{QLisse}(X)^{uni} \ar[rr, hook] && \Dmod(X)^{l.c.}.\end{tikzcd}$$
As $\Dmod(X)$ is left complete, we deduce a chain of fully faithful embeddings
$$\on{IndLisse}(X)^{uni} \hookrightarrow \on{QLisse}(X)^{uni} \hookrightarrow \Dmod(X).$$
We also deduce the existence of an adjunction 
\begin{equation} \label{e:trunks}\on{IndLisse}(X)^{uni} \rightleftarrows \on{QLisse}(X)^{uni}: i^!,\end{equation}i.e., that the inclusion $\on{IndLisse}(X)^{uni} \hookrightarrow \on{QLisse}(X)^{uni}$ admits a right adjoint.  

\sss
Suppose our ground field $k = \mathbb{C}$ for simplicity, and recall that we denote by $X_{\mathbb{C}}$ the analytification of a complex scheme $X$ of finite type. 
In this case, let us write $$C_\bullet(\Omega X_{\mathbb{C}})\mod_{\wh{0}} \hookrightarrow C_\bullet(\Omega X_{\mathbb{C}})\mod$$for the full subcategory of modules generated by the augmentation.

\begin{prop}The following hold. 

\begin{enumerate}[(i)]
    \item If $X$ is connected, one has a $t$-exact equivalence \begin{equation} \label{e:hansel} C_\bullet(\Omega X_{\mathbb{C}})\mod_{\wh{0}} \simeq \on{IndLisse}(X)^{uni},\end{equation}which exchanges the augmentation module $\mathbb{C}$ on the left-hand side with the constant sheaf $\underline{\mathbb{C}}_X$ on the right-hand side.

    \item Moreover, if $X$ is simply connected,  this prolongs to an equivalence
    \begin{equation} \label{e:gretel}C_\bullet(\Omega X_{\mathbb{C}})\mod \simeq \on{QLisse}(X)^{uni}.\end{equation}
In particular, the latter category is compactly generated by a single object $\Xi$. 

\item Finally, if $X$ is smooth and proper, or more generally a Poincar\'e duality space of dimension $2d$, one has the identity
\begin{equation} \label{e:WITCH}i^!(\Xi) \simeq \delta^\wedge[-2d].\end{equation}That is, the cofree-monodromic sheaf is obtained from the compact generator via the adjunction \eqref{e:trunks}. 
\end{enumerate}

\end{prop}

\begin{proof}Let us begin with \eqref{e:hansel}. To obtain an equivalence of the desired form, modulo the claim of $t$-exactness, it suffices to know that the augmentation module $\mathbb{C}$ for $C_\bullet(\Omega X_\mathbb{C})$ is compact, and its endomorphisms are canonically identified with $C^\bullet(X_{\mathbb{C}})$. However, these are standard parts of Koszul duality in topology, cf. Proposition 5.3 of \cite{dgi}. 

To see that the obtained equivalence is $t$-exact, we first claim  that the compact objects on both sides are closed under truncation. Explicitly, for $C_\bullet(\Omega X)\mod_{\wh{0}},$ the compact objects are $t$-bounded modules whose individual cohomology groups are finite iterated extensions of the augmentation module $\mathbb{C}$. For $\on{IndLisse}(X)^{uni}$, the compact objects similarly are bounded complexes of D-modules, whose individual cohomology D-modules are finite iterated extensions of the constant sheaf $\underline{\mathbb{C}}_X$. In particular, the claimed stability under truncation is manifest. As the $t$-structures are also compatible with filtered colimits, it is therefore enough to check that the $t$-structures agree on compact objects. 

To see this, it is straightforward to see that a compact object $c$ of $C_\bullet(\Omega X)\mod_{\wh{0}}$ lies in degrees $\geqslant 0$ if and only if 
$$\Hom(\mathbb{C}, c) \in \Vect^{\geqslant 0}.$$
Indeed, this follows from the fact that its lowest nonzero cohomology receives a degree zero map from $\mathbb{C}$. Similarly, a compact object $\sM$ of $\on{IndLisse}(X)^{uni}$ lies in degrees $\geqslant 0$ if and only if 
$$\Hom(\underline{\mathbb{C}}_X, \sM) \in \Vect^{\geqslant 0},$$
hence the claimed $t$-exactness follows. 

To prove \eqref{e:gretel}, note that $C_\bullet(\Omega X_{\mathbb{C}})\mod$, as with modules for any connective algebra, is left complete with respect to its natural $t$-structure. Therefore, the inclusion $$C_\bullet(\Omega X_{\mathbb{C}})\mod_{\wh{0}} \hookrightarrow C_\bullet(\Omega X_{\mathbb{C}})\mod$$
prolongs to an inclusion
\begin{equation} \label{e:tobyarnaulttoby}C_\bullet(\Omega X_{\mathbb{C}})\mod^{l.c.}_{\wh{0}} \hookrightarrow C_\bullet(\Omega X_{\mathbb{C}})\mod.\end{equation}To show \eqref{e:gretel}, we need to check that if $X$ is simply connected the latter inclusion is moreover essentially surjective. However, the latter is clear, as each truncation $$\tau^{\geqslant N} C_\bullet(\Omega X_{\mathbb{C}})$$is then a finite iterated extension of the augmentation module $\mathbb{C}$, hence $C_\bullet(\Omega X_{\mathbb{C}})$ lies in the essential image of  \eqref{e:tobyarnaulttoby}.

To see the final claim \eqref{e:WITCH}, it is enough to see that one has an isomorphism
$$\Hom_{C_\bullet(\Omega X_{\mathbb{C}})\mod}(\mathbb{C}, C_\bullet(\Omega X_{\mathbb{C}})) \simeq \mathbb{C}[-2d].$$
However, this is a known result, cf. Section 10.1 of \cite{dgi}. \end{proof}

\subsection{Nearly compact objects}

\label{s:lasts4}

\sss
Our next goal is to explain the relationship between cofree-monodromic sheaves and near compactness (see Definitions \ref{cofree mon def} and \ref{d:defnc1}). To do so, we first recall the definition of monodromic invariants.

\subsubsection{} As before, $H$ denotes an affine group scheme whose prounipotent radical is of finite codimension.  Recall that $\Dmod(H)$, viewed as a monoidal category under convolution, has a natural action on $\Vect$, whose underlying binary product is given by global flat sections, i.e., pushforward to a point $\pi: H \rightarrow \on{pt}$:$$\Dmod(H) \otimes \Vect \rightarrow \Vect, \quad \quad M \otimes V \mapsto \pi_*(M) \otimes V. $$
In particular, as this map commutes with the natural action of $H \times H$ on both, and $H$ carries a constant sheaf, we deduce an adjunction of $\Dmod(H)$-bimodules \begin{equation} \label{e:adjbasic}\pi^*: \Vect \rightleftarrows \Dmod(H): \pi_*.\end{equation}

Consequently, the $H$-invariants of a $\Dmod(H)$-module $\sC$ are defined by 
$$\sC^H := \Hom_{\Dmod(H)}(\Vect, \sC).$$
This comes equipped, by applying $\Hom_{\Dmod(H)}(-, \sC)$ to \eqref{e:adjbasic}, with a tautological $\Dmod(H)$-equivariant adjunction of forgetting and averaging functors $$\on{Oblv}: \sC^H \rightleftarrows \sC: \on{Av}_{*}^H.$$

\subsubsection{} \label{s:defmon} By definition, the monodromic invariants $\sC^{H\mon}$ are the full subcategory of $\sC$ generated by the essential image of $\on{Oblv}$, which carry a natural action of $\Dmod(H)$. As the constant sheaf on $H$ is compact in $\Dmod(H)$, one has an adjunction 
$$\on{Oblv}: \sC^{H\mon} \rightleftarrows \sC: \on{Av}_*^{H\mon}.$$

\subsubsection{} We recall from Section \ref{sss:nearcompactdef} that given a $\Dmod(H)$-module $\sC$, within its category of $H$-monodromic objects one has the full, but not cocomplete, subcategory of {\em nearly compact} objects
$$\sC^{H\mon, n.c.} \hookrightarrow \sC^{H\mon}, $$
i.e., those objects whose averages in $\sC^H$ are compact. 

As left adjoint functors (with continuous right adjoints) preserve compactness, one has the following the basic functoriality. 

\begin{lemma} Suppose $F: \sC \rightarrow \sD$ is a $\Dmod(H)$-equivariant functor which admits a (continuous) right adjoint. Then $F$ sends nearly compact objects of compact objects, i.e., induces a map $$F: \sC^{H\mon, n.c.} \rightarrow \sD^{H\mon, n.c.}.$$
\end{lemma}

\subsubsection{} To see the relationship with cofree-monodromic sheaves, we begin with the representative case of $\sC \simeq \Dmod(H)$, i.e., the group algebra itself viewed as a module under left convolution, recalling the discussion in Example \ref{Cofree group}.

\begin{lemma}
Suppose that $H$ has a finite set of components. Then an object of $\Dmod(H\mon\!\bs H)$ is nearly compact if and only if it lies in the pretriangulated hull of the cofree-monodromic sheaves $\delta^\wedge_\gamma, \gamma \in \pi_0(H)$. 
\end{lemma}

\begin{proof}The composition
$$\Dmod(H\mon\!\bs H ) \xrightarrow{\on{Av}^H_*} \Dmod(H \bs H) \simeq \Vect$$is computed by taking de Rham cohomology, so this follows from Corollary \ref{c:cofree=fdcoh}.\end{proof}

\subsubsection{} We now perform a similar analysis for nearly compact D-modules on certain homogeneous spaces. Let $K \subset H$ be a subgroup of $H$, and consider the category of monodromic D-modules 
$$\Dmod(H\mon\!\bs H / K).$$
Two basic objects in this category are the constant D-module $\underline{k}_{H/K}$ on $H/K$, and the cofree-monodromic sheaves  
$$\delta^\wedge_\gamma, \quad \quad \gamma \in \pi_0(H/K). $$
\begin{prop} \label{p:oneorbit}
Suppose that $K$ is connected. Then an object of $\Dmod(H\mon\!\bs H /K)$ is nearly compact if and only if it lies in the pretriangulated hull of the cofree-monodromic sheaves.
\end{prop}

\begin{proof}
By definition an object of $\Dmod(H\mon\!\bs H / K)$ is nearly compact if and only if it is sent to a compact object of $\Dmod(H \bs H/K) \simeq \Dmod(\pt /K)$ under the averaging functor $\Av_*^H$.  Recall that we may identify $$\Vect^K \simeq \Dmod(\pt/K)$$ with $C_\bullet(K)\mod$, where we view $C_\bullet(K)$ as an algebra under convolution, in such a way that $\on{Oblv}: \Vect^K \rightarrow \Vect$ is intertwined with the forgetful functor 
$$\on{Oblv}: C_\bullet(K)\mod \rightarrow \Vect;$$
cf. Section 7.2.2 of \cite{QCA}.

Consider the category $\Ind\Dmod_c(\pt/K)$ of ind-coherent D-modules on $\pt/K$.  By definition, this is ind-completion of the pretriangulated hull of the constant sheaf $\underline{k}_{\pt/K}$. Under the identification with $C_\bullet(K)\mod$, this is given by the augmentation module $k$ concentrated in degree zero. In particular, as $C_\bullet(K)$ is finite dimensional, the pretriangulated hull of $\underline{k}_{\pt/K}$ contains the compact objects of $C_\bullet(K)$, and ind-completing this inclusion gives rise to a fully faithful embedding 
$$\Dmod(\pt/K) \hookrightarrow \Ind\Dmod_c(\pt/K).$$
Note that the composition and identification
$$\Vect \simeq \Dmod(\pt) \xrightarrow{\on{Av}_*^K} \Dmod(\pt/K) \hookrightarrow \Ind\Dmod_c(\pt/K) \simeq C^\bullet(\pt/K)\mod$$
sends $k$ to the augmentation module $k$ for $C^\bullet(\pt/K)$. As a consequence, an object of $\Dmod(\pt/K)$ is compact if and only if its image in $C^\bullet(\pt/K)\mod$ is a perfect complex of vector spaces. 

By the preceding discussion, an object $M$ of $\Dmod(H\mon\!\bs H/K)$ is nearly compact if and only if the composition
$$\Dmod(H\mon\!\bs H / K) \xrightarrow{\on{Av}_*^H} \Dmod(\pt/K) \hookrightarrow \Ind\Dmod_c(\pt/K) \simeq C^\bullet(\pt/K)\mod$$
sends $M$ to a perfect complex of vector spaces.  However, under the identification $\Dmod(H\mon\!\bs H / K) \simeq C^\bullet(H/K)\mod$, this is simply the restriction functor 
$$C^\bullet(H/K)\mod \rightarrow C^\bullet(\pt/K)\mod,$$
so the assertion of the proposition follows from Corollary \ref{c:cofree=fdcoh}.
\end{proof}

\begin{rmk} If $K$ is disconnected, one may modify the preceding proposition as follows. One can replace the constant sheaf on $\pt/K$ with the local system $\mathscr{L}$ corresponding to the regular representation of $\pi_0(K)$.  Its pullback along $\rho: H/K \rightarrow \pt/K$ is again semisimple, and in particular its degree zero endomorphisms \begin{equation} \label{e:endl}H^0(\Hom_{H/K}(\rho^* \mathscr{L}, \rho^* \mathscr{L})) \end{equation} are again semisimple.  After adding to the previously discussed cofree (unipotently) monodromic sheaves the additional (non-unipotently) cofree-monodromic sheaves corresponding to the other simple modules for 
\eqref{e:endl}, the proposition again holds. 
\end{rmk}

\subsubsection{} Let us now describe the case of multiple strata. For motivation, consider an ind-scheme $X$ stratified by simply connected strata $$j_\lambda: C_\lambda \rightarrow X, \quad \quad \lambda \in \Lambda.$$If we consider the category $\Dmod(X)_\Lambda$ of D-modules on $X$ lisse along the strata, this is compactly generated by the $!$-extensions of the constant sheaves
$$j_{\lambda, !} := (j_\lambda)_! ( \underline{k}_{C_\lambda}), \quad \quad \lambda \in \Lambda.$$
Explicitly, an object of $\Dmod(X)_\Lambda$ is compact if and only if its $*$-restriction to each stratum is a finite extension of copies of the constant sheaf, and all but finitely many $*$-restrictions vanish. 

On the other hand, one has within $\Dmod(X)_\Lambda$ a larger small category of sheaves, namely the pretriangulated hull of the $!$-extensions of cofree-monodromic sheaves

$$j_{\lambda, !}^\wedge := (j_{\lambda})_!(\delta_{C_\lambda}^\wedge), \quad \quad \lambda \in \Lambda.$$
Explicitly, an object of $\Dmod(X)_\Lambda$ lies in this category if and only if its $*$-restriction to each stratum is a finite extension of copies of the cofree-monodromic sheaf, and all but finitely many $*$-restrictions vanish.

\subsubsection{} If the strata are not simply connected, one would need to also consider other local systems in the previous setup. However, if the strata are the orbits of the action of a group $H$ on $X$, and we suppose for simplicity that the stabilizers of points of $X$ are connected, we may circumvent this as follows,  beginning with some formal setup.  

Let $\Lambda$ be a filtered partially ordered set, with the property that for any $\lambda \in \Lambda$ the set of $\nu \in \Lambda$ such that $\nu \leqslant \lambda$ is finite. Consider a category $\sC$ equipped with a stratification by full  subcategories

$$\sC^{\leqslant \lambda} \hookrightarrow \sC, \quad \quad \varinjlim_\lambda \sC^{\leqslant \lambda} \simeq \sC.$$
For any $\lambda$, write $\sC^{< \lambda}$ for the full subcategory generated under colimits by $\sC^{\leqslant \nu}$, for all $\nu < \lambda$. We suppose in addition that the inclusion $$i_*: \sC^{< \lambda} \rightarrow \sC^{\leqslant \lambda}$$prolongs to a recollement 
\begin{equation} \label{e:barres}
	\begin{tikzcd}
	   \sC^{< \lambda} \arrow[r, "i_*"] & \arrow[l, "i^!"', shift left=5] \arrow[l, "i^*"',  shift right = 5] \ar[r, "j^!"]\sC^{\leqslant \lambda} & \arrow[l, "j_*"', shift left=5] \arrow[l, "j_!"',  shift right = 5]\sC^{\lambda}. 
	 	\end{tikzcd}
	\end{equation}
In particular, for each subquotient $\sC^\lambda, \lambda \in \Lambda$, we have associated $!$- and $*$- extension and restriction functors 
$$j_!: \sC^\lambda \rightleftarrows \sC: j^! \quad \quad j^*: \sC \rightleftarrows \sC^\lambda: j_*.$$

We then have the following basic observation. 

\begin{lemma} \label{l:whosnc}For $\sC$ as above, the compact objects of $\sC$ are the pretriangulated envelope of the $!$-extensions of compact objects from each $\sC^\lambda$.  In particular, an object $c$ of $\sC$ is compact if and only if its $*$-restriction to each stratum $\sC^\lambda$ is compact and all but finitely many $*$-restrictions vanish. 
\end{lemma}

\begin{proof} For each $\lambda \in \Lambda$, we have an adjunction 
$$i^{\leqslant \lambda}_*: \sC^{\leqslant \lambda} \rightleftarrows \sC: i^{\leqslant \lambda, !},$$
and correspondingly an identity $$\on{id}_{\sC} \simeq \varinjlim_\lambda i^{\leqslant \lambda}_* \circ i^{\leqslant \lambda, !}.$$
It follows that any compact object $c$ of $\sC$ is $*$-extended from some $\sC^{\leqslant \lambda}$.  Within there, consider as in \eqref{e:barres} the recollement triangle 
$$j_! j^!(c) \rightarrow c \rightarrow i^*i_*(c) \xrightarrow{+1}.$$
As $j_!$ and $j^!$ are both left adjoints, we have that $j^!(c)$ and $j_! j^!(c)$ are compact in $\sC^{<\lambda}$ and $\sC^{\leqslant \lambda}$, respectively. In particular, the object $i^* i_*(c)$ is again compact. As there are only finitely many $\nu < \lambda$, iterating this procedure for $i^*i_*(c)$ and inducting on the number of strata for which the $*$-restriction is nonzero yields the claim. 
\end{proof}

\subsubsection{}\label{renormalize autom} Suppose in the previous subsection $\sC$ is a $\Dmod(H)$-module, and each $\sC^{\leqslant \lambda}$ is essentially preserved by the action of $\Dmod(H)$. In this situation, all the recollement functors acquire data of $\Dmod(H)$-equivariance, as an adjoint of an equivariant functor inherits a datum of equivariance.

By applying Lemma \ref{l:whosnc} to $\sC^H$ and its induced filtration, we obtain the following.

\begin{cor} \label{c:whoisnc}For $\sC$ as above, the nearly compact objects $c$ of $\sC^{H\mon}$ are the pretriangulated envelope of the $!$-extensions of nearly compact objects from each $\sC^\lambda$. In particular, $c$ is nearly compact if and only if its $*$-restriction to each stratum $\sC^\lambda$ is nearly compact and all but finitely many $*$-restrictions vanish. 
\end{cor}

The particular case of the preceding corollary which is relevant for us is the following.

\begin{cor} \label{c:whoisncshv} Suppose each $\sC^\lambda$ is equivalent to $\Dmod(H/K_\lambda)$ for a connected subgroup $K_\lambda$ of $H$. Then the nearly compact objects of $\sC^{H\mon}$ are the pretriangulated hull of the $!$-extensions of the cofree-monodromic sheaves from each $\sC^\lambda$, for $\lambda \in \Lambda$. 
\end{cor}

We also record the following version for equivariant objects, which is essentially by definition.
\begin{cor}
In the setting of the above corollary, the coherent (or pseudo-compact) objects of $\sC^H$ are the pretriangulated hull of the $!$-extensions of the constant sheaves from each $\sC^\lambda$, for $\lambda \in \Lambda$.
\end{cor}

\subsubsection{}
Let us finally spell out the relationship between nearly compact objects and the  free-monodromic tilting sheaves of Bezrukavnikov--Yun \cite{BezYun}. In {\em loc. cit.}, the authors consider a certain monoidal category of tilting pro-sheaves on the basic affine space $\chN \bs \chG / \chN$ endowed with a  $!$-convolution product, building (substantially) on previous ideas of Beilinson--Ginzburg \cite{BeiGinz}.

Consider the homotopy category of  free-monodromic tilting sheaves. Verdier duality exchanges this with a full subcategory of $$\Dmod(\chB\mon\!\bs \chG / \chB \mon)$$preserved under the $*$-convolution product. One of the basic results of Section \ref{Cofree tilting sec} below is that the obtained subcategory is precisely the nearly compact objects. That is,  nearly compact objects, in favorable situations, may also be expressed as the homotopy category of cofree-monodromic tilting sheaves.

\section{Automorphic affine Hecke categories}\label{sec autom hecke}

\subsection{Overview}

\sss In this section we establish the automorphic counterparts to the results in Section \ref{spectral hecke sec}. The precise contents are as follows. 

\sss In Section \ref{group def}, we study the operations of taking strict and monodromic equivariants for a categorical representation of a group, and explain how to pass back and forth between them via turning on and off certain monodromy operators. These are analogous to the results of Section \ref{sec spectral monodrom} on the spectral side.

\sss In Section \ref{whit monad}, we study the monads arising from Whittaker averaging, and the corresponding algebra objects in the Hecke category. We then build on these results to study and classify in Section \ref{Cofree tilting sec} the cofree monodromic tilting sheaves, and use these to describe the categories of nearly compact objects in monodromic Hecke categories. The results of these sections may be compared to those of Sections \ref{sec algebra obj}  and \ref{spectral hecke algebra} on the spectral side. 

\sss Finally, in Section \ref{autom convolve sec} we record some basic relations between relative tensor products and spaces of homomorphisms for automorphic affine Hecke category modules and the corresponding categorical loop group representations. These in particular establish the automorphic counterpart to the main result of Section \ref{convolution sec}.

\subsection{Invariants and monodromy operators}\label{group def}

\sss
 
Let $H$ be an affine group scheme possibly of infinite type whose prounipotent radical is of finite codimension. Our goal in this section is to record some basic properties of the functor of monodromic invariants.   The main result is Theorem \ref{t:avenh}, which states that one can pass back and forth between strict and monodromic invariants by turning on and off certain monodromy operators.  

We highlight in particular that one has such a description not only when the reductive quotient of $H$ is a torus, where it is essentially well known in geometric representation theory. For general $H$, the monodromy operators are again a (completion of a) polynomial ring, namely the convolution algebra of chains on the based loop space $\Omega H$, but may now lie in nonpositive even degrees, and not only  degree zero as in the case of tori.

\subsubsection{} Let $\sC$ be a $\Dmod(H)$-module. Recall the notions of the strict and monodromic invariants

$$\sC^H = \Hom_{\Dmod(H)}(\Vect, \sC) \quad \quad \text{and} \quad \quad \sC^{H\mon} \simeq \Hom_{\Dmod(H)}(\Dmod(H/H\mon), \sC), $$
where $\Dmod(H/H\mon)$ is explicitly the full subcategory of $\Dmod(H)$ generated by the constant D-module.

Thanks to the $\DGCat$-linearity of $H$-invariants, one has natural equivalences 
\begin{align}\label{e:formula1}
&\sC^H \simeq \Dmod(H\backslash H/H\mon) \underset{\Dmod(H\mon\! \backslash H / H\mon)} \otimes \sC^{H\mon}, \\ \label{e:formula2}
&\sC^{H\mon} \simeq \Dmod(H\mon\! \backslash H/H) \underset{\Dmod(H \backslash H / H)} \otimes \sC^H.\end{align}
Indeed, these follow from the $\Dmod(H)$-equivariant equivalences $$\sC^H \simeq \Dmod(H \backslash H) \underset {\Dmod(H)} \otimes \sC \quad \text{and} \quad \sC^{H\mon} \simeq \Dmod(H\mon\! \backslash H) \underset{\Dmod(H)} \otimes \sC, $$i.e., the identifications with strict and monodromic coinvariants arising from continuity and $\DGCat$-linearity.

\subsubsection{} The identities of Equations \eqref{e:formula1} and \eqref{e:formula2}, and specifically the terms in the tensor products, may be rewritten more explicitly as follows. The contents of this subsubsection will not be needed in the sequel, and may be skipped by the reader. 

We begin with Equation \eqref{e:formula1}. Pushing forward to a point yields a continuous monoidal functor
$$C^\bullet(H, -): \Dmod(H\mon\! \backslash H / H\mon) \rightarrow \Vect.$$
As the source category is compactly generated by the constant D-module, this factors through a monoidal equivalence
$$\Dmod(H\mon\! \backslash H / H\mon) \simeq C^\bullet(H)\mod \xrightarrow{\on{Oblv}} \Vect,$$where $C^\bullet(H)$ denotes the de Rham cohomology of $H$, i.e., the pushforward of the constant D-module. Via this identification, we may termwise rewrite Equation \eqref{e:formula1} as 
$$\sC^H \simeq \Vect \underset{C^\bullet(H)\mod}{\otimes} \sC^{H\mon}.$$

For the 
second identity \eqref{e:formula2}, the constant D-module on $H$ compactly generates $\Dmod(H/H)$, yielding an identification 
$$\Dmod(H/H) \simeq \Vect.$$
By concatenating this with the comonadic adjunction
$$\on{Oblv}:\Dmod(H \backslash H / H) \rightleftarrows \Dmod(H/H): \on{Av}_*,$$
one obtains a monoidal identification of $\Dmod(H \backslash H / H)$ with the category of comodules for $C^\bullet(H)$, i.e. 
$$\Dmod(H \backslash H / H) \simeq C^\bullet(H)\on{-comod}.$$
With this, we may termwise rewrite Equation \eqref{e:formula2} as
$$\sC^{H\mon} \simeq \Vect \underset{C^\bullet(H)\on{-comod}}{\otimes} \sC^H.$$

\subsubsection{} \label{sss:intromonopers}Our main goal in this subsection is give an alternative to the formulas in Equation \eqref{e:formula1} and \eqref{e:formula2} by using monodromy operators rather than Hecke functors. To state this, we need to collect some preliminary observations. 

Recall that, given any module category $\mathbf{M}$ for a monoidal category $\mathbf{A}$, one has a tautological map from the endomorphisms of the monoidal unit $u$ of $\mathbf{A}$ to the center of $\mathbf{M}$, i.e.,
$$\End_{\mathbf{A}}(u) \rightarrow HH^\bullet(\mathbf{M}) := \End_{\cat{End}(\mathbf{M})}(\on{id}_{\mathbf{M}}).$$

In particular, if we write $\delta^\wedge$ for the monoidal unit of  $\Dmod(H\mon\! \bs H / H\mon)$ and $\delta$ for the monoidal unit of $\Dmod(H \backslash H / H)$, we have tautological maps 
\begin{align*}\End_{H}(\delta^\wedge) \rightarrow HH^\bullet(\sC^{H\mon}) \quad \text{and} \quad  \End_{\pt/H}(\delta) \rightarrow HH^\bullet(\sC^{H}).\end{align*}
By considering the natural actions of both the monodromic and strict Hecke categories by left and right convolution on $\Dmod(H\mon\! \backslash H / H) \simeq \Vect$ we obtain in particular augmentation maps 
$$\End_H(\delta^\wedge) \rightarrow HH^\bullet(\Vect)  \simeq k \leftarrow \End_{\pt/H}(\delta).$$

\subsubsection{}\label{averaging defin section} We may use these augmentations as follows. Consider the averaging functor 
$$\on{Av}_*^H: \sC^{H\mon} \rightarrow \sC.$$
By construction, this identifies with the map 
$$\sC^{H\mon} \simeq \Vect \otimes \hspace{.7mm} \sC^{H\mon} \simeq \Dmod(H \backslash H / H\mon) \otimes \sC^{H\mon} \xrightarrow{\on{act}} \sC^H,$$i.e., is given by convolution with the constant D-module. It follows that averaging kills our monodromy operators, i.e., induces an {\em enhanced} averaging functor

\begin{equation} \label{e:avenh}\on{Av}_*^{H, \on{enh}}: \Vect \underset{\End(\delta^\wedge)\mod}\otimes \sC^{H\mon} \rightarrow \sC^H.\end{equation}

Similarly, the forgetful functor $\on{Oblv}: \sC^H \rightarrow \sC^{H\mon}$ enhances to a functor 
\begin{equation} \label{e:oblvenh}\on{Oblv}^{\on{enh}}: \Vect \underset{\End(\delta)\mod} \otimes \sC^H \rightarrow \sC^{H\mon}.\end{equation}

Our desired formulas in terms of monodromy operators now reads as follows.

\begin{thm}\label{t:avenh}\label{autom central}
Suppose that $H$ is connected. For any $\Dmod(H)$-module $\sC$, the enhanced averaging and forgetful functors \eqref{e:avenh} and \eqref{e:oblvenh} are equivalences. 
\end{thm}

\begin{proof} By the $\DGCat$-linearity and continuity of the functors in the argument $\sC$, it is enough to verify both assertions in the case of $\sC \simeq \Dmod(H)$.

Let us begin with the enhanced forgetful functor \eqref{e:oblvenh}. In the present case of $\sC \simeq \Dmod(H)$, this takes the form 
$$\Vect \underset{\End(\delta)\mod} \otimes \Vect \simeq (k \underset{\End(\delta)} \otimes k)\mod \simeq C^\bullet(H)\mod.$$
I.e., this reduces to the standard relation between the (de Rham) cohomology of a group and its classifying space. 

Let us consider now the enhanced averaging functor \eqref{e:avenh}. In the present case of $\sC \simeq \Dmod(H)$, this takes the form 
\begin{equation} \label{e:alwaysapapaya}\Vect \underset{\End(\delta^\wedge)\mod} \otimes C^\bullet(H)\mod \xrightarrow{\on{id} \otimes \on{Oblv}} \Vect.\end{equation}
That is, the composite map 
$$C^\bullet(H)\mod \simeq \Vect \otimes \hspace{.7mm} C^\bullet(H)\mod \rightarrow \Vect \underset{\End(\delta^\wedge)\mod} \otimes C^\bullet(H)\mod \rightarrow \Vect$$sends a $C^\bullet(H)$-module $M$ to its underlying vector space $\on{Oblv}(M)$.

Let us write the relative tensor product more explicitly, beginning with $\End_H(\delta^\wedge)$.  Write $\delta_e$ for the monoidal unit of $\Dmod(H)$, i.e., the delta D-module at the identity element $e$. Note that $\delta^\wedge$ may be calculated as the image of $\delta_e$ under the monoidal functor 
$$\on{Av}_{*}^{H \times H \mon}: \Dmod(H) \rightarrow \Dmod(H\mon\!\bs H / H\mon),$$where the monoidality follows from the fact that $\sC^{H\mon} \rightarrow \sC$ is a fully faithful embedding. If we write $\underline{k}$ for the constant D-module on $H$, we may compute that the object of $C^\bullet(H)\mod$ corresponding to $\delta^\wedge$ is the augmentation module, namely 
$$\Hom_{H}(\underline{k}, \delta^\wedge) \simeq \Hom_{H}(\underline{k}, \on{Av}_*^{H \times H \mon} \delta_e) \simeq \Hom_{H}(\underline{k}, \delta_e) \simeq k.$$
By formality, we may fix an isomorphism of dg-algebras $C^\bullet(H) \simeq \Sym(\mathfrak{p}),$ where $\mathfrak{p}$ is a finite dimensional cohomologically graded vector space with trivial differential, concentrated in positive odd degrees. Using this, we have 
$$\End_H(\delta^\wedge) \simeq \Hom_{\Sym(\mathfrak{p})\mod}(k, k) \simeq \wh\Sym(\mathfrak{p}^*[-1]),$$
where $\wh\Sym$ denotes completion at the augmentation for the degree zero part of $\mathfrak{p}^*[-1]$.

We may therefore rewrite \eqref{e:alwaysapapaya} as 
\begin{equation}\label{e:alwayspapaya2}
\Vect \underset{ \wh\Sym(\mathfrak{p}^*[-1])\mod} \otimes \Sym(\mathfrak{p})\mod \xrightarrow{ \on{id} \otimes \on{Oblv}} \Vect.
\end{equation}
As $\fp$ is concentrated in odd degrees, $\fp^*[-1]$ is concentrated in even degrees. Therefore, the augmentation module $k$ of $\wh\Sym(\fp^*[-1])$ is a compact object of $\wh\Sym(\fp^*[-1])\mod$, as its Koszul resolution is finite. Let us denote the full subcategory generated by it as 
$$i_*: \begin{tikzcd} \wh\Sym(\fp^*[-1])\mod^\wedge_0 \arrow[r, hook, shift left] & \arrow[l, shift left] \wh\Sym(\fp^*[-1])\mod: i^!.\end{tikzcd}$$
By taking the endomorphisms of the augmentation $k$,  we obtain the $\wh\Sym(\fp^*[-1])\mod$-linear equivalence 
$$\wh\Sym(\fp^*[-1])\mod^\wedge_0 \simeq \Sym(\fp)\mod.$$

Via this Koszul dual picture, we may further rewrite \eqref{e:alwayspapaya2} as 
\begin{equation}
\Vect \underset{\wh\Sym(\fp^*[-1])\mod}\otimes \wh\Sym(\fp^*[-1])\mod^\wedge_0 \xrightarrow{ \on{id} \otimes \Hom(k, -)} \Vect.
\end{equation}
Recalling that the action of $\wh\Sym(\fp^*[-1])\mod$ on $\Vect$ in the tensor product was by tensoring with the augmentation module $k$, we may rewrite the appearing tensor product as 
$$\Vect \simeq \Vect \underset{\wh\Sym(\fp^*[-1])\mod} \otimes \wh\Sym(\fp^*[-1])\mod \overset{\on{id} \otimes i^!}{\simeq} \Vect  \underset{\wh\Sym(\fp^*[-1])\mod} \otimes \wh\Sym(\fp^*[-1])\mod^\wedge_0.$$
Using this, the composite map 
$$\Vect \simeq \Vect \underset{\wh\Sym(\fp^*[-1])\mod}\otimes \wh\Sym(\fp^*[-1])\mod^\wedge_0 \xrightarrow{ \on{id} \otimes \Hom(k, -)} \Vect$$is given by tensoring with the cohomologically graded line $\det(\fp)$, and in particular is an equivalence, as desired. \end{proof}

\subsection{Cofree-monodromic sheaves and the Whittaker monads}\label{whit monad}

\sss
We recall that our basic strategy in proving the isomorphisms of affine Hecke category modules conjectured by Bezrukavnikov is to give descriptions on the automorphic and spectral sides of the modules via monads acting on the affine Hecke category itself, and then match the monads. 

In this subsection, we describe the monads for possibly degenerate Iwahori--Whittaker categories on the automorphic side.

In addition, to match a renormalization of the corresponding category of coherent sheaves on the spectral side, we study the cofree-monodromic tilting sheaves. When the Whittaker character is trivial on both sides, we will show their homotopy category recovers the nearly compact objects of $$\Dmod(\chI\mon\!\bs \chG_F / \chI\mon),$$
as was promised in the previous section.

\subsubsection{} \label{s:defst}To describe the categories of Whittaker sheaves we will be concerned with, we next recall some relevant definitions from \cite{campbelldhillon}.  

Let $\chL$ denote a connected reductive group, and $B_{\chL} \subset \chL$ a Borel subgroup. Fix a $\Dmod(\chL)$-module $\sC$. Let us denote by $\sC^{B_{\chL}\on{-gen}} \subset \sC$ the $\Dmod(\chL)$-submodule generated by the essential image of the $B_{\chL}$-equivariant objects under the forgetful functor
$$\on{Oblv}: \sC^{B_{\chL}} \rightarrow \sC;$$note the resulting subcategory is independent of the choice of Borel. To avoid any confusion, recall that the $B_{\chL}$-monodromic invariants of $\sC$, denoted by $$\sC^{B_{\chL}\mon} \subset \sC,$$are the full subcategory (but not $\Dmod(\chL)$-module) generated by the essential image of $\on{Oblv}$.

For any maximal unipotent subgroup $U$ of $\chL$, equipped with a character $\Psi$, the corresponding (possibly degenerate) {\em Steinberg--Whittaker invariants} of $\sC$
$$\sC^{U, \Psi, \St} \subset \sC^{U, \Psi}$$ are by definition the corresponding invariants for $\sC^{B_{\chL}\on{-gen}}$, i.e., $$\sC^{U, \Psi, \St} := (\sC^{B_{\chL}\on{-gen}})^{U, \Psi} \subset \sC^{U, \Psi}.$$

We will be interested in categories of D-modules with Steinberg--Whittaker equivariance conditions on both sides. To aid in their study, we now collect some general assertions about averaging functors.

\subsubsection{}\label{sss:avfunctors} Let $\chL$ be a reductive group with opposite Borel subgroups 
$B_{\chL}$ and 
$B^-_{\chL}$, with unipotent radicals $N_{\chL}$ and $N^-_{\chL}$, and fix a generic, i.e., nondegenerate, character $\psi$ of $N^-_{\chL}$. We will presently describe 
the basic adjunctions relating the categories 
\[
   \sC^{B_{\chL}\hspace{-.6mm} \on{-mon}} \quad \text{and} \quad \sC^{N^-_{\chL}, \psi, \St}. 
\]
%
\label{s:defavfuncs}
Consider the composition of forgetting and averaging functors
\[
\sC^{N_{\chL}} \xrightarrow{\on{Oblv}} \sC \xrightarrow{ \on{Av}_{*}^{N_{\chL}^-, 
		\psi}[\dim N_{\chL}]} \sC^{N_{\chL}^-, \psi};
\]
note that this composition is not automatically the zero functor by our assumption that $B_{\chL}$ and $B^-_{\chL}$ are in generic position. Applying the same sequence instead to $\sC^{B_{\chL}\hspace{-.6mm} \on{-gen}}$, we obtain a functor which 
we denote by
\begin{equation} \label{e:avfunc1}
\on{Av}_{!*}^{\psi}: \sC^{B_{\chL}\hspace{-.6mm} \on{-mon}} \rightarrow \sC^{B_{\chL}\hspace{-.6mm} \on{-gen}} \ra \sC^{N_{\chL}^-, \psi, \St}. 
\end{equation}
Dually, the composition of forgetting and averaging functors 
\[
    \sC^{N_{\chL}^-, \psi} \xrightarrow{\on{Oblv}} \sC \xrightarrow{ 
    \on{Av}_*^{N_{\chL}}[\dim N_{\chL}]} \sC^{N_{\chL}}
\]
upon restriction to $\sC^{B_{\chL}\hspace{-.6mm} \on{-gen}}$ yields a functor 
\begin{equation} \label{e:avfunc2}
      \on{Av}_{!*}^{}: \sC^{N_{\chL}^-, \psi, \St} \rightarrow \sC^{B_{\chL}\hspace{-.6mm} \on{-gen}} \ra \sC^{B_{\chL}\hspace{-.6mm} \on{-mon}}. 
\end{equation}

We may now collect the basic properties of these functors. 
\begin{lemma} \label{l:avavpsi}The functor $\on{Av}_{!*}^\psi$ is canonically left and right adjoint to $\on{Av}_{!*}$. In addition, the functor $\on{Av}_{!*}$ 
is conservative.  
\end{lemma}

\begin{proof} Let us first consider the case of $$\sC \simeq \Dmod(B_{\chL} \backslash \chL).$$In this case, note that the appearing D-modules in \eqref{e:avfunc1} and \eqref{e:avfunc2} are ind-holonomic, and in particular the $!$-averaging functors left adjoint to $\on{Oblv}$ are defined. The claims therefore follows from the standard cleanness of the corresponding integral kernels in 
\[
     \Dmod(B_{\chL}\hspace{-.6mm} \on{-mon} \backslash \chL / N^-_{\chL}, \psi) \quad \text{and} \quad \Dmod(N^-_{\chL}, \psi \backslash \chL / B_{\chL}\hspace{-.6mm} \on{-mon}). 
\]

 For 
general $\sC$, the categories and functors of \eqref{e:avfunc1} and \eqref{e:avfunc2} are obtained by tensoring $\sC$ over $D(\chL)$ with
	\[
	     \Dmod(\chL/B_{\chL}) \underset{\Dmod(B_{\chL} \backslash \chL / B_{\chL})} \otimes \Dmod(B_{\chL} \backslash 
	     \chL / B_{\chL}\hspace{-.6mm} \on{-mon}) \quad  \text{and} \]
      
      \[ \Dmod(\chL/B_{\chL}) 
	     \underset{\Dmod(B_{\chL} 
	     \backslash \chL / 
	     B_{\chL})} \otimes \Dmod(B_{\chL} \backslash \chL / N_{\chL}^-, \psi). 
	\]

We claim the statements of the lemma follow formally. Indeed, the datum of an adjunction is inherited after tensoring up. Similarly, for the 
asserted conservativity, note that in an adjunction 
\[  
    F: \sC \rightleftarrows \sD: G,
\]
the conservativity of $G$ is 
equivalent to the generation of $\sD$ by the essential image of $F$ under 
colimits. It is straightforward to see that the latter property is preserved by 
tensoring up, whence the conservativity of $\on{Av}_{!*}$ follows.  
\end{proof}

Given the conservativity of $\on{Av}_{!*}$, we next turn to the corresponding monad on $\sC^{B_{\chL}\hspace{-.6mm} \on{-mon}}$.

\subsubsection{} To explicitly describe the monad, we will make use of the following general observation on averaging functors. For a $D(\chL)$-module $\sC$, recall the definitional adjunction 
\[
   \on{Oblv}: \sC^{B_{\chL}} \rightleftarrows \sC^{B_{\chL}\hspace{-.6mm} \on{-mon}}: \on{Av}_*^{B_{\chL}}. 
\]

As we now show, the functors are also adjoint in the other order, up to tensoring by a cohomologically graded line. Denote by $\ft$ and $\chT$ the abstract Cartan algebra and group of $\chL$, respectively.

\begin{prop}\label{p:avbbmon} For any $D(B_{\chL})$-module $\sC$, there is a canonical adjunction 
\[
    \on{Av}_{*}^{B_{\chL}} \otimes \det(\ft^*[-1]) : \sC^{B_{\chL}\hspace{-.6mm} \on{-mon}} \rightleftarrows: \sC^{B_{\chL}}: \on{Oblv}.
\]
\end{prop}

\begin{proof} It suffices to address the universal case of $\sC \simeq \Dmod(B_{\chL})$, i.e. to furnish a $\Dmod(B_{\chL})$-equivariant adjunction 
\begin{equation} \label{e:monadj}
    \on{Av}_*^{B_{\chL}}: \Dmod(B_{\chL} / B_{\chL}\hspace{-.6mm} \on{-mon}) \rightleftarrows  \Dmod(B_{\chL} / B_{\chL}): \on{Oblv} \otimes \det(\ft[1]).  
\end{equation}

Let us first construct a right adjoint to $\on{Av}_*^{B_{\chL}}$. To do so, as the dualizing sheaf compactly generates both categories, this yields equivalences 
\[
    \Dmod(B_{\chL}/B_{\chL}\hspace{-.6mm} \on{-mon}) \simeq \on{Sym}(\ft^*[-1])\mod \quad \text{and} \quad \Dmod(B_{\chL}/B_{\chL}) \simeq \on{Vect}.
\]
Moreover, under these equivalences, $\on{Av}_*^{B_{\chL}}$ identifies with the forgetful functor 
\begin{equation} \label{e:oblvbmon}
      \Omega: \on{Sym}(\ft^*[-1])\mod \rightarrow \on{Vect}. 
\end{equation}
As this functor visibly preserves compactness, it admits a right adjoint $\Omega^R$. Moreover, noting that the action of $\Dmod(B_{\chL})$ on both categories in \eqref{e:monadj} canonically factors through 
\begin{equation} \label{e:bimonBorel} \Dmod(B_{\chL}\hspace{-.6mm} \on{-mon} \backslash B_{\chL} / B_{\chL}\hspace{-.6mm} \on{-mon}),
\end{equation}
the equivariance of $\Omega^R$ follows from the semi-rigidity of \eqref{e:bimonBorel}. 

   It remains to identify $\Omega^R$. To do so, first observe the equivalence  
   \[
       {\mathbf{Fun}}_{\Dmod(B_{\chL})\mod}(   \Dmod(B_{\chL}/B_{\chL}), \Dmod(B_{\chL}/B_{\chL}\hspace{-.6mm} \on{-mon})) \simeq \Dmod(B_{\chL} \backslash B_{\chL} / B_{\chL}\hspace{-.6mm} \on{-mon}) \simeq \on{Vect}.
   \]
   Using this, the assertion follows from noting that the left and right adjoints to \eqref{e:oblvbmon} are given by tensoring by $\on{Sym}(\ft^*[-1])$ and its dual, respectively. In particular, we have  
   \[
    \Omega^R \simeq \Omega^L \otimes \det(\ft[1]), 
   \]
   as desired.     \end{proof}

\subsubsection{} \label{s:defmoncreatures}To describe the object underlying the monad, we now introduce some notation concerning objects of 
\begin{equation} \label{e:dogsonaflag}
     \Dmod(B_{\chL}\hspace{-.6mm} \on{-mon} \backslash \chL / B_{\chL}\hspace{-.6mm} \on{-mon}). 
\end{equation}

Write $W_{\chL}$ for the Weyl group of $\chL$. For each $w \in W_{\chL}$, note the tautological equivalence 
\[
      \Dmod(B_{\chL}\hspace{-.6mm} \on{-mon} \backslash B_{\chL}wB_{\chL} / B_{\chL}\hspace{-.6mm} \on{-mon}) \simeq \Dmod(\chT/\chT\hspace{-.6mm} \on{-mon}). 
\]
Let us write $\delta_w$ for the constant perverse sheaf on the stratum, and let us denote its injective hull by  
\[
    \widehat{\delta}_w := \delta_w \underset{\on{End}(\delta_w)}{\otimes}  k.
\]
 Finally, let us denote the $!$-extension and $*$-extension of $\widehat{\delta}_w$ to \eqref{e:dogsonaflag} by $\hj_{w, !}$ and $\hj_{w, *}$, respectively, and similarly for $\delta_w, j_{w,!},$ and $j_{w, *}$. Given an object $\sM$ of \eqref{e:dogsonaflag}, let us say it admits a cofree-monodromic standard filtration if it is a successive extension of the $\hj_{w, !}$, for $w \in W$, and a cofree-monodromic costandard filtration if it is a successive extension of the $\hj_{w, *},$ for $w \in W$. In particular, such an $\sM$ lies in the heart of the $t$-structure.

We now collect their standard properties. Write $\ell: W_L \rightarrow \mathbb{Z}^{\geqslant 0}$ for the length function on $W_L$, and denote the convolution product on \eqref{e:dogsonaflag} by 
$$ \Dmod(B_{\chL}\hspace{-.6mm} \on{-mon} \backslash \chL / B_{\chL}\hspace{-.6mm} \on{-mon})^{\otimes 2} \rightarrow  \Dmod(B_{\chL}\hspace{-.6mm} \on{-mon} \backslash \chL / B_{\chL}\hspace{-.6mm} \on{-mon}), \quad \sM_1 \boxtimes \hspace{.7mm}\sM_2 \mapsto \sM_1 \star \sM_2.$$
Note that by the definition of the convolution product, the monoidal unit is given by the cofree-monodromic sheaf on the unit stratum
$$\on{Av}_{B \times B\mon, *}(\delta_e) = \hj_{e, !} \simeq \widehat{\delta}_e \simeq \hj_{e, *}. $$

\begin{lemma}\label{l:norsupises} The following are true. 

\begin{enumerate}
    \item Given $y, w \in W_L$ such that $\ell(yw) = \ell(y) + \ell(w)$, one has equivalences
    \begin{equation} \hj_{y, !} \star \hj_{w, !} \simeq \hj_{yw, !} \quad \text{and} \quad \hj_{y, *} \star \hj_{w, *} \simeq \hj_{yw, *} .\label{e:blockofbones}\end{equation}

\item For any $w \in W$, one has an equivalence 
$$\hj_{w, !} \star \hj_{w^{-1}, *} \simeq \widehat{\delta}_e.$$

\end{enumerate} 
    
\end{lemma}

\begin{proof} This can be deduced from the  free-monodromic version by applying Verdier duality; for completeness we provide an argument essentially following \cite{BezYun} and \cite{BezRiche}.

We first note that $\hj_{w, !}$ and $\hj_{w, *}$ may be characterized as the unique objects of \eqref{e:dogsonaflag} whose left (or right) $B$-averages $\on{Av}_{B, *}$ yield $j_{w, !}$ and $j_{w, *},$ respectively.

Let us show this for $\hj_{w, !}$. So, suppose we are given an object $\sM$ satisfying 
$$\on{Av}_{B, *}(\sM) \simeq j_{w, !}.$$
From the fact that $\on{Av}_{B, *}$ is conservative on $B$-monodromic objects and commutes with $*$-restriction to a single stratum, cf. Proposition \ref{p:avbbmon}, it follows that $\sM$ is $!$-extended from the stratum $BwB$: $$j_{!} j^! \sM \simeq \sM.$$Within the stratum, which is equivalent to $\Dmod(T / T\mon)$, the condition $\on{Av}_{B, *} j^! \sM \simeq \delta_w$ is exactly what characterizes the cofree-monodromic sheaf. The argument for $\hj_{w, *}$ is similar. 

With this, let us show the left identity of \eqref{e:blockofbones}. Denote the convolution product on the equivariant Hecke category by 
$$\Dmod(B_{\chL} \bs \chL / B_{\chL}) \otimes \Dmod(B_{\chL} \bs \chL / B_{\chL}) \rightarrow \Dmod(B_{\chL} \bs \chL / B_{\chL}), \quad \quad \sM_1 \boxtimes \sM_2 \mapsto \sM_1 \overset{B_{\chL}} \star \sM_2,$$and recall that there it is standard that 
$$j_{y, !} \overset{B_{\chL}} \star j_{w, !} \simeq j_{yw, !}.$$By the preceding paragraph, the desired identity for cofree-monodromic sheaves follows by computing 
\begin{align*}\on{Av}_{B_{\chL}, *}( \hj_{y, !} \star \hj_{w, !})  &\simeq (\on{Av}_{B_{\chL}, *} \hj_{y, !}) \star \hj_{w, !}  \\ &\simeq \on{Oblv}(j_{y, !}) \star \hj_{w, !} \\ & \simeq j_{y, !} \overset{B_{\chL}} \star \on{Av}_{B_{\chL}, *}(\hj_{w, !}) \\ & \simeq j_{y, !} \overset{B_{\chL}} \star j_{w, !} \\& \simeq j_{yw, !}. \end{align*}
The other assertions of the lemma may be deduced from their equivariant versions similarly.  
\end{proof}

\subsubsection{} We now have enough ingredients to give the desired explicit description of the monad. As above, we write $\delta_e$ for the constant perverse sheaf on the minimal stratum of $B_{\chL} \backslash \chL / B_{\chL}$, and recall from Section \ref{sss:avfunctors} that $\psi$ is nondegenerate.  We normalize the $t$-structure so that $\delta_e$ is the monoidal unit of $\Dmod(B_{\chL} \backslash \chL / B_{\chL})$.

\begin{prop}\label{p:adjtilt} \label{p:avmonad}Let $\sC$ be a $\Dmod(\chL)$-module. The monad on $\sC^{B_{\chL}\hspace{-.6mm} \on{-mon}}$ corresponding to the adjunction of Lemma \ref{l:avavpsi} is given by left convolution with an algebra object 
\[
   \widehat{\Xi} \in \Dmod(B_{\chL}\hspace{-.6mm} \on{-mon} \backslash \chL / B_{\chL}\hspace{-.6mm} \on{-mon}). 
\]
Moreover, the underlying object of of $\widehat{\Xi}$ is perverse, and the canonical maps
\[
     \delta_e \rightarrow \widehat{\delta_e} \rightarrow \on{Av}_{!*} \circ \on{Av}_{!*}^\psi( \widehat{\delta_e})
\]
exhibit it as the injective envelope of $\delta_e$, with cofree-monodromic standard and costandard filtrations with each $\hj_{w, !}$ and $\hj_{w, *},$ for $w \in W_{\chL}$, occuring once, respectively. \end{prop}

\begin{proof} For the first assertion, recall from the proof of Lemma \ref{l:avavpsi} that the adjunction in question arises from a $\Dmod(\chL)$-equivariant adjunction 
\[
        \Dmod(\chL/B_{\chL}\hspace{-.6mm} \on{-mon}) \rightleftarrows \Dmod(\chL/N^-_{\chL}, \psi, \St).
\]
In particular, the monad is an algebra object $\widehat{\Xi}$ in
\[
    {\mathbf{Fun}}_{\Dmod(\chL)\mod}( \Dmod(\chL/B_{\chL}\hspace{-.6mm} \on{-mon}), \Dmod(\chL/B_{\chL}\hspace{-.6mm} \on{-mon})) \simeq \Dmod(B_{\chL}\hspace{-.6mm} \on{-mon}\backslash \chL / B_{\chL}\hspace{-.6mm} \on{-mon}), 
\]
which establishes the first assertions of the Proposition. 

   For the remaining assertions, note that $\widehat{\delta}_e$ is the monoidal identity of \eqref{e:dogsonaflag}, and hence $\widehat{\Xi}$ is given by   
   \[
      \widehat{\Xi} \simeq \on{Av}_{!*} \circ \on{Av}_{!*}^\psi (\widehat{\delta}_e). 
   \]
  To calculate it, write $j_{w, !}$ for the $!$-extension of $\delta_w$, cf. the discussion of Section \ref{s:defmoncreatures}. Recalling the standard equivalence 
  \begin{equation} \label{e:monwhitcat}
        \Dmod(B_{\chL} \hspace{-.6mm} \on{-mon} \backslash \chL / N^-_{\chL}, \psi) \simeq \Dmod(\chT/\chT\hspace{-.6mm} \on{-mon}),
  \end{equation}
  it is straightforward to see that $j_{w, !}, j_{w, *}$ and $\widehat{\delta}_e$ are sent by $\on{Av}_{!*}^\psi$ to simple, simple,  and indecomposable injective objects in the heart of \eqref{e:monwhitcat}, respectively. In particular, we have 
  \begin{equation} \label{e:reverbanimal}
       \on{Hom}( j_{w, !}, \widehat{\Xi}) \simeq \on{Hom}(\on{Av}_{!*}^\psi j_{w, !}, \on{Av}_{!*}^\psi \widehat{\delta}_e) \simeq k. 
  \end{equation}
  As $\widehat{\Xi}$ is obtained from an injective object by a functor with an exact left adjoint, it is injective. Recalling that each $j_{w, !}$ has socle $\delta_e$ and head the simple object with support the closure of $BwB$, it follows from \eqref{e:reverbanimal} that $\widehat{\Xi}$ is the injective envelope of $\delta_e$.

 It remains to describe the cofree-monodromic standard and costandard filtrations on $\widehat{\Xi}$. We will show this for the standard filtration, and the argument for the costandard filtration is similar. 
 
 For the existence of the desired cofree-monodromic standard filtration, we first claim that for any $w \in W_{\chL}$ that the  $\Dmod(\chT\mon\!\bs\chT) \simeq \Dmod(\BL\mon\!\bs \BL / \BL\mon)$-equivariant functor
\begin{align}\label{e:howlongwillyou}\Dmod(\chT\mon\!\bs \chT) &\simeq \Dmod(B_{\chL}\mon\!\bs \BL w \BL / \BL\mon) \xrightarrow{ j_{*}} \Dmod(\BL\mon\!\bs \chL / \BL\mon) \\ &\xrightarrow{\on{Av}_{!*}^\psi} \Dmod(\BL\mon\!\bs \chL / N^-_{\chL}, \psi)  \simeq \Dmod(\chT\mon\!\bs \chT).\end{align}
is equivalent to the identity. Indeed, by equivariance, it is enough to see it exchanges $\widehat{\delta}_e$ with itself. By the beginning of the proof of Lemma \ref{l:norsupises}, it is further enough to see it exchanges $\delta_e$ with itself. However, the latter is exactly the assertion that $\on{Av}_{!*}^\psi(j_{w, *})$ is simple, which was noted above. By passing to left adjoints in \eqref{e:howlongwillyou}, one obtains the Cousin filtration of $\widehat{\Xi}$ associated to $*$-restriction to strata has successive quotients $\hj_{w, !}$, for $w \in W_{\chL}$.   \end{proof}


\subsubsection{}\label{def whittaker} To match the arguments on the spectral side, we now set up a compatibility between these Whittaker averaging functors as one varies the parabolic.

Recall that $\chG$ comes equipped with a pinning $(\chB, \chH, \{ e_i \}_{i \in I} )$, where $I$ indexes the simple roots, and each $e_i$ is a generator for the corresponding simple root subspace  $\check{\mathfrak{\fg}}_{\halpha_i}$. For each $i$, let us write $f_i$ for the corresponding element of $\check{\mathfrak{\fg}}_{-\halpha_i},$ characterized by the property that $$f_i, \alpha_i, e_i $$form an $\fsl_2$-triple.

For a  standard parabolic $\chP$ of $\chG$, with corresponding subset $J \subset I$ of Dynkin nodes, we may use the pinning to fix a choice of associated Whittaker category, as follows. Write $w_{\circ}^{\chP}$ for the longest element of $W_{\chP}$, and consider the associated Borel subgroup $\on{Ad}_{w_{\circ}^{\chP}} \chB$. Denote its unipotent radical by $\PN$. This carries a unique additive character \begin{equation} \label{e:defpsi}\psi: \PN \rightarrow \mathbb{G}_a\end{equation}which is trivial on $\PN \cap \chN$ and whose differential sends $f_j$ to 1, for all $j \in J$. 

\label{sss:normchars}

\subsubsection{} \label{sss:normtilts}
Recall the notations $F = k(\!(t)\!)$ and $O = k[\![t]\!]$,  the associated notations for loop and arc groups, and that $\chI \subset \chG_O$ denotes the standard Iwahori subgroup attached to our pinning. 

Let us write $_{\chP}\mathring{I}$ for the preimage of $\PN$ under the evaluation map 
\begin{equation}\chG_O \rightarrow \chG,\label{e:K1}\end{equation}
and denote its induced character still by $\psi$. Let us denote the kernel of \eqref{e:K1} by $K$, i.e., the first congruence subgroup.

\begin{defn}\label{define steinberg whit} For a category $\sC$ acted on by $\Dmod(\chG_O)$, its \emph{Steinberg--Whittaker invariants} with respect to $_{\chP}\mathring{I}$ are given by 
$$\sC^{_{\chP}\mathring{I}, \psi, \St} := (\sC^{K})^{_{\chP}\check{N}, \psi, \St},$$where we use the natural action of $\Dmod(\chG)$ on $K$-invariants, and the normalized character of $_{\chP}\check{N}$ from \eqref{e:defpsi}. 
\end{defn}

\subsubsection{}

Before explaining the compatibility between the categories of Steinberg--Whittaker invariants as we vary the parabolic, let us spell out what form the monads we are interested in take. Namely, by applying the preceding Proposition \ref{p:avmonad} to $\chL$ the Levi factor of $\chP$, and denoting by $\widehat{\Xi}_{\chP}$ the corresponding algebra object of $$\Dmod(\chI\mon\!\bs \chG_F / \chI\mon),$$
we obtain the following.

\begin{cor}\label{cor banana peel} The adjunction of averaging functors, cf. Section \ref{s:defavfuncs},
\[  \on{Av}_{!*}^{\psi}:  \Dmod( \chI\mon\!\bs \check{G}_F / \chI\mon) \rightleftarrows  \Dmod( _{\chP}\mathring{I}, \psi \backslash \check{G}_F / \chI\mon): \on{Av}_{!*} 
\]
exhibits the right-hand category as modules for the algebra object $\widehat{\Xi}_P$, acting by left convolution, in the left-hand category. 
\end{cor}

\subsubsection{} As we vary the choice of standard parabolic $\chP$, the normalized characters introduced in Section \ref{sss:normchars} satisfy the following basic compatibility. 

For any inclusion $\chP \subset \chQ$ of standard parabolics, and any $\Dmod(\chG)$-module $\sC$, consider the associated averaging functor $$\sC^{\PN, \psi} \xrightarrow{\Oblv} \sC \xrightarrow{\on{Av}_{*}^{\QN, \psi}[\dim \QN/(\QN \cap \PN)]} \sC^{\QN, \psi}.$$which we denote below by $\on{Av}_{!*}^{\chQ, \chP, \psi}$. The desired compatibility is as follows. 

\begin{lemma} \label{l:steps}Given a triple $\chR \subset \chP \subset \chQ$, there is a canonical natural equivalence $$\on{Av}_{!*}^{\chQ, \chR, \psi} \simeq \on{Av}_{!*}^{\chQ, \chP, \psi} \circ \on{Av}_{!*}^{\chP, \chR, \psi}. $$  
\end{lemma}

\begin{proof} Note that $\on{Av}_{!*}^{\chQ, \chR}$ is given by convolution with a kernel sheaf $$\mathscr{K}_{\chQ, \chR} \in \Dmod( \QN/ (\QN \cap \RN, \psi)).$$If we write  ${}_{\chQ, \chR}\chN $ for the subgroup of $\chQ$ whose Lie algebra is the sum of the $\chH$-weight spaces of $\on{Lie}(\QN)$ not contained in $\on{Lie}(\RN)$, it follows that the composition $${}_{\chQ, \chR}\chN \rightarrow \QN \rightarrow \QN/(\QN \cap \RN)$$is an isomorphism, which we denote by $\iota$. Moreover, by the definitions and the unipotence of the appearing groups, it is straightforward to see that $\iota^!\mathscr{K}_{\chQ, \chR}$ canonically identifies with $$\exp(\psi)^{\chQ, \chR} := \psi^!(\exp(z))[ -\dim {}_{\chQ, \chR}\chN],$$where $\psi: {}_{\chQ, \chR}\chN \rightarrow \mathbb{G}_a$ is defined as in \eqref{e:defpsi}, and $\exp(z)$ denotes the exponential character D-module on $\mathbb{G}_a$ (which lies in cohomological degree minus one). 

With this, the lemma follows by noting that the multiplication map $$\mu: {}_{\chR, \chQ}\chN \times {}_{\chQ, \chP}\chN \rightarrow  {}_{\chR, \chP}\chN$$is an isomorphism of varieties, and that one has a canonical isomorphism $$\mu^! \exp(\psi)^{\chR, \chP} \simeq \exp(\psi)^{\chR, \chQ} \boxtimes \exp(\psi)^{\chQ, \chP},$$which follows from the multiplicativity isomorphisms of a character sheaf and the definition of the characters as in \eqref{e:defpsi}. \end{proof}

\subsubsection{}  As a special case of the setup of the preceding lemma, let us take $\chR = \chB$, and consider the adjunctions $$\sC^{\chB\mon} \rightleftarrows \sC^{\PN, \psi, \St} \rightleftarrows \sC^{\QN, \psi, \St}.$$  
By considering their composition, Lemma \ref{l:steps} yields a morphism of monads $$\on{Av}_{!*}^{\chB,\chP}\circ \on{Av}_{!*}^{\chP, \chB, \psi} \rightarrow \on{Av}_{!*}^{\chB,\chQ}\circ \on{Av}_{!*}^{\chQ, \chB, \psi}. $$
By Proposition \ref{p:adjtilt}, this is given by a morphism of algebra objects of $\Dmod(\chB\mon\!\bs \chG / \chB\mon)$
\begin{equation}  \label{e:maptilts} \widehat{\Xi}_{\chP} \rightarrow \widehat{\Xi}_{\chQ}.
\end{equation}

\begin{prop} \label{autom algebra}The morphism \eqref{e:maptilts} is an injection in $\Dmod(\chB\mon\!\bs \chG / \chB\mon)^\heartsuit$, and exhibits $\widehat{\Xi}_{\chP}$ as the maximal subobject of $\widehat{\Xi}_{\chQ}$ supported on $\chP \subset \chG$. 

\end{prop}

\begin{proof}For any $w \in W_P$, let us denote by $\hj_{w, !}$ the corresponding cofree-monodromic standard object. It is enough to show that the induced maps \begin{equation} \label{e:mapinq}
\Hom_{\Dmod(\chB\mon\!\bs \chG / \chB\mon)}(\hj_{w, !}, \widehat{\Xi}_{\chP}) \rightarrow \Hom_{\Dmod(\chB\mon\!\bs \chG / \chB\mon)}(\hj_{w, !}, \widehat{\Xi}_{\chQ})
\end{equation}
are isomorphisms. However, by definition these identify with the  map \begin{align*}\Hom_{\Dmod(\chB\mon\!\bs \chG / \PN, \psi)}(\on{Av}_{!*}^{\chP, \chB, \psi}(\hj_{w, !}), \on{Av}_{!*}^{\chP, \chB, \psi}(\hj_{e, !})) \\ 
\xrightarrow{\on{Av}_{!*}^{\chQ, \chP, \psi}}\hspace{4ex} & \hspace{-4ex}\Hom_{\Dmod(\chB\mon\!\bs \chG / \QN, \psi)}(\on{Av}_{!*}^{\chQ, \chB, \psi}(\hj_{w, !}), \on{Av}_{!*}^{\chQ, \chB, \psi}(\hj_{e, !})).\end{align*}
But it is standard that one has isomorphisms 
$$\on{Av}_{!*}^{\chP, \chB, \psi}(\hj_{w, !}) \simeq  \on{Av}_{!*}^{\chP, \chB, \psi}(\hj_{e, !})\quad \text{and} \quad \on{Av}_{!*}^{\chQ, \chB, \psi}(\hj_{w, !}) \simeq \on{Av}_{!*}^{\chQ, \chB, \psi}(\hj_{e, !})$$
as both sides of each isomorphism are sent to the cofree-monodromic standard object with minimal support in $\Dmod(\chB\mon\!\bs \chG / \PN, \psi)$ and $\Dmod(\chB\mon\!\bs \chG / \QN, \psi)$, respectively. With respect to these isomorphisms, the map \eqref{e:mapinq} in question reduces to $$\End_{\Dmod(\chB\mon\!\bs \chG / \PN, \psi)}( \on{Av}_{!*}^{\chP, \chB, \psi}(\hj_{e, !})) \simeq \Sym(\fchh)^{\wedge}_0  \xrightarrow{\on{id}} \Sym(\fchh)^{\wedge}_0 \simeq \End_{\Dmod(\chB\mon\!\bs \chG / \QN, \psi)}(\on{Av}_{!*}^{\chQ, \chB, \psi}(\hj_{e, !}))$$
where we have used the fact that both endomorphism algebras are generated freely by the monodromy operators arising from the left action of $\chH$.  
\end{proof}

\subsubsection{} Fix a pair of standard parabolics 
$$\chP_i \subset \chG, \quad 1 \leqslant i \leqslant 2.$$
We will be interested in the category of D-modules 
\[
\Dmod( _{\chP_1}\mathring{I}, \psi_1, \St \backslash \chG_F /  _{\chP_2}\mathring{I}, \psi_2, \St).
\]
Write $W_{a}$ for the extended affine Weyl group of $\chG$, and $W_{i}$ for its parabolic subgroups associated to $I_{\chP_i}$, $1 \leqslant i \leqslant 2$. 
The Bruhat decomposition 
gives a stratification of $\chG_F$ with strata 
\[
I_{\chP_1}wI_{\chP_2} , \quad \quad \text{for } w \in W_{1} \backslash W_a / W_{2}. 
\]
We now inspect the abelian category of Steinberg--Whittaker sheaves associated 
with a single stratum.

\begin{prop}\label{p:1stratum} For each stratum, the abelian category %
	\[
	\Dmod( _{\chP_1}\mathring{I}, \psi_1, \St \backslash I_{\chP_1}wI_{\chP_2} /  _{\chP_2}\mathring{I}, \psi_2, \St)^\heartsuit
	\]
	contains a unique up to isomorphism simple object $L_w$ and indecomposable 
	injective object $I_w$. 
\end{prop}

\begin{proof} As we first review, if either $I_{\chP_1}$ or $I_{\chP_2}$ is $\chI$, the result is standard. Indeed, let us without loss of generality take $I_{\chP_1} = \chI$. If we write $\dot{w}$ for the minimal length element of $wW_{1}$, then $*$-pushforward yields an equivalence
\[
      \Dmod(\chI\hspace{-.6mm} \on{-mon} \backslash \chI\dot{w} _{\chP_2}\mathring{I} /  _{\chP_2}\mathring{I}, \psi) \xrightarrow{\sim} \Dmod(\chI\hspace{-.6mm} \on{-mon} \backslash \chI wI_{\chP_2} /  _{\chP_2}\mathring{I}, \psi).  
\]
I.e., every appearing D-module is cleanly extended from the minimal Bruhat cell. The claims of the proposition then straightforwardly follow from the identification of the former category with 
$$\Dmod(\chT\hspace{-.6mm} \on{-mon} \backslash \chT).$$

For the general case, consider as in Lemma \ref{l:avavpsi} the adjunction
\begin{equation} \label{e:adjbiwhit}
     \on{Av}_{!*}^\psi: \Dmod(\chI\hspace{-.6mm} \on{-mon} \backslash I_{\chP_1}wI_{\chP_2} /  _{\chP_2}\mathring{I}, \psi_2) \rightleftarrows \Dmod( _{\chP_1}\mathring{I}, \psi, \St \backslash I_{\chP_1}wI_{\chP_2} /  _{\chP_2}\mathring{I}, \psi_2): \on{Av}_{!*}.
\end{equation}
As $\on{Av}_{!*}$ is conservative, we will prove the proposition by a monadicity argument. However, it will be convenient to replace the left-hand category of \eqref{e:adjbiwhit} as follows. 

For an element $y \in W_1wW_2/W_2$ we denote by $M_y$ and $L_y$ the corresponding standard and simple objects of the left-hand side of \eqref{e:adjbiwhit}, i.e. the $!$-extension and intermediate extension of the simple object of 
\[
             \Dmod(\chI\on{-mon} \backslash \chI yI_{\chP_2} /  _{\chP_2}\mathring{I}, \psi_2)^\heartsuit. 
\]
If we write $\dot{w}$ for the minimal length element of $W_1wW_2/W_2$, we claim we have the vanishing 
\begin{equation} \label{e:only1simp}
    \on{Av}_{!*}^\psi(L_y) \simeq 0, \quad \quad \text{for } y \neq \dot{w}. 
\end{equation}

To see this, we will need to record a general observation. Fix $w \in W_1$, and let us denote by 
\[
     j_{w, !} \in \Dmod(\chI \backslash I_{\chP_1} / \chI)
\]
the corresponding standard object. For any $\Dmod(\chP_1)$-module $\sC$, let us denote the binary product underlying the convolution action of $\Dmod(\chI \backslash I_{\chP_1} / \chI)$ on $\sC^{\chI}$ by 
\[
    (-) \overset{\chI} \star (-): \Dmod(\chI \backslash I_{\chP_1} / \chI) \otimes \sC^{\chI} \rightarrow \sC^{\chI}. 
\]
Recall that $j_{e, !}$ is the monoidal unit of $\Dmod(\chI \backslash I_{\chP_1} / \chI)$, and fix an embedding $$j_{e, !} \hookrightarrow j_{w, !}.$$With this notation, we claim that the embedding yields a natural equivalence 
\begin{equation} \label{e:convwhitiso}
         \on{Av}_{!*}^\psi(-) \simeq \on{Av}_{!*}^\psi(\hspace{.5mm} j_{e, !} \overset{\chI} \star - ) \xrightarrow{\sim} \on{Av}_{!*}(\hspace{.5mm}j_{w, !} \overset{\chI} \star -): \sC^{\chI} \rightarrow \sC^{ _{\chP_2}\mathring{I}, \psi_2, \St}. 
\end{equation}
Indeed, one may immediately reduce to the case of $\sC \simeq \Dmod(\chP_1)$, where it is standard.

With this observation, let us prove the vanishing \eqref{e:only1simp}. For $y \neq \dot{w}$, fix a factorization 
\[
        y = x\dot{w}, \quad \quad \text{for some } x \in W_1.
\]
It is standard that convolution with $j_{e, !} \hookrightarrow j_{x, !}$ yields a monomorphism 
\begin{equation} \label{e:embwhitverms}
      M_{\dot{w}} \simeq j_{e, !} \overset{\chI} \star M_{\dot{w}} \hookrightarrow j_{x, !} \overset{\chI} \star M_{\dot{w}} \simeq M_y.
\end{equation}
As $L_y$ is the simple quotient of the cokernel of \eqref{e:embwhitverms}, the desired vanishing \eqref{e:only1simp} follows from \eqref{e:convwhitiso} and the exactness of $\on{Av}_{!*}^\psi$. 
    
    We are now ready to obtain the desired monad. Let us concatenate the adjunction \eqref{e:adjbiwhit} and the adjunction associated to a closed embedding  
    \[
       i_*: \Dmod( \chI\hspace{-.6mm} \on{-mon} \backslash \chI\dot{w}I_{\chP_2}/ _{\chP_2}\mathring{I}, \psi_2) \rightleftarrows   \Dmod(\chI\hspace{-.6mm} \on{-mon} \backslash I_{\chP_1}wI_{\chP_2} /  _{\chP_2}\mathring{I}, \psi_2): i^!.
    \]
    It follows from the vanishing \eqref{e:only1simp} that in the resulting adjunction 
    \[
            \on{Av}_{!*}^\psi \circ \hspace{1mm} i_*: \Dmod(\chI\hspace{-.6mm} \on{-mon} \backslash \chI\dot{w}I_{\chP_2} /  _{\chP_2}\mathring{I}, \psi_2) \rightleftarrows  \Dmod( _{\chP_1}\mathring{I}, \psi, \St \backslash I_{\chP_1}wI_{\chP_2} /  _{\chP_2}\mathring{I}, \psi, \St): i^! \circ \on{Av}_{!*},
    \]
    the right adjoint is again conservative. 
    
        To describe the monad, recall that by the minimality of $\dot{w}$ the group
        \begin{equation} \label{e:p12}
                W_{I_{\chP_1}wI_{\chP_2} } := \{ x \in W_1: \hspace{1mm} \dot{w}^{-1}x\dot{w} \in W_2 \}
        \end{equation}
        is a parabolic subgroup of $W_1$. Let us write $I_{\chP_w}$ for the corresponding standard parahoric subgroup of $\chG_F$. In addition,  for a standard parahoric subgroup $\check{Q}$ of $\chG_F$, let us write $\widehat{\Xi}_{\check{Q}}$ for the algebra object of 
        \[
            \Dmod(\chI\hspace{-.6mm} \on{-mon} \backslash \chG_F / \chI\hspace{-.6mm} \on{-mon})
        \]
        associated to it as in Proposition \ref{p:avmonad}. With this, by Proposition \ref{p:avmonad}, the underlying endofunctor of our monad is given by 
        \begin{equation} \label{e:ourmonad}
             i^! \circ ( \hspace{1mm}  \widehat{\Xi}_{\chP_1} \star - ) .   
        \end{equation}
       Note the homomorphism of algebra objects 
       \[
          \widehat{\Xi}_{\chP_w} \rightarrow \widehat{\Xi}_{\chP_1}
       \]
       associated to the corresponding factorization of averaging functors, cf. Lemma \ref{l:steps}. It is straightforward to see using the costandard filtration on $\widehat{\Xi}_{\chP_1}$ that this induces an equivalence of endofunctors
       \[
            \widehat{\Xi}_{\chP_w} \star - \simeq  i^! \circ (\widehat{\Xi}_{\chP_1} \star - ): \Dmod(\chI\hspace{-.6mm} \on{-mon} \backslash \chI \dot{w}I_{\chP_2} /  _{\chP_2}\mathring{I}, \psi_2) \rightarrow \Dmod(\chI\hspace{-.6mm} \on{-mon} \backslash \chI\dot{w}I_{\chP_2} /  _{\chP_2}\mathring{I}, \psi_2).
       \]

       We may rewrite this monad as follows. Recall that $\check{\mathfrak{h}}$ denotes the Cartan subalgebra of $\chG$, write $R$ for the completion of $\Sym(\check{\mathfrak{h}})$ at the augmentation ideal, and recall that to a choice of coset representative $\ddot{w}$ of $\dot{w}$ in $\check{H}\dot{w}$ one has the $t$-exact equivalence 
       \[
          \Dmod(\chI\hspace{-.6mm} \on{-mon} \backslash \chI \dot{w}I_{\chP_2} /  _{\chP_2}\mathring{I}, \psi_2) \simeq \IndCoh(R),
       \]      
      where $\IndCoh(R)$ denotes the ind-completion of $\on{Coh}(R)$, the category of bounded complexes of $R$-modules with finite dimensional cohomology. Explicitly, this sends a D-module to its $*$-stalk at $\ddot{w}$, equipped with its logarithm of monodromy operators. If we write $R^w$ for the subalgebra of $R$ fixed by \eqref{e:p12}, Soergel's Endomorphismsatz identifies convolution with $\widehat{\Xi}_{\chP_w}$ with the standard monad with underlying endofunctor 
      \begin{equation} \label{e:nicemonadsoergel}
         R \underset{R^w}\otimes (-): \IndCoh(R) \rightarrow \IndCoh(R).
      \end{equation}
      %
      %
      %
      %
      %
    In particular we deduce a $t$-exact equivalence
      \begin{equation} \label{e:ssbims}
           \IndCoh(R^w) \simeq \Dmod( _{\chP_1}\mathring{I}, \psi, \St \backslash I_{\chP_1}wI_{\chP_2} /  _{\chP_2}\mathring{I}, \psi_2, \St),
      \end{equation}
      which immediately implies the claims of the proposition. 
\end{proof}

Note that the proof of Proposition \ref{p:1stratum} shows the following, which will be of use later. 

\begin{cor}\label{c:OxYtOcIn}Consider the adjunction of two sided averaging functors $$\on{Av}_{!*}: \Dmod( _{\chP_1}\mathring{I}, \psi, \St \backslash I_{\chP_1}wI_{\chP_2} /  _{\chP_2}\mathring{I}, \psi_2, \St)^\heartsuit \rightleftarrows \Dmod( \chI \mon\!\bs I_{\chP_1} w I_{\chP_2} / \chI \mon)^\heartsuit: \on{Av}_{!*}^\psi.$$An object $\sM$ of the left-hand side is injective if and only if its average $\on{Av}_{!*}(\sM)$ is injective. Moreover, it is a finite direct sum of indecomposable injective if and only if its average is a finite direct sum of indecomposable injectives.   
\end{cor}

\begin{rmk} The Soergel-theoretic identity \eqref{e:ssbims} for a single stratum  extends to an equivalence, for a connected reductive group or a Kac--Moody group, between bi-Steinberg--Whittaker sheaves and a category of singular Soergel bimodules.  
\end{rmk}

\subsection{Nearly compact objects and cofree monodromic tilting sheaves}
\label{Cofree tilting sec}

\subsubsection{} Having described the Steinberg--Whittaker sheaves on each stratum, i.e. 
\begin{equation} \label{e:1stratumbiwhit}
     \Dmod( _{\chP_1}\mathring{I}, \psi_1, \St \backslash I_{\chP_1} w I_{\chP_2} /  _{\chP_2}\mathring{I}, \psi_2, \St), \quad \quad \text{for } w \in W_1 \backslash W / W_2,
\end{equation}
we may now introduce the cofree-monodromic tilting sheaves and classify the indecomposable ones. 

\begin{defn}\label{d:freemontil} Let us call an object of $\Dmod( _{\chP_1}\mathring{I}, \psi_1, \St \backslash \chG_F /  _{\chP_2}\mathring{I}, \psi_2, \St)$ {\em cofree-monodromic tilting} if the following two conditions hold:
\begin{enumerate}\item  it is $*$-extended from a subscheme of $\chG_F$, i.e., is supported on finitely many strata, and 
\item  its $*$-restriction and $!$-restriction to each stratum \eqref{e:1stratumbiwhit} is isomorphic to a finite direct sum of the indecomposable injective object.\end{enumerate} \end{defn}

We now set up the accompanying notions of (co)standard filtrations. Note that a simple object $\delta_w$ of \eqref{e:1stratumbiwhit} is canonically given by averaging the constant perverse sheaf on $\chI\dot{w}\chI$, cf. the proof of Proposition \ref{p:1stratum}, and in particular Equation \eqref{e:nicemonadsoergel}. We have a canonical choice of indecomposable injective
\begin{equation} \label{e:indinjdisc}
    \widehat{\delta}_w := \delta_w \underset{\on{Hom}(\delta_w, \delta_w)} \otimes k.
\end{equation}
Recall that we denote the $!$-extension and $*$-extension of $\widehat{\delta}_w$ to all of $\chG_F$ by $\hj_{w, !}$ and $\hj_{w, *}$, respectively, and say an object $\sM$ admits a cofree-monodromic standard (resp. costandard) filtration if it is a successive extension of the $\hj_{w, !}$ (resp. $\hj_{w, *}$). 

\begin{lemma} \label{l:NeXt}
An object is cofree-monodromic tilting if and only if it admits finite cofree-monodromic standard and costandard filtrations. In particular, any cofree-monodromic tilting sheaf lies in the heart of the $t$-structure, and for cofree-monodromic tilting sheaves $T_1$ and $T_2$ one has the Ext-vanishing
\[
     \on{Hom}(T_1, T_2) \simeq \tau^{\geqslant 0}\tau^{\leqslant 0} \on{Hom}(T_1, T_2). 
\]
\end{lemma}

\begin{proof}The first assertion is immediate from the definitions. For the second, it is enough to show that $\hj_{w, !}$ and $\hj_{w, *}$ lie in the heart of the $t$-structure. Consider the locally closed embedding $$j: I_{\chP_1}wI_{\chP_2} \rightarrow G.$$As the source of $j$ is not in general affine, the claim is not automatic, so we argue by averaging.

 Recall the minimal length element $\dot{w}$ of $W_1wW_2$, and the the corresponding injective object 
\[
     \widehat{\delta}_{\dot{w}} \in \Dmod(\chI\hspace{-.6mm} \on{-mon} \backslash \chI\dot{w}I_{\chP_2} /  _{\chP_2}\mathring{I}, \psi_2, \St).
\]
Note that $\on{Av}_{!*}^\psi \widehat{\delta}_{\dot{w}}$ is isomorphic to a (finite) direct sum of copies of $\widehat{\delta}_w$, as follows for example from Proposition \ref{p:1stratum}.  It is therefore enough to see that $!$-extension and $*$-extension of the former lie in the heart of the $t$-structure. However, it follows from Lemma \ref{l:avavpsi} that we have 
\begin{align*}
   & j_{!} \circ \on{Av}_{!*}^\psi \hspace{.5mm} (\widehat{\delta}_{\dot{w}}) \simeq \on{Av}_{!*}^\psi \circ \hspace{1mm} j_{!} \hspace{.5mm} (\widehat{\delta}_{\dot{w}}) \simeq \on{Av}_{!*}^\psi \hspace{.5mm} (\hj_{\dot{w}, !}) 
\end{align*}
The claim for the $!$-extension then follows from the $t$-exactness of $\on{Av}_{!*}^\psi$ implied by Lemma \ref{l:avavpsi} and recalling that $\hj_{w, !}$ lies in the heart due to the affineness of $\chI\dot{w}I_{\chP_2}$ and separatedness of $\chG_F$. The same argument applies, {\em mutatis mutandis}, to the $*$-extension. Finally, the desired Ext vanishing follows from those between the $\hj_{w, !}$ and $\hj_{w', *}$, which is immediate by adjunction.   \end{proof}

\subsubsection{} It is clear that the collection of cofree-monodromic tilting sheaves is closed under finite direct sums and summands. In the following proposition, we show that the classification of indecomposable cofree-monodromic tilting sheaves follows the usual pattern, i.e., they are indexed by strata. 

\begin{prop}\label{p:classtilting} \label{p:classtilts}For each stratum $\eqref{e:1stratumbiwhit}$, there is a unique up to (non-unique) isomorphism indecomposable cofree-monodromic tilting sheaf $T_w$ such that 
\begin{enumerate}
    \item it is supported on the closure of $I_{\chP_1}wI_{\chP_2} $, and 
    
    \item its restriction to the open stratum \eqref{e:1stratumbiwhit} is an indecomposable injective object. 
\end{enumerate}
Moreover, any indecomposable cofree-monodromic tilting sheaf is isomorphic to exactly one of the $T_w$. 
\end{prop}

\begin{proof} Recall that the closure of $I_{\chP_1}wI_{\chP_2} $ contains only finitely many strata. We first construct such a $T_w$ inductively by extending it over strata of higher and higher codimension.

Suppose $X$ is an open union of strata in the closure of $I_{\chP_1}wI_{\chP_2} $. We will construct inductively an indecomposable cofree-monodromic tilting sheaf $T_X$ on $X$ satisfying condition (2). If $X$ is only the open stratum, set $T_X$ to be the object $\widehat{\delta}_z$. Otherwise, fix a stratum which is closed in $X$, and denote it and its open complement by$$i: Z \rightarrow X \quad \quad \text{and} \quad \quad j: U \rightarrow X,$$respectively. By induction, we may take as given an indecomposable cofree-monodromic tilting sheaf $T_U$ on $U$ satisfying condition (2) above. 

We will construct a similar sheaf on $X$ by modifying $j_! T_U$ as follows. Recall the simple object $\delta_z$ and indecomposable injective $\widehat{\delta}_z$ on $Z$, cf. the discussion near \eqref{e:indinjdisc}. We first claim that 
\begin{equation}\label{e:claim0n1}
      \on{Hom}( \delta_z, j_! T_U) \text{ is concentrated in cohomological degrees 0 and 1.}
\end{equation}
To see this, recall from the proof of Proposition \ref{p:1stratum} that $\delta_z$ is the average of the simple $\chI$-equivariant sheaf $\delta_{\dot{z}}$ on the minimal $\chI$-orbit $\dot{Z}$ in $Z/\chP_2$, hence by adjunction
\begin{align}
  \nonumber \on{Hom}_{\Dmod( _{\chP_1}\mathring{I}, \psi \backslash X /  _{\chP_2}\mathring{I}, \psi)}( \delta_z, j_! T_U) &\simeq \on{Hom}_{\Dmod( _{\chP_1}\mathring{I}, \psi \backslash X /  _{\chP_2}\mathring{I}, \psi)}( \on{Av}_{!*}^\psi \delta_{\dot{z}}, j_! T_U) \\\nonumber  &\simeq \on{Hom}_{\Dmod(\chI \backslash X /  _{\chP_2}\mathring{I}, \psi)  }( \delta_{\dot{z}}, \on{Av}_*^{\chI} \circ \on{Av}_{!*} \circ \hspace{.5mm}j_! (T_U)). 
  \intertext{By Proposition \ref{p:avbbmon} and Lemma \ref{l:avavpsi}, we may rewrite the latter as}
  \label{e:0and1} & \simeq \on{Hom}_{\Dmod({\chI} \backslash X /  _{\chP_2}\mathring{I}, \psi)}(\delta_{\dot{z}}, j_! \circ \on{Av}_*^{\chI} \circ \on{Av}_{!*} (T_U)).
\end{align}
Again using Proposition \ref{p:avbbmon} and Lemma \ref{l:avavpsi}, it follows that $\on{Av}_*^{\chI} \circ \on{Av}_{!*}$ sends (co)standard objcts to (co)standard objects, and thus 
\[
    \Upsilon := \on{Av}_*^{\chI} \circ \on{Av}_{!*} (T_U)
\]
is a tilting object of $\Dmod({\chI} \backslash X /  _{\chP_2}\mathring{I}, \psi)$. 

The desired vanishing \eqref{e:claim0n1} now follows straightforwardly from this observation by a general argument about highest weight categories. Namely, to see that 
\begin{equation} \label{e:bingbong}
\on{Hom}(\delta_z, j_! T_U) \simeq \tau^{\geqslant 0} \on{Hom}(\delta_z, j_! T_U),
\end{equation}
note that the standard filtration on $\Upsilon$ implies that $j_!(\Upsilon)$ again has a standard filtration, and hence lies in the heart of the $t$-structure. Dually, to see that 
\begin{equation} \label{e:bongbing}
\on{Hom}(\delta_z, j_! T_U) \simeq \tau^{\leqslant 1} \on{Hom}(\delta_z, j_! T_U),
\end{equation}
one may use the standard filtration on $\Upsilon$. Namely if we denote the inclusion of the minimal stratum and its complement by 
\[
     \iota: \dot{Z} \rightarrow X \quad \quad \text{and} \quad \quad \zeta: \dot{U} \rightarrow X,
 \]
the claim follows by applying the triangle 
\[
       \iota^!\circ  \zeta_!\circ \zeta^! \rightarrow \iota^! \rightarrow \iota^! \circ \iota_* \circ \iota^* \xrightarrow{+1}
\]
to a standard object $j_{u, *}$ from $\dot{U}$ and using the vanishing of  $\iota^! j_{u, *}$ and the right exactness of $\iota^*$.

     Let us apply \eqref{e:0and1} to obtain the desired cofree-monodromic tilting extension. Namely, note that the tautological equivalence 
     \[
           \on{Hom}(\delta_z, - ): \Dmod( _{\chP_1}\mathring{I}, \psi, \St \backslash Z /  _{\chP_2}\mathring{I}, \psi, \St) \simeq \on{Hom}(\delta_z, \delta_z)^{op}\on{-mod}. 
     \]
exchanges the injective envelope $\widehat{\delta}_z$ and the augmentation module $k$. In particular, from \eqref{e:0and1} and the fact that appearing endomorphism algebra is coconnective and has $$\on{Ext}^0(\delta_z, \delta_z) \simeq k,$$we obtain, cf. Proposition \ref{p:coconnectivealgs} and Example \ref{ex:minres}, a triangle 
\[
        \widehat{\delta}_z \underset k \otimes \on{Ext}^1(\delta_z, j_! T_U)[-1]  \rightarrow    i^!j_! T_U  \rightarrow \widehat{\delta}_z \underset k \otimes \on{Ext}^0(\delta_z, j_! T_U) \xrightarrow{+1}.
\]
Consider the adjoint of the left-hand morphism. We claim its cone $T_X$, i.e.,
\begin{equation} \label{e:birthofT}
      i_* \widehat{\delta}_z \underset k \otimes \on{Ext}^1(\delta_z, j_! T_U)[-1] \rightarrow  j_! T_U \rightarrow T_X \xrightarrow{+1}
\end{equation}
is the sought-for indecomposable cofree-monodromic tilting extension. Indeed, by construction it has the correct $*$-stalks and $!$-stalks, i.e., is  cofree-monodromic tilting. 
  By induction, we may take $X$ to be closure of $I_{\chP_1}wI_{\chP_2} $, and thereby obtain some  cofree-monodromic tilting object $T_w$ satisfying (1) and (2).  
  
  To prove the remaining assertions of the proposition, and in particular that $T_w$ is indecomposable,\footnote{This indecomposability is also straightforward to see directly from the construction.} we  claim that any endomorphism $\phi$ of $T_w$ which is an isomorphism upon restriction to $I_{\chP_1}wI_{\chP_2} $ is an isomorphism. Indeed, let us prove this by induction on the strata, as above. It follows from the existence of cofree monodromic standard and costandard filtrations on $T_X$ that one has a  short exact sequence 
$$0 \rightarrow \Hom(i^* T_X, i^! T_X) \rightarrow \Hom(T_X, T_X) \rightarrow \Hom(j^! T_X, j^! T_X)  \rightarrow 0.$$
 Explicitly, to a map $\mu: i^* T_X \rightarrow i^! T_X$, the associated endomorphism of $T_X$ is given by 
$$T_X \rightarrow i_* i^* T_X \xrightarrow{\mu} i_! i^! T_X \rightarrow T_X.$$
In particular, the composition of any two such maps includes the canonical map 
$$i_! i^! T_X \rightarrow T_X \rightarrow i_* i^* T_X.$$
By applying the sequence 
$$i_! i^! \rightarrow \on{id} \rightarrow i_* i^*$$
to every term of \eqref{e:birthofT}, it follows that the map $i_! i^! T_X \rightarrow i_* i^* T_X$ identifies with the differential in our chosen minimal resolution of $i^! j_! T_U$
$$ i_* \widehat{\delta}_x \underset{k} \otimes \hspace{.7mm} \Ext^0(\delta_z, j_! T_U) \rightarrow i_* \widehat{\delta}_x \underset{k} \otimes  \hspace{.7mm} \Ext^1(\delta_z, j_! T_U),$$
cf. Example \ref{ex:minres}. In particular, under the identification of Steinberg--Whittaker sheaves on our stratum $Z$ with $\on{QC}^!(R^w)$, cf. Equation \eqref{e:ssbims}, if we write $\mathfrak{m}$ for the maximal ideal of $R^w$, the differential lies in 
\begin{equation} \label{e:lambdaofgod}\mathfrak{m} \cdot \Hom(\widehat{\delta}_x \underset{k} \otimes \hspace{.7mm} \Ext^0(\delta_z, j_! T_U), \widehat{\delta}_x \underset{k} \otimes \hspace{.7mm} \Ext^1(\delta_z, j_! T_U)). \end{equation}
With this observation in hand, let us prove that any endomorphism of $T_X$ which induces an isomorphism of $j^! T_X$ is an isomorphism itself. To see this, note that $R^{W_1}$ acts by monodromy operators on all of $T_X$, as it does on any object in any category of Steinberg--Whittaker invariants 
$$R^{W_1} \rightarrow HH^\bullet(\sC^{ _{\chP_1}\mathring{I}, \psi, \St}),$$cf. Proposition 9.2.9 of \cite{campbelldhillon}.\footnote{Note that, while what literally is written there identifies $R^{W_1}$ with the equivariant Hochschild cohomology of a certain category of Lie algebra representations. However, the latter is identified in Theorem 7.2.5 of {\em loc. cit.} with a category of Steinberg--Whittaker sheaves, and in particular the above equivariant Hochschild cohomology is identified with the endomorphisms of the monoidal unit of 
$$\Dmod( _{\chP_1}\mathring{I}, \psi, \St \bs \chG_F / _{\chP_1}\mathring{I}, \psi, \St).$$} Write $\mathfrak{m}_1$ for the maximal ideal of $R^{W_1}$. By considering e.g. a cofree-monodromic standard filtration on $T_X$, one sees that if for any integer $N \geqslant 0$ we write $T_X^{\mathfrak{m}_1^N}$ for the maximal subobject of $T_X$ annihialted by $\mathfrak{m}_1^N$, this is finite length and we have that 
$$T_X \simeq \varinjlim_N T_X^{\mathfrak{m}_1^N}.$$
The claim now follows by noting that the natural map 
$$\Hom(T_X, T_X) \rightarrow \Hom(T_X^{\mathfrak{m}_1^N}, T_X^{\mathfrak{m}_1^N})$$sends $\Hom(i^* T_X, i^! T_X)$ to nilpotent endomorphisms of $T_X^{\mathfrak{m}_1^N}$ thanks to the observation preceding \eqref{e:lambdaofgod}.

 With this, let us now deduce our final claim regarding isomorphism classes of indecomposable cofree-monodromic tilting sheaves. Let $T$ be any cofree-monodromic tilting sheaf. Pick a stratum $$j: I_{\chP_1}wI_{\chP_2} \rightarrow G$$ maximal in its support. Pick any retraction, i.e. a pair of homomorphisms 
 \[
          \phi: \widehat{\delta}_w \xrightarrow{} j^*T \quad \text{and} \quad \psi: j^*T \xrightarrow{} \widehat{\delta}_w, \quad \text{where } \psi \circ \phi = \on{id}_{\widehat{\delta}_w}.  
 \]
As $T$ and $T_w$ are cofree-monodromic tilting, a standard argument using (co)standard filtrations shows these are the restrictions of some homomorphisms 
\[
     \widetilde{\phi}: T_w \rightarrow T \quad \text{and} \quad \widetilde{\psi}: T \rightarrow T_w.
\]
By the preceding paragraph, the composition $\widetilde{\psi} \circ \widetilde{\phi}$ is an automorphism of $T_w$, i.e., $T_w$ is a summand of $T$, as desired. \end{proof}

\subsubsection{} Finally, let us unwind the promised connection between cofree-monodromic tilting sheaves and near compactness.

Let us denote the additive category of cofree-monodromic tilting sheaves by 
$$\on{Tilt}_{\chP_1, \chP_2} \hookrightarrow \Dmod( _{\chP_1}\mathring{I}, \psi_1, \St \backslash \chG_F /  _{\chP_2}\mathring{I}, \psi_2, \St)^\heartsuit.$$
This inclusion into the abelian category prolongs into a fully faithful embedding from the homotopy category of bounded complexes of cofree-monodromic tiltings into the derived category of Steinberg--Whittaker sheaves
\begin{equation} \label{e:eScHaT} K^b(\on{Tilt})_{\chP_1, \chP_2} \hookrightarrow \Dmod( _{\chP_1}\mathring{I}, \psi_1, \St \backslash \chG_F /  _{\chP_2}\mathring{I}, \psi_2, \St),  \end{equation}cf. Lemma \ref{l:NeXt}.

\begin{cor} \label{c:PoSt}Fix an object $\sM$ of $\Dmod( _{\chP_1}\mathring{I}, \psi_1, \St \backslash \chG_F /  _{\chP_2}\mathring{I}, \psi_2, \St)$. The following are equivalent. 

\begin{enumerate}
    \item  $\sM$ belongs to the pre-triangulated hull of the $\hj_{w, !}$ for $w \in W_1 \bs W_a / W_2.$

    \item $\sM$ belongs to the essential image of \eqref{e:eScHaT}. 

    \item The left $\chI$-average $\on{Av}_{\chI, *} \sM \in \Dmod(\chI \bs \chG_F / _{\chP_2}\mathring{I}, \psi_2, \St)$ is compact. 

    \item The right $\chI$-average $\on{Av}_{\chI, *} \sM \in \Dmod( _{\chP_1}\mathring{I}_1, \psi, \St \bs \chG_F / \chI)$ is compact. 
    
    \item The two-sided $\chI$-average $\on{Av}_{\chI \times \chI, *} \sM \in \Dmod(\chI \bs \chG_F / \chI) $ is compact. 
\end{enumerate}
\end{cor}

\begin{proof}The equivalence of 1. and 2. is immediate from Proposition \ref{p:classtilting}. That they imply 3., 4., and 5. follows from the proof of Proposition \ref{p:1stratum}. Finally, that 5. implies 1. and 2. follows from Corollaries \ref{c:OxYtOcIn} and \ref{c:whoisncshv}.    
\end{proof}

With the previous corollary in mind, let us introduce the following terminology. 
\begin{defn} \label{d:ncnear2}
Let $\chL$ be as in Section \ref{s:defst}, and consider a $\Dmod(\chL)$-module $\sC$. We define the {\em nearly compact} objects 
$$(\sC^{N^-_{\chL}, \psi, \St})^{n.c.} \subset \sC^{N^-_{\chL}, \psi, \St}$$
to be the full subcategory of objects $\xi$ such that their average 
$\on{Av}_{B_{\chL}, *} \xi \in \sC^{\BL}$ is compact. Note this is equivalent to asking that their monodromic average $\on{Av}_{\BL\mon, *} \xi \in \sC^{\BL\mon}$ be nearly compact, in the sense of Section \ref{sss:nearcompactdef}. 
\end{defn}

\subsubsection{} We may rephrase Corollary \ref{c:PoSt} by saying that the nearly compact 
objects in a bi-Steinberg--Whittaker category are the pretriangulated hull of the cofree-monodromic tilting sheaves.

In the remainder of this subsection, we would like to give similarly explicit descriptions of the nearly compact objects in the other categories involving a Whittaker condition on one side appearing in Bezrukavnikov's conjectures.

That is, we would like to analyze the nearly compact objects in categories of two types. First, we consider categories of the form 
\begin{equation} \label{e:aLlaRms}\Dmod(I_{\chP_1} \bs \chG_F / _{\chP_2}\mathring{I}, \psi, \St) \simeq \Dmod(I_{\chP_1} \bs \chG_F / _{\chP_2}\mathring{I}, \psi),\end{equation}i.e., with a strict parahoric equivariance condition on the other side. 

To describe the second type of category, recall that $I'_{\chP}$ denotes the derived subgroup of $I_{\chP}$. For a category $\sC$ acted on by $\Dmod(I_{\chP})$, we will use the notation 
$$\sC^{(I_{\chP}, I'_{\chP})\mon} := (\sC^{I'_{\chP}})^{I_{\chP}/I'_{\chP}\mon}.$$


The second type of category of interest to us will be the nearly compact objects within 
\begin{equation}\Dmod( (I_{\chP_1}, I'_{\chP_1})\mon\!\bs \chG_F / _{\chP_2}\mathring{I}, \psi, \St) \simeq \Dmod( (I_{\chP_1}, I'_{\chP_1})\mon\!\bs \chG_F / _{\chP_2}\mathring{I}, \psi).\label{e:aRouNdYoU}\end{equation}
As we will see, in the case \eqref{e:aLlaRms} of strict equivariance, there will be no difference between compactness and near compactness, but in the case \eqref{e:aRouNdYoU} there will be, due to the central monodromicity. In both cases, the category of interest will be the pretriangulated hull of certain (cofree-monodromic) standard or tilting sheaves.

\subsubsection{} Let us begin with the case \eqref{e:aLlaRms} of strict equivariance. Fix a stratum 
$$I_{\chP_1} w I_{\chP_2}, \quad \quad w \in W_1 \bs W_a / W_2. $$
Recall in addition the associated parabolic subgroup  
$W_{I_{\chP_1}wI_{\chP_2}} \subset W_1,$ 
cf. Equation \eqref{e:p12}. Let us say that a double coset $w \in W_1 \bs W_a / W_2$ is {\em relevant} if $W_{I_{\chP_1}wI_{\chP_2}}$ is the trivial subgroup of $W_1$, and otherwise {\em irrelevant}.

\begin{prop} \label{p:sInGsTraT}The category associated to a single stratum
$$\Dmod( I_{\chP_1} \bs I_{\chP_1} w I_{\chP_2} / _{\chP_2}\mathring{I}, \psi)$$vanishes if $w$ is irrelevant. Otherwise, if we write $\dot{w}$ for the minimal length element of $W_1wW_2$, taking the $!$-fiber at $I_{\chP_1} \dot{w}$  yields an equivalence 
$$\Dmod(I_{\chP_1} \bs I_{\chP_1} w I_{\chP_2} / _{\chP_2}\mathring{I}, \psi) \simeq \Vect, \quad \quad \text{for $w$ relevant}.$$\end{prop}

\begin{proof} Recall that $\chL_2$ denotes the Levi factor of $I_{\chP_2}$, and let us view $\chL_2$ as a subgroup of $\chG_F$ via our chosen Cartan. The right action of $\chL_2$ on $I_{\chP_1}\dot{w} \in I_{\chP_1} \bs \chG_F$ yields a map 
$$I_{\chP_1} \bs I_{\chP_1} w I_{\chP_2} \leftarrow \dot{w}^{-1} I_{\chP_1} \dot{w} \cap \chL_2 \bs \chL_2.  $$
Upon $!$-restricting D-modules, this yields an identification 
$$\Dmod(I_{\chP_1} \bs I_{\chP_1} w I_{\chP_2} / _{\chP_2}\mathring{I}, \psi) \simeq \Dmod( w^{-1} I_{\chP_1} w \cap \chL_2 \bs \chL_2 / N^-_{\chL_2}, \psi).$$
To analyze this, consider the tautological inclusion \begin{equation} \label{e:surfactant} w^{-1} I_{\chP_1} w \cap \chL_2 \supset w^{-1} I w \cap \chL_2.\end{equation} As $\dot{w}$ is of minimal length in $\dot{w} W_2$, we have that $w^{-1} I w \cap \chL_2$ is the standard Borel subgroup of $\chL_2$, i.e., the image of $I$. Recalling that $\psi$ is a nondegenerate character of $N^-_{\chL_2}$, it follows that if the inclusion \eqref{e:surfactant} is strict, the appearing category vanishes, and otherwise $!$-restricting to the identity element of $\chL_2$ yields an equivalence with $\Vect$.
\end{proof}

Note that the inclusion 
$I_{\chP_1} \bs I_{\chP_1} \dot{w} _{\chP_2}\mathring{I} \rightarrow I_{\chP_1} \bs \chG_F $ is affine, and in particular the $!$- and $*$- extension functors 
$$j_!, j_*: \Dmod(I_{\chP_1} \bs I_{\chP_1} w I_{\chP_2} /  _{\chP_2}\mathring{I}, \psi) \rightarrow \Dmod(I_{\chP_1} \bs \chG_F / _{\chP_2}\mathring{I}, \psi)$$
are $t$-exact. It therefore follows from Proposition \ref{p:sInGsTraT} that the finite length objects in the abelian category 
$$\Dmod( I_{\chP_1} \bs \chG_F / _{\chP_2}\mathring{I}, \psi)^{\heartsuit, f.l.}$$
form a highest weight category, with standard and costandard objects the $!$- and $*$- extensions of the simple object from each relevant stratum, which we denote by $j^\psi_{w, !}$ and $j^\psi_{w, *}$, respectively. 

Moreover, the compact objects in $\Dmod(I_{\chP_1} \bs \chG_F / _{\chP_2}\mathring{I}, \psi)$ are simply the pretriangulated hull of the (co)standard objects, i.e., bounded complexes with holonomic cohomology. Let us now show these agree with the nearly compact objects. 

\begin{prop}\label{p:bIgMisTake} An object $\xi$ of $\Dmod(I_{\chP_1} \bs \chG_F / _{\chP_2}\mathring{I}, \psi)$ is nearly compact if and only if it is compact.     
\end{prop}

\begin{proof} Consider the case of a single relevant stratum, i.e., the adjunction 

\begin{equation} \label{e:deadred}\on{Av}_{!*}^\psi: \Dmod(I_{\chP_1} \bs I_{\chP_1}\dot{w}I_{\chP_2} / I\mon) \rightleftarrows \Dmod(I_{\chP_1} \bs I_{\chP_1}\dot{w}I_{\chP_2} / _{\chP_2}\mathring{I}, \psi): \on{Av}_{!*}. \end{equation}
We need to show an object $\xi$ of the right-hand side is compact if and only if $\on{Av}_{!*}(\xi)$ is nearly compact, i.e., the further averaging\begin{equation} \label{e:trueblue}\on{Av}_{I,*}: \Dmod(I_{\chP_1} \bs I_{\chP_1}\dot{w}I_{\chP_2} / I\mon) \rightarrow \Dmod(I_{\chP_1} \bs I_{\chP_1}\dot{w}I_{\chP_2}/I)\end{equation}is compact. We first claim that this second averaging is unnecessary, i.e., that an object $\zeta$ of the left-hand side of \eqref{e:trueblue} is compact if and only if $\on{Av}_{I, *}(\zeta)$ is compact. Indeed, this follows from the fact that \eqref{e:trueblue} is filtered by finitely many strata of the form 
$$\on{Av}_{\chH, *}: \Vect \simeq \Dmod( \pt ) \rightarrow \Dmod( \pt / \chH),$$and an object of $\Vect$ is compact if and only if its $\chH$-average is. 
Therefore, we must show that an object $\xi$ of the left-hand side of \eqref{e:deadred} is compact if and only if $\on{Av}_{!*}(\xi)$ is compact.

Let us denote the simple object on the minimal stratum of the left-hand side of \eqref{e:deadred} by 
$$\delta_{\dot{w}} \in \Dmod(I_{\chP_1} \bs I_{\chP_1}\dot{w}I_{\chP_2} / I).$$
If we write $\chL$ for the Levi factor of $\chP_2$, we claim that the assumption that $\dot{w}$ is relevant implies (and is in fact equivalent to the fact that) convolution with $\delta_{\dot{w}}$ defines a $\Dmod( \BL\mon\!\bs \chL / \BL\mon)$-equivariant equivalence 
\begin{equation} \label{e:fartherfrom}  \Dmod(I_{\chP_1} \bs I_{\chP_1} \dot{w} I_{\chP_2} / I\mon)  \overset{\hspace{.7mm}\sim}\leftarrow \Dmod( \BL \bs \chL / \BL\mon) : \delta_{\dot{w}} \overset{\BL} \star - .\end{equation}
To see this, note that, by the relevance of $\dot{w}$, the strata
$$I_{\chP_1} \dot{w}y I, \quad \quad y \in W_{\chL}$$
are all distinct, and for any $y \in W_{\chL}$ convolution defines an isomorphism 
$$I_{\chP_1} \dot{w}y I \overset{\hspace{.7mm}\sim}\leftarrow I_{\chP_1}\dot{w}I \overset I \times IyI.$$It follows that \eqref{e:fartherfrom} exchanges the (co)standard objects on either side, and identifies the homomorphisms from standard to costandard objects, and is hence an equivalence. 

It therefore remains to see that in the adjunction 
$$\on{Av}_{!*}^\psi: \Dmod( \BL \bs \chL / \BL\mon) \rightleftarrows \Dmod(\BL \bs \chL / N^-_{\chL}, \psi),$$an object $\xi$ in the right-hand side is compact if and only if $\on{Av}_{!*}(\xi)$ is compact, but this is clear.

Having handled the case of a single stratum, the statement of the proposition follows straightforwardly by considering Cousin filtrations. \end{proof}

\subsubsection{} Let us now turn to the case \eqref{e:aRouNdYoU} of monodromic sheaves. For ease of notation, let us denote the quotient $I_{\chP_1}/I'_{\chP_1}$ by $Q$. 

\begin{prop} \label{p:sInGsTraTMoNo}The category associated to a single stratum
$$\Dmod( (I_{\chP_1}, I'_{\chP_1})\mon\! \bs I_{\chP_1} w I_{\chP_2} / _{\chP_2}\mathring{I}, \psi)$$vanishes if $w$ is irrelevant. Otherwise, if we write $\dot{w}$ for the minimal length element of $W_1wW_2$, $!$-restricting to $I_{\chP_1} \dot{w}$  yields an equivalence 
$$\Dmod((I_{\chP_1}, I'_{\chP_1})\mon\!\bs I_{\chP_1} w I_{\chP_2} / _{\chP_2}\mathring{I}, \psi) \simeq \Dmod(Q\mon\!\bs Q), \quad \quad \text{for $w$ relevant}.$$\end{prop}

\begin{proof} This follows from Proposition \ref{p:sInGsTraT} by turning on the monodromy operators for $Q$, i.e., applying $$\Vect \underset{ \Dmod(\pt/Q)} \otimes -,$$cf. Equation \eqref{e:formula2}.
\end{proof}

For a relevant stratum $w \in W_1 \bs W_a / W_2$, let us denote the $!$- and $*$- extensions of the object corresponding to the cofree-monodromic sheaf in $\Dmod(Q\mon\!\bs Q)$, normalized to lie in cohomological degree zero, by 
$$\hj_{w, !}, \quad \hj_{w, *} \in \Dmod((I_{\chP_1}, I'_{\chP_1})\mon\!\bs \chG_F / _{\chP_2}\mathring{I}, \psi)^\heartsuit, $$
respectively. We will refer to these as cofree-monodromic standard and costandard objects, respectively, below. 

Let us say that an object of $\Dmod((I_{\chP_1}, I'_{\chP_1})\mon\!\bs \chG_F / _{\chP_2}\mathring{I}, \psi)$ is cofree-monodromic tilting if it admits a finite filtration by cofree-monodromic standard objects and a finite filtration by cofree-monodromic costandard objects. 

The classification of indecomposable cofree-monodromic tilting objects reads as follows.

\begin{prop}For each relevant stratum $w \in W_1 \bs W_a / W_2$, there is a unique up to (non-unique) isomorphism indecomposable cofree-monodromic tilting sheaf $T_w$ such that 
\begin{enumerate}
    \item it is supported on the closure of $I_{\chP_1}wI_{\chP_2} $, and 
    
    \item its restriction to the open stratum $I_{\chP_1}wI_{\chP_2}$ is an indecomposable injective object. 
\end{enumerate}
Moreover, any indecomposable cofree-monodromic tilting sheaf is isomorphic to exactly one of the $T_w$.     
\end{prop}

\begin{proof}As we spell out below, the argument of Proposition \ref{p:classtilts} applies {\em mutatis mutandis}.

Namely, one may again construct an object $T_w$ by downward induction on the strata in the closure of $I_{\chP_1}wI_{\chP_2}$. In the notation of the above argument, the fact that $\Hom(\delta_z, j_! T_U)$ is concentrated in cohomological degrees 0 and 1 follows in this case by noting the adjunction
$$\Hom_{\Dmod((I_{\chP_1}, I'_{\chP_1})\mon\!\bs \chG_F/ _{\chP_2}\mathring{I}, \psi)}(\delta_z, j_! T_U) \simeq \Hom_{\Dmod(I_{\chP_1} \bs \chG_F/ _{\chP_2}\mathring{I}, \psi)}(\delta_z, \on{Av}_{Q, *}(j_! T_U)),$$
and that the averaged object $\on{Av}_{Q, *}(j_! T_U))$ now admits standard and costandard filtrations, i.e., is a tilting object of $\Dmod(I_{\chP_1} \bs \chG_F/ _{\chP_2}\mathring{I}, \psi)$. Therefore the general observation about highest weight categories recorded near Equations \eqref{e:bingbong} and \eqref{e:bongbing} applies in the present setting.

The construction of $T_Z$ given the above Ext vanishing proceeds identically, and hence by induction one obtains $T_w$. The remaining assertions of the proposition again follow from the fact that any endomorphism of $T_w$ which is an isomorphism after restriction to the open stratum, is an isomorphism. This may be proven exactly as in the proof of Proposition \ref{p:classtilting}, after replacing $R^{W_1}$ in its argument with the monodromy operators for $Q$, cf. Section \ref{sss:intromonopers}.\end{proof}

As before, if we denote the inclusion of the additive category of cofree-monodromic tilting sheaves by 
$$_{\chP_1}\on{Tilt}_{\chP_2} \hookrightarrow \Dmod((I_{\chP_1}, I'_{\chP_1})\mon\!\bs \chG_F/ _{\chP_2}\mathring{I}, \psi)^\heartsuit$$this prolongs to a fully faithful embedding from its homotopy category of bounded complexes 
\begin{equation}K^b(_{\chP_1}\on{Tilt}_{\chP_2}) \hookrightarrow \Dmod((I_{\chP_1}, I'_{\chP_1})\mon\!\bs \chG_F/ _{\chP_2}\mathring{I}, \psi). \label{e:tiltywhilty}\end{equation}
This may equivalently be characterized as the pre-triangulated hull of the cofree-monodromic standard objects, tilting objects, or costandard objects.

\begin{prop}\label{c:post3} An object $\xi$ of $\Dmod((I_{\chP_1}, I'_{\chP_1})\mon\!\bs \chG_F/ _{\chP_2}\mathring{I}, \psi)$ is nearly compact if and only if it belongs to the essential image of \eqref{e:tiltywhilty}.
\end{prop}

\begin{proof}Recall that in the proof of \ref{p:bIgMisTake} for a single relevant stratum we saw an equivalence
$$\Dmod( I_{\chP_1} \bs I_{\chP_1} \dot{w} I_{\chP_2} / I\mon) \simeq \Dmod(\BL \bs \chL / \BL\mon). $$
In particular, recalling that $Q$ denotes the quotient $I_{\chP_1}/I'_{\chP_1}$, by Theorem \ref{t:avenh}, we may turn back on the monodromy to obtain \begin{align*} \Dmod( (I_{\chP_1}, I'_{\chP_1})\mon\!\bs I_{\chP_1} \dot{w} I_{\chP_2} / I\mon) & \simeq 
\Vect \underset{C^\bullet(BQ)\mod}\otimes \Dmod( I_{\chP_1} \bs I_{\chP_1} \dot{w} I_{\chP_2} / I\mon) \\ & \simeq \Vect \underset{C^\bullet(BQ)\mod}\otimes 
\Dmod(\BL \bs \chL / \BL\mon).\end{align*}
To identify the latter more explicitly, note that the map 
$$C^\bullet(BQ) \rightarrow HH^\bullet(\Dmod(B_L \bs \chL / \BL\mon)) = \End_{\cat{End}(\Dmod(B_L \bs \chL / \BL\mon))}(\on{id}),$$
as it comes from the left convolution, canonically lifts to a map to the Hecke equivariant Hochschild cohomology with respect to right convolution, namely 
$$C^\bullet(BQ) \rightarrow \End_{{\mathbf{End}}_{\Dmod(\BL\mon\!\bs \chL / \BL\mon)}(\Dmod(B_L \bs \chL / \BL\mon))}(\on{id}) \simeq \End_{\Dmod(\BL \bs \chL / \BL)}(\delta_e),$$
i.e., is completely determined by its action on the delta D-module on the minimal stratum $I_{\chP_1} \dot{w} I_{\chP_2}.$

Therefore, writing $\BL' := \BL \cap \on{Ad}_{\dot{w}^{-1}}(I_{\chP_1})$, we obtain an equivalence
$$\Dmod( I_{\chP_1} \bs I_{\chP_1} \dot{w} I_{\chP_2} / I\mon) \simeq \Dmod( (\BL, \BL')\mon\!\bs \chL / \BL\mon). $$
To proceed, we next claim that an object $\zeta$ of $\Dmod((\BL, \BL')\mon\!\bs \chL / \BL\mon)$ is nearly compact with respect to its right $\BL$-monodromicity if and only if it is nearly compact with respect to its left $(\BL, \BL')$-monodromicity. That is, with respect to  the  averaging maps
$$\begin{tikzcd}
& \arrow[dl, "{\on{Av}_{\BL\BL', *}}"'] \Dmod((\BL, \BL')\mon\!\bs \chL / \BL\mon) \arrow[dr, "{\on{Av}_{\BL, *}}"] & \\
\Dmod( \BL \bs \chL / \BL\mon) & & \Dmod((\BL, \BL')\mon\!\bs \chL / \BL), 
\end{tikzcd}$$
we have that $\on{Av}_{\BL/\BL', *}(\zeta)$ is compact if and only if $\on{Av}_{\BL, *}(\zeta)$ is compact. This assertion may checked stratum by stratum, i.e., after replacing $\chL$ with $\BL y \BL$, for $y \in W_{\chL}$. There, if we write $\chH' := \chH \cap \BL'$, the above identifies with the sequence 

$$\Vect \simeq \Dmod(\chH \bs \chH) \xleftarrow{\on{Av}_{\chH/\chH', *}} \Dmod(\chH' \bs \chH/ \chH\mon) \xrightarrow{\on{Av}_{\chH, *}} \Dmod(\chH' \bs \chH / \chH) \simeq \Dmod (\chH' \bs \pt),$$and hence the claim follows by noticing that $\on{Av}_{\chH, *}$ factors through $\on{Av}_{\chH/\chH'}$, and recalling that the averaging map $\Dmod(\pt) \rightarrow \Dmod(\chH' \bs \pt)$ reflects and preserves compactness. 

Having shown the claim, let us deduce the statement of the proposition. By a Cousin filtration argument it is enough to show that for a relevant stratum, an object $\xi$ of $\Dmod((I_{\chP_1}, I'_{\chP_1})\mon\!\bs I_{\chP_1}wI_{\chP_2} / _{\chP_2}\mathring{I}, \psi)$ is nearly compact if and only it lies the pretriangulated hull of the cofree-monodromic object, cf. Proposition \ref{p:sInGsTraTMoNo}. 

On the one hand, the pretriangulated hull of the cofree-monodromic object is the preimage of compact objects under the averaging functor 
\begin{equation} \label{e:circuit1}\Dmod((I_{\chP_1}, I'_{\chP_1})\mon\!\bs I_{\chP_1}wI_{\chP_2} / _{\chP_2}\mathring{I}, \psi) \xrightarrow{\on{Av}_{I_{\chP_1}/I'_{\chP_1}}} \Dmod(I_{\chP_1} \bs I_{\chP_1}wI_{\chP_2} / _{\chP_2}\mathring{I}, \psi).\end{equation}
On the other hand, the nearly compact objects are the preimage of the compact objects under the composition of averaging functors
\begin{align*}\Dmod((I_{\chP_1}, I'_{\chP_1})\mon\!\bs I_{\chP_1}wI_{\chP_2} / _{\chP_2}\mathring{I}, \psi) &\xrightarrow{\Av_{I\mon, *}} \Dmod((I_{\chP_1}, I'_{\chP_1})\mon\!\bs I_{\chP_1}wI_{\chP_2} / I\mon)\\ &\xrightarrow{\Av_{I, *}} \Dmod((I_{\chP_1}, I'_{\chP_1})\mon\!\bs I_{\chP_1}wI_{\chP_2} / I).\end{align*}
By the claim shown above, they are equivalently the preimage of the compact objects under the composition of averaging functors 
\begin{align*}\Dmod((I_{\chP_1}, I'_{\chP_1})\mon\!\bs I_{\chP_1}wI_{\chP_2} / _{\chP_2}\mathring{I}, \psi) &\xrightarrow{\Av_{I\mon, *}} \Dmod((I_{\chP_1}, I'_{\chP_1})\mon\!\bs I_{\chP_1}wI_{\chP_2} / I\mon) \\ &\xrightarrow{\Av_{I_{\chP_1}/I'_{\chP_1}, *}} \Dmod(I_{\chP_1} \bs I_{\chP_1}wI_{\chP_2} / I\mon).\end{align*}
However, the latter two functors commute, i.e., equivalently we are considering the preimage of the compact objects under the composition 
\begin{align*}\Dmod((I_{\chP_1}, I'_{\chP_1})\mon\!\bs I_{\chP_1}wI_{\chP_2} / _{\chP_2}\mathring{I}, \psi) &\xrightarrow{\Av_{I_{\chP_1}/I'_{\chP_1}, *}} \Dmod(I_{\chP_1} \bs I_{\chP_1}wI_{\chP_2} / _{\chP_2}\mathring{I}, \psi) \\ &\xrightarrow{\Av_{I\mon,  *}} \Dmod(I_{\chP_1} \bs I_{\chP_1}wI_{\chP_2} / I\mon).\end{align*}

By comparing this with \eqref{e:circuit1}, the proposition now follows from the fact that the functor 
$$ \Dmod(I_{\chP_1} \bs I_{\chP_1}wI_{\chP_2} / _{\chP_2}\mathring{I}, \psi) \xrightarrow{\on{Av}_{I\mon, *}} \Dmod(I_{\chP_1} \bs I_{\chP_1}wI_{\chP_2} / I\mon) $$
preserves and reflects compact objects. Indeed, that it preserves compactness follows from the fact this is also a left adjoint, cf. Lemma \ref{l:avavpsi}. That it reflects compactness follows from the fact that it may be identified with the composition 
$$\Vect \simeq \Dmod(\BL \bs \chL / N^-_{\chL}, \psi) \xrightarrow{\on{Av}_{\BL\mon, *}} \Dmod( \BL \bs \chL / \BL\mon),$$
and in particular is non-zero. 
\end{proof}

\subsection{Convolution of Iwahori monodromic categories}\label{autom convolve sec}

\sss
We would like to reduce verifying the main equivalences of the paper to the case where, on the automorphic side, one of the equivariances imposed is Iwahori monodromicity. The goal of this subsection is to supply the details of this reduction on the automorphic side.

\subsubsection{} The main technical result here is the following.

\begin{prop} The $\Dmod(\chG_F \times \chG_F)$-equivariant convolution map 
$$\Dmod(\chG_F / I\mon) \underset{\Dmod(I\mon\!\bs \chG_F/I\mon)}\otimes \Dmod(I\mon\!\bs \chG_F) \rightarrow \Dmod(\chG_F)$$
is a fully faithful embedding which admits a (continuous) right adjoint. In particular, for any $\Dmod(\chG_F)$-module $\sC$, the canonical map $$\Dmod(\chG_F/I\mon) \underset{\Dmod(I \mon\!\bs \chG_F / I\mon)} \otimes \sC^{I\mon} \rightarrow \sC,$$is a fully faithful embedding of $\Dmod(\chG_F)$-modules, with essential image the submodule of $\sC$ generated by the Iwahori equivariant objects. 
\end{prop}

\begin{proof}
This may be proven exactly as in Theorem  3.1.5 of \cite{campbelldhillon}, where one replaces the strict invariants therein with monodromic invariants. 
\end{proof}

\begin{cor}For two $\Dmod(\chG_F)$-modules $\sC$ and $\sD$, where $\sC$ is dualizable and generated by its Iwahori equivariant objects, one has a canonical equivalence $${\mathbf{Fun}}_{\Dmod(\chG_F)\on{-mod}}(\sC, \sD) \simeq (\sC^{\vee})^{I\mon} \underset{\Dmod(I\mon\!\bs \chG_F / I\mon)} \otimes \sD^{I\mon}. $$
\label{autom convolution}
\end{cor}



\begin{proof} By the previous Proposition, for $\sC$ to be generated by its Iwahori equivariant objects is equivalent to the natural map $$\Dmod(\chG_F/I\mon) \underset{\Dmod(I\mon\!\bs\chG_F/I\mon)} \otimes \sC^{I\mon} \rightarrow \sC $$being an equivalence.

In particular, one obtains by adjunction \begin{align*}{\mathbf{Fun}}_{\Dmod(\chG_F)\mmod}(\sC, \sD) 
& \simeq {\mathbf{Fun}}_{\Dmod(I\mon\!\bs\chG_F/I\mon)}(\sC^{I\mon}, \sD^{I\mon}).
\intertext{Using the semi-rigidity of $\Dmod(I\mon\!\bs\chG_F/I\mon)$, we way further rewrite this as} 
&\simeq (\sC^{I\mon})^\vee \underset{\Dmod(I\mon\!\bs\chG_F/I\mon)} \otimes \sD^{I\mon},\intertext{which by the $\DGCat$-linearity of $I$-monodromic invariants may be written as} & \simeq (\sC^\vee)^{I\mon} \underset{\Dmod(I\mon\!\bs\chG_F/I\mon)} \otimes \sD^{I\mon}, \end{align*}
as desired. \end{proof}

\begin{exmp} If as before we pick two parabolics $\chP, \chQ$ of $\chG$, and write $I_{\chP}$ and $I_{\chQ}$ for the corresponding parahoric subgroups of $\chG_F$, we deduce that convolution yields a canonical equivalence $$\Dmod(I_{\chP} \backslash \chG_F/ I\mon) \underset{\Dmod(I\mon\!\bs\chG_F/I\mon)} \otimes \Dmod(I\mon\!\bs \chG_F/I_{\chQ}) \simeq \Dmod( I_{\chP}\backslash \chG_F/ I_{\chQ}), $$
by taking $\sC \simeq \Dmod(\chG_F/I_{\chP})$ and $\sD \simeq \Dmod(\chG_F/I_{\chQ})$ in the previous corollary. Similarly statements apply to the other equivariance conditions appearing in Bezrukavnikov's conjectures. 
\end{exmp}

\begin{rmk}The results of this subsection hold after replacing Iwahori monodromicity  with Iwahori equivariance equivariance throughout. 
\end{rmk}


\section{Langlands dual equivalences of affine Hecke categories}\label{s:mainresults}

In this section we will prove the main results of our paper, i.e. the Langlands equivalences between automorphic and spectral partial affine Hecke categories.  We first introduce some notation (recall Definition \ref{steinberg notation}):
$$\cZ/G := \cZ_{\mf{n}}/G = \wt{\cN}/G \times_{\mf{g}/G} \wt{\cN}/G, \;\;\;\;\;\;\;\; \cZ'/G := \cZ_{\mf{b}, \mf{n}}/G = \wt{\mf{g}}/G \utimes{\mf{g}/G} \wt{\cN}/G, \;\;\;\;\;\;\;\; \cZ_{\mf{b}}/G  = \wt{\mf{g}}/G \utimes{\mf{g}/G} \wt{\mf{g}}/G,$$
$$\wh{\mf{h}} := \wh{\mf{h}}_{\{0\}}, \;\;\;\;\;\;\;\;\;\; \wh{\mf{h}/\!/W} := \wh{\mf{h}/\!/W}_{\{0\}}, \;\;\;\;\;\;\;\;\;\;\wh{\cZ}/G := \cZ_{\mf{b}}/G \utimes{\mf{h}/\!/W} \wh{\mf{h}/\!/W} = (\wh{\cZ_{\mf{b}}})_{\cN}/G.$$
All fiber products are derived, but we emphasize that the fiber products defining $\cZ'$ and $\cZ_{\mf{b}}$ may be taken classically.

Recall that $I = I_{\chB} \subset \chG_F$ is the Iwahori subgroup corresponding to $\chB \subset \chG$, and $\mathring{I}$ its pro-unipotent radical.

\subsubsection{}\label{conventions}
Our starting point is the following theorem of Bezrukavnikov in \cite{roma hecke}, which we lift to an $\infty$-categorical statement and re-state in our notation and conventions.  We define, using the discussion in Section \ref{Cofree tilting sec}, $\wh{\Xi}$ to be the cofree-monodromic tilting sheaf attached to the big cell $\chB \bs \chG / \chB \subset {I}_{\chB} \bs \chG_F / {I}_{\chB}$, which is Grothendieck dual to the big tilting sheaf in \emph{op. cit.}. 
\begin{thm}[Bezrukavnikov]\label{roman theorem}
There is an exact monoidal $\cO(\wh{\mf{h}} \times \wh{\mf{h}})$-linear equivalence of $\cO(\wh{\mf{h}} \times \wh{\mf{h}})$-linear monoidal stable $\infty$-categories
$$\wh{\Phi}: \begin{tikzcd}
\Dmod(\mathring{I}\mon\!\bs \chG_F / \mathring{I}\mon) \arrow[r, "\simeq"] & \IndCoh_{\cN}(\cZ_{\mf{b}}/G) = \IndCoh(\wh{\cZ}/G)
\end{tikzcd}$$
where $\cO(\wh{\mf{h}})$ acts by (left and right) log monodromy on $\Dmod(\mathring{I}\mon\!\bs \chG_F / \mathring{I}\mon)$ and via pullback along the map $\wh{\wt{\mf{g}}}_{\cN} \rightarrow \wh{\mf{h}}$ on $\IndCoh(\wh{\cZ}_{\cN}/G)$, such that $\wh{\Phi}(\wh{\Xi}) \simeq \omega_{\wh{\cZ}/G}$.  
Furthermore, we have the following.
\begin{enumerate}[(a)]
\item By applying $- \otimes_{\Mod(\cO(\wh{\mf{h}} \times \wh{\mf{h}}))} \Vect_k$ to $\wh{\Phi}$, we obtain an exact $(\Sym \mf{h}[-2])^{\otimes 2}$-linear monoidal equivalence
$$\Phi: \begin{tikzcd}
   \Dmod(I \bs \chG_F / I) \arrow[r, "\simeq"] & \IndCoh_{\cN[1]}(\cZ/G).
\end{tikzcd}$$
\item By applying $- \otimes_{\Mod(\cO(\wh{\mf{h}}))} \Vect_k$ on the right to $\wh{\Phi}$, we obtain an $\cO(\wh{\mf{h}}) \otimes \Sym \mf{h}[-2]$-linear equivalence of bimodule categories
$$\Phi': \begin{tikzcd}
\Dmod(\mathring{I}\bs \chG_F / I) \arrow[r, "\simeq"] & \IndCoh(\cZ'/G),
\end{tikzcd}$$
i.e. equivariant on the left under the equivalence $\wh{\Phi}$ and on the right under $\Phi$. 
\item Consider the maps
$$\begin{tikzcd}
\mathring{I}\bs G_F / \mathring{I} \arrow[r, "\pi"] & \mathring{I} \bs G_F / I \arrow[r, "p"] & I \bs G_F / I & \cZ/G \arrow[r, "i"] & \cZ'/G \arrow[r, "\iota"] & \wh{\cZ}/G.
\end{tikzcd}$$
The corresponding functors restrict to the following small subcategories, and are compatible with the equivalences above:
$$\begin{tikzcd}[column sep=3.7ex]
\Dmod_{n.c.}(\mathring{I}\mon\!\bs \chG_F / \mathring{I}\mon) \arrow[r, "\wh{\Phi}", "\simeq"'] \arrow[d, "\pi_*"'] & \hCoh_{\cN}(\cZ_{\mf{b}}/G) = \hCoh(\wh{\cZ}/G) \arrow[d, "\iota^!"]  & {\Dmod}_c(I \bs \chG_F / I) \arrow[r, "\Phi", "\simeq"'] \arrow[d, "p^*"'] & \Coh(\cZ/G) \arrow[d, "i_*"] \\
\Dmod_c(\mathring{I} \bs \chG_F / I) \arrow[r, "\Phi'", "\simeq"'] \arrow[d, "\pi^*"'] & \Coh(\cZ'/G) \arrow[d, "\iota_*"] & \Dmod_s(\mathring{I} \bs \chG_F / I) \arrow[r, "\Phi'", "\simeq"'] \arrow[d, "p_*"'] & \Coh(\cZ'/G) \arrow[d, "i^!"] \\
\Dmod_c(\mathring{I}\mon\!\bs \chG_F / \mathring{I}\mon) \arrow[r, "\wh{\Phi}", "\simeq"'] & \Coh_{\cN}(\cZ_{\mf{b}}/G) = \Coh(\wh{\cZ}/G) & \Dmod_s(I \bs \chG_F / I) \arrow[r, "\Phi", "\simeq"'] & \Coh_{\cN[1]}(\cZ/G).
\end{tikzcd}$$
\end{enumerate}
\end{thm}
\begin{proof}
We introduce the notation 
$$\mf{h}^\wedge = \Spec \cO(\wh{\mf{h}}), \;\;\;\;\;\;\;\;\;\; \cZ^\wedge/G := \wh{\cZ}/G \utimes{\wh{\mf{h}/\!/W}} \Spec \cO(\wh{\mf{h}/\!/W})$$
and let $\Dmod^{pro}_{n.c.}$ denote the dg nerve of the dg derived category of pro-monodromic sheaves defined in Section 3.1 of \cite{roma hecke}. We first lift the main theorem of \cite{roma hecke}, which provides an exact $\cO(\wh{\mf{h}}^2)$-linear monoidal functor on homotopy categories
$$\wh{\Phi}_h: h\!\Dmod^{pro}_{n.c.}(\mathring{I}\mon\!\bs \chG_F / \mathring{I}\mon) \longrightarrow h\!\Coh(\cZ^\wedge/G)$$
to the $\infty$-categorical setting.  Let $\cat{T} \subset \Dmod^{pro}_{n.c.}(\mathring{I}\mon\!\bs \chG_F / \mathring{I}\mon)$ be the full (not stable) subcategory of sums of  free-monodromic tilting objects; see Section 3.4 of \cite{roma hecke} for definitions, and $\cat{T}'$ the full subcategory of corresponding objects on the spectral side under $\wh{\Phi}_h$.  It is a monoidally closed tilting subcategory in the sense of Definition \ref{def tilting subcat} by Proposition 7 of \cite{roma hecke}.  Furthermore, the full subcategory of $\QCoh(\mf{h}^\wedge)$ of finite rank free modules is a tilting subcategory acting on $\cat{T}, \cat{T}'$.  Then, by Proposition \ref{prop additive nice}, the functor $\Ch(h\cat{T}) \rightarrow {\Dmod}_{n.c.}^{pro}(\mathring{I}\mon\!\bs \chG_F / \mathring{I}\mon)$ is an exact monoidal equivalence of $\cO(\wh{\mf{h}}^2)$-linear stable $\infty$-categories by Proposition \ref{prop additive nice}, and likewise for the functor $\Ch(h\cat{T}') \rightarrow \Coh(\cZ^\wedge/G)$.  Thus we have a commuting diagram of monoidal triangulated categories
$$\begin{tikzcd}[column sep=large]
h\!\Ind(\Dmod^{pro}_{n.c.}(\mathring{I}\mon\!\bs \chG_F / \mathring{I}\mon)) \arrow[r, "\wh{\Phi}_h", "\simeq"'] &  h\!\IndCoh(\cZ^\wedge/G) \\
h\!\Ch(h\cat{T}) \arrow[r, "h\!\Ch(\wh{\Phi}_h|_{\cat{T}})"] \arrow[u, "\simeq"] & h\!\Ch(h\cat{T}'). \arrow[u, "\simeq"]
\end{tikzcd}$$
Thus, the functor $\Ch(\wh{\Phi}_h|_{\cat{T}}): \Ch(h\cat{T}) \rightarrow \Ch(h\cat{T}')$ is a monoidal $\cO(\wh{\mf{h}}^2)$-linear equivalence of $\infty$-categories, and we define the monoidal $\cO(\wh{\mf{h}}^2)$-linear equivalence $\wh{\Phi}$ of $\infty$-categories by inverting the equivalences in the diagram
$$\begin{tikzcd}[column sep=large]
\Ind(\Dmod^{pro}_{n.c.}(\mathring{I}\mon\!\bs \chG_F / \mathring{I}\mon)) \arrow[r, "\wh{\Phi}", "\simeq"', dotted] &  \IndCoh(\cZ^\wedge/G) \\
\Ch(h\cat{T}) \arrow[r, "\Ch(\wh{\Phi}_h|_{\cat{T}})", "\simeq"'] \arrow[u, "\simeq"] & \Ch(h\cat{T}'). \arrow[u, "\simeq"]
\end{tikzcd}$$
By construction (see Section 6 of \cite{roma hecke}), we have $\wh{\Phi}(\wh{\Xi}^{pro}) \simeq \cO_{\cZ^\wedge/G}$.  To deduce the statement in the theorem, we first pass through the Grothendieck existence theorem and then intertwine with Verdier and Grothendieck-Serre duality.  For the first step, we apply Theorem \ref{groth exist}, taking $S = \mf{h}^\wedge/\!/W$, to deduce a diagram of monoidal (i.e. by base change on the diagrams defining convolution) equivalences
\begin{equation}\label{steinberg groth exist}
\begin{tikzcd}[column sep=large]
\hCohst(\wh{\cZ}/G)^{\opp} \arrow[r, "\wh{j}_*^{\opp} \; \simeq"', shift right] \arrow[d, "\simeq", "\bD_{\wh{\cZ}/G}"'] & \arrow[l, "\wh{j}^{*,\opp} \; \simeq"', shift right] \Coh(\cZ^\wedge/G)^{\opp} \arrow[d, "\simeq", "\bD_{\wh{\cZ}/G}"'] & \Coh(\cZ_{\mf{b}}/G)^{\opp} \arrow[l, "j^{*,\opp}"'] \arrow[d, "\simeq", "\bD_{\cZ_{\mf{b}}/G}"'] \\
\hCohsh(\wh{\cZ}/G) \arrow[r,  "\wh{j}_* \; \simeq"', shift right] & \Coh(\cZ^\wedge/G) \arrow[l, "\wh{j}^{!} \; \simeq"', shift right] & \Coh(\cZ_{\mf{b}}/G). \arrow[l, "j^{!}"']
\end{tikzcd}
\end{equation}
The verification that the conditions in Theorem \ref{groth exist} are satisfied was proven in Corollary \ref{completed spectral coalgebra}.  That $j^!$ sends $\Coh(\cZ_{\mf{b}/G})$ to $\hCoh(\wh{\cZ}_{\cN})$ by Proposition \ref{wcoh prop}(b).  Furthermore, the equivalences in (\ref{steinberg groth exist}) are compatible with these identifications by the discussion in Section \ref{sec completed spectral coalgebra}.

Together, this provides an $\infty$-categorical $\cO(\wh{\mf{h}}^2)$-linear monoidal equivalence
$$\begin{tikzcd}[column sep=large]
\Ind(\Dmod^{pro}_{n.c.}(\mathring{I}\mon\!\bs \chG_F / \mathring{I}\mon)) \arrow[r, "\wh{\Phi}", "\simeq"'] &  \Ind(\hCohst(\wh{\cZ}/G))
\end{tikzcd}$$
such that $\wh{\Phi}(\wh{\Xi}^{pro}) \simeq \cO_{\wh{\cZ}/G}$.  Twisting by Verdier and Grothendieck-Serre duality, we obtain the equivalence
\begin{equation}\label{warbler}
\begin{tikzcd}[column sep=large]
\Ind(\Dmod_{n.c.}(\mathring{I}\mon\!\bs \chG_F / \mathring{I}\mon)) \arrow[r, "\wh{\Phi}", "\simeq"'] &  \Ind(\hCoh(\wh{\cZ}/G))
\end{tikzcd}
\end{equation}
where $\wh{\Phi}(\wh{\Xi}) \simeq \omega_{\wh{\cZ}/G}$.

To deduce the main statements (i.e. to pass from the above renormalized form to the standard form), we first deduce the equivalence in (b).  Specializing at $0 \in \mf{h}$ on the right, gives the strict category on the automorphic side by Proposition \ref{t:avenh}.  On the spectral side, we have
\begin{align*}
    \IndCoh(\cZ'/G) & \simeq \IndCoh(\cZ^\wedge/G) \otimes_{\QCoh(\cZ^\wedge/G)} \QCoh(\cZ'/G) & \text{\cite[Cor. 7.4.10]{AG}} \\
    & \simeq \IndCoh(\cZ^\wedge/G) \otimes_{\QCoh(\cZ^\wedge/G)} \QCoh(\cZ^\wedge/G) \otimes_{\QCoh(\mf{h})} \QCoh(\{0\}) & \text{\cite[Thm. 4.7]{BFN}} \\
    & \simeq \IndCoh(\cZ^\wedge/G) \otimes_{\QCoh(\mf{h})} \QCoh(\{0\}) & 
\end{align*}
Since $\cZ^\wedge$ is not finite type, the application of the above two results requires some justification.  For the proof of Corollary 7.4.10 of \cite{AG}, we first note that the $*$-restriction is fully faithful by Proposition 4.4.2 of \cite{indcoh} which only assumes that the scheme is locally Noetherian.  We then conclude that the $!$-restriction is fully faithful since it is a twist of the $*$-restriction by the relative dualizing complex, which is a line bundle since $\{0\} \hookrightarrow \mf{h}^\wedge$ is quasi-smooth.  For essential surjectivity, as in Proposition 7.4.3. in \emph{op. cit.} we need only show that the pushforward functor is conservative, which is clear.  For the application of Theorem 4.7 of \cite{BFN}, we need to show that the stacks involved are perfect.  Every affine stack is perfect, so the only question is whether $\cZ^\wedge/G$ is perfect.  By Proposition 3.21 of \emph{op. cit.} it suffices to show that $\mf{g}^\wedge/G$ is perfect, where $\mf{g}^\wedge = \Spec \lim_i \cO(G)/I_{\cN}^i$ and $I_{\cN}$ is the ideal sheaf cutting out the nilpotent cone.  Since $\mf{g}^\wedge$ is affine, so this follows by \emph{loc. cit.} from the fact that $BG$ is perfect.

Putting them together we have compatible equivalences
$$\begin{tikzcd}[column sep=large]
\Ind(\Dmod_{n.c.}(\mathring{I}\mon\!\bs \chG_F / \mathring{I}\mon)) \arrow[r, "\wh{\Phi}", "\simeq"'] \arrow[d, "\pi_*"] &  \Ind(\hCoh(\wh{\cZ}/G)) \arrow[d, "\iota^!"] \\
\Dmod(\mathring{I}\mon\!\bs \chG_F /I)  \arrow[r, "\Phi'", "\simeq"']  &  \IndCoh({\cZ}'/G)
\end{tikzcd}$$
By the discussion in Sections \ref{renormalize autom} and \ref{renormalize indcoh}, the full small subcategories $\Dmod_c(\mathring{I}\mon\!\bs \chG_F / \mathring{I}\mon)$ and $\Coh_{\cN}(\cZ_{\mf{b}}/G) \simeq \Coh(\wh{\cZ}/G)$ are generated by the essential images of the left adjoint functors $\pi^*$ and $\iota_*$.  Thus, these subcategories can be matched in (\ref{warbler}), and we deduce the equivalence from the theorem statement, as well as the equivalence $\Phi'$ in (b) and the left-hand side of (c).  By construction, $\Phi'$ is equivariant on the left for the equivalent monodromic affine Hecke categories under $\wh{\Phi}$; we will address equivariance on the right later.

To deduce (a), we apply the functor $- \otimes_{\cO(\wh{\mf{h}})} k$ on the left to deduce a diagram
$$\begin{tikzcd}[column sep=large]
\Dmod(\mathring{I}\mon\!\bs \chG_F /I)  \arrow[r, "\Phi'", "\simeq"']  \arrow[d] &  \IndCoh({\cZ}'/G) \arrow[d]\\
\Dmod(\mathring{I}\mon\!\bs \chG_F /I) \otimes_{\cO(\wh{\mf{h}})} k  \arrow[r, "\Phi' \otimes_{\cO(\wh{\mf{h}})} k", "\simeq"'] \arrow[d]  &  \IndCoh({\cZ}'/G) \otimes_{\cO(\wh{\mf{h}})} k\arrow[d] \\
\Dmod(I \bs \chG_F / I) \arrow[r, dotted, "\Phi", "\simeq"'] & \IndCoh_{\cN[1]}(\cZ/G).
\end{tikzcd}$$
To deduce the equivalence $\Phi$ from $\Phi'$, we need to show that the bottom vertical arrows are equivalences, and that the compositions may be identified with $p_*$ and $i^!$ respectively.  On the automorphic side, we use Proposition \ref{t:avenh} again.  On the spectral side, we note that the map $i: \cZ \hookrightarrow \cZ'$ is quasi-smooth, and that
\begin{align*}
    \IndCoh_\Lambda(\cZ/G) & \simeq \IndCoh(\cZ'/G) \otimes_{\QCoh(\cZ'/G)} \QCoh(\cZ/G) & \text{\cite[Cor. 7.4.10]{AG}} \\
    & \simeq \IndCoh(\cZ'/G) \otimes_{\QCoh(\cZ'/G)} \QCoh(\cZ'/G) \otimes_{\QCoh(\mf{h})} \QCoh(\{0\}) & \text{\cite[Thm. 4.7]{BFN}} \\
    & \simeq \IndCoh(\cZ'/G) \otimes_{\QCoh(\mf{h})} \QCoh(\{0\}) & 
\end{align*}
where $\Lambda$ is the restriction of the total singular support condition on $\cZ'$ to $\cZ$.  A calculation analogous to that in Proposition \ref{reduction using monoidal} shows that this is exactly $\cN[1]$.

Furthermore, the equivalence $\Phi$ is automatically monoidal, since the monoidal structure for $\wh{\Phi}$ is compatible with the $\Mod(\cO(\wh{\mf{h}}))$-bimodule structure.  Equivariance for $\Phi'$ follows similarly.

To deduce the right-hand side of (c), we recall that $\Dmod_c(I \bs \chG_F / I)$ and $\Coh(\cZ/G)$ may be characterized as the full small subcategories of objects which are mapped to compact objects by the left adjoints $p^*$ and $i_*$ respectively, i.e. as discussed in Sections \ref{dmod placid equiv} and \ref{renormalize indcoh}.  

Finally, we note that in general, given a $\Perf(\mf{h})$-module category $\cat{C}$, the category $\cat{C} \otimes_{\Perf(\mf{h})} \Perf(k)$ is a module category for $\cat{End}_{\Perf(\mf{h})}(\Perf(k)) \simeq \Coh(\{0\} \times_{\mf{h}} \{0\})$.  Passing throgh Koszul duality, this gives rise to a $\Perf(\Sym \mf{h}[-2])$-module structure.  In particular, for any compactly generated $\QCoh(\mf{h})$-module category $\cat{C}$ whose module structure restricts to compact objects, the category $\cat{C} \otimes_{\QCoh(\mf{h})} \Vect_k$ automatically acquires a $\Mod(\Sym \mf{h}[-2])$-module structure.  This confirms the claims regarding this Koszul dual $\Sym \mf{h}[-2]$-module structure in the theorem statements.  This concludes the proof of the theorem.
\end{proof}

\subsubsection{}  We also recall the following variant of Theorem 6.1 of \cite{BL}, with monodromic Iwahori equivariance on the left and strict parahoric equivariance on the right.  Before doing so we review some relevant functorialities for parabolic subgroups $P \subset Q$.  
On the automorphic side, we take the pushforward and pullback functors for the proper map $\pi_P^Q: - \bs \chG_F/I_{\chP} \rihgtarrow - \bs \chG_F/I_{\check{Q}}$:
$$\begin{tikzcd}
\Dmod(- \bs \chG_F / I_{\check{Q}}) \arrow[r, "(\pi_P^Q)^*", shift left] & \Dmod(- \bs \chG_F / I_{\check{P}}). \arrow[l, "(\pi_P^Q)_*", shift left] 
\end{tikzcd}$$
When one of the parabolics is $B$, we take for simplicity $\pi_B^P = \pi^P$.  On the spectral side, note that when $P \subset Q$, we have the opposite inclusion $\mf{n}_P \supset \mf{n}_Q$, thus functoriality arises via the Lagrangian correspondence:
$$\begin{tikzcd}
- \times_{\mf{g}} \na{Q} & - \times_{\mf{g}} \left(G/P \times_{G/Q} \na{Q}\right) \arrow[l, "p"'] \arrow[r, "i"] & - \times_{\mf{g}} \na{P}
\end{tikzcd}$$
and we define functors via the integral transforms, morally the classical limit (up to a shift) of the adjoint pair of functors $(\pi^*, \pi_*)$ on left D-modules
$$\begin{tikzcd}[column sep=large]
\Coh(- \times_{\mf{g}} \na{Q}) \arrow[rr, "(\Lag_P^Q)^* := i_* p^*", shift left] & & \Coh(- \times_{\mf{g}} \na{P}). \arrow[ll, "(\Lag_P^Q)_* := p_* i^!", shift left] 
\end{tikzcd}$$
When one of the parabolics is $B$, we write $(\Lag^{P,*},\Lag^P_*)$ for simplicity.  We state the theorem of \emph{loc. cit.} using our conventions above.
\begin{thm}[Bezrukavnikov, Losev]\label{thm monodromic strict}
Let $P \supset B$ be a parabolic subgroup.  There are equivalences of $\infty$-categories
$$\begin{tikzcd}[column sep=large]
\Dmod(\mathring{I} \bs \chG_F / I) \arrow[r, "\simeq"', "\Phi'"] \arrow[d, shift right, "\pi^P_*"'] &  \IndCoh(\wt{\mf{g}}/G \times_{\mf{g}/G} \wt{\cN}/G) \arrow[d, shift right, "\Lag_*^P"'] \\ 
\Dmod(\mathring{I} \bs \chG_F / I_{\chP}) \arrow[r, "\simeq"', "\Phi_{\mathring{B}, P}"] \arrow[u, shift right, "\pi^{P,*}"'] &  \IndCoh(\wt{\mf{g}}/G \times_{\mf{g}/G} \wt{\cN}_P/G) \arrow[u, shift right, "\Lag^{P,*}"'] 
\end{tikzcd}$$
as module categories for the left actions of the categories $\Dmod(\mathring{I} \bs \chG_F / \mathring{I}) \simeq \IndCoh(\wh{\cZ}/G)$ and compatible with the functorialities discussed above.
\end{thm}
\begin{proof}
The statement may be lifted to an $\infty$-categorical statement with the same arguments using tilting subcategories as in Theorem \ref{roman theorem}, using Corollary 5.5.2 in \cite{BezYun} to deduce that the convolution action by tilting objects preserves tilting objects.  We relate our conventions to those in Theorem 6.1 of \cite{BL}.  The functors $\phi_1, \phi_2$ in \emph{loc. cit.} are related to ours by
$$\phi_1 = \Lag^*(-\rho, \rho - 2\rho_{P}) = \Lag^!(-\rho, \rho)[-\dim P/B], \;\;\;\;\;\;\;\;\;\; \phi_2 = \pi^*[\dim P/B] = \pi^![-\dim P/B]$$
where $\rho_P$ is half the sum of positive roots in $\mf{n}_P$.  By the discussion preceding Theorem 6.1 in \cite{BL}, $\phi_1$ is not equivariant for the right $*$-convolution action but for a $\rho$-twisted action; however, our functors are equivariant, and thus we take the non-twisted action and conjugate the equivalence by $\cO(\rho)$ instead. 
\end{proof}

\begin{rmk}
Morally, the proof of Theorem 6.1 in \cite{BL} is parallel to our approach to Theorem \ref{thm monodromic whittaker}.  Namely, Lemma 6.5 in \emph{op. cit.} identifies the corresponding monads with right convolution by certain objects in the affine Hecke category.  While these objects live in the heart, their iterated convolutions pick up tensor copies of $H^\bullet(P/B)$ and therefore do not.  Thus, identifying them as algebra objects in the $\infty$-categorical sense requires extra work.  One may observe that in the setting of mixed sheaves, the graded vector space $H^\bullet(P/B)$ is Tate; thus one expects that these convolutions should lie in the heart of a ``Koszul'' $t$-structure in the mixed setting.  This issue is overcome in \emph{loc. cit.} using different methods.
\end{rmk}

\begin{rmk}
Note that in the above, we could have alternatively taken the adjunction $(\pi_!, \pi^!)$ on the automorphic side; note that $\pi_! = \pi_*$ while $\pi^! = \pi^*[2\dim(Q/P)]$.  On the spectral side, we would then take the adjunction $(\Lag_! := p_* i^*, \Lag^! := i_* p^!)$, i.e. the classical limit up to shift of the functors $(\pi_!, \pi^!)$ on right D-modules. 
\end{rmk}

\subsubsection{}\label{functoriality whit} We now prove a Whittaker-monodromic version of the above theorem.  Let us briefly recall the various functorialities, fixing parabolics $P \subset Q$ on the spectral side.  Automorphically, functoriality arises via averaging functors (see Section \ref{p:adjtilt}), i.e.
$$\begin{tikzcd}
\Dmod(I\mon\!\bs \chG_F / \PN, \psi_P) \arrow[rr, shift left, "\on{Av}_{!*}^{\chQ, \chP, \psi}"]&  & \Dmod(I\mon\!\bs \chG_F / {}_{\chQ}\check{N}, \psi_Q).  \arrow[ll, "\on{Av}_{!*}^{\chP, \chQ, \psi}", shift left] 
\end{tikzcd}$$
For simplicity, when one of the parabolics is $B$ we denote
$$\Av_{!*}^{\chP,\chB,\psi} = \Av_{!*}^{\psi_P}, \;\;\;\;\;\;\;\;\;\;\Av_{!*}^{\chB,\chP, \psi} = \Av_{!*}.$$
On the spectral side these correspond to the pushforward and pullback by the proper map $\phi_P^Q: \wt{\mf{g}}_P \rightarrow \wt{\mf{g}}_Q$:
$$\begin{tikzcd}
\Coh(- \times_{\mf{g}} \wt{\mf{g}}_P) \arrow[rr, shift left, "(\phi_P^Q)_*"] & & \Coh(- \times_{\mf{g}} \wt{\mf{g}}_Q). \arrow[ll, shift left, "(\phi_P^Q)^!"]  
\end{tikzcd}$$
Again for simplicity, when one of the parabolics is $B$, we denote $\phi^P = \phi_B^P$.  For the following theorem, it may be useful to recall that the character $\psi$ on ${}_{\chB}\mathring{I} = \mathring{I}$ is the trivial character.
\begin{thm}\label{thm monodromic whittaker}
There are equivalences of $\infty$-categories
$$\begin{tikzcd}[column sep=large]
\Dmod(\mathring{I}\mon\!\bs \chG_F / {}_{\chB}\mathring{I}, \psi) \arrow[r, "\simeq"', "\wh{\Phi}"] \arrow[d, shift right, " \Av_{!*}^{\psi_P}"']  & \IndCoh_{\cN}(\wt{\mf{g}}/G \utimes{\mf{g}/G} \wt{\mf{g}}/G) \arrow[d, shift right, "\phi^{P}_*"'] \\
\Dmod(\mathring{I}\mon\!\bs \chG_F / {}_{\chP}\mathring{I}, \psi) \arrow[r, "\simeq"', "\Phi_{\mathring{B}, \psi_{P}}"] \arrow[u, shift right, "\Av_{!*}"'] &  \IndCoh_{\cN}(\wt{\mf{g}}/G \utimes{\mf{g}/G} \wt{\mf{g}}_P/G) \arrow[u, "\phi^{P,!}"', shift right]
\end{tikzcd}$$
compatible with the left actions of the categories $\Dmod(\mathring{I}\mon\!\bs \chG_F / \mathring{I}\mon) \simeq \IndCoh_{\cN}(\wh{\cZ}/G)$, as well as the $\cO(\wh{\mf{h}} \times \wh{\mf{h}})$-actions, and compatible with the functorialities in Section \ref{functoriality whit}.
\end{thm}
\begin{proof}
The top equivalence is by Theorem \ref{roman theorem}, and we study the monads given by the vertical adjoint functors.  By Barr--Beck--Lurie it suffices to identify these monads.  Recall that $\wh{\mf{h}}$ is the formal completion of $\mf{h}$ at $0$, and $\mf{h}^\wedge = \Spec \cO(\wh{\mf{h}})$.

To do so, we need the following categorical notion of flatness.  Let $A$ be a commutative algebra, and let $\cat{C}$ be an $A$-linear category equipped with a $t$-structure.  We say that an object $c \in \cat{C}$ is \emph{flat} if the action map $\mathrm{act}_c: A\dmod \rightarrow \cat{C}$ is $t$-exact.  For example, when $X$ lives over $\Spec A$, and $\cat{C} = \QCoh(X)$, this recovers the usual notion of a flat sheaf.  Moreover, the Verdier dual of any cofree monodromic sheaf, i.e. free monodromic sheaf, is flat over its ring of monodromy operators.

On the automorphic side, by Corollary \ref{cor banana peel} and Proposition \ref{autom algebra}, the monad is given by convolution with an algebra subobject $\wh{\Xi}_P \subset \wh{\Xi}$ in the category $\Perv(\chB\mon\! \bs \chG / \chB\mon)$ of $\chB$-monodromic perverse sheaves on the finite flag variety.   Moreover, Proposition \ref{autom algebra} identifies $\wh{\Xi}_{\chP}$ as the unique cofree monodromic tilting subobject whose support is exactly $\chP$.  Then, its Verdier dual $\wh{\Xi}^{pro}_{\chP}$ is the unique free pro-monodromic tilting quotient whose support is exactly $\chP$, or equivalently, whose generic fiber over $\mf{h}^\wedge$ matches the description in Proposition \ref{spectral whittaker characterization}.


On the spectral side, recall that $\cZ_P = \wt{\mf{g}} \times_{G \times^P \mf{p}} \wt{\mf{g}}$, $\cZ^\wedge_P$ its base-change to $\wh{\mf{h}}$, $\wh{\cZ}_P$ is its completion along the nilpotent cone, and $i: \cZ_P/G \ra \cZ/G$ is the closed immersion. 
By Corollary \ref{completed spectral coalgebra} the monad is given by convolution with the subalgebra object $i_* \omega_{\wh{\cZ}_P/G} \subset \omega_{\wh{\cZ}/G}$, or equivalently passing through Grothendieck existence and Grothendieck duality, the quotient oblect $\cO_{\cZ^\wedge/G} \twoheadrightarrow i_* \cO_{\cZ^\wedge_P/G}$, which is an object in the heart of the standard $t$-structure, and the unique flat quotient of $\cO_{\cZ^\wedge/G}$ whose restriction to the regular semisimple locus matches that on the automorphic side.

Now, the equivalence $\wh{\Phi}$, twisted by Verdier and Grothendieck duality, matches $\wh{\Xi}^{pro}$ with $i_* \cO_{\cZ^\wedge/G}$, and takes the abelian category $\Perv(\chB\mon\! \bs \chG / \chB\mon)$ to an abelian subcategory of $\Coh(\cZ^\wedge/G)^\heartsuit$.  Furthermore, $\wh{\Phi}$ is $O(\wh{\mf{h}})$-bilinear, and our argument above shows that the automorphic and spectral quotient objects above are characterized by the same conditions.  This completes the proof.
\end{proof}

\subsubsection{} We also have the following centrally monodromic variant, which has the following functorialities.  On the automorphic side, recall that $I'_{\chP}$ denotes the derived subgroup of the parahoric $I_{\chP}$.  We may consider the non-proper map $\pi: - \bs \chG_F / I'_{\chP} \rightarrow - \bs \chG_F / I_{\chP}$, yielding adjoint functors
$$\begin{tikzcd}
\Dmod(- \bs \chG_F / I_{\chP}) \arrow[r, "\pi^*", shift left] & \Dmod(- \bs \chG_F / I'_{\chP}). \arrow[l, "\pi_*", shift left] 
\end{tikzcd}$$
On the spectral side, the functoriality arises via the proper map $\kappa: - \times_{\mf{g}} \wt{\cN}_P \hookrightarrow - \times_{\mf{g}} \wt{\cN}'_P$.
$$\begin{tikzcd}
\Coh(- \times_{\mf{g}} \wt{\cN}_P) \arrow[rr, "\kappa_*", shift left] & & \Coh(- \times_{\mf{g}} \wt{\cN}_P'). \arrow[ll, "\kappa^!", shift left] 
\end{tikzcd}$$
We now deduce the theorem; it may be helpful to recall that $I'_{\chB} = \mathring{I}$ and $\wt{\cN}'_B \simeq \wt{\mf{g}}$.
\begin{thm}\label{thm monodromic central}
There are compatible equivalences of categories:
$$\begin{tikzcd}
\Dmod(\mathring{I}\mon\!\bs \chG_F / \mathring{I}\mon) \arrow[r, "\simeq"', "\wh{\Phi}"] \arrow[d, shift left]  & \IndCoh_{\cN}(\wt{\mf{g}}/G \utimes{\mf{g}/G} \wt{\cN}'_B/G) \arrow[d, shift left] \\
\Dmod(\mathring{I}\mon\!\bs \chG_F / (I_{\chP}, I'_{\chP})\mon) \arrow[r, "\simeq"', "\Phi_{\mathring{B}, P'}"] \arrow[u, shift left] &  \IndCoh_{\cN}(\wt{\mf{g}}/G \utimes{\mf{g}/G} \wt{\cN}'_P/G)  \arrow[u, shift left]
\end{tikzcd}$$
compatible with the left actions of the categories $\Dmod(\mathring{I}\mon\!\bs \chG_F / \mathring{I}\mon) \simeq \IndCoh(\wh{\cZ}/G)$, as well as the $\cO(\wh{\mf{h}} \times \wh{\mf{h}})$-actions.  Furthermore, the equivalences restrict to small categories putting $\Dmod_{n.c.}$ on the left and $\hCoh_{\cN}$ on the right. 
\end{thm}
\begin{proof}
Let $\fH$ denote both the bi-monodromic automorphic and spectral affine Hecke categories, identified in Theorem \ref{roman theorem}, and consider the $\fH$-linear Hochschild cohomology complexes of the categories appearing in Theorem \ref{thm monodromic strict}.  We will apply the formula \eqref{e:oblvenh}  in Proposition \ref{autom central} for the torus $\chH = I_{\chP}/I'_{\chP} \simeq \chP/[\chP, \chP]$, where we have
$$C^\bullet(BH) \hookrightarrow C^\bullet(BP) \simeq \End(\delta),$$
i.e. the algebra $\End(\delta)$ is formal, and $H^\bullet(BH)$ is realized as the subalgebra generated in degree 2.  Note that the right action of $\Dmod(H \bs H / H)$ evidently commutes with the left action of $\fH$, so the map factors through $HH^\bullet_{\fH}$, i.e. we have an automatically injective map
$$C^\bullet(BH) \hookrightarrow \End(\delta) \hookrightarrow HH^\bullet_{\fH}(\Dmod(\mathring{I}\mon\!\bs \chG_F / I_{\chP})).$$
We now match this with the spectral side statement in Proposition \ref{spectral central} to deduce the theorem.
\end{proof}

\begin{rmk}
    Since $\pi$ is not proper, thus $\pi_* \ne \pi_!$, and we could have alternatively considered the adjoint functors $(\pi_! = \pi_*[-\dim(Z(L_P))], \pi^! = \pi^*[2 \dim(Z(L_P))]).$  On the spectral side, we have $(\kappa^* = \kappa^![\dim(Z(L_P))], \kappa_*.)$
\end{rmk}

\subsubsection{}
We now deduce the equivalences in Conjecture 59 of \cite{roma hecke} from the above.
\begin{thm}\label{thm conjectures}
There are equivalences of categories between automorphic and spectral affine Hecke categories
$$\begin{tikzcd}[row sep=0ex, column sep=tiny]
\Dmod(I_{\chP} \bs \chG_F / I_{\chQ}) \arrow[r, "\Phi_{P,Q}","\simeq"'] & \IndCoh_{\cN[1]}(\wt{\cN}_P/G \utimes{\mf{g}/G} \wt{\cN}_Q/G) & \text{(strict-strict)}  \\
\Dmod((I_{\chP}, I'_{\chP})\mon\!\bs \chG_F / I_{\chQ}) \arrow[r, "\simeq"', "\Phi_{P', Q}"]& \IndCoh_{\cN[1]}(\wt{\cN}'_P/G \utimes{\mf{g}/G} \wt{\cN}_Q/G) & \text{(monodromic-strict)} \\
\Dmod(\PN, \psi_P \bs \chG_F / I_{\chQ}) \arrow[r, "\simeq"', "\Phi_{\psi_P, Q}"] & \IndCoh(\wt{\mf{g}}_P/G \utimes{\mf{g}/G} \wt{\cN}_Q/G)  &  \text{(Whittaker-strict)} \\
\Dmod((I_{\chP},I'_{\chP})\mon\!\bs \chG_F / (I_{\chQ}, I'_{\chQ})\mon) \arrow[r, "\simeq"', "\Phi_{P', Q'}"]& \IndCoh_{\cN \cap \cN[1]}(\wt{\cN}'_P/G \utimes{\mf{g}/G} \wt{\cN}'_Q/G) & \text{(monodromic-monodromic)} \\
\Dmod(\PN,\psi_P \bs \chG_F / (I_{\chQ}, I'_{\chQ})\mon) \arrow[r, "\simeq"', "\Phi_{\psi_P, Q'}"]& \IndCoh_{\cN}(\wt{\mf{g}}_P/G \utimes{\mf{g}/G} \wt{\cN}'_Q/G) & \text{(Whittaker-monodromic)} \\
\Dmod(\PN, \psi_P, \St \bs \chG_F / \QN, \psi_Q, \St) \arrow[r, "\simeq"', "\Phi_{\psi_P, \psi_Q}"] & \IndCoh_{\cN}(\wt{\mf{g}}_P/G \utimes{\mf{g}/G} \wt{\mf{g}}_Q/G) & \text{(bi-Whittaker)}
 \end{tikzcd}$$
 compatible with monoidal and module structures, and the functorialities discussed above.

All spectral side categories all have the classical support condition $\cN$ and the singular support condition $\cN[1]$; we omit them where they are automatically satisfied.  Furthermore, the following small subcategories are matched:
$$\Dmod_s \leftrightarrow \Coh_{\cN \cap \cN[1]}, \;\;\;\;\;\;\;\;\;\; \Dmod_c \leftrightarrow \Coh_{\cN},$$
$$\Dmod_{s, n.c.} \leftrightarrow \wh{\Coh}_{\cN \cap \cN[1]}, \;\;\;\;\;\;\;\;\;\; \Dmod_{c, n.c.} \leftrightarrow \hCoh_{\cN}.$$
\end{thm}
\begin{proof}
By tensoring the (compatible) equivalences from Theorems \ref{thm monodromic strict}, \ref{thm monodromic whittaker}, and \ref{thm monodromic central} over the equivalence in Theorem \ref{roman theorem}, we obtain (for example):
\begin{align*}
\Dmod(\mathring{I}, \psi_P \bs \chG_F / \mathring{I}) \tens{\Dmod(\mathring{I} \bs \chG_F / \mathring{I})} \Dmod(\mathring{I} \bs \chG_F / I_{\chQ}) & \overset{\simeq}{\longrightarrow}  \\
 & \hspace{-7ex}\IndCoh_{\cN}(\wt{\mf{g}}_P/G \times_{\mf{g}/G} \wt{\mf{g}}/G) \tens{\IndCoh(\wt{\mf{g}}/G \times_{\mf{g}/G} \wt{\mf{g}}/G)} \IndCoh(\wt{\mf{g}}/G \times_{\mf{g}/G} \wt{\cN}_Q/G)
\end{align*}
from which we deduce all the equivalences in the theorem, along with functoriality, by applying Theorems \ref{autom convolution} and \ref{reduction using monoidal} (and Lemma \ref{auto nilpotent}); that is, on the spectral side we always see a nilpotent support and nilpotent singular support condition appearing.  
The equivalences are compatible with the monoidal structure by construction, e.g. in the above, $\Dmod(\mr{I}\bs \chG_F / I_{\chQ})$ is acted on the right by $\cat{End}_{\Dmod(\mr{I} \bs \chG_F / \mr{I})}(\Dmod(\mr{I}\bs \chG_F / I_{\chQ}))$.  Likewise, $\IndCoh(\wt{\mf{g}}/G \times_{\mf{g}/G} \wt{\cN}_Q/G)$ is acted on the right by $\cat{End}_{\IndCoh(\wh{\cZ}/G)}(\IndCoh(\wt{\mf{g}}/G \times_{\mf{g}/G} \wt{\cN}_Q/G))$.  By self-duality, these endomorphisms may be written as a tensor product, i.e. the actions are matched with our construction of the equivalences
\begin{align*}\Dmod(I_{\chQ} \bs \chG_F / I_{\chQ}) & \simeq \Dmod(I_{\chQ} \bs \chG_F / \mr{I}) \tens{\Dmod(\mr{I} \bs \chG_F / \mr{I})} \Dmod(\mr{I} \bs \chG_F / I_{\chQ}) \\
& \simeq \IndCoh(\wt{\cN}_Q/G \utimes{\mf{g}/G} \wt{\mf{g}}/G) \tens{\IndCoh(\wt{\mf{g}}/G \times_{\mf{g}/G} \wt{\mf{g}}/G)} \IndCoh(\wt{\mf{g}}/G \utimes{\mf{g}/G} \wt{\cN}_Q/G).
\end{align*}
The other cases may be similarly deduced.

For the statements on small categories, for $\Dmod_s \leftrightarrow \Coh_{\cN \cap \cN[1]}$ we simply pass to compact objects.  In the following, we let $K \subset G_F$ abusively denote one of the subgroups and monodromicity conditions appearing in the theorem statement, and $\cV/G$ the corresponding spectral stack over $\mf{g}/G$; they will not be relevant to the arguments.  To enlarge to $\Dmod_c \leftrightarrow \Coh_{\cN}$ we use the characterizations in Sections \ref{renormalize autom} and \ref{renormalize indcoh}.  We first treat the case where the right action is ``strict'', i.e. consider the $G$-equivariant correspondence
$$\wt\cN_P = G \times^P \mf{n}_P \overset{p}{\longleftarrow}   G\times^B \mf{n}_P \overset{i}{\longrightarrow} G \times^B \mf{b}$$
and the diagram
$$\begin{tikzcd}
\Dmod_c(K\bs \chG_F / I_{\chP}) \arrow[r, hook]  & \Dmod(K \bs \chG_F / I_{\chP}) \arrow[r, "\pi^*"] \arrow[d, "\simeq"', "\Phi_{P}"] & \Dmod(K \bs \chG_F/\mathring{I}\mon) \arrow[d, "\simeq"', "\Phi_{\mathring{B}}"] \\
\Coh(\cV/G \utimes{\mf{g}/G} \wt\cN_P/G) \arrow[r, hook] & \IndCoh_{\cN[1]}(\cV/G \utimes{\mf{g}/G} \wt\cN_P/G) \arrow[r, "i_*p^*"] & \IndCoh_{\cN \cap \cN[1]}(\cV/G \utimes{\mf{g}/G} \wt{\mf{g}}/G).
\end{tikzcd}$$
We will show that the full subcategories on the left consist precisely of the objects which are mapped to compact objects on the right.  On the automorphic side, we note that $\pi: K \bs \chG_F / \mr{I} \rightarrow K \bs \chG_F / I_{\chP}$ is a smooth covering and that $K \bs \chG_F / \mr{I}$ is the quotient of a placid ind-scheme $K \bs \chG_F$ by the pro-unipotent group $\mr{I}$, thus by Corollary \ref{coherent is compact} the pseudo-compact, i.e. coherent, objects are precisely the compact objects. On the spectral side, the map $p$ is smooth, and pullback along it reflects the property of pseudo-compactness, i.e. coherence.  Likewise, the map $i$ is a closed embedding, and pushforward along it reflects pseudo-compactness, which are compact by definition.  Thus the equivalence restricts to an equivalence
$$\Phi_{P}: \Dmod_c(K \bs \chG_F / I_{\chP}) \overset{\simeq}{\longrightarrow} \Coh(\cV/G \times_{\mf{g}/G} \wt{\cN}_P/G).$$
The notions of nearly compact constructible sheaves and continuous ind-coherent sheaves coincides with compactness in their respective categories when one side is strict, so this concludes the case where one side is strict.

We now consider the case where one side is monodromic, where the same argument shows that the equivalence $\Phi_{\mr{P}}$ restricts
$$\begin{tikzcd}
\Dmod_c(K \bs \chG_F / I'_{\chP}) \arrow[r, hook] \arrow[d, dotted, "\simeq"', "\Phi_{\mr{P}}"]  & \Dmod(K \bs \chG_F / I'_{\chP}) \arrow[r, "\pi"] \arrow[d, "\simeq"', "\Phi_{\mathring{P}}"] & \Dmod(K \bs \chG_F/\mathring{I}\mon) \arrow[d, "\simeq"', "\Phi_{\mathring{B}}"] \\
\Coh_{\cN}(\cV/G \utimes{\mf{g}/G} \wt\cN'_P/G) \arrow[r, hook] & \IndCoh_{\cN \cap \cN[1]}(\cV/G \utimes{\mf{g}/G} \wt\cN'_P/G) \arrow[r] & \IndCoh_{\cN \cap \cN[1]}(\cV/G \utimes{\mf{g}/G} \wt{\mf{g}}/G).
\end{tikzcd}$$
We now wish to enlarge to nearly compact and continuous ind-coherent sheaves, i.e. we let 
$$i: \wt{\cN}_P = G \times^P \mf{n}_P \hookrightarrow \wt{\cN}'_P  = G \times^P \mf{z}_P$$
denote the $G$-equivariant inclusion and consider
$$\begin{tikzcd}
\Dmod_{s, n.c.}(K \bs \chG_F / I'_{\chP}) \arrow[r, hook]  & \Dmod(K \bs \chG_F / I'_{\chP}) \arrow[r] \arrow[d, "\simeq"', "\Phi_{\mathring{P}}"] & \Dmod(K \bs \chG_F/I_{\chP}) \arrow[d, "\simeq"', "\Phi_{P}"] \\
\hCoh_{\cN[1]}(\cV/G \utimes{\mf{g}/G} \wt\cN'_P/G) \arrow[r, hook] & \IndCoh_{\cN \cap \cN[1]}(\cV/G \utimes{\mf{g}/G} \wt\cN'_P/G) \arrow[r, "i^!"] & \IndCoh_{\cN[1]}(\cV/G \utimes{\mf{g}/G} \wt{\cN}_P/G).
\end{tikzcd}$$
On the spectral side, we note that every nilpotent element of $\mf{z}_P$ belongs to $\mf{n}_P$, therefore $i$ is a quasi-smooth closed embedding geometrically surjecting onto the nilpotent locus.  Thus, the objects in the full subcategory may be characterized, using Theorem \ref{wcoh prop qs}, as those which are compact after application of $i^!$.  This matches the definition of near compactness on the automorphic side, thus $\Phi_{\mr{P}}$ restricts to an equivalence
$$\Phi_{\mr{P}}: \Dmod_{s,n.c.}(K \bs \chG_F / \mr{I}_{\chP}) \overset{\simeq}{\longrightarrow} \hCoh_{\cN[1]}(\cV/G \times_{\mf{g}/G} \wt{\cN}'_P/G).$$
This argument may be repeated without the singular support condition $\cN[1]$ to conclude that the equivalence also restricts to an equivalence
$$\Phi_{\mr{P}}: \Dmod_{c,n.c.}(K \bs \chG_F / \mr{I}_{\chP}) \overset{\simeq}{\longrightarrow} \hCoh(\cV/G \times_{\mf{g}/G} \wt{\cN}'_P/G).$$

Finally, we consider the case where one side is monodromic for a Whittaker character.  Here, the singular support is already nilpotent, so we only need to consider enlarging with respect to the classical support.  To this end, we consider the diagram
$$\begin{tikzcd}
{\Dmod}_{c, n.c.}(K\bs \chG_F / {}_{\chP}\chN, \psi, \St) \arrow[r, hook]  & \Dmod(K \bs \chG_F / {}_{\chP}\chN, \psi, \St) \arrow[r] \arrow[d, "\simeq"', "\Phi_{ \phi_P}"] & \Dmod(K \bs \chG_F/I) \arrow[d, "\simeq"', "\Phi_{B}"] \\
\hCoh_{\cN}(\cV/G \utimes{\mf{g}/G} \wt{\mf{g}}_P/G) \arrow[r, hook] & \IndCoh_{\cN}(\cV/G \utimes{\mf{g}/G} \wt{\mf{g}}_P/G) \arrow[r, "\mu^!"] & \IndCoh_{\cN[1]}(\cV/G \utimes{\mf{g}/G} \wt{\cN}/G)
\end{tikzcd}$$

We wish to show that the full subcategories on the left consist of precisely the objects which map to compact objects on the right.  In the bottom row, i.e. spectrally, we factor
$$\mu: \cV/G \utimes{\mf{g}/G} \wt{\cN}/G \overset{i}{\hookrightarrow} \cV/G \utimes{\mf{g}/G} (\wt{\mf{g}}_P \times_{G/P} G/B)/G \overset{p}{\hookrightarrow}  \cV/G \utimes{\mf{g}/G} \wt{\mf{g}}_P/G$$
where $i$ is a quasi-smooth closed embedding, and $p$ is smooth.  We will prove that the statement of Theorem \ref{wcoh prop qs} holds for the proper map $\mu$, arguing along similar lines.  First, we note that $p^!$ reflects and preserves $t$-boundedness since $p$ is smooth, and by Lemma \ref{reflect t boundedness} the same is true for $i^!$.  Now, suppose $Z/G \subset \cV/G \times_{\mf{g}/G} \wt{\mf{g}}_P/G$ is a closed substack nil-isomorphic to the nilpotent locus; we need to show that the $!$-restriction of a left $t$-bounded object of $\IndCoh(\cV/G \times_{\mf{g}/G} \wt{\mf{g}}_P/G)$ to $Z/G$ has coherent cohomology if and only if its $!$-pullback to $\cV/G \times_{\mf{g}/G} \wt{\cN}/G$ does.  Consider its derived fiber 
$$\begin{tikzcd}
Z'/G \arrow[r, hook, "z'"] \arrow[d, "p'"] &  \cV/G \times_{\mf{g}/G} (G \times^B \mf{p})/G \arrow[d, "p"] \\
Z/G \arrow[r, hook, "z"] & \cV/G \times_{\mf{g}/G} \wt{\mf{g}}_P/G.
\end{tikzcd}$$
It is straightforward to verify that $Z'/G$ is nil-isomorphic to $\cV/G \times_{\mf{g}/G} \wt\cN/G.$  Assuming that $\cF$ is left $t$-bounded, suppose that $\mu^! \cF$ has coherent cohomology.  By Lemma \ref{technical lemma coh com}, it follows that $z'^!p^!\cF \simeq p'^!z^!\cF$ has coherent cohomology.  Furthermore, pullback along smooth morphisms preserves and reflects coherence of cohomology, so we conclude that $z^!\cF$ has coherent cohomology as desired.

In particular, the category $\hCoh_{\cN}(\cV/G \utimes{\mf{g}/G} \wt{\mf{g}}_P/G)$ is characterized as the full subcategory of sheaves $\cF$ such that $\mu^! \cF$ is compact.  This matches the automorphic description, as shown in Corollary \ref{c:PoSt}, Proposition \ref{p:bIgMisTake}, and Proposition \ref{c:post3}, thus the equivalence restricts to an equivalence
$$\Phi_{\phi_P}: \Dmod_{c,n.c.}(K \bs \chG_F / {}_{\chP}\chN, \psi, \St) \overset{\simeq}{\longrightarrow} \hCoh(\cV/G \times_{\mf{g}/G} \wt{\mf{g}}_P/G).$$
This completes the proof.
\end{proof}

\appendix

\section{\texorpdfstring{$\infty$}{Infinity}-categorical considerations}\label{app infcat}

Our results use certain abstract $\infty$-categorical results, e.g. Lurie-Barr-Beck, in an essential way.  To identify the $\infty$-categories on the automorphic and spectral sides of Langlands, we will boil down the comparison to an identification of objects at the level of 1-categories, e.g. in Corollary \ref{completed spectral coalgebra} and Proposition \ref{autom algebra}.    In this appendix we establish our language for discussing $\infty$-categories, following the approach in \cite{HA, HTT}, and prove foundational results for deducing results on algebra objects in monoidal $\infty$-categories by arguing at the level of 1-categorical algebra objects.

\subsubsection{}
We model spaces by $\infty$-groupoids or Kan complexes, i.e. simplicial sets such that all horns (an $n$-simplex minus its interior and any codimension 1 face) can be filled.  We model $\infty$-categories via \emph{quasi-categories} or weak Kan complexes, i.e. simplicial sets such that all \emph{inner} horns can be filled.  We denote these categories, realized as full subcategories of the category of simplicial sets, by $\cat{Gpd}_\infty$ and $\cat{Cat}_\infty$ respectively.  The category of simplicial sets is canonically simplicially enriched, and we denote their simplicial enrichments by $\cat{Grd}_\infty^\Delta$ and $\cat{Cat}_\infty^\Delta$.  These categories may also be considered as $\infty$-categories (i.e. quasi-categories, with an underlying simplicial set) via the simplicial nerve construction discussed in Section \ref{homotopy coherent nerve} below, i.e. we have the quasi-categories $N_\Delta(\cat{Gpd}_\infty^\Delta)$ and $N_\Delta(\cat{Cat}_\infty^\Delta).$

\subsubsection{} 
A \emph{stable $\infty$-category} is an $\infty$-category that has a zero object, all fibers (cocones) and cofibers (cones), and such that fiber sequences coincide with cofiber sequences.  By Theorem 1.1.2.14 of \cite{HA} the homotopy category of a stable $\infty$-category has a canonical triangulated structure.  Conversely, if $\cat{C}$ is a pre-triangulated dg category then $N_{dg}(\cat{C})$ is stable by Theorem 4.3.1 of \cite{dg nerve}.  Thus, one can view stable $\infty$-categories as a generalization of triangulated categories; however, note that stability is a \emph{condition} on an $\infty$-category (which, in some sense, already has all the extra structure needed to identify exact triangles), while a triangulated category is an ordinary category equipped with extra \emph{structure} (i.e. a specification of which triangles are exact).  In stable $\infty$-categories exact triangles correspond to biCartesian squares.

\medskip

A $t$-structure on a stable $\infty$-category is no more than a $t$-structure on the (triangulated) homotopy category; see Section 1.2.1 of \cite{HA} for details.  Given a $t$-structure on a stable $\infty$-category $\cat{C}$, we define its \emph{heart} to be the full $\infty$-subcategory $\cat{C}^\heartsuit \subset \cat{C}$ consisting of objects which are in the heart of the homotopy category $h\cat{C}$, which is an abelian 1-category.  The canonical functor $\cat{C}^\heartsuit \rightarrow N(h\cat{C}^\heartsuit$ is an equivalence, and particular the heart, a priori an $\infty$-category, is discrete, i.e. for any $X, Y \in \cat{C}^\heartsuit$, we have $\pi_i(\Map(X, Y)) = 0$ unless $i = 0$.

\subsubsection{}
There are variants of the nerve construction that produce an $\infty$-category (i.e. quasi-category) from other kinds of categories.  The most basic example is the ordinary or 1-categorical nerve which produces a quasi-category from an ordinary 1-category, and is defined (see Definition 1.1.2.1 in \cite{HTT})
$$N: \cat{Cat} \rightarrow \cat{Cat}_\infty, \;\;\;\;\;\;\;\;\;\; N(\cat{C})([n]) := \Hom_{\cat{Cat}}([n], \cat{C})$$
where $[n] := \{0, 1, \ldots, n\}$ is an ordered set viewed as a category (with a unique morphism between $i, j$ if $i \leq j$), i.e. $N$ assigns to a 1-category $\cat{C}$ the simplicial set whose $n$-simplices consists of $n$-chains of morphisms in the category.  The nerve has a left adjoint $h: \cat{Cat}_\infty \rightarrow \cat{Cat}$ which is the \emph{homotopy category}; this is discussed in Section 1.2.3 of \emph{loc. cit.}, with an explicit characterization discussed around Proposition 1.2.3.5.

\subsubsection{}\label{homotopy coherent nerve}
We will also require functors for passing between quasi-categories and simplicially enriched categories, which may morally be viewed as two different models for a ``Platonic'' $\infty$-category.  There is a functor $\mf{C}: \cat{Cat}_\infty \rightarrow \cat{Cat}_\Delta$ defined in Definition 1.1.5.1 and above Example 1.1.5.9 of \cite{HTT}, whose right adjoint $N_\Delta: \cat{Cat}_\Delta \rightarrow \cat{Cat}_\infty$ is the simplicial (or homotopy coherent) nerve (see Definition 1.1.5.5 of \emph{op. cit.}):
$$N_\Delta(\cat{C})([n]) := \Hom_{\cat{Cat}_\Delta}(\mf{C}[n], \cat{C}).$$
We are particularly interested in the case where $\cat{C} = \cat{Cat}_\infty^\Delta$ is the category of $\infty$-categories (i.e. quasi-categories), which is simplicially enriched.  The simplicial nerve defines for us a quasi-category of quasi-categories (or $\infty$-category of $\infty$-categories) $\cat{Cat}_\infty := N_\Delta(\cat{Cat}_\infty^\Delta)$.  In particular, the homotopy category assignment $h: \cat{Cat}_\infty \rightarrow \cat{Cat}$ only depends on simplices in $\cat{Cat}_\infty$ of dimension $\leq 3$.  Given a simplicially enriched category $\cat{C}$, we have the following description of these simplices of small dimension in the homotopy coherent nerve $N_\Delta(\cat{C})$:
\begin{enumerate}
    \item 0-simplices are the objects of $\cat{C}$,
    \item 1-simplices are given by pairs of objects $x_0, x_1 \in \mathrm{Ob}(\cat{C})$ along with a 0-simplex in the mapping simplicial set $\Hom_{\cat{C}}(x_0, x_1)$,
    \item 2-simplices are given by a triple of objects $x_0, x_1, x_2$, along with 0-simplices in respective mapping simplicial sets $f_{10}: x_0 \rightarrow x_1$, $f_{21}: x_1 \rightarrow x_2$, $f_{20}: x_0 \rightarrow x_2$, and a 1-simplex in $\Hom_{\cat{C}}(x_0, x_2)$ from $f_{21} \circ f_{10}$ to $f_{20}$, and
    \item 3-simplices are given by a quadruple of objects $x_0, x_1, x_2, x_3$, 0-simplices in mapping simplicial sets $f_{ji}: x_i \rightarrow x_j$ (for $i < j$), and denoting $f_{kji} = f_{jk} \circ f_{ij}$ et cetera, and 1- and 2-simplices in $\Hom_{\cat{C}}(x_0, x_3)$:
$$\begin{tikzcd}
& f_{310} \arrow[dr] & \\
f_{30} \arrow[rr] \arrow[dr] \arrow[ur]  & & f_{3210}.  \\
& f_{320} \arrow[ur] &    
\end{tikzcd}$$
\end{enumerate}

\subsubsection{} 
We will also use the differential graded or dg nerve $N_{dg}: \cat{dgCat} \rightarrow \cat{Cat}_\infty^{st}$ (see Construction 1.3.1.6 in \cite{HA}), which assigns to every dg category a stable $\infty$-category (here, we view $\cat{dgCat}$ as an ordinary 1-category).  There is also a ``large'' version which assigns a simplicially enriched category to a dg category; these two constructions are compatible via the homotopy coherent nerve and dg nerve (see Proposition 1.3.1.17 of \cite{HA}).

\subsubsection{}\label{inert fact} One important $\infty$-category for us will be the nerve of the opposite simplex category $N(\Delta^\opp)$.  The simplex category $\Delta$ is the 1-category consisting of objects $[n] := \{0, 1, \ldots, n\}$ for $n \in \bZ^{\geq 0}$, with morphisms given by increasing maps.  We denote a morphism in the opposite category $f \in \Hom_{\Delta^\opp}([m], [n])$ by $f: [m] \leftarrow [n]$.  We say a morphism $f: [m] \leftarrow [n]$ in $\Delta^\opp$ is \emph{inert} if it is sequential, i.e. if $f(i+1) = f(i) + 1$, and \emph{active} if $f(0) = 0$ and $f(n) = m$.  For example, there are two inert maps $[2] \leftarrow [1]$, and one active map.  It is an easy fact that every morphism in $\Delta^\opp$ factors uniquely as the composition of an inert morphism followed by an active morphism (for further discussion, see Remark 2.1.2.2 and Proposition 2.1.2.4 \cite{HA}, and note that the notions of inert and active in $\Delta^\opp$ are inherited from the operad $\cat{Assoc}^\otimes$ by Construction 4.1.2.9 in \emph{op. cit.}).

\subsection{Monoidal \texorpdfstring{$\infty$}{infinity}-categories and \texorpdfstring{$t$}{t}-structures}

We begin by reviewing the notion of a monoidal $\infty$-category in both straightened and unstraightened incarnations; see Section 4.1.3 (and in particular Definition 4.1.3.6) in \cite{HA} for details.  Straightening and unstraightening functors are defined around Theorem 3.2.0.1 of \cite{HTT} and form Quillen equivalences.  We refer the reader to \emph{loc. cit.} for details and precise statements.

\subsubsection{}\label{section straightened}
A \emph{straightened associative (or $A_\infty$)-monoidal $\infty$-category} (see Definition 4.2.1.5 of \cite{HA}) is a map of simplicial sets $F_{\cat{A}}: N(\Delta^\opp) \rightarrow N_\Delta(\cat{Cat}_\infty^\Delta)$, satisfying the Segal condition that the induced maps
$$\prod_{i=1}^n F_{\cat{A}}(\rho^i): F_{\cat{A}}([n]) \longrightarrow \prod_{i=1}^n F_{\cat{A}}([1])$$
are equivalences for $i \geq 0$, where $\rho^i: [n] \leftarrow [1]$ are the inert maps sending $0 \mapsto i-1$ and $1 \mapsto i$.  The Segal condition for $n=0$ implies that $F_{\cat{A}}([0]) = \asterisk$, i.e. the terminal category.  We refer to $\cat{A} := F_{\cat{A}}([1])$ as the underlying $\infty$-category, and the Segal condition defines a canonical equivalence $F_{\cat{A}}([n]) \simeq \cat{A}^{\times n}$.  Note that by the usual adjunction (see Section \ref{homotopy coherent nerve}), a map of simplicial sets $N(\Delta^\opp) \rightarrow N_\Delta(\cat{Cat}^\Delta_\infty)$ is equivalent to a functor $\mf{C}[N(\Delta^\opp)] \rightarrow \cat{Cat}_\infty^\Delta$ of simplicially enriched categories.

\medskip

A module category for a straightened monoidal category $F_{\cat{A}}$ is a functor $F_{\cat{M}}: N(\Delta^{\opp}) \rightarrow N_\Delta(\cat{Cat}_\infty^\Delta)$  satisfying the relative Segal condition that the induced maps
$$\prod_{i=1}^n F_{\cat{M}}(\rho^i): F_{\cat{M}}([n]) \longrightarrow F_{\cat{M}}([1]) \utimes{F_{\cat{M}}([0])} \cdots \utimes{F_{\cat{M}}([0])} F_{\cat{M}}([1]).$$
are equivalences for all $i \geq 0$, and equipped with a natural map $F_{\cat{M}} \rightarrow F_{\cat{A}}$ such that the induced maps
$$F_{\cat{M}}([n]) \longrightarrow F_{\cat{A}}([n]) \utimes{F_{\cat{A}}([0])} F_{\cat{M}}([0])$$
are equivalences.  We refer to $\cat{M} := F_{\cat{M}}([0])$ as the underlying module category (note that we take the evaluation at $[0]$, rather than $[1]$ as above).  A map of module categories is map $F_{\cat{M}} \rightarrow F_{\cat{N}}$ over $F_{\cat{A}}$.

\subsubsection{} An \emph{unstraightened associative (or $A_\infty$)-monoidal $\infty$-category} (see Definition 4.1.3.6 of \cite{HA}) is a category $\cat{A}^\otimes$ equipped with a coCartesian fibration $p: \cat{A}^\otimes \rightarrow N(\Delta^{\opp})$ satisfying the Segal condition that the maps
$$\prod_{i=1}^n \rho_!^i: \cat{A}^\otimes_{[n]} \longrightarrow \prod_{i=1}^n \cat{A}^\otimes_{[1]}$$
are equivalences, where for a 0-simplex $\sigma$ the category $\cat{A}_{\sigma}^\otimes$ is fiber over a simplex $\sigma \in N(\Delta^{\opp})$, and for a 1-simplex $\alpha: [m] \leftarrow [n]$ of $N(\Delta^\opp)$ the functor $\alpha_!: \cat{A}^\otimes_{[m]} \rightarrow \cat{A}^\otimes_{[n]}$ is defined via a homotopy coherent choice of coCartesian lifts (see the discussion following Example 2.1.1.2 in \cite{HTT}).  The Segal condition for $n=0$ implies that $\cat{A}^\otimes_{[0]} \simeq \asterisk$.  We refer to $\cat{A} := \cat{A}^\otimes_{[1]}$ as the underlying $\infty$-category, and we have a \emph{non-canonical} equivalence $\cat{A}^\otimes_{[n]} \simeq \cat{A}^{\times n}$.  Furthermore, by the coCartesian lifting property we have
\begin{equation*}
\begin{split}
\Map_{\cat{A}^\otimes}((X_1, \ldots, X_n), (X_1, \ldots, X_m)) & \simeq \coprod_{\alpha: [n] \leftarrow [m]} \Map_{\cat{A}_{[m]}^\otimes}(\alpha_!(X_1,\ldots, X_n), (X_1, \ldots, X_m)) \\
& \simeq \coprod_{\alpha: [n] \leftarrow [m]} \prod_{i=1}^m \Map_{\cat{A}}(\alpha_!(X_1,\ldots, X_n)_i, X_i)
\end{split}
\end{equation*}
where $(-)_i$ picks out the $i$th component in the $m$-tuple.

\medskip

An unstraightened module category for $\cat{A}^\otimes$ is an $\infty$-category $\cat{M}^\otimes$ equipped with a coCartesian fibration $\cat{M}^\otimes \rightarrow \cat{A}^\otimes$ satsifying the relative Segal condition that the maps
$$\prod_{i=1}^n \rho_!^i: \cat{M}^\otimes_{[n]} \longrightarrow  \cat{M}^\otimes_{[1]} \utimes{\cat{M}^\otimes_{[0]}} \cdots \utimes{\cat{M}^\otimes_{[0]}} \cat{M}^\otimes_{[1]} $$
are equivalences.  A map of module categories is a map of categories $\cat{M}^\otimes \rightarrow \cat{N}^\otimes$ over $\cat{A}^\otimes$.  We refer to $\cat{M} := \cat{M}^\otimes_{[0]}$ as the underlying $\infty$-category.  Applying the coCartesian lifting property to the degeneracy maps from $\cat{A}^\otimes_{[0]} \simeq \asterisk \rightarrow \cat{A}^\otimes_{[n]}$, we have a non-canonical equivalence $\cat{M}^\otimes_{[n]} \simeq \cat{A}^{\times n} \times \cat{M}$.

\subsubsection{} We will generally use the superscript $(-)^\otimes$ to denote unstraightened categories, and denote straightened categories by a functor $F_{-}$.  The two are related by straightening and unstraightening (or relative nerve) functors; we refer the reader to Section 3.2.5 of \cite{HTT} for details.

\medskip

Before continuing, let us discuss informally how the above notions are meant to capture the idea of a monoidal category.  An object $[n] \in \Delta^\opp$ may be thought of as labelling a $n$-tuple of objects in a monoidal category $(X_1, \ldots, X_n)$; it is convenient to think of the labels of the $X_i$ are not correpsonding to elements of $[n]$ but (length 1) gaps in $[n] = \{0, 1, \ldots, n\}$.  A map $\alpha: [m] \leftarrow [n]$ in $\Delta^\opp$ may be thought of an assignment of a sequence of gaps in $[n]$ to each gap in $[m]$.  The functors $F_{\cat{A}}(\alpha)$ or $\alpha_!$ tensor the objects in each sequence (or inserts the monoidal unit if the gap is trivial, i.e. length 0), and discard any objects labeled by gaps ``at the ends'' outside the range of $\alpha$.  For example, the inert maps $\rho^i$ above precisely pick out one object in the tuple and throw out the rest, the active map $[n] \leftarrow [1]$ tensors all objects in the tuple together, and the active map $[0] \leftarrow [n]$ inserts the monoidal unit into every component.

\subsubsection{}\label{section monoidal functor} A \emph{(strong) monoidal functor} between two unstraightened associative monoidal $\infty$-categories is a map of simplicial sets $F: \cat{A}^\otimes \rightarrow \cat{B}^\otimes$ over $N(\Delta^\opp)$ that takes coCartesian simplices to coCartesian simplices.  A \emph{(right) lax monoidal functor} is only required to take coCartesian lifts of inert simplices to coCartesian simplices.  In particular, in the setting of unstraightened monoidal categories a monoidal functor is a functor with \emph{extra structure}, while in the setting of unstraightened monoidal categories it is defined by a \emph{property}.  

\medskip

We will not define the notion of monoidal functor in the straightened setting; the difficulty of giving an enumerative description of this extra structure (along with the difficulty of defining the notion of an algebra object) here is a significant motivation for working in the unstraightened setting.  However, we will require straightening to pass between familiar 1-categorical constructions, and cannot avoid its discussion.

\



\subsubsection{} We are often interested in stable monoidal $\infty$-categories equipped with a $t$-structure; in such a setting it is often far too strong to require that the monoidal structure restricts to the heart.  However, we can still consider a certain subcategory of ``objects in the monoidal heart'', i.e. objects in the heart whose iterated tensor products also live in the heart.  We now formally define this notion.

\begin{defn}\label{obj in heart}
A \emph{$t$-structure} on a monoidal stable $\infty$-category is a $t$-structure on the underlying (stable) category.  Note that $\cat{A}^\otimes_{[n]}$ inherits a $t$-structure from $\cat{A}$ via the Segal condition, i.e. such that $(\cat{A}^\otimes_{[n]})^{\heartsuit} \simeq (\cat{A}^\heartsuit)^n$, independently of any choices made.   We define the \emph{monoidal heart} $\cat{A}^{\otimes, \heartsuit} \subset \cat{A}^\otimes$ to be the full subcategory of objects $A \in \cat{A}^\otimes$ such that
\begin{enumerate}
    \item the underlying object lives in the heart, i.e. $A \in (\cat{A}^\otimes_{[p(A)]})^\heartsuit$, and
    \item for every 1-simplex $\alpha: [p(A)] \leftarrow [n]$ and every choice of coCartesian pushforward $\alpha_!: \cat{A}^\otimes_{[p(A)]} \rightarrow \cat{A}^\otimes_{[n]}$, we have $\alpha_!(A) \in (\cat{A}^\otimes_{[n]})^\heartsuit$, i.e. the target of the coCartesian lift of $\alpha$ with source $A$ is in the heart.
\end{enumerate}
The mapping spaces of $\cat{A}^{\otimes, \heartsuit}$ are 0-truncated by definition and by the description of mapping spaces in Section \ref{section straightened}).  We often denote $(\cat{A}^{\otimes, \heartsuit})_{[1]}$ by $\cat{A}^{\heartsuit_\otimes}$; this is the underlying category of the monoidal heart (which we often simply call the monoidal heart).
\end{defn}

\medskip

We establish a few basic first properties of the monoidal heart, including an easier and more intuitive criterion for checking when an object belongs to the monoidal heart.
\begin{prop}
Let $\cat{A}^\otimes$ be an unstraightened associative monoidal stable $\infty$-category  with a $t$-structure such that the monoidal unit is in the heart.\footnote{Note that without the assumption that $1_{\cat{A}}$ is in the heart, $\cat{A}^{\otimes, \heartsuit}$ is the empty category, e.g. by taking the coCartesian lift of any composition $[p(A)] \rightarrow [0] \rightarrow [1]$.  For this reason we will always impose this condition when discussing the monoidal heart.}  Then $\cat{A}^{\otimes, \heartsuit} \subset \cat{A}^\otimes$ is a monoidal subcategory; in particular, $\cat{A}^{\heartsuit_\otimes}$ is a (discrete) monoidal $\infty$-category.  Furthermore, an object $A \in \cat{A}$ is in the monoidal heart if and only if its underlying object is in the heart and its iterated products are in the heart.
\end{prop}
\begin{proof}
The first claim is immediate.  For the second claim, by the Segal condition, to verify condition (1) it suffices to check  for $[n] = [1]$.  For condition (2), we use the factorization discussed in Section \ref{inert fact}.  Since $s$ sends inert maps to coCartesian maps and by the Segal condition, it suffices to verify the condition for active maps, i.e. we require that the iterated products of $A$ are in $\cat{A}^\heartsuit$.
\end{proof}

\subsubsection{} The monoidal heart $\cat{A}^{\heartsuit_\otimes}$ is ``essentially'' a 1-category, but to obtain an actual 1-category we need to take the homotopy category.  We now show that taking the homotopy category recovers the usual notion of a monoidal 1-category.  We recall the notion of a monoidal 1-category as well as its equivalent ``unbiased'' version, which is closer to the $\infty$-categorical definition.

\begin{defn}\label{monoidal 1cat}
An \emph{associative monoidal 1-category} is a 1-category $\cat{A}$ equipped with functors
$$\tau: \cat{A} \times \cat{A} \longrightarrow \cat{A}, \;\;\;\;\;\;\;\;\; \upsilon: \asterisk \longrightarrow \cat{A}$$
as well as associator and left/right unitor natural isomorphisms
$$\alpha: \tau \circ (\mathrm{id}_{\cat{A}} \times \tau) \simeq \tau \circ (\tau \times \mathrm{id}_{\cat{A}}), \;\;\;\;\;\;\; \lambda: \tau \circ (\upsilon \times \mathrm{id}_{\cat{A}}) \simeq \mathrm{id}_{\cat{A}}, \;\;\;\;\;\;\; \rho: \tau \circ (\mathrm{id}_{\cat{A}} \times \upsilon) \simeq \mathrm{id}_{\cat{A}}$$
satisfying the usual pentagon and triangle relations.  A \emph{lax monoidal functor} between monoidal 1-categories $F: \cat{A} \rightarrow \cat{B}$ is a functor equipped with natural maps of functors
$$\mu: \tau_{\cat{B}} \circ F \rightarrow F \circ \tau_{\cat{A}}, \;\;\;\;\;\;\;\;\;\; \epsilon: \upsilon_{\cat{B}} \circ F \rightarrow F \circ \upsilon_{\cat{A}}$$
satisfying the usual compatibilities with the associator and unitor.  A \emph{strong monoidal functor} is a lax monoidal functor where the maps are equivalences.

By MacLane's coherence theorem for monoidal categories, this notion of a monoidal 1-category is equivalent to the notion of an \emph{unbiased associative monoidal 1-category} (see Proposition 1.5 of \cite{DM} for details),\footnote{The data we specify here is more than that in \emph{loc. cit.}, which restricts attention to only the active maps $[m] \leftarrow [1]$.  From the data in \emph{loc. cit.} one can recover the additional 1-simplices above by defining $\mu_f(A_1, \ldots, A_m) = (\mu(A_{f(1)+1}, \ldots, A_{f(2)}), \ldots, \mu(A_{f(n-1)+1}, \ldots, A_{f(n)}))$ where $\mu$ abusively denotes $\mu_{[k] \leftarrow [1]}$ for various $k$, and similarly for the corresponding natural isomorphisms.} i.e. a category $\cat{A}$ and the assignment of:
\begin{enumerate}
    \item for every 1-simplex $f: [m] \leftarrow [n]$ in $N(\Delta^\opp)$, a functor $\tau_f: \cat{A}^{\times m} \rightarrow \cat{A}^{\times n}$; 
    \item for every 2-simplex corresponding to the composition $g \circ f$ in $N(\Delta^\opp)$, a natural isomorphism $\alpha_{g,f}: \tau_g \circ \tau_f \simeq \tau_{g \circ f}$;
    \item For every 3-simplex corresponding to the composition $h \circ g \circ f$ in $N(\Delta^\opp)$, an equality $\alpha_{h, g \circ f} (\mathrm{id}_{\tau_h} \circ \alpha_{g,f}) = \alpha_{h \circ g, f} (\alpha_{h,g} \circ \mathrm{id}_{\tau_f})$ (here, we reserve $\circ$ for functor composition, and denote composition of natural transformations with no symbol).
\end{enumerate}

An \emph{unbiased lax monoidal functor} is a functor $F: \cat{A} \rightarrow \cat{B}$ between unbiased monoidal 1-categories is the assignment of:
\begin{enumerate}
\item for every 1-simplex $f: [m] \leftarrow [n]$ in $N(\Delta^\opp)$, a natural transformation $\mu_f: \tau_f \circ F^{m} \rightarrow F^{n} \circ \tau_f$,
\item for every 2-simplex corresponding to the composition $g \circ f: [m] \leftarrow [n]$ in $N(\Delta^\opp)$, an equality $\mu_{f \circ g} (\alpha_{f,g} \circ F^m) = (F^n \circ \alpha_{f,g})(\mu_f \circ \mathrm{id}_{\tau_g})(\mathrm{id}_{\tau_f} \circ \mu_g)$.
\end{enumerate}
\end{defn}

\medskip

The following claim is surely known to experts, though we could not find an explicit reference.
\begin{prop}\label{monoidal infinity is monoidal 1}
Let $\cat{A}$ be an associative monoidal $\infty$-category.  The homotopy category $h\cat{A}$ is a monoidal category in the usual 1-categorical sense.  A monoidal functor $F: \cat{A} \rightarrow \cat{B}$ induces a monoidal structure on the functor $hF: h\cat{A} \rightarrow h\cat{B}$ in the usual 1-categorical sense.  If $\cat{A}$ is an associative monoidal stable $\infty$-category, then the homotopy category $h\cat{A}$ is a monoidal triangulated category\footnote{This means that tensoring takes exact triangles to exact triangles.} in the usual 1-categorical sense, and monoidal exact functors induce monoidal triangulated functors.
\end{prop}
\begin{proof}
We apply unstraightening to obtain a map of simplicial sets $F: N(\Delta^\opp) \rightarrow \cat{Cat}_\infty$.  Since $n$-simplices in $\cat{Cat}_\infty$ for $n \geq 4$ have no bearing on the homotopy category, the description of simplices of dimension $\leq 3$ in $\cat{Cat}_\infty$ from Section \ref{homotopy coherent nerve} give a full description of the resulting structure on $h\cat{A}^{\heartsuit_\otimes}$.  It is immediate to see that these simplicies give rise to exactly the data in Definition \ref{monoidal 1cat}; note that we omit the natural transformation $f_{30} \rightarrow f_{3210}$ from the data, i.e. we do not specify a natural isomorphism $\alpha_{h,g,f}: \tau_h \circ \tau_g \circ \tau_f \simeq \tau_{h \circ g \circ f}$ since this isomorphism is determined by the requirement that $\alpha_{h, g \circ f} \circ \alpha_{g,f} = \alpha_{h \circ g \circ f} = \alpha_{h,g} \circ \alpha_{h \circ g, f}$.  The claim regarding stable categories is immediate from the definition.  The claims regarding monoidal functors follows by a similar diagram chase for the simplices enumerated in the above definition.
\end{proof}

\subsubsection{} As mentioned, it is often far too strict to require that the monoidal structure restricts to the heart.  However, it is sometimes possible to find a generating \emph{tilting subcategory} of the heart.  In this case, we are able to study the monoidal category via 1-categorical methods.  Our goal now is to formalize this discussion.

\medskip

First, we note that by Lemma 1.1.2.9 of \cite{HA}, the homotopy category $h\cat{C}$ of any stable $\infty$-category $\cat{C}$ is an additive category.  In particular, for $\cat{T} \subset \cat{C}$ any (possibly not stable) full subcategory which is closed under finite direct sums, the homotopy category $h\cat{T}$ is also an additive category.  This allows us to make the following definition.
\begin{defn}\label{def tilting subcat}
Let $\cat{C}$ be a compactly generated stable $\infty$-category and let $\cat{T} \subset \cat{C}^\heartsuit$ be a full subcategory of compact objects.  We say $\cat{T}$ is a \emph{tilting subcategory} if $\Ext^i(X, Y) = \pi_0(\Map(X, Y[i])) = 0$ for every $X, Y \in \cat{T}$ and $i \ne 0$, and if $\Ind(\cat{T}^{st}) \rightarrow \cat{C}$ is an equivalence of categories (i.e. the ind-completion of the stable closure of $\cat{T}$).
\end{defn}

\medskip

As usual, we wish to be able to build a category equivalent to $\cat{C}$ from a tilting subcategory $\cat{T}$, or rather its homotopy category $h\cat{T}$ since the mapping spaces of $\cat{T}$ are discrete.  Given an additive 1-category $\cat{T}'$, one may form a stable $\infty$-category as follows (see Sections 1.1.2, 1.2.3, and 1.3.2 of \cite{HA}).  First, one may form a pretriangulated dg category $\Ch_{dg}(\cat{T}')$ of chain complexes in $\cat{T}'$, and then one may apply the dg nerve to obtain a stable $\infty$-category $\Ch(\cat{T}') := N_{dg}(\Ch_{dg}(\cat{T}'))$ (see Remark 1.3.2.2 in \emph{op. cit.}).  The resulting stable $\infty$-category is equipped with the usual canonical $t$-structure.  This procedure is functorial, and sends additive functors to exact and $t$-exact functors.  Furthermore, the homotopy category $h\Ch(\cat{T}')$ is the usual category of chain complexes in $\cat{T}'$.  

\medskip

In our setting, we are interested in monoidal structures as well, and we discuss their interaction.  If $\cat{C}$ is a monoidal additive 1-category, then $\Ch_{dg}(\cat{C})$ is a \emph{strictly monoidal dg category}, i.e. a monoidal 1-category in the sense of Definition \ref{monoidal 1cat} such that all functors involved are dg functors, and the natural isomorphisms are cycles in degree 0 (and are isomorphisms, not homotopy equivalences).  The dg nerve $\Ch(\cat{C})$ of a strict monoidal dg category $\Ch_{dg}(\cat{C})$ is naturally a straightened monoidal $\infty$-category, which can then be unstraightened into a category $\Ch(\cat{C})^\otimes$.  In addition, a monoidal functor of additive 1-categories determines a canonical (strong) monoidal functor of straightened monoidal $\infty$-categories, which determines a (strong) monoidal functor after unstraightening.



\medskip

The following technical result produces the desired comparison functor between the categories $\Ch(h\cat{T})$ and $\cat{C}$.
\begin{prop}\label{prop additive nice}
Let $\cat{C}$ be a compactly generated stable monoidal $\infty$-category, and let $\cat{T} \subset \cat{C}$ be a full (but not necessarily stable) subcategory of compact objects with discrete mapping spaces and closed under finite direct sums.  Then, there is a functor $\Phi: \Ch(h\cat{T}) \rightarrow \cat{C}$ defined uniquely up to contractible homotopy by the property that the restriction of $h\Phi$ to $h\cat{T} \subset h\Ch(h\cat{T})$ is equal to the inclusion $h\cat{T} \hookrightarrow h\cat{C}$.  Furthermore, if $\cat{T}$ is a tilting subcategory (i.e. it generates and its $\Ext$ groups vanish), then $\Phi$ is an equivalence.

If, in addition, $\cat{C}$ is a monoidal and $\cat{T}$ is closed under the monoidal structure, then $\Phi$ has a monoidal structure uniquely defined up to contractible homotopy by the property that it recovers the canonical monoidal structure on the inclusion $h\cat{T} \hookrightarrow h\cat{C}$.
\end{prop}
\begin{proof}
By construction, $N(h\cat{T}) \subset \Ch(h\cat{T})$ is a full subcategory.  Because $\cat{T}$ has discrete mapping spaces, the canonical map $\cat{T} \rightarrow N(h\cat{T})$ is an equivalence, thus it has an inverse equivalence $\Phi_0: N(h\cat{T}) \rightarrow \cat{T}$.  We extend $\Phi_0$ to a functor $\Ind(N(h\cat{T})^{st})\rightarrow \cat{C}$ by extending the functor under finite limits and small colimits, which exist in $\cat{C}$ since $\cat{C}$ is stable and cocomplete.  Finally, the identification $\Ind(\cat{T}^{st}) \simeq \Ch(h\cat{T})$ gives an equivalence $\Phi: \Ch(h\cat{T}) \rightarrow \cat{C}$ as desired.


For the statement on monoidal categories, we apply unstraightening via the relative nerve (see Definition 3.2.5.2 and Proposition 3.2.5.18 in \cite{HTT}) where the coCartesian maps are the marked edges (see Definition 3.2.5.12 and Proposition 3.2.5.21 in \emph{op. cit.}) to obtain the category $\Ch(\cat{T})^\otimes$.  We wish to extend $\Phi$ to a functor $\Phi^\otimes: \Ch(h\cat{T})^\otimes \rightarrow \cat{C}^\otimes$ such that $\Phi^\otimes_{[1]} = \Phi$.  It suffices to produce a functor $\Phi_0^\otimes: N(h\cat{T})^\otimes \rightarrow \cat{T}^\otimes$ and extend by loop spaces (i.e. shifts in the positive direction) and small colimits, since colimits preserve the coCartesian property of morphisms, and since the suspension and loop space functors are inverse equivalences.

Since the category $N(h\cat{T})^\otimes$ is an unstraightening, for any we have \emph{given} equivalences $N(h\cat{T})^\otimes_{[n]} \simeq N(h\cat{T})^{\times n}$.  Thus, we can apply $\Phi_0$ to give the map on 0-simplices, i.e. we map objects in the homotopy category to lifts chosen in the first part of this proof.  For a 1-simplex $\alpha \in N(\Delta^\opp)$ we have a \emph{given} choice of compatible $\alpha_!$, i.e. a given choice of coCartesian lifts.  In particular, every 1-simplex in $N(h\cat{T})^\otimes$ has a given factorization as a coCartesian morphism followed by a morphism in a fiber $N(h\cat{T})^\otimes_{[n]} \simeq \Ch(h\cat{T})^{\times n}$.  Thus, a map of functors is determined precisely by a corresponding choice of coCartesian lifts in $\cat{T}^\otimes$, i.e. a choice of tensor product $\Phi_0(X) \otimes \Phi_0(Y)$, and for the functor to be monoidal is precisely to require that $\Phi_0(X) \otimes \Phi_0(Y) \simeq \Phi_0(X \otimes Y)$.  By the definition of an unbiased monoidal functor in Definition \ref{monoidal 1cat}, for any 2-simplex in $N(h\cat{T})^\otimes$ we may choose a corresponding 2-simplex in $\cat{T}^\otimes$ (i.e. roughly, asserting that elements in the same $\pi_0$ in the mapping space map to the same $\pi_0$).  We may define the map $\Phi_0^\otimes$ on $n$-simplices for $n \geq 3$ since the mapping spaces of $\cat{T}$ are discrete.  These choices may be made up to contractible homotopy, i.e. the subspace of $\Map(N(h\cat{T})^\otimes, \cat{T}^\otimes)$ satisfying the required properties is contractible, since in the mapping space we are able to make choices for $n$-simplices for $n \geq 2$ for free since the mapping spaces of $\cat{T}^\otimes$ are discrete, while the 0-simplices are completely determined by our requirements and the existence of 1-simplices (i.e. identifying our different choices for the tensor product $\Phi_0(X) \otimes \Phi_0(Y)$) follows since any two coCartesian lifts are homotopy coherently equivalent.
\end{proof}

\subsection{Algebra objects and \texorpdfstring{$t$}{t}-structures}

\subsubsection{}
Given an unstraightened monoidal $\infty$-category $p: \cat{A}^\otimes \rightarrow N(\Delta^\opp)$, an \emph{algebra object} $A$ in $\cat{A}$ is defined to be a section $A: N(\Delta^\opp) \rightarrow \cat{A}^\otimes$ of $p$ that takes \emph{inert} maps to coCartesian maps; in other words, an algebra object is a lax monoidal functor of operads $N(\Delta^\opp) \rightarrow \cat{A}^\otimes$ (see Definition 4.1.3.16 in \cite{HA}).  The object underlying the algebra is $A([1]) \in \cat{A}^\otimes_{[1]} \simeq \cat{A}$, which we sometimes abusively denote by simply $A$.

\medskip
Given $\cat{A}^\otimes$ as above, and algebra object $A$ as above, and a module category $\cat{M}^\otimes$, an $A$-module object in $\cat{M}$ is a section $M: N(\Delta^\opp) \rightarrow \cat{M}^\otimes$ that takes inert maps to coCartesian maps such that the composition of $N(\Delta^\opp) \rightarrow \cat{M}^\otimes \rightarrow \cat{A}^\otimes$ is the section defining $A$.

\subsubsection{} Let us spell out the above definition for the uninitiated.  Let $s: N(\Delta^\opp) \rightarrow \cat{A}^\otimes$ denote the section.  The condition on inert maps says there is a equivalence $s([n]) \simeq A^{\times n}$ up to contractible choice.  Let $\alpha: [2] \leftarrow [1]$ denote the active map; the monoidal product $A \otimes A$ is given by a (contractible) choice of $\alpha_!(A, A)$, i.e. the target of a choice of coCartesian lift of $\alpha$ with initial value $(A, A)$ over $[2]$.  The section defining the algebra object $A$ defines a non-coCartesian map $s(\alpha): (A, A) \rightarrow A$ in $\cat{A}^\otimes$, which by the universal factorization through the coCartesian lift defines a multipication map $\mu := \alpha_!(s(\alpha)): A \otimes A \rightarrow A$, i.e.
$$\begin{tikzcd}[column sep=80]
(A, A) \arrow[r, "{\text{coCart}}"] \arrow[dr, "s(\alpha)"'] &  \alpha_!(A, A) \simeq A \otimes A \arrow[d, dotted, "\mu := \alpha_!s(\alpha)"] \\
 & A.
\end{tikzcd}$$
Similarly,  the value of the section on the active map $\sigma: [0] \leftarrow [1]$ encodes the unit, i.e. we have that $\cat{A}^\otimes_{[0]} = \asterisk$ is the terminal category, and $\sigma_! \asterisk \simeq 1_{\cat{A}}$ is equivalent to the monoidal unit, while the section determines a unit map via universal factorization $\eta: \sigma_! \asterisk \simeq 1_{\cat{A}} \rightarrow A$.  

The other active maps encode (higher) associativity relations between these maps (see also Chapter 10, Section 3 of \cite{DAG}).  For example, the value of the section on the active map $\alpha_3: [3] \leftarrow [1]$ sending $(0,1) \mapsto (0,3)$ encodes a choice of tertiary product $\alpha_{3!}(A, A,A) \simeq  A \otimes A \otimes A$ along with a tertiary multiplication $\mu_3 := \alpha_{3!}s(\alpha_3): A \otimes A \otimes A \rightarrow A$.  On the other hand, one may take iterated products and multiplications: let $\alpha_{2,1}, \alpha_{2,2}: [3] \leftarrow [2]$ denote the active maps sending $(0,1,2) \mapsto (0,1,3), (0,2,3)$ respectively.  The composition $\alpha_! s(\alpha) \circ \alpha_{2,i!}s(\alpha_{2,i}): [3] \rightarrow [1]$ encodes the iterated products with the two groupings $(A \otimes A) \otimes A$ and $A \otimes (A \otimes A)$.  The value of $s$ on the two 2-simplices which witness the identities $\alpha \circ \alpha_{2,i} \sim \alpha_3$ in $N(\Delta^\opp)$ encode the  associativity structure, and similarly for higher associativities, et cetera.

\subsubsection{}
The category $\Alg(\cat{A})$ of algebra objects is defined in Definition 4.1.3.16 of \cite{HA} to be the full subcategory of $\Fun_{N(\Delta^\opp)}(N(\Delta^\opp), \cat{A}^\otimes)$ spanned by the above maps (see Definition 1.2.7.2 \cite{HTT} for notions of functor $\infty$-categories).  Note that the notation $\Fun_{N(\Delta^\opp)}$ does not indicate an overcategory in any sophisticated $\infty$-categorical sense, and only encodes the condition that $p \circ s = \mathrm{id}_{N(\Delta^\opp)}$.



\subsubsection{} The ultimate goal is to be able to say when two algebra objects in an $\infty$-category are equivalent using only 1-categorical arguments.  This will be possible when the algebra objects which live in the heart of a given $t$-structure in the following sense.
\begin{defn}\label{alg in heart}
Let $\cat{A}^\otimes$ be an unstraightened associative monoidal $\infty$-category whose underlying category $\cat{A}$ is stable with a $t$-structure, such that the monoidal unit is in the heart.  
We define the category of \emph{algebra objects in the heart} by $\Alg^\heartsuit(\cat{A}) = \Alg(\cat{A}^{\otimes, \heartsuit}).$
\end{defn}



We will now see that the category $\Alg^\heartsuit(\cat{A})$ is essentially 1-categorical, and can be fully described by 1-categorical data.  We first define algebra objects in a 1-category, as well as an ``unbiased'' version.
\begin{defn}\label{unbiased algebra}
Let $\cat{A}$ be an associative monoidal 1-category.  An \emph{algebra object} of $\cat{A}$ is an object $A \in \cat{A}$ equipped with morphisms
$$\mu: A \otimes A \rightarrow A, \;\;\;\;\;\;\;\;\; \eta: 1_{\cat{A}} \rightarrow A$$
satisfying the usual associativity and unitor relations.  Now let $\cat{A}$ be an unbiased monoidal 1-category; an \emph{unbiased algebra object} is an object $A \in \cat{A}$ equipped with the data:
\begin{enumerate}
    \item for every 1-simplex $f: [m] \leftarrow [n]$ in $N(\Delta^\opp)$, a morphism $\mu_f: A^{\otimes m} \rightarrow A^{\otimes n}$,
    \item for every 2-simplex corresponding to the composition $g \circ f$, an equality $\mu_g \circ \mu_f = \mu_{g \circ f}$.
\end{enumerate}
In both settings the morphisms are defined as usual, i.e. $f: A \rightarrow B$ is a morphism of $\cat{A}$ commuting with all decsribed structure maps.  Again by MacLane's coherence theorem these two notions are equivalent.  We also denote the category of algebra objects in a monoidal 1-category by $\Alg(\cat{A})$.
\end{defn}

\begin{defn}
Let $\cat{A}$ denote a stable monoidal $\infty$-category with $t$-structure, such that the monoidal unit is in the heart.  Recall that the homotopy category of the monoidal heart $h\cat{A}^{\heartsuit_\otimes}$ is a monoidal 1-category, and define $\Alg^\dagger(\cat{A}) = \Alg(h\cat{A}^{\heartsuit_\otimes})$ to be the 1-category of algebra objects.  There is a natural functor $h\Alg^\heartsuit(\cat{A}) \rightarrow \Alg^\dagger(\cat{A})$.
\end{defn}

\medskip

Finally, the following proposition allows us to lift any identification of algebra objects in the heart in the 1-categorical sense to an identification in the $\infty$-categorical sense, uniquely up to contractible homotopy.  We note that similar results were established in Section 4.1.6 of \cite{HA}.
\begin{prop}\label{algebra in heart prop}
The category $\Alg^\heartsuit(\cat{A})$ is equivalent to a 1-category.\footnote{Recall the notion of an $n$-category (Definition 2.3.4.1 of \cite{HTT}) and the notion of being equivalent to a 1-category (Proposition 2.3.4.18 of \emph{op. cit.}).}  Equivalently, the canonical functor $\Alg^\heartsuit(\cat{A}) \rightarrow N(h\Alg^\heartsuit(\cat{A}))$ is an equivalence, $\Alg^\heartsuit(\cat{A})$ is equivalent to the nerve of a 1-category, or the mapping spaces of $\Alg^\heartsuit(\cat{A})$ are discrete (or 0-truncated).  Furthermore, the natural functor $h\Alg^\heartsuit(\cat{A}) \rightarrow \Alg^\dagger(\cat{A})$ is an equivalence of 1-categories; thus, the functor $\Alg^\heartsuit(\cat{A}) \rightarrow N(\Alg^\dagger(\cat{A}))$ is an equivalence of $\infty$-categories.
\end{prop}
\begin{proof}
First, note that by Definition 2.1.2.7 of \cite{HA}, $\Alg(\cat{A})$ is a full subcategory of $\cat{Fun}(N(\Delta^\opp), \cat{A}^\otimes)$ (note that the overcategory in \emph{loc. cit.} is taken in the category of simplicial sets, not the category of $\infty$-categories).   By assumption, $A$ and $A'$ take values in $\cat{A}^{\otimes, \heartsuit}$, so it suffices to show that the mapping spaces of the category $\cat{Fun}(N(\Delta^\opp), \cat{A}^{\otimes, \heartsuit})$ are 0-truncated, which follows by Corollary 2.3.4.20 of \cite{HTT} since the mapping spaces of $\cat{A}^{\otimes, \heartsuit}$ are 0-truncated by the description of mapping spaces in \ref{section straightened}.

We now show that the map $h\Alg^\heartsuit(\cat{A}) \rightarrow \Alg^\dagger(\cat{A})$ is an equivalence.  For essential surjectivity we need to lift objects, that is we need to produce, for any $A \in \Alg^\dagger(\cat{A})$, an object $\wt{A} \in \Alg^\heartsuit(\cat{A})$, i.e. a section $\wt{A}: N(\Delta^\opp) \rightarrow \cat{A}^\otimes$, lifting it.   Recall that there is a non-canonical identification of the fibers $\cat{A}_{[n]}^\otimes \simeq (\cat{A}_{[1]}^\otimes)^{\times n}$; choose such identifications for all $n$.  We define $\wt{A}$ on 0-simplices by choosing a lift $\wt{A}([1])$, and then define $\wt{A}([n]) := \wt{A}([1])^{\times n}$ via these identifications.  Lifts for 1-simplices and 2-simplices in $N(\Delta^\opp)$ may be chosen by the data specified in the definition of unbiased algebra objects (see Definition \ref{unbiased algebra}); that is, lifting the 1-simplices amounts to choosing points in the mapping space $\Map_{\cat{A}^\otimes}(\wt{A}([i]), \wt{A}([j]))$, and lifting 2-simplices amounts to connecting two points.  For $n \geq 3$, lifting an $n$-simplex amounts to filling in an $(n-2)$-boundary in the mapping space, which may always be done since $\pi_{n-2}(\Map_{\cat{A}^\otimes}(\wt{A}([i]), \wt{A}([j]))) = 0$ by 0-truncatedness of mapping spaces.  


To see that the functor is full, we need to lift morphisms in the category $\Alg^\dagger(\cat{A})$, i.e. for any $f: A \rightarrow B$ in the 1-category $\Alg^\dagger(\cat{A})$ we need to provide a $\wt{f}: \wt{A} \rightarrow \wt{B}$ in the $\infty$-category $\Alg(\cat{A})$ for our choice of lifts $\wt{A}, \wt{B}$.  Our model for the mapping space $\Map_{\Alg(\cat{A})}(\wt{A}, \wt{B})$ will be the left mapping space of Section 1.2.2 in \cite{HTT}.  That is, a map $\wt{f}$ is an element of $\Map(N(\Delta^\opp) \times \Delta^1, \cat{A}^\otimes)$, i.e. a map of simplicial sets, which restricts along the inclusion of the 0-simplices in $\Delta^1$ to $\wt{A}, \wt{B}$.  This latter condition determines the value of $\wt{f}$ on 0-simplices in $N(\Delta^\opp) \times \Delta^1$ entirely.  For 1-simplices, denote the unique non-degenerate 1-simplex in $\Delta^1$ by $\eta$, and the detenerate 1-simplices by $\sigma_0, \sigma_1$; we lift the 1-simplices $(-, \sigma_i)$ to lifts of the structure maps for $\wt{A}, \wt{B}$ as above, and we lift the 1-simplices $(u, \alpha)$ as follows.  If $u = \mathrm{id}_{[0]}$, then we lift it to the identity map on the monoidal unit.  If $u = \mathrm{id}_{[1]}$, we choose a lift $\wt{f}: \wt{A} \rightarrow \wt{B}$ in $\cat{A}_{[1]}^\otimes$, and for $u = \mathrm{id}_{[n]}$ we then lift to $\wt{f}^{\times n}$.  Otherwise, $u$ is lifted to a choice of composition of the structure maps and the maps $\wt{f}^{\times n}$, e.g. for the active map $\alpha: [2] \leftarrow [1]$, the 1-simplex $(\alpha, \eta)$ is sent to a choice for the diagonal composition (note the abusive notation; all objects and maps are labeled after being pushed forward along $\alpha$ to $\cat{A}_{[1]}^\otimes$):
$$\begin{tikzcd}[column sep = large]
A \otimes A \arrow[r, "{(\alpha, \sigma_0)}"] \arrow[d, "{(\mathrm{id}_{[2]}, \alpha)}"'] \arrow[dr, dotted, "{(\alpha, \eta)}"] & A \arrow[d, "{(\mathrm{id}_{[1]}, \alpha)}"] \\
B \otimes B \arrow[r, "{(\alpha, \sigma_1)}"'] & B.
\end{tikzcd}$$
The 2-cells may be filled in because of the equalities that arise from a morphism of algebra objects in the 1-categorical sense.  For $n \geq 3$, the $n$-simplices for $n \geq 3$ may be chosen inductively as above, since $\cat{A}^{\otimes, \heartsuit}$ is 0-truncated.  

To see that the functor is faithful, we need to lift homotopies between morphisms, i.e. for any lifts $\wt{f}, \wt{f}' \in \Map(N(\Delta^\opp) \times \Delta^1, \cat{A}^\otimes)$ as above we need to provide a 1-simplex in the left mapping space (see 1.2.2 in \cite{HTT}) identifying them, i.e. a map of simplicial sets $\Map(N(\Delta^\opp) \times \Delta^2, \cat{A}^\otimes)$ such that the restriction to the faces are given by:
$$\begin{tikzcd}
& \wt{B} \arrow[dr, "{\mathrm{id}_{\wt{B}}}"] & \\
\wt{A} \ar[rr, "\wt{f}'"'] \arrow[ur, "\wt{f}"] & & \wt{B}.
\end{tikzcd}$$
The 0-simplices and 1-simplices of such a map are determined by the above property, and the 2-simplices may be chosen since $\wt{f}, \wt{f}'$ define the same map in the homotopy category (i.e. all diagrams arising from the above commute).  For $n \geq 3$, this follows by the same 0-truncatedness argument as above.
\end{proof}

\begin{rmk}
It is not difficult to adapt everything we have written about monoidal categories to module categories, and algebra objects to module objects.
\end{rmk}

\endgroup


\begin{thebibliography}{99}

\bibitem[AB09]{AB}
    Sergey Arkhipov and Roman Bezrukavnikov,
    Perverse sheaves on affine flags and Langlands dual group,
    Israel Journal of Mathematics volume 170, Article number: 135 (2009).


\bibitem[ABG04]{ABG}
 Sergey Arkhipov, Roman Bezrukavnikov, and Victor Ginzburg,
 J. Amer. Math. Soc. 17(3) (2004), 595-678. 


\bibitem[Ar18]{Arinkinnotes}
     Dima Arinkin,
     Day III, Talk 3: Spectral side in the classical case (GL-6),
     \url{https://lysenko.perso.math.cnrs.fr/Notes_talks_winter2018/GL-6(Dima).pdf}


\bibitem[AG15]{AG}
	Dmitri Arinkin and Dennis Gaitsgory,
	Singular support of coherent sheaves and the geometric Langlands conjecture,
	Selecta Math. (N.S.) 21 (2015), no. 1, 1-199.


\bibitem[AGKRRV20]{AGKRRV}
Dima Arinkin, Dennis Gaitsgory, David Kazhdan, Sam Raskin, Nick Rozenblyum, and Yakov Varshavsky, The stack of local systems with restricted variation and geometric Langlands theory with nilpotent singular support, 
arXiv:2010.01906 (2020).  



\bibitem[BG99]{BeiGinz}
    Alexander Beilinson and Victor Ginzburg,
    Wall crossing functors and D-modules, 
    Representation Theory 3.1 (1999), 1-31.  


\bibitem[BCHN21]{BCHN}
	David Ben-Zvi, Harrison Chen, David Helm and David Nadler,
	Coherent Springer theory and the categorical Deligne-Langlands correspondence,
	arXiv:2010.02321v2 (2021).


\bibitem[BCHN23]{ihes} David Ben-Zvi, Harrison Chen, David Helm and David Nadler,  Between coherent and constructible local Langlands correspondences. To appear in Proc. Symp. Pure Math. (2023).



\bibitem[BFN10]{BFN}
    David Ben-Zvi, John Francis and David Nadler,
    Integral transforms and Drinfeld centers in derived algebraic geometry,
    J. Amer. Math. Soc. 23 (2010), no. 4, 909-966.



\bibitem[Be16]{roma hecke}
	Roman Bezrukavnikov,
	On two geometric realizations of an affine Hecke algebra,
	Publ. Math. Inst. Hautes Études Sci. 123 (2016), 1–67.



\bibitem[Be09]{BezNilCone}
 Roman Bezrukavnikov, 
 Perverse sheaves on affine flags and nilpotent cone of the Langlands dual group,
 Israel J. Math. 170 (2009), 185-206. 

\bibitem[BF08]{BF}
    Roman Bezrukavnikov and Michael Finkelberg,
    Equivariant Satake Category and Kostant–Whittaker Reduction,
    Moscow Math. J. 8 (2008), no. 1, 39-72.


    
\bibitem[BL20]{BL}
	Roman Bezrukavnikov and Ivan Losev,
	Dimensions of modular irreducible representations of semisimple Lie algebras,
	J. Amer. Math. Soc.36 (2023), no. 4, 1235-1304. 



\bibitem[BeRi21]{BezRiche}
Roman Bezrukavnikov and Simon Riche, A topological approach to Soergel theory, Representation Theory and Algebraic Geometry: A Conference Celebrating the Birthdays of Sasha Beilinson and Victor Ginzburg (2021), 267-343. 

\bibitem[BRR20]{BezRicheRider}
Roman Bezrukavnikov, Simon Riche, and Laura Rider,
Modular affine Hecke category and regular unipotent centralizer, I, arXiv:2005.05583 (2020).




\bibitem[BN15]{bn ctcg} David Ben-Zvi and David Nadler, The Character Theory of a Complex Group, arXiv:0904.1247v3 (2015).

\bibitem[BNP17a]{BNP convolution}
	David Ben-Zvi, David Nadler and Anatoly Preygel,
	Integral transforms for coherent sheaves,
	J. Eur. Math. Soc. (JEMS) 19 (2017), no. 12, 3763-3812. 




\bibitem[BNP17b]{BNP affine}
	David Ben-Zvi, David Nadler and Anatoly Preygel,
	A spectral incarnation of affine character sheaves,
	Compos. Math. 153 (2017), no. 9, 1908-1944.

\bibitem[BY13]{BezYun}
    Roman Bezrukavnikov and Zhiwei Yun,
    On Koszul duality for Kac-Moody groups, 
    Represent. Theory 17 (2013), 1-98.


\bibitem[CD21]{campbelldhillon}
Justin Campbell and Gurbir Dhillon, 
Affine Harish--Chandra bimodules and Steinberg--Whittaker localization,
arXiv:2108.02806v1 (2021). 


\bibitem[CG97]{CG}
    Neil Chriss and Victor Ginzburg. Representation theory and complex geometry. Reprint of the 1997 edition. Modern Birkhäuser Classics. Birkhäuser Boston, Ltd., Boston, MA, 2010. 






\bibitem[DM82]{DM}
    Pierre Deligne and James S. Milne. \emph{Tannakian Categories}. In: Hodge Cycles, Motives, and Shimura Varieties. Lecture Notes in Mathematics, vol 900. Springer, Berlin, Heidelberg (1982).

\bibitem[DG13]{QCA}
	Vladimir Drinfeld, and Dennis Gaitsgory. On some finiteness questions for algebraic stacks. Geom. Funct. Anal. 23 (2013), no. 1, 149-294. 

\bibitem[DGI06]{dgi} William Dwyer, John Greenlees, and Srikanth Iyengar. Duality in algebra and topology. Advances in Math. 200 (2006), no. 2, 357-402.   



\bibitem[Fa17]{dg nerve}
    Giovanni Faonte,
    Simplicial nerve of an $A_\infty$ category,
    Theory and Appl. of Categories, Vol. 32, 2017, No. 2, pp 31-52 (2017).

    
\bibitem[Ga13]{indcoh}
    Dennis Gaitsgory,
    Ind-coherent sheaves,
    Mosc. Math. J. 13 (2013), no. 3, 399-528, 553.

\bibitem[Ga15]{1affine}
    Dennis Gaitsgory,
    Sheaves of categories and the notion of 1-affineness,
    Stacks and categories in geometry, topology, and algebra, 127-225,
    Contemp. Math., 643, Amer. Math. Soc., Providence, RI, 2015. 


\bibitem[Ga18]{GaIC}
Dennis Gaitsgory,
The semi-infinite intersection cohomology sheaf,
Advances in Math. 327 (2018), 789-868. 



\bibitem[Gr65]{EGAIV.2}
    Alexander Grothendieck. \'{E}l\'{e}ments de g\'{e}om\'{e}trie alg\'{e}brique. IV. \'{E}tude locale des sch\'{e}mas et des morphismes de sch\'{e}mas. II. (French) Inst. Hautes Études Sci. Publ. Math. No. 24 (1965).

\bibitem[GR14]{dgind}
    Dennis Gaitsgory and Nick Rozenblyum,
    DG indschemes,
    Perspectives in representation theory, 139-251,
Contemp. Math., 610, Amer. Math. Soc., Providence, RI, 2014. 
    
\bibitem[GR17]{DAG}
    Dennis Gaitsgory and Nick Rozenblyum,
    \emph{A study in derived algebraic geometry},
    Mathematical Surveys and Monographs, 221 (2017), American Mathematical Society.

\bibitem[HP19]{HLP}
    Daniel Halpern-Leistner and Anatoly Pregel,
    Mapping stacks and categorical notions of properness,
    Compos. Math. 159, No. 3, 530-589 (2023). 

\bibitem[HL22]{ho li}
    Quoc P. Ho and Penghui Li,
    Revisiting mixed geometry,
    arXiv:2202.04833v2 (2022).

\bibitem[Lu09]{HTT}
    Jacob Lurie,
    \emph{Higher Topos Theory},
    Annals of Mathematics Studies, 170, Princeton University Press,
    2009. 

\bibitem[Lu17]{HA}
    Jacob Lurie,
    \emph{Higher Algebra},
    \url{http://people.math.harvard.edu/~lurie/papers/HA.pdf},
    Sept. 18, 2017.
    
    
\bibitem[MR10]{MR} Ivan Mirkovi\'{c} and Simon Riche, Linear Koszul duality. Compos. Math. 146 (2010), No. 1, 233-258.

    
\bibitem[Pr11]{preygel MF}
	Anatoly Preygel,
	Thom-Sebastiani \& Duality for Matrix Factorizations,
	arXiv:1101.5834v1,
	2011.

\bibitem[Ra14]{raskin dg lie}
    Sam Raskin,
    Coherent sheaves on formal complete intersections via DG Lie algebras,
    Math. Res. Lett. 21, no. 1, 207-223,
    2014.

\bibitem[Ra15]{raskindmod}
 Sam Raskin,
 D-modules on infinite dimensional varieties,
 available at \url{https://web.ma.utexas.edu/users/sraskin/dmod.pdf}, 2015.



\bibitem[SP21]{stacks-project}
    The Stacks project authors,
    The Stacks project,
    \url{https://stacks.math.columbia.edu},
    2021.
		
\end{thebibliography}
\end{document}